\pdfoutput=1
\documentclass[10pt,a4paper]{amsart}

\usepackage{mathrsfs}

\usepackage{stmaryrd}
\usepackage{tikz}
\usetikzlibrary{calc,shapes, decorations.pathmorphing, tikzmark, arrows.meta}
\usepackage{float}
\usepackage{adjustbox}

\usepackage{hyperref}
\usepackage{amsmath,amssymb,amsfonts}
\usepackage{bbm}

\usepackage{mathtools}

\def\equationautorefname~#1\null{Equation~(#1)\null}

\usepackage[marginpar=2cm,ignoremp,margin=3cm]{geometry}

\usepackage{mathdots}
\usepackage{etoolbox}
\usepackage{shuffle}
\usepackage{thmtools}

\makeatletter
\def\@endtheorem{\endtrivlist}
\makeatother

\declaretheorem[
style=plain,
name=Theorem,
numbered=yes,
numberwithin=section,
refname={Theorem,Theorems},
Refname={Theorem,Theorems}
]{Thm}

\declaretheorem[
style=plain,
name=Proposition,
numberlike=Thm,
refname={Proposition,Propositions},
Refname={Proposition,Propositions}
]{Prop}
\declaretheorem[
style=definition,
name=Remark,
numberlike=Thm,
refname={Remark,Remarks},
Refname={Remark,Remarks}
]{Rem}
\declaretheorem[
style=plain,
name=Conjecture,
numberlike=Thm,
refname={Conjecture,Conjectures},
Refname={Conjecture,Conjectures}
]{Conj}
\declaretheorem[
style=definition,
name=Definition,
numberlike=Thm,
refname={Definition,Definitions},
Refname={Definition,Definitions}
]{Def}
\declaretheorem[
style=definition,
name=Example,
numberlike=Thm,
refname={Example,Examples},
Refname={Example,Examples}
]{Eg}
\declaretheorem[
style=plain,
name=Corollary,
numberlike=Thm,
refname={Corollary,Corollaries},
Refname={Corollary,Corollaries}
]{Cor}
\declaretheorem[
style=plain,
name=Lemma,
numberlike=Thm,
refname={Lemma,Lemmas},
Refname={Lemma,Lemmas }
]{Lem}

\DeclareMathOperator{\Li}{Li}
\newcommand{\LiL}{\Li^{\mathcal{L}}}
\newcommand{\IL}{I^{\mathcal{L}}}
\newcommand{\lmot}{\mathcal{L}}
\newcommand{\logL}{\log^{\mathcal{L}}}

\newcommand{\Sym}{\mathfrak{S}}

\let\phi\varphi

\DeclareMathOperator{\id}{id}
\DeclareMathOperator{\gr}{gr}
\DeclareMathOperator{\CoLie}{CoLie}

\newcommand{\sym}{\mathrm{sym}}
\newcommand{\Mod}[1]{\quad(\text{\rm mod #1})}

\newcommand{\modone}{\quad(\text{\rm mod dp 1})}
\newcommand{\modonei}{\mkern4mu(\text{\rm mod dp 1})}

\newcommand{\modtwo}{\quad(\text{\rm mod dp 2})}
\newcommand{\modtwoi}{\mkern4mu(\text{\rm mod dp 2})}

\newcommand{\moddp}[1]{\quad(\text{\rm mod dp #1})}
\newcommand{\moddpi}[1]{\mkern4mu(\text{\rm mod dp #1})}

\newcommand{\eps}{\varepsilon}

\usepackage{enumitem}

\allowdisplaybreaks

\def\;{\,{\boldsymbol{;}}\,}
\def\={\mathrel{\,{=}\,}}

\def\+{\!+\!}
\def\-{\!-\!}

\newcommand{\lif}{\LiL_{3\;1,1,1}}
\newcommand{\liftwo}{\LiL_{2\;1,1}}
\newcommand{\crv}{\throwerror}
\newcommand{\QU}{\widetilde{\mathbf{Q}}}
\newcommand{\QUf}{\mathbf{Q}}
\newcommand{\three}{\mathcal{T}}
\newcommand{\four}{\mathcal{F}}

\DeclareMathOperator{\cor}{Cor}
\newcommand{\corL}{\cor^\lmot}

\newcommand{\abs}[1]{\left| #1 \right|}

\DeclareMathOperator{\sgn}{sgn}
\DeclareMathOperator{\sign}{sign}
\DeclareMathOperator{\QLi}{QLi}

\DeclareMathOperator{\spec}{Sp}

\newcommand{\revsc}{\textsc{Rev}}
\newcommand{\invsc}{\textsc{Inv}}
\newcommand{\revinvsc}{\textsc{RevInv}}
\newcommand{\redid}[1]{\textsc{Red}_{#1}}
\newcommand{\degsym}[1]{\textsc{DegSym}_{#1}}
\newcommand{\fullsym}[1]{\textsc{FullSym}_{#1}}
\newcommand{\shsym}[1]{\textsc{Sh}_{#1}}

\usepackage[normalem]{ulem}

\makeatletter
\newcommand{\case}[1]{
  \par
\addvspace{\medskipamount}
\noindent%
{\ttfamily#1\,\@addpunct{\normalfont\bf:}}\enspace\ignorespaces%
}

\newcommand{\note}[1]{
\noindent {\ttfamily#1\,\@addpunct{\normalfont\bf:}}\enspace\ignorespaces%
}
\makeatother

\makeatletter
\def\paragraph{\@startsection{paragraph}{4}%
	{0em}{0.5ex \@plus 0.5ex \@minus 0.25ex}{-\fontdimen2\font}%
	{%
	\normalfont\itshape%
}}
\makeatother

\newcommand{\biggerskip}{
	\setlength{\abovedisplayskip}{9pt}%
	\setlength{\belowdisplayskip}{9pt}%
	\setlength{\abovedisplayshortskip}{6pt}%
	\setlength{\belowdisplayshortskip}{6pt}%
}

\makeatletter
\@namedef{subjclassname@2020}{%
	\textup{2020} Mathematics Subject Classification}
\makeatother

\newcommand{\wtfourfilename}{\texttt{weight4\_symmetries\_v1.nb}}
\newcommand{\wtsixfilename}{\texttt{weight6\_symmetries\_v1.nb}}
\newcommand{\wtsixidentities}{\texttt{weight6\_results\_Li3s111deg.txt}}
\newcommand{\wtsixdepthtwo}{\texttt{weight6\_results\_depth2.txt}}
\newcommand{\wtsixidentitieslist}{\texttt{weight6\_resultsList\_Li3s111deg.txt}}
\newcommand{\wtsixdepthtwolist}{\texttt{weight6\_resultsList\_depth2.txt}}

\begin{document}
	
	\title[Symmetries of weight 6 MPL's]{Symmetries of weight 6 multiple polylogarithms \\ and Goncharov's Depth Conjecture}

	\author{Steven Charlton}
	\address{Max Planck Institute for Mathematics, Vivatsgasse 7, Bonn 53111, Germany}
	\email{charlton@mpim-bonn.mpg.de}

	\begin{abstract}
		We prove that the weight 6, depth 3, multiple polylogarithm \( \Li_{4,1,1}((xyz)^{-1}, x, y) \), or rather its more natural `divergent' incarnation \(\Li_{3\;1,1,1}(x,y,z) \),  satisfies the 6-fold anharmonic symmetries of the dilogarithm \( \Li_2 \), \( \lambda \mapsto 1-\lambda \) and \( \lambda \mapsto \lambda^{-1} \), in each of \( x, y \) and \( z \) independently, modulo terms of depth \( {\leq}2 \).  This establishes the `higher Zagier' part of the weight 6, depth 3, reduction conjectured by Matveiakin and Rudenko \cite{MR22}.  Together with their proof of the `higher Gangl' part of the weight 6, depth 3, reduction (which is formulated \emph{modulo} the `higher Zagier' part), we establish Goncharov's Depth Conjecture in the case of weight 6, depth 3.
	\end{abstract}

	\date{May 22, 2024}
	

	\maketitle
	
	\vspace{-1em}
	
	\setcounter{tocdepth}{1}
	\tableofcontents
	\setcounter{tocdepth}{2}
	
	\vspace{-1em}

	\section{Introduction}\label{sec:intro}
	
	Multiple polylogarithms (MPL's) are multivalued analytic functions of variables \( x_1,\ldots,x_k \in \mathbb{C} \), depending on some given positive integer parameters \( n_1,\ldots,n_k \in \mathbb{Z}_{>0} \), called the \emph{indices}.  In the region \( \abs{x_j} < 1 \), they are defined by the following power series
	\begin{equation}\label{eqn:lidef}
		\Li_{n_1,\ldots,n_k}(x_1,\ldots,x_k) \coloneqq \sum_{0<i_1 < \cdots < i_k} \frac{x_1^{i_1} \cdots x_k^{i_k}}{i_1^{n_1} \cdots i_k^{n_k}} \,, \quad \abs{x_j} < 1 \,.
	\end{equation}
	The number of indices \( k \) is called the \emph{depth} of the function, and the sum \( n = n_1 + \cdots + n_k \) of the indices is called the \emph{weight} of the multiple polylogarithm. \medskip
	
	One of the key challenges in the study of multiple polylogarithms is understanding the nature and properties of the filtration by depth, and in particular determining what the true depth of a given function is (or more generally special combinations of functions).  It is well-known that every function of weight 2 and weight 3 can be expressed via depth 1, for example \cite[Proposition 1]{zagierDilog} (c.f. also \cite[Eqn. (71)]{GoncharovMSRI})
	\[
		\Li_{1,1}(x,y)	= \Li_2\Big(\frac{y(x-1)}{1-y}\Big) - \Li_2\Big(\frac{y}{y-1}\Big) - \Li_2\big(x y\big) \,, \quad \abs{x y} < 1, \abs{y} < 1
	\]
	as a power-series identity.  Similar formulae also exist expressing \( \Li_{1,2} \), \( \Li_{2,1} \), and \( \Li_{1,1,1} \) via the trilogarithm \( \Li_3 \): these can be traced back to Lewin \cite[\S8.4.3]{Lewin}, but reappeared in a more conceptual context in work of Goncharov \cite[Eqn. (72)--(73)]{GoncharovMSRI}, later \cite[\S6.8]{Gon01}, and Zhao \cite[Remark 3.18 \& Corollary 3.25]{ZhaoMotivic3}.  However, it is known that \( \Li_{3,1} \), and indeed the generic weight 4 multiple polylogarithm, is not expressible via \( \Li_4 \).  In particular, Theorem D of \cite{WojtkowiakStructure} establishes there is no polynomial relation between the generic weight 4 MPL, and classical polylogarithms \( \Li_4 \) of rational function arguments.\medskip
	
	A complete answer to the question of determining the true depth multiple polylogarithms is given via an ambitious conjecture posed by Goncharov, which gives  necessary and sufficient conditions for a combination of multiple polylogarithms to have depth \( {\leq}k \).  This criterion is formulated in terms of the truncated coproduct \( \overline{\Delta} \) on the Lie coalgebra of motivic multiple polylogarithms, as we will describe just below in \autoref{conj:depthsimp}.
	
	Recalling the Chen iterated path integrals \cite{chenIterated77} 
	\[
	I_{(\gamma)}(x_0; x_1, \ldots, x_N; x_{N+1}) \coloneqq \int_{x_0 < t_1 < \dots < t_N < x_{N+1}} \frac{\mathrm{d}t_1}{t_1 - x_1} \wedge \cdots \wedge \frac{\mathrm{d}t_N}{t_N - x_N} \,,
	\]
	(which implicitly depend on some choice of path \( \gamma \) from \( x_0 \) to \( x_{N+1} \)) one can write the multiple polylogarithms via an iterated integral as follows
	\begin{equation}\label{eqn:liasii}
	\Li_{n_1,\ldots,n_k}(x_1,\ldots,x_k) = 
	(-1)^k I\bigl(0; \tfrac{1}{x_1\cdots x_k}, \underbrace{0, \ldots, 0}_{n_1 - 1}, \tfrac{1}{x_2\cdots x_k}, \underbrace{0, \ldots, 0}_{n_2-1}, \ldots, \tfrac{1}{x_k}, \underbrace{0, \ldots, 0}_{n_k-1}; 1\bigr) \,.
	\end{equation}
	Goncharov \cite{GoncharovGalois01} upgraded these iterated integrals , \( x_i \in \overline{\mathbb{Q}} \), to framed mixed Tate motives, in order to define motivic iterated integrals \( I^{\mathfrak{u}}(x_0; x_1,\ldots,x_N; x_{N+1}) \) in a connected graded (by weight) Hopf algebra, equipped with a coproduct \( \Delta \) given by a combinatorial formula (c.f. Theorem 1.2 \cite{GoncharovGalois01}).  In particular this gives a motivic versions of the multiple polylogarithms.  (Informally \( (2\pi\mathrm{i})^\mathfrak{u} = 0 \), so any path dependence drops out for \( I^{\mathfrak{u}} \).) 
	Quotienting out by products gives the motivic Lie coalgebra \( \mathcal{L}_\bullet(\overline{\mathbb{Q}}) \) of multiple polylogarithms, spanned by elements \( \LiL_{n_1,\ldots,n_k}(x_1,\ldots,x_k) \), which are multiple polylogarithms viewed modulo products.  (In \autoref{sec:coalg:def}, below, we will actually define a version of \( \mathcal{L}_\bullet(F) \) over a field \( F \) inductively, following the setup from Matveiakin-Rudenko \cite{MR22}, in order to by-pass the question of finding \emph{explicit} defining relations.)
	
	In the motivic Lie coalgebra, the coproduct\footnote{More properly this is a cobracket, typically denoted by \( \delta \).  We keep with the notation and nomenclature established by Matveiakin-Rudenko \cite{MR22} for consistency.} is computed by
	\begin{align*}
		\Delta \colon \mathcal{L}_\bullet(\overline{\mathbb{Q}}) &\to \bigwedge\nolimits^2 \mathcal{L}_\bullet(\overline{\mathbb{Q}}) \\
		 \IL(x_0; x_1,\ldots,x_m; x_{m+1}) &\mapsto \sum_{0 \leq i < j \leq m+1} \begin{aligned}[t] \IL(x_0; x_1,\ldots,x_i, x_j,&\ldots, x_m; x_{m+1}) \\[-0.5ex] & {} \wedge \IL(x_i; x_{i+1}, \ldots, x_{j-1}; x_j) \,. \end{aligned} 
	\end{align*}
	One can check that the \emph{truncated} coproduct \( \overline{\Delta} \), defined by omitting the \( \mathcal{L}_1(\overline{\mathbb{Q}}) \wedge \mathcal{L}_{n-1}(\overline{\mathbb{Q}}) \) part from the coproduct, annihilates all depth 1 polylogarithms, i.e. \( \overline{\Delta} \LiL_n(x) = 0 \).  Since \( \overline{\Delta} \LiL_{3,1}(x,y) = -\LiL_2(x y) \wedge \Li_2(y) \neq 0 \), (see \autoref{lem:coprodLik111}, more generally), we deduce \( \LiL_{3,1} \) cannot be expressed via \( \Li_4 \).  The expectation is that \( \overline{\Delta} \neq 0 \) is the only obstruction to a combination of multiple polylogarithms being depth 1.  In particular, this explains why all weight 2 and weight 3 functions can be expressed via depth 1: there is no room for a non-trivial contribution to \( \overline{\Delta} \).

	One can repeatedly apply the truncated coproduct (see \autoref{def:truncatedcoproduct}, for the formal definition), to obtain \( \overline{\Delta}^{[k]} \) the \( k \)-th iterated truncated coproduct.  A direct calculation (c.f. Proposition 4.1 in \cite{MR22}) shows that \( \overline{\Delta}^{[k]} \) annihilates all depth \( {\leq}k \) multiple polylogarithms.  Likewise, this is expected to be the only obstruction to being depth \( {\leq}k \).
	\begin{Conj}[{Goncharov's Depth Conjecture, simplified version, c.f. \cite[Conjecture 7.6]{Gon01}}]
		\label{conj:depthsimp}
		A linear combination of multiple polylogarithms has depth \( {\leq}k \) if and only if its \( k \)-th truncated iterated coproduct \( \overline{\Delta}^{[k]} \) vanishes.
	\end{Conj}
	The more precise version in \autoref{conj:depthv2}, uses the truncated iterated coproduct \( \smash{\overline{\Delta}^{[k-1]}} \) to set up a conjectural isomorphism  from the associated graded \( \gr_{k}^\mathcal{D} \) of the depth filtration, to a the Lie coalgebra cogenerated by depth 1 polylogarithms. \medskip
	
	The Depth Conjecture can be used to obtain number of precise predictions about the behaviour of higher-depth multiple polylogarithms.  Introduce the `divergent' generalisations \( \LiL_{n_0\;n_1,\ldots,n_k}\) of the usual multiple polylogarithms, defined by
	\begin{equation}\label{eqn:divlitoii}
	\begin{aligned}[c]
		& \LiL_{n_0\;n_1,\ldots,n_k}(x_1,\ldots,x_k) \\
		& = (-1)^k \IL\bigl(0; \underbrace{0,\ldots,0}_{n_0}, \tfrac{1}{x_1\cdots x_k}, \underbrace{0, \ldots, 0}_{n_1 - 1}, \tfrac{1}{x_2\cdots x_k}, \underbrace{0, \ldots, 0}_{n_2-1}, \ldots, \tfrac{1}{x_k}, \underbrace{0, \ldots, 0}_{n_k-1}; 1\bigr) \,, \end{aligned}
	\end{equation}
	with the regularisation \( \IL(0; 0; 1) = 0 \) extended to all integrals via the shuffle-product.  These functions have better structured coproducts (c.f. \autoref{lem:coprodLik111}), and play a crucial role already in the quadrangular polylogarithms and their functional equations \autoref{sec:quadrangular}, which is our key tool for proving the identities we discuss below.  One can, of course, write them via the `convergent' multiple polylogarithms if desired (following from certain dihedral symmetries and correlator expressions, c.f. \autoref{sec:coalg:corandint}):
	\begin{align*}
		\LiL_{2\;1,1}(x,y) & = \LiL_{3,1}(\tfrac{1}{x y}, x) - 3 \LiL_4(x y) - \LiL_4(x) \,, \\
			\LiL_{3\;1,1,1}(x,y,z) &= \begin{aligned}[t] 
		\LiL_{4,1,1}\bigl(\tfrac{1}{x y z},x,y\bigr) & {} + \LiL_{4,2}\bigl(\tfrac{1}{x y z},y z\bigr)+\LiL_{5,1}\bigl(\tfrac{1}{x y},x\bigr) \\
		& {} - 5 \LiL_{6}(x y)-\LiL_{6}(x)+5 \LiL_{6}(y z) \,. \end{aligned} 
	\end{align*}
	
	\paragraph{Weight 4:} One computes  (c.f. \autoref{lem:coprodLik111}) that
	\[
		\overline{\Delta} \liftwo(x,y) = -\LiL_{2}(x) \wedge \LiL_2(y) \,,
	\]
	so the function \( \liftwo(x,y) \) should satisfy dilogarithm functional equations in each variable \( x \), and \( y \) independently.  For the two-term dilogarithm symmetries \( \LiL_2(x) + \LiL_2(1-x) = 0 \) and \( \LiL_2(x) + \LiL_2(x^{-1}) = 0 \) (note: the apparently missing constant \( \LiL_2(1) = \zeta^\mathcal{L}(2) = 0 \), as it is a product), the reductions of \( \liftwo(x,y) + \liftwo(1-x,y) \) and \( \liftwo(x,y) + \liftwo(x^{-1},y) \) were essentially found by Zagier (and Gangl) \cite[\S4.1]{gangl-4} and \cite[\S6, Remark p.~7]{ganglSome}.  (A reduction for  \( \liftwo(1, y) = \LiL_{3,1}(1,x) \) follows from this, and is related to a formula for the Nielsen polylogarithm \( S_{2,2}(x) \) via \( \Li_4 \), c.f. \cite{koelbig}, and \cite[p.~204]{Lewin}).  The fundamental functional equation for the dilogarithm, generating all others, is the  five-term relation, 
	\[
		\sum_{i=0}^4 (-1)^i \LiL_2([w_0,\ldots,\widehat{w_i},\ldots,w_4]) = 0 \,, \quad \text{with } [a,b,c,d] \coloneqq \frac{(a-b)(c-d)}{(b-c)(d-a)} \,,
	\]
	the cross-ratio (see \cite{deJeu} for a formalisation of this folklore statement).   A reduction essentially for 
	\[
		\sum_{i=0}^4 (-1)^i \liftwo([w_0,\ldots,\widehat{w_i},\ldots,w_4], y) = \sum_{j=1}^{122} \LiL_4(\xi_j(w_0,\ldots,w_4, y)) \,,
	\]
	was found by Gangl \cite{gangl-4} with 122 explicit rational function arguments \( \xi_j(w_0,\ldots,w_4,y) \) obtained as products of up to 4 cross-ratios, and later understood conceptually by Goncharov-Rudenko \cite{GR-zeta4}.  This settles the Depth Conjecture in weight 4, as all expected reductions are known explicitly. 
	
	\paragraph{Weight 6:} One computes  (c.f. \autoref{lem:coprodLik111}) that
	\[
			\overline{\Delta}^{[2]} \lif(x,y,z) = -\LiL_2(x) \otimes \LiL_2(y) \otimes \Li_2(z) \quad ( \in \CoLie_3 ) \,,
	\]
	For the function \( \lif(x,y,z) \), one therefore expects the following reductions to hold (in any of \( x, y \) or \( z \) independently), in order of increasing generality
	\begin{enumerate}[label=\roman*),itemsep=0.5ex]
		\item \( \lif(1, y, z) \equiv 0 \modtwoi \), 
		\item \(\lif(x, y, z) \equiv - \lif(1-x, y, z) \equiv -\lif(x^{-1}, y, z) \modtwoi \), and 
		\item \( \sum_{i=0}^4 (-1)^i \lif([w_0, \ldots,\widehat{w_i},\ldots,w_4], y, z) \equiv 0 \modtwoi \).
	\end{enumerate}
	These respectively generalise the following weight 4 results: i) the formula reducing the Nielsen pollogarithm, ii) the Zagier's formulae for reducing the two-term symmetries, and iii) Gangl's 122-term formula for reducing the five-term relation.
	
	In \cite[Theorem 1.4]{MR22} Matveiakin and Rudenko were able to prove that the five-term combination \( \sum_{i=0}^4 (-1)^i \lif([w_0, \ldots,\widehat{w_i},\ldots,w_4], y, z)  \) from iii) can be expressed via functions of depth 2, \emph{assuming} that the reductions for the two-term symmetries from ii) already hold.
	
	The main result of the paper is to establish the reductions in ii) for the two-term symmetries, as follows.
	\begin{Thm}[\autoref{thm:i411sixfold}, the Zagier formulae in weight 6]\label{thm:intro:zagierformulae}
		The function \(
		\lif(x,y,z) 
		\) satisfies the following 6-fold dilogarithm symmetries in each argument \( x, y \) and \( z \) independently, modulo terms of depth \( {\leq}2 \),
		\begin{align*}
		\begin{aligned}
			\lif(x,y,z) &\equiv -\lif(1-x, y, z) \modtwo \,, \\
			\lif(x,y,z) &\equiv -\lif(x^{-1}, y, z) \modtwo \,.
			\end{aligned}
		\end{align*}
	\end{Thm}
	In order to establish this, we need to combine and utilise proven relations for multiple polylogarithms.  In general, we do not know the \emph{explicit} defining relations even for \( \mathcal{L}_4(F) \) let alone \( \mathcal{L}_6(F) \).  However, a family of non-trivial functional equations for higher-polylogarithms was established by Matveiakin-Rudenko in \cite{MR22} (using the quadrangular polylogarithms \( \QLi \) introduced in \cite{Rud20}, extending the cluster-polylogarithm relation \( \QUf_4 \) from \cite{GR-zeta4}, and the higher-weight analogues we established in \cite{CGRpolygon}); this family seem to be powerful enough to establish all the currently known identities and reductions.  We use the functional equation (denoted here by \( \QUf_6 \), in the spirit of \cite{GR-zeta4})  for the weight 6, 8-point quadrangular polylogarithm \( \QLi_6(x_1,\ldots,x_{8}) \) (c.f. \autoref{sec:quadrangular}, and  \cite[Proposition 3.6]{MR22}), which by the quadrangulation formula (\autoref{sec:quadrangular:quadrangulation}, \autoref{fig:qli3polygon}, and \cite[Theorem 1.2]{Rud20}), gives a relation between (many) \( \lif \)'s of cross-ratios of 9 points, modulo depth \( {\leq}2 \).  By specialising and degenerating \( \QUf_6 \) to various strata of \( \overline{\mathfrak{M}}_{0,9} \) (\autoref{sec:quadrangular:M0nstable}) we obtain (\autoref{sec:higherZagier6}) a number of symmetries and short relations which conspire to show the reductions in \autoref{thm:intro:zagierformulae}.
	
	As part of this process we must first establish reductions when any variable is specialised to 1.
	\begin{Thm}[\autoref{thm:onexy_dp2}, the Nielsen formulae in weight 6]
		The function
		\(
		 \lif(x,y,z)
		\)
		is expressible via depth 2, when any of \( x \), \( y \) or \( z \), is specialised to 1,\biggerskip
		\[
			\lif(1,y,z) \equiv \lif(x,1,y) \equiv \lif(x,y,1) \equiv 0 \modtwo \,.
		\]
	\end{Thm}
	
	\Autoref{thm:intro:zagierformulae} provides the missing ingredient to the result from Matveiakin-Rudenko \cite{MR22}, which established that modulo these symmetries and terms of depth \( {\leq}2 \), \( \lif(x,y,z) \) satisfies the dilogarithm five-term relation.  Consequently, the five-term reduction of \( \lif(x,y,z) \) holds unconditionally.  As all expected reductions are now known, together with Matveiakin-Rudenko \cite[Theorem 1.4]{MR22}, we have established Goncharov's Depth Conjecture in weight 6, depth 3. 
	
	\begin{Cor}[\autoref{cor:gon:wt6depth3}, Goncharov's Depth Conjecture, special case]
		A linear combination of weight 6 multiple polylogarithms has depth \( {\leq}2 \) if and only if its 2-nd iterated truncated coproduct \( \overline{\Delta}^{[2]} \) vanishes.
	\end{Cor}
	
	The rest of the paper is structured as follows.  In \autoref{sec:coalg} we recall the definition of the motivic Lie coalgebra following \cite{MR22}, and state the Goncharov's Depth Conjecture (\autoref{conj:depthv2}).  In particular, we recall the specialisation homomorphism and inductive definition of relations in \autoref{sec:coalg:def}, we clarify the relationship between motivic iterated integrals and motivic correlators in \autoref{sec:coalg:corandint}, and we establish or recall some basic relations and specialisations (of arguments to \( 0 \) or \( \infty \)) for motivic multiple polylogarithms in \autoref{sec:coalg:mplpropspec}.  In \autoref{sec:coalg:depth} we formally state Goncharov's Depth Conjecture, and precisely categorise a special case (weight \( 2k \), depth \( k \)) via the higher Nielsen formulae (\autoref{conj:highernielsen}), higher Zagier formulae (\autoref{conj:higherzagier}) and the higher Gangl formula (\autoref{conj:highergangl}).
	
	In \autoref{sec:quadrangular} we recall the framework of quadrangular polylogarithms \cite{Rud20,MR22} and how to specialise them to strata of \( \overline{\mathfrak{M}}_{0,N} \).  In \autoref{sec:quadrangular:construction}, we give the construction of \( \QLi_{n+k}(x_0,\ldots,x_{2n+1}) \) via correlators, then in \autoref{sec:quadrangular:quadrangulation} we use the quadrangulation formula to express the 4-, 6- and 8-point quadrangular polylogarithms via depth 1, depth 2 and depth \( {\leq}3 \) multiple polylogarithms respectively.  (We graphically present these quadrangulations in \autoref{fig:qli1polygon}, \autoref{fig:qli2polygon} and \autoref{fig:qli3polygon} respectively.)  Then in \autoref{sec:quadrangular:femoddp}, we describe the functional equations for quadrangular polylogarithms, modulo lower depth.  In \autoref{sec:quadrangular:anharmonic} we recall some properties of cross-ratios, including the action of \( \Sym_3 \), and in \autoref{sec:quadrangular:M0nstable} we recall how boundary components of the Deligne-Mumford compactification of \( \mathfrak{M}_{0,N} \) can be described via stable curves, and what it means (conceptually and computationally) to specialise a multiple polylogarithm functional equation to such a stable curve.
	
	In \autoref{sec:higherZagier4}, we revisit the Zagier and Gangl formulae in weight 4.  We aim to understand the weight 4 and weight 6 versions from a uniform point of view, in order to set the groundwork for generalisations to higher weight.  (We also wish to clarify a point elided over in \cite[Definition 1.8, relation 2.]{GR-zeta4}, where a version of specialisation \( \liftwo(1,x) = -\LiL_4(x) + \LiL_4(\frac{x-1}{x}) + \LiL_4(1-x) \) is imposed; we show in \autoref{lem:wt4deg} how to deduce it from a version of \( \QUf_4 \).)  Since the combinatorics of the degenerations in weight 6 are rather involved, revisiting weight 4 gives us the chance to explain the principles of degenerating the quadrangular polylogarithm functional equations in a simpler -- and more explicit -- context.
	
	Finally in \autoref{sec:higherZagier6}, we establish the higher Zagier formulae in weight 6.  We start in \autoref{sec:wt6:sym} by re-deriving a number of basic symmetries and relations which are known to hold generally; we re-derive them from \( \QUf_6 \) as evidence that \( \QUf_6 \) is indeed the basic functional equation for weight 6 multiple polylogarithms (c.f. the discussion after \cite[Thoerem 1.1]{MR22}).  In \autoref{sec:wt6:nielsen:11x} we show how to deduce a reduction for \( \lif(1,1,x) \equiv 0 \modtwoi \) from \( \QUf_6 \).  In \autoref{sec:wt6:nielsen:1xy}, we start by establishing two non-trivial symmetries of \( \lif(1,x,y) \) in \autoref{lem:onexy_sym1}, and \autoref{lem:onexy_sym2}; we then show in \autoref{thm:onexy_dp2} how these symmetries (together with the known inversion symmetry \autoref{cor:inv}) imply that \( \lif(1,x,y) \equiv 0 \modtwoi \), which establishes the Nielsen formulae in weight 6 (\autoref{conj:highernielsen}, $k = 3$).  In \autoref{sec:wt6:sym} we derive two simple symmetries \autoref{lem:fullsym1}, and \autoref{lem:fullsym2} of \( \lif(x,y,z) \), and a four-term term relation \autoref{prop:fourterm} by specialising \( \QUf_6 \).  From an interplay between these symmetries and the four-term relation, we deduce \autoref{cor:zagier:wt6} that \( \lif(x, y^{-1}, z) + \lif(x, y^{-1}, 1-z) \equiv 0 \modtwoi \), the first of the Zagier formulae.  In \autoref{sec:wt6:conclusion} we conclude with \autoref{thm:i411sixfold}, establishing the 6-fold dilogarithm symmetries in each argument of \( \lif(x,y,z) \) independently.  This establishes the higher Zagier formulae in weight 6 (\autoref{conj:higherzagier}, $k = 3$).
	
	In \autoref{app:explicit}, we explicitly present some of the shorter identities: namely, those in \autoref{sec:wt6:sym}--\autoref{sec:wt6:nielsen:1xy}, along with the two symmetries \autoref{lem:fullsym1}, \autoref{lem:fullsym2} of \( \lif(x,y,z) \), and the four-term term relation \autoref{prop:fourterm}.
	\medskip
	
	The following ancillary files are attached to the \texttt{arXiv} submission:\hypertarget{filedescriptions}{}
	\begin{itemize}[itemsep=0.5ex]
		\item \wtfourfilename, a \texttt{Mathematica} worksheet which verifies the calculations and proofs from \autoref{sec:higherZagier4}, re-proving the Zagier and Gangl formulae in weight 4,
		\item \wtsixfilename, a \texttt{Mathematica} worksheet which verifies the calculations and proofs from \autoref{sec:higherZagier6}, establishing the Nielsen and Zagier formulae in weight 6,
		\item \wtsixidentities, a text file in \texttt{Mathematica} syntax, for all identities and reductions found in \autoref{sec:wt6:sym}--\autoref{sec:wt6:6fold}, with  \autoref{sec:wt6:6fold} written via \( \lif(1, x, y) \) and \( \lif(x, 1, y) \) for simplicity,  
		\item \wtsixidentitieslist, the corresponding text file, giving expressions in the form
		\texttt{identity = [ [coeff, func, [arg1, ..., argd]], ... ];} \, ,
		\item \wtsixdepthtwo, a text file in \texttt{Mathematica} syntax, for the reductions and identities in \autoref{sec:wt6:6fold}, via purely depth 2 functions, and 
		\item \wtsixdepthtwolist, the corresponding text file, giving expressions in the form
		\texttt{identity = [ [coeff, func, [arg1, ..., argd]], ... ];} \,.
	\end{itemize}
	
	\par\medskip
	\noindent {\bf Acknowledgements.}  Initial developments, research and results on this problem began to take shape while at Universit\"at Hamburg, supported by DFG Eigene Stelle grant CH 2561/1-1, for Projektnummer 442093436.  The remaining results were then obtained at the Max Planck Institute for Mathematics, Bonn, for whose support, hospitality and excellent working conditions I am grateful.
	
	This work benefited from many enlightening discussions with Herbert Gangl, Danylo Radchenko (in particular for \autoref{prop:zagcombRad} below), and especially with Andrei Matveiakin and Daniil Rudenko.   Their questions about techniques and strategies to establish the (then missing) `higher Zagier formulae' (\autoref{conj:higherzagier} below) in weight 6, by matching parts of the coproduct, provided continual motivation to better understand the structure of the weight 6 quadrangular polylogarithm relation \( \QUf_6 \) (see \autoref{sec:quadrangular} below), its degenerations and their implications.
	
	\section{The motivic Lie coalgebra, and Goncharov's Depth Conjecture}\label{sec:coalg}
	
	To formally state Goncharov's Depth Conjecture, and to identify the expectations of it (in weight 6, depth 3, for the main result of this paper, as well as weight 4, depth 2, by comparison, and general expectations in weight \( 2k \), depth \( k \)), we need to recall the definition of the (motivic) Lie coalgebra \( \mathcal{L}_\bullet(F) \) of (formal) multiple polylogarithms with values in a field \( F \).  For the background and formal definitions, we largely follow the narrative given in \cite{MR22}.
	
	\subsection{Degenerations and inductive definition of the Lie coalgebra of multiple polylogarithms}\label{sec:coalg:def}  First, let us recall from \cite[\S2.1]{MR22} the formal degeneration procedure.  Let \( F \) be a field of characteristic 0, and for \( n \geq 1 \), let \( \mathcal{A}_n(F) \) be the \( \mathbb{Q} \)-vector space generated by ordered tuples of points \( (x_0,x_1,\ldots,x_n) \in F^{n+1} \) modulo the relations
	\begin{align*}
		&(x_0, x_1,\ldots,x_{n-1}, x_2) = (x_1, x_2, \ldots, x_n, x_n) \\	
		&(x, x, \ldots, x) = 0 \,.
	\end{align*}
	Define a Lie cobracket\footnote{As before, this should more properly be denoted \( \delta \), as \( \Delta \) denotes the coproduct (or coaction) in a motivic Hopf algebra (comodule).  As we work exclusively in the Lie coalgebra, so confusion can arise, and to be consistent with the notation in Matveiakin-Rudenko \cite{MR22}, we write \( \Delta \).} \( \Delta \colon \mathcal{A}_\bullet(F) \to \bigwedge^2 \mathcal{A}_\bullet(F) \) by
	\begin{equation}\label{eqn:liecobr}
		\Delta \, (x_0, \ldots, x_n) = \sum_{\mathrm{cyc}} \sum_{i=1}^{n-1} (x_0, x_1, \ldots, x_i) \wedge (x_0, x_{i+1}, \ldots, x_n) \,,
	\end{equation}
	where
	\[
		\sum_{\mathrm{cyc}} f(x_0, \ldots, x_n) \coloneqq \sum_{j=0}^n f(x_j, x_{j+1}, \ldots, x_n, x_0, \ldots, x_{j-1}) \,.
	\] 
	Let \( K \) be a field with discrete valuation \( \nu \colon K \to \mathbb{Z}\cup \{ \infty \} \) and residue field \( k \).  Given a uniformiser \( \pi \in F \), Matveiakin and Rudenko define the specialisation homomorphism as follows.  Let \( m = \min_{0 \leq j \leq n} \nu(x_j) \), then
	\[
		\spec_{\nu, \pi} \, (x_0, x_1,\ldots,x_n) \coloneqq (y_0,y_1,\ldots,y_n)
	\]
	where 
	\[
		y_i = \begin{cases}
			\overline{x_i \pi^{-m}} & \text{if \( \nu(x_i) = m\)\,,} \\
			0 & \text{if \( \nu(x_i) > m \)\,.}
		\end{cases}
	\]
	In Lemma 2.1 \cite{MR22} they show specialisation commutes with the Lie cobracket. \medskip
	
	Since the explicit relations for multiple polylogarithms are mostly unknown, Matveiakin and Rudenko give an following inductive definition for the space of relations \( \mathcal{R}_n(F) \subseteq \mathcal{A}_n(F) \), to bypass this issue (mimicking the construction of higher Bloch groups \cite{Gon95B},\cite{zagierDedekind} and similar constructions in \cite{GoncharovMSRI},\cite{GKLZ}.  The coalgebra of multiple polylogarithms is then defined as the quotient
	\begin{equation}\label{eqn:mplcoalg}
		\mathcal{L}_n(F) = \frac{\mathcal{A}_n(F)}{\mathcal{R}_n(F)} \,,
	\end{equation}
	and the projection of \( (x_0, x_1, \ldots, x_n) \) to \( \mathcal{L}_n(F) \), is denoted by \( \corL(x_0,x_1,\ldots,x_n) \) and called the \emph{correlator}.  
	
	\paragraph{Weight one:} In weight one, \( \mathcal{R}_1(F) \) is defined as the kernel of the map \( (x_0, x_1) \in \mathcal{A}_1(F) \) to \( (x_0 - x_1) \in \mathbb{Q}[F^\times_\mathbb{Q}] \), whence by definition \( \mathcal{L}_1(F) \cong F^\times_\mathbb{Q} \).  One can then denote \( \logL(a) \coloneqq  \corL(0, a) \), agreeing with the identity
	\[
		\corL(0, ab) = \corL(0, a) +  \corL(0,b) \,,
	\]
	  which follows from the relation \( (0, ab) - (0,a) - (0,b) \in \mathcal{R}_1(F) \).
	
	\paragraph{Higher weight:} Suppose that the spaces \( \mathcal{R}_i(F) \) are defined in weights \( {<}n \).  Then consider the field \( K = F(t) \), and let \( \nu_a \) be the discrete valuation corresponding to \( a \in \mathbb{P}^1(F) \).  Write \( \spec_{t\to a} = \spec_{\nu_a, t-a} \), and \( \spec_{t\to \infty} = \spec_{\nu_\infty, {t^{-1}}} \).  Then the relations in weight \( n \) are given by
	\[
		\mathcal{R}_n(F) = \Big\{ \spec_{t\to 0} R(t) - \spec_{t\to \infty} R(t) \, \Big| \, R(t) \in \mathcal{A}_n(F(t)) \,,\, \Delta \, R(t) = 0 \in \bigwedge\nolimits^2 \mathcal{L}_\bullet(F(t))  \Big\} \,.
	\]
	Intuitively: the coproduct condition tells us that \( R(t) \) is constant, as a function of \( t \), and by taking the difference of two specialisations, we guarantee this combination vanishes, hence gives us a relation.  With the relations defined, \( \mathcal{L}_n(F) = \mathcal{A}_n(F) / \mathcal{R}_n(F) \) gives the weight \( n \) component of the Lie coalgebra. \medskip
	
	Since the cobracket commutes with specialisation \cite[Lemma 2.1]{MR22}, we have
	\[
		\Delta \big( \spec_{t\to 0} R(t) - \spec_{t\to \infty} R(t) \big) = 0 \,,
	\]
	so \( \Delta \) descends to \( \mathcal{L}_\bullet(F) \), indeed giving a Lie coalgebra structure to the quotient.  It then also follows \cite[Eq. 2.5]{MR22} that
	\begin{equation}\label{eqn:deltacor}
		\Delta \corL(x_0, x_1, \ldots, x_n) = \sum_{\mathrm{cyc}} \sum_{i=1}^{n-1} \corL(x_0, x_1, \ldots, x_i) \wedge \corL(x_0, x_{i+1}, \ldots, x_n) \,.
	\end{equation}
	Likewise, the specialisation homomorphisms \cite[Remark 2.2]{MR22} are well-defined on \( \mathcal{L}_\bullet(F) \), and independent of the choice of uniformiser (by the affine invariance of correlators) for weight \( n \geq 2 \).  The affine invariance can be seen via a coproduct calculation.
	
	\begin{Lem}[Affine invariance of correlators]\label{lem:cor:affine}
		Correlators of weight \( n \geq 2\) are invariance under affine transformations, i.e. for any \( a \in F^\times \), \( b \in F \),
		\[
			\corL(x_0, x_1,\ldots,x_n) = \corL(a x_0 + b, a x_1 + b, \ldots, a x_n + b) \,.
		\]
		
		\begin{proof}
			Firstly, note that in weight 1, we have an extra correction term
			\[
				\corL(a x_0 + b, a x_1 + b) = \corL(x_0, x_1) + \underbrace{\corL(0, a)}_{\logL(a)} \,,
			\]
			on account of the relation \( [(a x_0 + b) - (a x_1 + b)] - [(x_0 - x_1)] - [a] \) in \( \mathcal{R}_1(F) \), corresponding to the identity \( [(a x_0 + b) - (a x_1 + b)] = [a(x_0 - x_1)] = [a] + [x_0 - x_1] \) in \( \mathbb{Q}[F^\times_\mathbb{Q}] \). \medskip
			
			Spelling out all of the details in weight 2, we compute directly the coproduct
			\begin{align*}
				& \Delta \corL((t+a) x_0 +b , (t+a)x_1 +b, (t+a)x_2 +b) \\
				&= \sum_{\mathrm{cyc}} \corL((t+a) x_0 + b, (t+a) x_1 + b) \wedge \corL((t+a) x_0 + b, (t+a) x_2 + b) \\
				& = \sum_{\mathrm{cyc}} (\corL(x_0, x_1) + \logL(t+a)) \wedge (\corL(x_0, x_2) + \logL(t+a)) \\
				& = \sum_{\mathrm{cyc}} \corL(x_0, x_1) \wedge \corL(x_0, x_2) \,.
			\end{align*}
			The final line follows since \( \logL(t+a) \wedge \logL(t+a) = 0 \), and the equality
			\[
				\sum_{\mathrm{cyc}} \corL(x_0,x_1) = \sum_{\mathrm{cyc}} \corL(x_0, x_2) = \corL(x_0,x_1) + \corL(x_1,x_2) + \cor(x_2,x_0) \,,
			\]
			implies the terms \( \sum_{\mathrm{cyc}} \logL(t+a) \wedge \corL(x_0,x_2) \) and \( \sum_{\mathrm{cyc}} \corL(x_0, x_1) \wedge \logL(t+a) \) are are equal up to sign (arising from switching the order of factors in the wedge), hence cancel out.
			
			Hence \vspace{-1ex} \[
			\Delta \big( \overbrace{ \corL((t+a) x_0 +b , (t+a) x_1 +b, (t+a) x_2 +b) - \corL(x_0, x_1, x_2)}^{R(t)\coloneqq} \big)  = 0 \in \bigwedge\nolimits^2 \mathcal{L}_\bullet(F(t))
			\]
			Then
			\begin{align*}
				\spec_{t\to0} R(t) &= \spec_{t\to0} \corL((t+a) x_0 +b , (t+a) x_1 +b, (t+a) x_2 +b) - \corL(x_0, x_1, x_2) \\
				& = \corL(a x_0 + b ,a x_1+ b, a x_2 + b) - \corL(x_0, x_1, x_2)  \,;
			\end{align*}
			the valuation of every term with respect to \( \nu_0 \), i.e. divisibility by \( t-0 \), is 0, hence minimal, so we can just reduce modulo \( t - 0 \), by setting \( t = 0 \)).
			Whereas
			\begin{align*}
			\spec_{t\to\infty} R(t) &= \spec_{t\to\infty} \corL((t+a) x_0 +b , (t+a) x_1 +b, (t+a) x_2 +b) - \corL(x_0, x_1, x_2) \\
			& = \spec_{t\to0} \corL((t^{-1}+a) x_0 +b , (t+a) x_1 +b, (t+a) x_2 +b) - \corL(x_0, x_1, x_2) \\
			& = \spec_{t\to0} \corL(\tfrac{(1+at) x_0 +bt}{t}, \tfrac{(1+at) x_1 +bt}{t}, \tfrac{(1+at) x_2 +bt}{t}) - \corL(x_0, x_1, x_2)
			\end{align*}
			For the first term, every argument has valuation \( -1 \) with respect to \( \nu_0 \), so we must multiply by \( t \), and then reduce modulo \( t-0 \) by setting \( t = 0 \), the specialisation of the second term is computed directly, and we have
			\[
				\spec_{t\to\infty} R(t) = \corL(x_0, x_1, x_2) - \corL(x_0, x_1, x_2) = 0 \,.
			\]
			We thus have
			\[
				\spec_{t\to0} R(t) - \spec_{t\to\infty} R(t) = \corL(a x_0 + b ,a x_1+ b, a x_2 + b) - \corL(x_0, x_1, x_2) \in \mathcal{R}_2(F) \,,
			\]
			giving us then the affine invariance in \( \mathcal{L}_2(F) \). \medskip
			
			\noindent$\llbracket$This is the formal and rigorous way of computing more informally, as follows.  Since
			\[
				\Delta \big(\corL(a x_0 + b ,a x_1+ b, a x_2 + b) - \corL(x_0, x_1, x_2) \big) = 0 \,,
			\]
			the difference is constant as a function of the \( x_i \), (\( a, b \) were fixed).  Setting \( x_0 = x_1 = x_2 = 0 \), shows that this constant is 0, hence we obtain the affine invariance.  For simplicity, we will use this informal approach in future.$\rrbracket$ \medskip
			
			In weight \( {>}2 \), we compute the coproduct using the assumption of affine invariance (except in weight 1) inductively
				\begin{align*}
				& \Delta \corL(a x_0 +b , b x_1 +b, \ldots, ax_n +b) \\
				&= \sum_{\mathrm{cyc}} \sum_{i=1}^{n-1} \corL(a x_0 +b , b x_1 +b, \ldots, ax_i +b) \wedge \corL(a x_0 +b , b x_{i+1} +b, \ldots, ax_n +b) \\
				& = \begin{aligned}[t] 
					& \sum\nolimits_{\mathrm{cyc}} \sum\nolimits_{i=1}^{n-1}  \corL(x_0 +b , x_1, \ldots, x_i) \wedge \corL(x_0 , x_{i+1}, \ldots, x_n) \\
					& + \sum\nolimits_{\mathrm{cyc}} \logL(a) \wedge  \corL(a x_0 +b , b x_{2} +b, \ldots, ax_n +b) \\ 
					& + \sum\nolimits_{\mathrm{cyc}} \corL(a x_0 +b , b x_{1} +b, \ldots, ax_{n-1} +b) \wedge \logL(a) \,.
				\end{aligned}
			\end{align*}
			As before the \( \sum_{\mathrm{cyc}} \logL \wedge \corL \) and \( \sum_{\mathrm{cyc}} \corL \wedge \logL \) terms cancel, and so we have
			\[
				\Delta \big( \corL(a x_0 +b , b x_1 +b, \ldots, ax_n +b) - \corL(x_0 , x_1, \ldots, x_n) \big) = 0
			\]
			implying the difference is constant; this constant is fixed to 0 by taking \( x_0 = x_1 = \cdots = x_n = 0 \).  This has established the affine invariance for weights \( n \geq 2 \).
		\end{proof}
	\end{Lem}

	\subsection{{Interlude:} Relation of correlators and iterated integrals} \label{sec:coalg:corandint} Although the fundamental objects defining the above model of the Lie coalgebra of multiple polylogarithms are the correlators, it behooves us to clarify the relation between correlators, and iterated integrals in the context of the motivic Lie coalgebra. \medskip

	\begingroup
	\renewcommand{\IL}{I^{\mathscr{L}}}
	\renewcommand{\corL}{\cor^{\mathscr{L}}}
	\renewcommand{\logL}{\log^{\mathscr{L}}}
	\note{Note} In this section we briefly forget the definition of the Lie coalgebra of multiple polylogarithms above.  We consider the iterated integral and correlator as elements of the motivic Lie coalgebra, which we denote via \( (\bullet)^\mathscr{L} \).  In this setting the integral and the correlator have cobrackets which are proven to be computed (not imposed as in \autoref{sec:coalg:def} and \autoref{eqn:liecobr}) by the following formulae
	\begin{align}
		\label{eqn:ildelta} \Delta \IL(x_0; x_1,\ldots,x_m; x_{m+1}) &= \sum_{0 \leq i < j \leq m+1} \begin{aligned}[t] \IL(x_0; x_1,\ldots,x_i, x_j,&\ldots, x_m; x_{m+1}) \\[-0.5ex] & {} \wedge \IL(x_i; x_{i+1}, \ldots, x_{j-1}; x_j) \,, \end{aligned} \\
		\notag \Delta \corL(x_0, x_1, \ldots, x_n) &= \sum_{\mathrm{cyc}} \sum_{i=1}^{n-1} \corL(x_0, x_1, \ldots, x_i) \wedge \corL(x_0, x_{i+1}, \ldots, x_n) \,.
	\end{align}
	
	\begin{Prop}[Iterated integrals via correlators, Rudenko]\label{lem:intascor}
		The following equality holds
		\[
			\IL(x_0; x_1,\ldots,x_m; x_{m+1}) = \corL(x_{m+1}, x_1,\ldots,x_m) - \corL(x_0, x_1,\ldots, x_m) \,.
		\]
		
		\begin{proof}[Proof {\normalfont (Rudenko)}]
			In weight 1, we can check this directly, as one has
			\[
				\corL(x_0, x_1) = \logL(x_1 - x_0) \,,
			\]
			so
			\[
				\IL(x_0; x_1; x_2) = \logL\Big(\frac{x_1 - x_2}{x_1 - x_0}\Big) = \corL(x_2, x_1) - \corL(x_0, x_1) \,.
			\]
			
			Higher weight is shown by induction, as follows.  Assume the statement holds for \( m' \leq m-1 \).  We compute
			\begin{align*}
				&\Delta \IL(x_0; x_1,\ldots,x_m; x_{m+1}) \\
				&  = \sum\nolimits_{0 \leq i < j \leq m+1} \IL(x_0; x_1,\ldots,x_i, x_j,\ldots, x_m; x_{m+1}) \wedge \IL(x_i; x_{i+1}, \ldots, x_{j-1}; x_j)  \\			
				&  = \sum\nolimits_{0 \leq i < j \leq m+1} \begin{aligned}[t] 
					\big(\corL&(x_{m+1}, x_1,\ldots,x_i, x_j,\ldots, x_m) - \corL(x_0, x_1,\ldots,x_i, x_j,\ldots, x_m)\big) \\
					& \wedge \big(\corL(x_j, x_{i+1}, \ldots, x_{j-1}) - \corL(x_i, x_{i+1}, \ldots, x_{j-1}) \big) \,. \end{aligned}
 			\end{align*}
 			Rewrite the last sum as follows
 			{\small
 			\begin{align*}
 				& = \begin{aligned}[t]
 				& \sum\nolimits_{1 \leq p < q \leq m} A_{p,q} \wedge \corL(x_p, x_{p+1},\ldots, x_{q}) \\
 				& {} + \sum\nolimits_{1 \leq q \leq m} \big( {-}\corL(x_q, x_{q+1}, \ldots, x_{m+1}) + \corL(x_0, x_q, x_{q+1},\ldots,x_m) \big) \wedge \corL(x_0, x_1,\ldots,x_q)  \\
 				& {} + \sum\nolimits_{1 \leq p \leq m} \big( \corL(x_1, x_2, \ldots, x_p, x_{m+1}) - \corL(x_0, x_1, \ldots,x_p) \big) \wedge \corL(x_p, x_{p+1},\ldots,x_{m+1}) 
 				\end{aligned} \\[1ex]
 				& = \begin{aligned}[t]
 				& \sum\nolimits_{1 \leq p < q \leq m} A_{p,q} \wedge \corL(x_p, x_{p+1},\ldots, x_{q}) \\
 				& {} + \sum\nolimits_{1 \leq q \leq m} \corL(x_0, x_{q}, x_{q+1}, \ldots, x_{m}) \wedge \corL(x_0, x_1,\ldots,x_q)  \\
 				& {} + \sum\nolimits_{1 \leq p \leq m} \corL(x_1, x_2, \ldots, x_p, x_{m+1}) \wedge \corL(x_p, x_{p+1},\ldots,x_{m+1}) \,, \end{aligned}
 			\end{align*}
 			}
 			where
 			\begin{align*}
 				A_{p,q} = {} & \corL(x_1,\ldots,x_{p-1},x_q,\ldots,x_{m+1})- \corL(x_0,\ldots,x_{p-1},x_q,\ldots,x_{m}) \\
 				&  - \corL(x_1,\ldots,x_{p},x_{q+1},\ldots,x_{m+1})+ \corL(x_0,\ldots,x_{p},x_{q+1},\ldots,x_{m}) \,.
 			\end{align*}
 			
 			On the other hand,
 			\begin{align*}
 				& \Delta \corL(x_{m+1},x_1,x_2,\ldots,x_{m}) \\
 				& = \sum_{\text{cyclic}} \sum\nolimits_{1 \leq i < m+1} \corL(x_{m+1}, x_1,x_2,\ldots,x_i) \wedge \corL(x_{m+1}, x_i, x_{i+1},\ldots,x_m)  \\
 				& = \begin{aligned}[t]
 					& \sum\nolimits_{1 \leq p < q \leq m} A'_{p,q} \wedge \corL(x_p, x_{p+1}, \ldots, x_q) \\
	 				& + \sum\nolimits_{1 \leq p \leq m} \corL(x_{m+1}, x_1,x_2,\ldots,x_p) \wedge \corL(x_{m+1}, x_p, x_{p+1},\ldots,x_m)  \,,
	 				\end{aligned}
 			\end{align*}
 			where 
 			\[
 				A'_{p,q} = -\corL(x_1,\ldots,x_p,x_{q+1},\ldots,x_{m+1}) + \corL(x_1,\ldots,x_{p-1},x_q,\ldots,q_{m+1}) \,.
 			\]
 			Likewise
 			\begin{align*}
 			& \Delta \corL(x_0,x_1,\ldots,x_{m}) \\
 			& = \sum_{\text{cyclic}} \sum\nolimits_{1 \leq i < m+1} \corL(x_{0}, x_1,x_2,\ldots,x_i) \wedge \corL(x_{0}, x_i, x_{i+1},\ldots,x_m)  \\
 			& = \begin{aligned}[t]
 			& \sum\nolimits_{1 \leq p < q \leq m} A''_{p,q} \wedge \corL(x_p, x_{p+1}, \ldots, x_q) \\
 			& + \sum\nolimits_{1 \leq p \leq m} \corL(x_{0}, x_1,x_2,\ldots,x_p) \wedge \corL(x_{0}, x_p, x_{p+1},\ldots,x_m) \,,
 			\end{aligned}
 			\end{align*}
 			where 
 			\[
 			A''_{p,q} = -\corL(x_0,\ldots,x_p,x_{q+1},\ldots,x_{m}) + \corL(x_0,\ldots,x_{p-1},x_q,\ldots,q_{m}) \,.
 			\]
 			
 			Since \( A_{p,q} = A'_{p,q} + A''_{p,q} \), we see
 			\[
 				\Delta \big( \IL(x_0; x_1,\ldots,x_m; x_{m+1}) - \big( \corL(x_{m+1}, x_1,\ldots,x_m)  - \corL(x_0, x_1,\ldots,x_m) \big) \big) = 0 \,,
 			\]
 			thus
 			\[
 				\IL(x_0; x_1,\ldots,x_m; x_{m+1}) - \big( \corL(x_{m+1}, x_1,\ldots,x_m)  - \corL(x_0, x_1,\ldots,x_m) \big) 
 			\]
 			is constant.  By setting \( x_{m+1} = x_0 \), this constant must be 0, which proves the result.
		\end{proof}
	\end{Prop}
	
	\begin{Cor}\label{cor:corasint}
		The correlator may be expressed via the iterated integral as follows,
		\begin{align*}
			 \corL(x_0, x_1,\ldots,x_m) 
			& = \sum_{i=0}^m \IL(0; \{0\}^i, x_0, \ldots, x_{m-1-i}; x_{m-i}) \\
			& = \begin{aligned}[t]
			& \IL(0; x_0, \ldots, x_{m-1}; x_m) + \IL(0; 0, x_0, \ldots, x_{m-2}; x_{m-1}) \\
			& \quad +  \IL(0; 0, 0, x_0, \ldots, x_{m-3}; x_{m-2}) + \cdots + \IL(0; 0,\ldots,0;x_0) \,. \end{aligned}
		\end{align*}
		
		\begin{proof}
			There is a pair-wise cancellation of terms on the right-hand side, when expressed via \autoref{lem:intascor}, leaving \( \corL(x_0,x_1,\ldots,x_m) - \corL(0,\ldots,0) \).  The term \( \corL(0, \ldots,0) = 0 \) by properties of correlators.
		\end{proof}
	\end{Cor}
	\endgroup

	\note{Note:} We now return to the Lie coalgebra of multiple polylogarithms defined in \autoref{sec:coalg:def}.  We can now impose by definition
	\[
		\IL(x_0; x_1,\ldots,x_n; x_{n+1}) \coloneqq \corL(x_1,\ldots, x_n, x_{n+1}) - \corL(x_0, x_1,\ldots, x_n) \,.
	\]
	The expression of \( \corL \) via \( \IL \) in \autoref{cor:corasint} still holds, as does the formula for the coproduct \( \Delta \IL \) from \autoref{eqn:ildelta}.  Based on the affine invariance of correlators of weight \( {\geq}2 \), we obtain the same for the iterated integral.
	
	\begin{Cor}[Affine invariance of \( \IL \)]\label{cor:int:affine}
		For \( n \geq 2 \), the iterated integral \(
			\IL(x_0; x_1,\ldots,x_n; x_{n+1}) 
		\)
		is invariant under affine transformations, i.e. for \( a \in F^\times, b \in F \), we have
		\[
			\IL(x_0; x_1, \ldots, x_n; x_{n+1}) = \IL(ax_0 + b; ax_1 + b, \ldots, ax_n + b; ax_{n+1} + b) \,.
		\]
		
		\begin{proof}
			This follows directly from the same result on correlators, in \autoref{lem:cor:affine}, after writing
			\[
				\IL(x_0; x_1, \ldots, x_n; x_{n+1}) =  \corL(x_1,\ldots, x_n, x_{n+1}) - \corL(x_0, x_1,\ldots, x_n) \,. \qedhere
			\]
		\end{proof}
	\end{Cor}

	 We can also give a more precise meaning to informal viewpoint that correlators are integrals with lower bound equal to \( \infty \).  \Autoref{lem:intascor} is then nothing but the path decomposition and path reversal properties of iterated integral, namely write the path \( x_0 \to x_n \) as a concatenation of \( \infty \to x_0 \) reversed followed by \( \infty \to x_m \).
	
	\begin{Cor}
		The following identify holds
		\[
			\spec_{x_0 = \infty} \IL(x_0; x_1,\ldots,x_m; x_{m+1}) = \corL(x_1,\ldots,x_m, x_{m+1}) \,.
		\]
		
		\begin{proof}
			By the specialisation prescription, we computer \( \spec_{x_0 = \infty} \) by setting \( x_0 = y_0^{-1} \), and taking \( \spec_{y_0 = 0} \).  From \autoref{lem:intascor}, we have
			\[
				\IL(y_0^{-1}; x_1,\ldots,x_m; x_{m+1}) = \corL(x_{m+1}, x_1,\ldots,x_m) - \corL(y_0^{-1}, x_1,\ldots,x_m) \,.
			\]
			We note that the minimum of the valuation \( \nu_{y_0} \) is 0 for the first correlator, and \( -1 \) for the second.  By the specialisation prescription, the first term remains unchanged; we multiply each argument in the second term by \( y_0 \) and set \( y_0  = 0 \).  Since
			\[
				\corL(1, 0, \ldots, 0) = \corL(0, \ldots, 0, 1) = \IL(0, 0, \ldots, 0, 1) + \IL(0, 0, \ldots, 0, 0) =  0 \,,
			\]
			the result follows.
		\end{proof}
	\end{Cor}

	Although not imposed as a relation in \autoref{sec:coalg:def}, one can show that correlators satisfy a reversal symmetry using the coproduct.  (This would already directly follow from the shuffle-product of correlators, but this is not discussed herein.)  Hence overall the correlators are dihedrally symmetric.

	\begin{Lem}\label{lem:correv}
		The following reversal symmetry holds
		\begin{equation}\label{eqn:correv}
			\corL(x_0,x_1,\ldots,x_n) = (-1)^{n+1} \corL(x_n, \ldots, x_1, x_0) \,.
		\end{equation}
		
		\begin{proof}
			For \( n = 1 \), this is trivial by the cyclic symmetry of the correlators, namely  \( \corL(x_0,x_1) = \corL(x_1,x_0) \).  Higher weight holds by induction as follows.  Computing the coproduct, and then applying the reversal to each factor, we have
			\begin{align*}
				& \Delta \corL(x_0, x_1,\ldots,x_n) \\
				 &= \sum_{\mathrm{cyc}} \sum_{1 \leq i \leq n-1} \corL(x_0, x_1,\ldots,x_i) \wedge \corL(x_0, x_{i+1},\ldots,x_n) \\
				&= \sum_{\mathrm{cyc}} \sum_{1 \leq i \leq n-1} (-1)^{i+1} \corL(x_i, \ldots, x_1, x_0) \wedge (-1)^{n+1-i} \corL(x_n, \ldots, x_{i+1}, x_0) \\
				&= (-1)^{n} \sum_{\mathrm{cyc}} \sum_{1 \leq i \leq n-1} \corL(x_i, \ldots, x_1, x_0) \wedge \corL(x_n, \ldots, x_{i+1}, x_0) 
			\end{align*}
			Apply the cyclic symmetry of correlators to put \( x_0 \) at the start, and switch the order of the wedge, we recognise the result as the coproduct of another correlator,
			\begin{align*}
				&= (-1)^{n+1} \sum_{\mathrm{cyc}} \sum_{1 \leq i \leq n-1} \corL(x_0, x_n, \ldots, x_{i+1}) \wedge \corL(x_0, x_i, \ldots, x_1)  \\
				&= (-1)^{n+1} \Delta \corL(x_0, x_n, \ldots, x_1) \\
				&= (-1)^{n+1} \Delta \corL(x_n, \ldots, x_1, x_0) \,.
			\end{align*}
			This shows that the two side of \autoref{eqn:correv} differ by a constant.  Specialising to \( x_0 = x_1 = \cdots = x_n = 0 \) shows that the constant is 0.
		\end{proof}
	\end{Lem}
		
	We can then also more directly give a dihedral symmetry for the iterated integrals.  (The reversal in part ii) below would again already follow from the shuffle product of iterated integrals.)
	
	\begin{Cor}[Dihedral symmetries of \( \IL \)]\label{cor:int:dihedral}
		For an iterated integral
		\(
			\IL(0; x_1,\ldots,x_n; x_{n+1})
		\)
		the following results hold.
		
		\begin{itemize}[leftmargin=1.5em,labelwidth=0em]
		\item[i)] It is equivalent to any non-vanishing cyclic shift, modulo terms with more \( 0 \)'s, i.e. for \( x_i \neq 0 \),\vspace{-0.5ex}
		\[
			\IL(0; x_1,\ldots,x_n; x_{n+1}) \equiv \IL(0; x_{i+1}, \ldots, x_{n+1}, x_1, \ldots, x_{i-1}; x_{i}) \Mod{\text{lower depth}} \,,
		\] 
		
		\item[ii)] It satisfies a reversal symmetry\vspace{-0.5ex}
		\[
			\IL(0; x_1,\ldots,x_n; x_{n+1}) = (-1)^{n+1} \IL(0; x_n,\ldots,x_1; x_{n+1}) \,.
		\]
		\end{itemize}
		
		\begin{proof}
			Property i) follows by using \autoref{lem:intascor} to write
			\begin{align*}
				\IL(0; x_1,\ldots,x_n; x_{n+1}) &= \corL(x_1,\ldots,x_n, x_{n+1}) - \corL(0, x_1,\ldots,x_n) \,, \\
				\IL(0; x_{i+1}, \ldots, x_{n+1}, x_1, \ldots, x_{i-1}, x_{i}) &= \begin{aligned}[t] \corL(&x_{i+1}, \ldots, x_{n+1}, x_1, \ldots, x_{i-1}, x_{i})  \\[-0.5ex]
				& - \corL(0,x_{i+1}, \ldots, x_{n+1}, x_1, \ldots, x_{i-1}) \,. \end{aligned}
			\end{align*}
			When taking the difference, the first correlator terms cancel by their cyclic symmetry.  Writing out the remaining two, which start with 0, via \autoref{cor:corasint} gives integrals with more \( 0 \)'s, (since \( x_i \neq 0 \) is removed and replaced with the starting 0) i.e. lower depth.
			
			Property ii) follows from \autoref{lem:correv}.  Write the second integral as correlators via \autoref{lem:intascor}, then apply \autoref{lem:correv}, to see immediately
				\begin{align*}
			\IL(0;x_n,\ldots,x_1;x_{n+1}) &= \corL(x_n,\ldots,x_1, x_{n+1}) - \corL(0, x_n,\ldots,x_1) \\
			&= (-1)^{n+1} \big( \corL(x_{n+1}, x_1, \ldots, x_n) - \corL(x_1,\ldots,x_n, 0) \big) \\
			&= (-1)^{n+1} \big( \corL(x_1, \ldots, x_n,x_{n+1}) - \corL(0, x_1,\ldots,x_n) \big)  \\
			&= (-1)^{n+1} \IL(0;x_1,\ldots,x_n;x_{n+1}) \,.
			\end{align*}
			Hence the result follows.
		\end{proof}
	\end{Cor}

	\subsection{Multiple polylogarithms, properties, and specialisations}  \label{sec:coalg:mplpropspec}
	
	Taking by definition the iterated integral in \( \mathcal{L}_n(F) \) as
	\[
		\IL(x_0; x_1,\ldots,x_n; x_{n+1}) \coloneqq \corL(x_1,\ldots, x_n, x_{n+1}) - \corL(x_0, x_1,\ldots, x_n) \,,
	\]
	we next introduce the multiple polylogarithms as elements of \( \mathcal{L}_n(F) \).  	For integers \( n_0 \geq 0 \), and \( n_1, \ldots, n_k \geq 1 \), and elements \( x_1, \ldots, x_k \in F^\times \), we define the multiple polylogarithm of depth \( k \), and weight \( n = n_0 + n_1 + \cdots + n_k \) as
	{\small
		\begin{equation}\label{eqn:coalglitoii}
		\begin{aligned}[c]
		& \LiL_{n_0 \; n_1,\ldots,n_k}(x_1,\ldots,x_k) \\[0.5ex]
		& \coloneqq (-1)^k \IL(0; \underbrace{\{0\}^{n_0}}_{n_0}, \underbrace{1, \{0\}^{n_1-1}}_{\smash{n_1}}, \underbrace{x_1, \{0\}^{n_2-1}}_{\smash{n_2}}, \underbrace{x_1 x_2 , \{0\}^{n_3-1}}_{\smash{n_3}}, \ldots, \underbrace{x_1 x_2 \cdots x_{k-1}, \{0\}^{n_{k}-1}}_{\smash{n_k}}; x_1 x_2 \cdots x_k ) \,,
		\end{aligned}
	\end{equation}}%
	where here \( \{x\}^n \)  denotes the string \( x, \ldots, x \), with \( n \) repetitions of \( x \).  (This is a small extension of the integral representation from \autoref{eqn:liasii} and \autoref{eqn:divlitoii}, obtained after applying the affine invariance, \autoref{cor:int:affine}.)  This is an element of \( \mathcal{L}_n(F) \).  When \( n_0 = 0 \), this may be omitted from the notation for simplicity: \( \LiL_{n_1,\ldots,n_k}(x_1,\ldots,x_k) \coloneqq \LiL_{0 \; n_1,\ldots,n_k}(x_1,\ldots,x_k) \). \medskip
	
	\subsubsection*{Specialisations of multiple polylogarithms}  From the viewpoint of power-series, the following limit is straightforward: for any fixed \( 1 \leq i \leq k \),
\begin{align*}
&\lim_{x_i \to 0} \Li_{n_1,\ldots,n_k}(x_1,\ldots,x_k) = 0 \,.
\end{align*}
This holds because when \( x_i = 0 \), every term in the series definition \autoref{eqn:lidef} vanishes.
On the other hand, for any fixed \( 1 \leq i \leq k \), we have
\begin{align*}
& \lim_{x_i \to \infty} \Li_{n_1,\ldots,n_k}(x_1,\ldots,x_k) \equiv 0 \moddp{${<}k$} \,,
\end{align*}
by considering the integral representation.  Namely,
\[
\Li_{n_1,\ldots,n_k}(x_1,\ldots,x_k) = (-1)^k I(0; \tfrac{1}{x_1\cdots x_k}, \{0\}^{n_1-1}, \ldots, \tfrac{1}{x_i \cdots x_k}, \{0\}^{n_i-1},\ldots, \tfrac{1}{x_k}, \{0\}^{n_k-1}; 1) \,.
\]
For \( 1 \leq i \leq k-1 \), one can then take
\[
\lim_{x_i \to \infty}  \Li_{n_1,\ldots,n_k}(x_1,\ldots,x_k)
\]
simply by putting \( x_i = y_i^{-1} \), and setting \( y_i = 0 \).  Although this leads to initial zeros in the integral, these can be shuffled out of the starting positions beforehand to give a meaningful result after specialisation, which has strictly \( {<}k \) non-zero entries, i.e. depth \( {<}k \).  However, taking \( \lim_{x_k\to\infty} \Li_{n_1,\ldots,n_k}(x_1,\ldots,x_k) \) requires more care as the terms which become 0 (all of them) cannot be shuffle away from the starting positions.  However by the stuffle-product one can write \( \Li_{n_1,\ldots,n_k}(x_1,\ldots,x_k) \equiv (-1)^k  \Li_{n_k,\ldots,n_1}(x_k,\ldots,x_1) \moddpi{${<}k$} \), in order to reduce to the previous case.\footnote{One can write \( \Li_{n_1,\ldots,n_k}(x_1,\ldots,x_k) \equiv \pm \Li_{n_1,\ldots,n_k}(x_1^{-1},\ldots,x_k^{-1}) \moddpi{${<}k$} \), by applying the parity theorem \cite{PanzerParity,Gon01} in the analytic case, in order to reduce directly to the specialisation \( y_i = x_i^{-1} \to 0 \).} \medskip

We make this formal in the Lie coalgebra of multiple polylogarithm by using the specialisation framework from \cite[\S2.1]{MR22}, as applied to correlators; this is related to Lemma 2.5 \cite{MR22}.
\begin{Lem}[Arguments to 0]\label{lem:li:argto0}
	The following specialisation holds
	\begin{align*}
	& \spec_{x_i\to0} \LiL_{n_0\;n_1,\ldots,n_k}(x_1,\ldots,x_k) = 0
	\end{align*}
	
	\begin{proof}
		For brevity, write \( X_{i,j} = \prod_{\ell=i}^j x_\ell \).  We have that
		\begin{align*}
		& \LiL_{n_0\;n_1,\ldots,n_k}(x_1,\ldots,x_k) \\[0.5ex]
		& = (-1)^k \IL(0; \{0\}^{n_0}, \tfrac{1}{X_{1,k}}, \{0\}^{n_1\-1}, \ldots, \tfrac{1}{X_{i,k}}, \{0\}^{n_i\-1},\ldots, \tfrac{1}{X_{k,k}}, \{0\}^{n_k\-1}; 1)  \\[0.5ex]
		& = \begin{aligned}[t] 
		& (-1)^k \cor(\{0\}^{n_0},\tfrac{1}{X_{1,k}}, \{0\}^{n_1\-1}, \ldots, \tfrac{1}{X_{i,k}}, \{0\}^{n_i\-1},\ldots, \tfrac{1}{X_{k,k}}, \{0\}^{n_k\-1}, 1)  \\[-0.5ex]
		& \quad - (-1)^k \cor(\{0\}^{n_0},\tfrac{1}{X_{1,k}}, \{0\}^{n_1\-1}, \ldots, \tfrac{1}{X_{i,k}}, \{0\}^{n_i\-1},\ldots, \tfrac{1}{X_{k,k}}, \{0\}^{n_k\-1}, 0) \,.
		\end{aligned}
		\end{align*}
		
		As \( \nu_{x_i}(\tfrac{1}{X_{j,k}}) = -1 \), for \( j \leq i \) and \( \nu_{x_i}(\tfrac{1}{X_{j,k}}) = 0 \), for \( j > i \), the minimum of the valuation \( \nu_{x_i} \) of the arguments for each of the two correlator terms is \( -1 \).  The two correlators differ only in the final argument, which is respectively 0 (with valuation \( \infty \)) or 1 (with valuation \( 0 \)).  In the specialisation prescription above, these arguments are both replaced by 0, therefore both correlators specialise to the same result, which cancels.  We immediately obtain
		\begin{align*}
		& \spec_{x_i=0} \LiL_{n_0\;n_1,\ldots,n_k}(x_1,\ldots,x_k) = 0 \,,
		\end{align*}
		as claimed.
	\end{proof}
\end{Lem}

\begin{Lem}[Arguments to \( \infty \)]\label{lem:li:argtoinfy}
	The following specialisation holds
	\[
	\spec_{x_i\to\infty} \LiL_{n_0\;n_1,\ldots,n_k}(x_1,\ldots,x_k) \equiv 0 \moddp{${<}k$} \,.
	\]
	
	\begin{proof}	
		By the specialisation prescription above, we compute \( \spec_{x_i=\infty} \) by setting \( x_i = y_i^{-1} \), and taking \( \spec_{y_i=0} \) instead.  Write \( X_{i,j} = \prod_{\ell=i}^j x_\ell \) for brevity.  We have
		\begin{align*}
		\spec_{x_i\to\infty} \LiL_{n_0\;n_1,\ldots,n_k}(x_1,\ldots,x_k) 
		& = \spec_{y_i\to0} \LiL_{n_0\;n_1,\ldots,n_k}(x_1,\ldots,x_{i-1},y_i^{-1},x_{i+1},\ldots,x_k) \,.
		\end{align*}
		As before{ \small
			\begin{align*}
			& \LiL_{n_0\;n_1,\ldots,n_k}(x_1,\ldots,x_{i-1},y_i^{-1},x_{i+1},\ldots,x_k) \\[0.5ex]
			& = (-1)^k \begin{aligned}[t] \IL(0; \{0\}^{n_0}, \tfrac{y_i}{X_{1,i\-1} X_{i\+1,k}}, \{0\}^{n_1\-1}, & \ldots, \tfrac{y_i}{X_{i\+1,k}}, \{0\}^{n_i\-1}, \tfrac{1}{X_{i\+1,k}}, \{0\}^{n_{i\+1}\-1}, \ldots, \tfrac{1}{X_{k,k}}, \{0\}^{n_k\-1};1)  \end{aligned} \\[0.5ex]
			& = \begin{aligned}[t]
			&  (-1)^k \begin{aligned}[t] \cor(\{0\}^{n_0}, \tfrac{y_i}{X_{1,i\-1}X_{i\+1,k}}, \{0\}^{n_1\-1}, & \ldots, \tfrac{y_i}{X_{i\+1,k}}, \{0\}^{n_i\-1}, \tfrac{1}{X_{i\+1,k}}, \{0\}^{n_{i\+1}\-1}, \ldots, \tfrac{1}{X_{k,k}}, \{0\}^{n_k\-1};1)  \end{aligned} \\[-0.5ex]
			& - (-1)^k \begin{aligned}[t] \cor(\{0\}^{n_0},\tfrac{y_i}{X_{1,i\-1}X_{i\+1,k}}, \{0\}^{n_1\-1}, & \ldots, \tfrac{y_i}{X_{i\+1,k}}, \{0\}^{n_i\-1}, \tfrac{1}{X_{i\+1,k}}, \{0\}^{n_{i\+1}\-1}, \ldots, \tfrac{1}{X_{k,k}}, \{0\}^{n_k\-1};0) \,. \end{aligned}
			\end{aligned}
			\end{align*}
		}
		
		\case{Case \( i < k \)} We see that the minimum of the valuation \( \nu_{y_i} \) both terms is 0; under the specialisation prescription above, we keep arguments with valuation 0 unchanged, and replace the terms \( \tfrac{y_i}{X_{j,i\-1} X_{i\+k,k}} \) by 0 as they have valuation \( 1 > 0 \).  We obtain
		\begin{align*}
		& \spec_{x_i\to\infty} \LiL_{n_0\;n_1,\ldots,n_k}(x_1,\ldots,x_k)  \\[0.5ex]
		& = \begin{aligned}[t]
		&  (-1)^k  \cor(\{0\}^{n_0}, 0, \{0\}^{n_1\-1}, \ldots, 0, \{0\}^{n_i\-1}, \tfrac{1}{X_{i\+1,k}}, \{0\}^{n_{i\+1}\-1}, \ldots, \tfrac{1}{X_{k,k}}, \{0\}^{n_k\-1};1)   \\[-0.5ex]
		& - (-1)^k \cor(\{0\}^{n_0}, 0, \{0\}^{n_1\-1}, \ldots, 0, \{0\}^{n_i\-1}, \tfrac{1}{X_{i\+1,k}}, \{0\}^{n_{i\+1}\-1}, \ldots, \tfrac{1}{X_{k,k}}, \{0\}^{n_k\-1};0)  \\
		\end{aligned} \\[0.5ex]
		&= (-1)^k \IL(0; \{0\}^{n_0}, 0, \{0\}^{n_1\-1}, \ldots, 0, \{0\}^{n_i\-1}, \tfrac{1}{X_{i\+1,k}}, \{0\}^{n_{i\+1}\-1}, \ldots, \tfrac{1}{X_{k,k}}, \{0\}^{n_k\-1}; 1) \\[0.5ex]
		&= (-1)^{i} \LiL_{n_0 + n_1 + \cdots + n_i \; n_{i+1}, \ldots,n_k}(x_{i+1}, \ldots, x_k)
		\end{align*}
		which has depth \( k-i < k \).
		
		\case{Case \( i = k \)} The minimum of the valuation \( \nu_{y_i}\) for the first correlator is 0, whilst for the second correlator it is 1 (every argument is either 0, with valuation \( \infty \), or \( \tfrac{y_k}{X_{j,k\!-\!1}} \) with valuation 1).  In the first correlator, the terms \( \tfrac{y_k}{X_{j,k\!-\!1}} \) are all replaced by 0, whilst in the second correlator we essentially multiply every argument by \( y_k^{-1} \), and then set \( y_k = 0 \).  After this prescription we obtain
		\begin{align*}
		& \spec_{x_k\to\infty} \LiL_{n_0\;n_1,\ldots,n_k}(x_1,\ldots,x_k)  \\[0.5ex]
		& = \begin{aligned}[t]
		&  (-1)^k  \cor(\{0\}^{n_0}, 0, \{0\}^{n_1\-1}, \ldots, 0, \{0\}^{n_k\-1};1)   \\[-0.5ex]
		& - (-1)^k \cor(\{0\}^{n_0}, \tfrac{1}{X_{1,k\-1}}, \{0\}^{n_1\-1}, \ldots,  \tfrac{1}{X_{i,k\-1}}, \{0\}^{n_i\-1}, \ldots, 1, \{0\}^{n_{k}\-1}, 0)  \\
		\end{aligned} \\[0.5ex]
		& = - (-1)^k \cor(\{0\}^{n_{k}\-1}, 0, \{0\}^{n_0}, \tfrac{1}{X_{1,k\-1}}, \{0\}^{n_1\-1}, \ldots,  \tfrac{1}{X_{i,k\-1}}, \{0\}^{n_i\-1}, \ldots, \tfrac{1}{X_{k\-1,k\-1}}, \{0\}^{n_{k\-1}\-1}, 1)  \,,
		\end{align*}
		where the last equality is obtained by the dihedral symmetry of correlators, and the identity \( \cor(0, \ldots, 0, 1) = \IL(0; 0, \ldots, 0; 1) = 0 \).  Writing this as a multiple polylogarithm via \autoref{cor:corasint}, 
		\[
		= \sum_{\ell=0}^{k-1} (-1)^\ell \LiL_{n_{k\-\ell} + \cdots + n_k + n_0 \; n_1,\ldots,n_{k\-\ell\-1}}(x_1,\ldots,x_{k-\ell-1})
		\]
		which has depth \( k-1 < k \).
	\end{proof}
\end{Lem}

	\begin{Rem}\label{rem:generaldeg}
		We should note that if several arguments degenerate simultaneously to 0 and \( \infty \), one cannot directly apply the identities in \autoref{lem:li:argto0} or \autoref{lem:li:argtoinfy}, however one will still get a reduction to lower (na\"ive) depth: if all arguments of \( \corL \) attain the same minimum valuation \( \nu_{x} \) they would be identically 1 (the first non-zero argument in \autoref{eqn:coalglitoii}), so at least one term vanishes, after rescaling, under \( x \mapsto 0 \).  For example
		\[
			\spec_{x\to0} \LiL_{0\;1,1,1}(x, x^{-1}, y) = -\LiL_{0\;2,1}(1,y) \,.
		\]
		This follows by directly computing as before,
		\begin{align*}
			\spec_{x\to0} \LiL_{0;1,1,1}(x,x^{-1},y) & = \spec_{x\to0} \bigl( -\corL(1, x, 1, y)  + \corL(0; 1, x, 1) \bigr) \\
			& = -\corL(0, 1, 0, 1) - \corL(1, 0, 1, y) \\
			& = \Li_{0\;2,1}(1,y) \,.
		\end{align*}
	\end{Rem}

	\subsubsection*{Quasi-shuffle of multiple polylogarithms} We recall now quasi-shuffle algebras \cite{HoffmanQuasi00,HoffmanIharaQuasi17}, the quasi-shuffle (also called the stuffle) product of multiple polylogarithms, and how this extends \cite{Rud20} to the the motivic multiple polylogarithms \( \LiL_{n_0\;n_1,\ldots,n_k} \) above.  Let \( \mathcal{A} = \{ (n_i, x_i) \mid n_i \in \mathbb{Z}_{>0}, x_i \in F^\times  \} \) be an alphabet.  There is a product on \( \mathcal{A} \) defined by \( (n,x) \cdot (m,y) = (n+m, x y) \).  Let \( \mathbb{Q}\langle \mathcal{A} \rangle \) is be the vector space of \( \mathbb{Q} \)-linear combinations of words (non-commutative monomials) over \( \mathcal{A}\), which we write as \( [n_1,x_1\mid n_2,x_2\mid \cdots\mid  n_k,x_k] \) for notational ease.  Finally define the quasi-shuffle product \( \star \) on \( \mathbb{Q}\langle \mathcal{A} \rangle \) recursively by
	\begin{equation}\label{eqn:quasishuffle:def}
	\left\{ \begin{aligned}
			\mathbbm{1} \star \omega  &= \omega \star \mathbbm{1} = \omega \,, \\
			 \quad a \omega \star b\eta& = (a\cdot b) \big(\omega \star \eta\big) + b \big( a \omega \star \eta \big) + a \big( \omega \star b \eta \big) \,,
	\end{aligned} \right.
	\end{equation}
	where \( a, b \in \mathcal{A} \) are letters, \( \omega, \eta \in \mathbb{Q}\langle\mathcal{A}\rangle \) are words, and \( \mathbbm{1} \) denotes the empty word.  This gives \( (\mathbb{Q}\langle\mathcal{A}\rangle, \star) \) the structure of a quasi-shuffle algebra.  
	
	Define a map
	\begin{align*}
		\Li \colon (\mathbb{Q}\langle \mathcal{A} \rangle, \star) &\to ( \mathbb{C}, \cdot)  \\
			[n_1,x_1 \mid  n_2, x_2 \mid  \cdots \mid  n_k,x_k] &\mapsto \Li_{n_1,n_2,\ldots,n_k}(x_1,x_2,\ldots,x_k) \,.
	\end{align*}
	From the series definition of the multiple polylogarithm, it follows\footnote{Or rather in part, this product structure of multiple polylogarithms, and multiple zeta values in particular, motivated the definition and study of quasi-shuffle algebras.} that this map is a homomorphism of algebras, i.e.
	\[
		\Li(\omega \star \eta) = \Li(\omega) \Li(\eta) \,.
	\]
	In this analytic case, the quasi-shuffle relation for multiple polylogarithms is shown in \cite[\S2.5]{Gon01}, although appears earlier in the context of multiple zeta values \cite{HoffmanQuasi00}.  For motivic objects, Goncharov showed the quasi-shuffle relation holds in \cite[Theorem 1.2]{GonPeriods02} (therein called the first shuffle relation).
	
	Via a coproduct calculation Rudenko \cite[Proposition 3.9]{Rud20} showed again that this holds for motivic multiple polylogarithms, and in fact extended it to include the generalised multiple polylogarithms \( \LiL_{n_0\;n_1,\ldots,n_k} \), \( n_0 \geq 0 \).  More precisely, write
	\[
		\LiL_{\bullet\;n_1,\ldots,n_k}(x_1,\ldots,x_k) = \sum_{n_0=0}^\infty \LiL_{n_0\;n_1,\ldots,n_k}(x_1,\ldots,x_k) \,,
	\]
	and define the map 
	\begin{align*}
		 \LiL_{\bullet} \colon \mathbb{Q}\langle \mathcal{A}\rangle &\to \mathcal{L}_\bullet(F)  \\
		 [n_1,x_1 \mid  n_2, x_2 \mid  \cdots \mid  n_k,x_k] &\mapsto \LiL_{\bullet\;n_1,n_2,\ldots,n_k}(x_1,x_2,\ldots,x_k) \,.
	\end{align*}
	Then we have the following result.
	\begin{Prop}[{Rudenko, Proposition 3.10 \cite{Rud20}}]\label{prop:lil:quasishuffle}
		The generalised quasi-shuffle relation
		\[
			\LiL_\bullet(\omega \star \eta) = 0 \,,
		\]
		holds for all elements \( \omega, \eta \neq \mathbbm{1} \in \mathbb{Q}\langle \mathcal{A} \rangle \).
	\end{Prop}
	
	In particular, by extracting the weight-graded pieces, we have
	\[
		\LiL_{n_0}(\omega \star \eta) = 0 \,,
	\]
	for any \( n_0 \), so any quasi-shuffle relation remains valid after increasing the index \( n_0 \) in \( \LiL_{n_0\;n_1,\ldots,n_k} \).  

	\begin{Eg}\label{eg:lil:stufflerels}
		We have that
		\[
			[1,x] \star [1,y] = [1,x\mid 1,y] + [1,y\mid 1,x] + [2,x y] \,,
		\]
		hence, for any \( n_0 \geq 0 \),
		\[
			\LiL_{n_0\;1,1}(x,y) + \LiL_{n_0\;1,1}(y,x) + \LiL_{n_0\;2}(x y) = 0 \,.
		\]
		
		Likewise, we have that
		\begin{align*}
			[1,x\mid 1,y] \star  [1,z] = {}
			& [1,x\mid 1,y\mid 1,z] + [1,x\mid 1,z\mid 1,y] + [1,z\mid 1,x\mid 1,y] \\
			& \quad + [2,x z\mid 1,y] + [1,x\mid 2,y z] \,,
		\end{align*}
		hence for any \( n_0 \geq 0 \),
		\begin{align*}
			\LiL_{n_0\;1,1,1}(x,y,z) + \LiL_{n_0\;1,1,1}(&x,z,y) + \LiL_{n_0\;1,1,1}(z,x,y) \\
			& {} + \LiL_{n_0\;2,1}(x z, y) + \LiL_{n_0\;1,2}(x, y z) =  0 \,.
		\end{align*}
	\end{Eg}

	\subsubsection*{General multiple polylogarithm symmetries and identities}\label{sec:mpl:identities}  Using the quasi-shuffle structure, and the dihedral symmetries of correlators (resp. iterated integrals, modulo lower depth), one can establish a number of useful general identities in the motivic Lie coalgebra of multiple polylogarithms.  We point these out as results which can be used to simplify the calculations in \autoref{sec:higherZagier4}, and \autoref{sec:higherZagier6}, however we will also show how (in weight 4 and weight 6) they can be obtained directly from degenerations of the quadrangular polylogarithm relations (introduced in \autoref{sec:quadrangular} below), showing the primacy of the quadrangular relation.

	\begin{Lem}[Stuffle antipode]\label{lem:stuffle:antipode}
		The following symmetry holds
		\[
			\LiL_{n_0\;n_1,\ldots,n_k}(x_1,\ldots,x_k) \equiv (-1)^{k-1} \LiL_{n_0\;n_k,\ldots,n_1}(x_k,\ldots,x_1) \moddp{${<}k$} \,.
		\]
		
		\begin{proof}
			If we work modulo depth \( <k \), then the quasi-shuffle product \( \star \) is equivalent to the shuffle product \( \shuffle \) of the corresponding letters \( \mathcal{A} = \{ (n_i,x_i) \mid n_i \in \mathbb{Z}_{>0} , x_i \in F^\times \} \), defined by
				\begin{equation}\label{eqn:shuffledef}
				\left\{ \begin{aligned}
			\mathbbm{1} \shuffle \omega  &= \omega \shuffle \mathbbm{1} = \omega \,, \\
			\quad a \omega \shuffle b\eta& = b \big( a \omega \shuffle \eta \big) + a \big( \omega \shuffle b \eta \big) \,,
			\end{aligned} \right.
			\end{equation}
			as we ignore the lower depth contribution which comes only from the factor \( (a\cdot b) \big(\omega \star \eta\big) \) in \autoref{eqn:quasishuffle:def}.
			
			In an algebra with shuffle product, the following identity holds for any \( k \geq 1 \),
			\[
				\sum_{i=0}^k (-1)^i \cdot (a_1 a_2 \cdots a_i) \shuffle (a_k \cdots a_{i+2} a_{i+1}) = 0 \,.
			\]
			The case \( k = 1 \) is trivial, \( a_1 \shuffle \mathbbm{1} - \mathbbm{1} \shuffle a_1 = 0 \).  The general case follows using \autoref{eqn:shuffledef}, by induction 
			\begin{align*}
				& \sum\nolimits_{i=0}^k (-1)^i \cdot (a_1 a_2 \cdots a_i) \shuffle (a_k \cdots a_{i+2} a_{i+1}) \\[1ex]
				& = \begin{aligned}[t] 
					&\sum\nolimits_{i=1}^{k-1} (-1)^i \cdot 
					\begin{aligned}[t] \big\{
					& a_1 \big(  a_2 \cdots a_i  \shuffle  a_k \cdots a_{i+2} a_{i+1} \big) \\[-1ex]
					& + a_k \big( a_1 a_2 \cdots a_i \shuffle a_{k-1} \cdots a_{i+2} a_{i+1} \big)
					\big\} \\ \end{aligned} \\
					& + (-1)^0 \cdot \mathbbm{1} \shuffle (a_k \cdots a_1) +  (-1)^k \cdot (a_1 \cdots a_k) \shuffle \mathbbm{1} \end{aligned} \\[0.5ex]
				& = \begin{aligned}[t] 
						&a_1 \overbrace{\sum\nolimits_{i=1}^{k} (-1)^i \cdot \big(  a_2 \cdots a_i  \shuffle  a_k \cdots a_{i+2} a_{i+1} \big) }^{=0}  \\
						& + a_k \underbrace{\sum\nolimits_{i=0}^{k-1} (-1)^i \cdot \big( a_1 a_2 \cdots a_i \shuffle a_{k-1} \cdots a_{i+2} a_{i+1} \big)}_{=0}
					 \end{aligned} \\[-1ex]
				& = 0 \,.
			\end{align*}
			Thus for any letters \( a_1,\ldots,a_k \in \mathcal{A} \), we obtain
			\[
				\sum_{i=0}^k (-1)^i \LiL_\bullet( (a_1 a_2 \cdots a_i) \star (a_k \cdots a_{i+2} a_{i+1}) ) \equiv 0 \moddp{${<}k$}
			\]
			By \autoref{prop:lil:quasishuffle}, the summand vanishes for \( 1 \leq i \leq k-1 \), as both words are non-empty, so we are left with
			\[
				(-1)^k \LiL_\bullet( a_1 a_2 \cdots a_k) + (-1)^0 \LiL_\bullet(a_k \cdots a_{2} a_{1}) \equiv 0 \moddp{${<}k$} \,,
			\]
			Specialising to \( a_i = (n_i, x_i) \), and extracting the weight-graded pieces of \( \LiL_\bullet \) gives us the desired identity.
		\end{proof}
	\end{Lem}

	\begin{Rem}
		It is possible to give a more precise description of the lower depth terms, using interpolated multiple polylogarithms \( \Li_{n_1,\ldots,n_k}^t \) (c.f. interpolated multiple zeta values \cite{YamamotoInterpolation13}).  See \S4.2.1.2 and Lemma 4.2.2 in \cite{GlanoisThesis16} or Lemma 3.3 in \cite{GlanoisBasis16}, in particular.  Alternatively \cite[\S3]{HoffmanQuasi20}, and the proof of Theorem 1.3 \cite{HofChMtVSym22} give a result stated in terms of an general interpolated product.
	\end{Rem}

\begin{Cor}[Reversal and \( (k-1,1) \)-shuffle]
	The following identities hold for any \( n_0 \geq 0 \)
	\begin{align*}
	& \LiL_{n_0\;1,\ldots,1}(x_1,\ldots,x_k) \equiv (-1)^{k-1} \LiL_{n_0\;1,\ldots,1}(x_k,\ldots,x_1) \moddp{${<}k$} \\ 
	& \sum_{i=0}^{k} \LiL_{n_0\;1,\ldots,1}(x_1,\ldots,x_i, y, x_{i+1},\ldots,x_{k-1}) \equiv 0 \moddp{${<}k$} \,.
	\end{align*}
	
	\begin{proof}
		The first identity is just the case \( n_1 = \cdots = n_k = 1 \) of \autoref{lem:stuffle:antipode}.  The second identity holds using \autoref{prop:lil:quasishuffle}, by the computation
		\[
			[1,x_1 \mid \cdots \mid 1,x_{k-1}] \star [1,y] = \begin{aligned}[t] 
				& \sum\nolimits_{i=0}^{k} [1,x_1 \mid \cdots \mid 1,x_i \mid 1,y \mid 1,x_{i+1}\cdots  \mid 1,x_{k-1}] \\
				& + \sum\nolimits_{i=1}^{k-1} [1,x_1 \mid \cdots \mid 1,x_{i-1} \mid 1, x_i y \mid 1,x_{i+1} \mid \cdots  \mid 1,x_{k-1}] \,,
				\end{aligned}
		\]
		where the second sum consists of depth \( k-1 \) terms.  (The first sum is the aforementioned \( (k-1,1) \)-shuffle part.)  Applying \( \LiL_{\bullet} \) and extracting the weight graded pieces establishes the result.  (C.f. also \autoref{eg:lil:stufflerels})
	\end{proof}
	\end{Cor}

	These identities in the case \( n_0 = k \) will be useful in \autoref{sec:higherZagier4} and \autoref{sec:higherZagier6}, to manipulate and simplify various identities which arise by degeneration of the corresponding quadrangular polylogarithm functional equation (\autoref{sec:quadrangular}).  We will actually re-prove the relevant identities using the quadrangular polylogarithm functional equation (viz. \autoref{lem:wt4:sh11}, \autoref{lem:21shuffleDeriv}, and \autoref{cor:reverse}) to give further evidence for the (informal) conjecture that all functional equations for multiple polylogarithms follow from the quadrangular polylogarithm functional equation (see the discussion after Theorem 1.1 in \cite[p. 7]{MR22}).

	\begin{Lem}[Inversion, {c.f. Goncharov \cite[\S2.6]{Gon01}, Panzer \cite{PanzerParity} for an analytic version}]
		In weight \( n = n_0 + n_1 + \cdots + n_k \geq 2 \), the following symmetry holds,
		\[
			\LiL_{n_0\;n_1,\ldots,n_k}(x_1,\ldots,x_k) \equiv (-1)^{n + k} \LiL_{n_0 \; n_1,\ldots,n_k}(x_1^{-1}, \ldots, x_k^{-1}) \moddp{${<}k$} \,,
		\]
		
		\begin{proof}
			This follows by the reversal and cyclic symmetry of integrals \autoref{cor:int:dihedral}, the stuffle antipode \autoref{lem:stuffle:antipode}, and the affine invariance of integrals \autoref{cor:int:affine}.  More precisely, write \( X_{i,j} = \prod_{\ell=i}^j x_i \) for brevity (with \( X_{1,0} = 1 \) as an empty product):
			{\small
			\begin{align*}
				& \LiL_{n_0\;n_1,\ldots,n_k}(x_1,\ldots,x_k) \\
				& = \begin{aligned}[t] (-1)^k \IL(0; \{0\}^{n_0}, X_{1,0}, \{0\}^{n_1\-1}, X_{1,1}, \{0\}^{n_2\-1}, \ldots, X_{1,k\-1}, \{0\}^{n_k\-1}; X_{1,k} ) \end{aligned} \\[1ex]
			\tag{reverse}	& = \begin{aligned}[t] (-1)^{k+n+1} \IL(0; \{0\}^{n_k\-1}, X_{1,k\-1}, \{0\}& ^{n_{k\-1}\-1} , X_{1,k\-2}, \ldots, \\[-0.5ex] 
			& \{0\}^{n_2\-1}, X_{1,1}, \{0\}^{n_1\-1}, X_{1,0}, \{0\}^{n_0} ; X_{1,k} ) \end{aligned} \\[1ex]
			\tag{affine}	& = \begin{aligned}[t] (-1)^{k+n+1} \IL(0; \{0\}^{n_k\-1}, \tfrac{X_{1,k\-1}}{X_{1,k\-1}}, \{0\}& ^{n_{k\-1}\-1} ,   \tfrac{X_{1,k\-2}}{X_{1,k\-1}}, \ldots, \\[-0.5ex] & \{0\}^{n_2\-1}, \tfrac{X_{1,1}}{X_{1,k\-1}}, \{0\}^{n_1\-1}, \tfrac{X_{1,0}}{X_{1,k\-1}}, \{0\}^{n_0} ; \tfrac{X_{1,k}}{X_{1,k\-1}} ) \end{aligned} \\[1ex]
				& = (-1)^{n+1} \LiL_{n_k\-1 \; n_{k\-1}, \ldots, n_1, n_0\+1}\big( \tfrac{X_{1,k\-2}}{X_{1,k\-1}} \big/ \tfrac{X_{1,k\-1}}{X_{1,k\-1}}, \ldots,  \tfrac{X_{1,0}}{X_{1,k\-1}} \big/ \tfrac{X_{1,1}}{X_{1,k\-1}} , \tfrac{X_{1,k}}{X_{1,k\-1}} \big/ \tfrac{X_{1,0}}{X_{1,k\-1}} \big) \\
				& = (-1)^{n+1} \LiL_{n_k\-1 \; n_{k\-1}, \ldots, n_1, n_0\+1}\big( x_{k-1}^{-1}, \ldots, x_1^{-1} , X_{1,k} \big) \\[1ex]
			\tag{stuffle antipode}	& \equiv (-1)^{k + n} \LiL_{n_k\-1 \; n_0\+1, n_1, \ldots, n_{k\-1}}\big( X_{1,k}, x_1^{-1}, \ldots, x_{k-1}^{-1} \big) \moddp{${<}k$} \,.
			\end{align*}
			}
			Now write this again as an integral
			{\small
			\begin{align*}
				& = (-1)^{n} \IL(0; \{0\}^{n_k\-1}, 1, \{0\}^{n_0}, \tfrac{X_{1,k}}{X_{1,0}}, \{0\}^{n_1\-1}, \tfrac{X_{1,k}}{X_{1,1}}, \ldots, \tfrac{X_{1,k}}{X_{1,k-2}}, \{0\}^{n_{k\-1}-1} ; \tfrac{X_{1,k}}{X_{1,k-1}} )  \\
				\tag{cycle}
				& \begin{aligned} \equiv (-1)^{n} \IL(0; \{0\}^{n_0}, \tfrac{X_{1,k}}{X_{1,0}}, \{0\}^{n_1\-1}, \tfrac{X_{1,k}}{X_{1,1}}, \ldots, \tfrac{X_{1,k}}{X_{1,k-2}}, \{0\}^{n_{k\-1}-1} , \tfrac{X_{1,k}}{X_{1,k-1}}, \{0\}^{n_k\-1}; 1 ) \qquad \\ \moddp{${<}k$} \end{aligned}\\[1ex]
				\tag{affine} 
				& = (-1)^{n} \IL(0; \{0\}^{n_0}, \tfrac{1}{X_{1,0}}, \{0\}^{n_1\-1}, \tfrac{1}{X_{1,1}}, \ldots, \tfrac{1}{X_{1,k-2}}, \{0\}^{n_{k\-1}-1} , \tfrac{1}{X_{1,k-1}}, \{0\}^{n_k\-1}; \tfrac{1}{X_{1,k}} )  \\
				& = (-1)^{n+k} \LiL_{n_0\;n_1,\ldots,n_k}(x_1^{-1}, \ldots, x_k^{-1}) \,.
			\end{align*}
			}
			 This is the result we wanted to show, and so the lemma is proven.
		\end{proof}
	\end{Lem}

	Likewise, we will show that the weight 4 case \( \LiL_{2\;1,1}(x,y) \) and the weight 6 case \( \LiL_{3\;1,1,1}(x,y,z) \) of inversion follow from the quadrangular polylogarithm relations (in \autoref{cor:wt4:inv}, \autoref{cor:inv}).  However, for future investigations into the identities of higher weight multiple polylogarithms, and the degenerations of higher weight quadrangular polylogarithms, being able to take basic results in this section for granted would be simpler and more convenient, initially.

	\subsection{The depth filtration and the Depth Conjecture} \label{sec:coalg:depth}

	We now come to the Depth Conjecture, which from a high-level viewpoint is the claim that the na\"ive --- \emph{not} motivically defined --- filtration induced by multiple polylogarithm depth in fact agrees with the motivic filtration induced by the coproduct.  Consequently, if this is correct, we expect certain concrete depth reduction identities to hold; these identities should play an important role (see the discussion after Conjecture 1.2 in \cite{MR22}) in establishing various (further) cases of Zagier's polylogarithm conjecture \cite{zagierDedekind}. \medskip

	On the Lie coalgebra \( \mathcal{L}_\bullet(F) \), we have a filtration by the depth of multiple polylogarithms.  Denote by \( \mathcal{D}_k \mathcal{L}_\bullet(F) \) the subspace spanned by polylogarithms of depth \( {\leq}k \).  The space \( \mathcal{D}_1 \mathcal{L}_n(F) \) is usually identified with the higher Bloch group\footnote{In some references (for \( n = 2 \)) this is the pre-Bloch group, and the Bloch group is the subgroup obtained as the kernel of the coboundary map, which is then spanned by special (`primitive') combinations of classical polylogarithms.} \( \mathcal{B}_n(F) \), spanned by classical polylogarithms \( \LiL_n(a) \), \( a \in F^\times \).  The associated graded to the depth filtration is denoted by \( \gr^\mathcal{D}_\bullet(F) \).
	
	We decompose the cobracket \( \Delta = \sum_{1 \leq i \leq j} \Delta_{ij} \) into the weight-graded pieces, where the piece \( \Delta_{ij} \colon \mathcal{L}_{i+j}(F) \to \mathcal{L}_i(F) \wedge \mathcal{L}_j(F) \).  The truncated cobracket \( \overline{\Delta} \) is defined as follows.
	
	\begin{Def}[Truncated cobracket]\label{def:truncatedcoproduct}
		The truncated cobracket is the map
		\begin{equation*}
			\overline{\Delta} \colon \mathcal{L}_\bullet(F) \to \bigwedge\nolimits^2 \mathcal{L}_\bullet(F) 
		\end{equation*}
		defined by the formula \( \overline{\Delta} = \sum_{2 \leq i \leq j} \Delta_{ij} \), i.e. \( \overline{\Delta} \) is obtained from \( \Delta \) by omitting the \( \mathcal{L}_1(F) \wedge \mathcal{L}_{n-1}(F) \) component from the cobracket.
	\end{Def}
	
	In any Lie coalgebra \( (\mathcal{L}, \Delta) \), one has an \( k \)-th iterated cobracket \( \Delta^{[k]} \colon \mathcal{L} \to \CoLie_{k+1}(\mathcal{L}) \) defined as follows (note the index shift).  Let \( \Delta^{[1]} = \Delta \).  Then for \( k > 1 \), define \( \Delta^{[k]} \) as the composition of \( (\Delta^{[k-1]} \otimes \id) \circ \Delta \) with projection \( \CoLie_{k}(\mathcal{L}) \otimes \mathcal{L} \to \CoLie_{k+1}(\mathcal{L}) \).  (By the coJacobi identity, this is well defined, on \( a \wedge b = -b \wedge a \).)  Correspondingly the \( k \)-th iterated truncated cobracket \( \overline{\Delta}^{[k]} \) is defined  and gives a map \( \mathcal{L}_\bullet(F) \to \CoLie_{k+1}(\mathcal{L}_{>1}(F)) \).  As verified in \cite[Proposition 4.1]{MR22}, the truncated iterated coproduct \( \overline{\Delta}^{[k]} \) vanishes on \( \mathcal{D}_{k} \mathcal{L}_\bullet(F) \) by a direct calculation via the formula for the coproduct of an iterated integral (see  \autoref{eqn:ildelta}, from which we deduce this version by the definition of \( \IL \) via \( \corL \))
	\[
		\Delta \IL(x_0; x_1,\ldots,x_m; x_{m+1}) = \sum_{0 \leq i < j \leq m+1} \begin{aligned}[t] \IL(x_0; x_1,\ldots,x_i, x_j,&\ldots, x_m; x_{m+1}) \\[-0.5ex] & {} \wedge \IL(x_i; x_{i+1}, \ldots, x_{j-1}; x_j) \,. \end{aligned} 
	\]
	Moreover \( \overline{\Delta}^{[k]} \gr_{k+1}^\mathcal{D} \mathcal{L}_\bullet(F) \) is expressible via multiple polylogarithms of depth at most 1, and hence one in fact obtains \cite[Proposition 4.1]{MR22} a map
	\(
		\overline{\Delta}^{[k-1]} \colon \gr_k^\mathcal{D} \mathcal{L}_\bullet(F) \to \CoLie_k\big( \bigoplus\nolimits_{n\geq2} \mathcal{B}_n(F) \big)
	\).
	
	The precise version of the Depth Conjecture is then as following.
	\begin{Conj}[Depth Conjecture, {c.f. \cite[Conjecture 7.6]{Gon01}}]\label{conj:depthv2}
		For \( k \geq 2 \), the map
		\[
			\overline{\Delta}^{[k-1]} \colon \gr_k^\mathcal{D} \mathcal{L}_\bullet(F) \xrightarrow{\cong?} \CoLie_k\Big( \bigoplus\nolimits_{n\geq2} \mathcal{B}_n(F) \Big)
		\]
		is an isomorphism.
	\end{Conj}
	For a quadratically closed field, the surjectivity of \( \overline{\Delta}^{[k-1]} \) was proven in \cite[Theorem 5]{CGRR22}. Moreover it was shown \cite[Corollary 6]{CGRR22} that for quadratically closed fields, the Depth Conjecture for \( k = 2 \) implies the Depth Conjecture \( k > 2 \). \medskip
	
	In the case of depth \( k > \frac{n}{2} \), one directly sees that \( \overline{\Delta}^{[k]} \) vanishes on \( \mathcal{L}_n(F) \), as a non-zero result would necessarily be of overall weight \( \geq 2 \cdot k > n \).   In this case, Rudenko \cite[Theorem 1]{Rud20} showed that \( \gr_k^{\mathcal{D}} \mathcal{L}_n(F) = 0 \), by explicitly expressing every weight \( n \) multiple polylogarithm through quadrangular polylogarithms of depth \( \leq \lfloor n/2 \rfloor \). \medskip
	
	The next non-trivial case arises for weight \( n = 2k \), depth \( k \), and will be the focus of this article for \( k = 2, 3 \).  Here, the weight \( 2k \) component of \( \CoLie_k\big( \bigoplus_{n\geq2}  \mathcal{B}_n(F) \big) \) is \( \CoLie_k(\mathcal{B}_2(F)) \) (since the minimal weight is obtained by taking weight 2 in each component, and this is already \( 2k \)), and \autoref{conj:depthv2} claims
	\begin{equation}\label{eqn:deltaneq2k}
		\overline{\Delta}^{[k-1]} \colon \gr_k^\mathcal{D} \mathcal{L}_{2k}(F) \xrightarrow{\cong?} \CoLie_k(\mathcal{B}_2(F)) \,,
	\end{equation}
	is an isomorphism.  Following Matveiakin-Rudenko \cite[\S4.1]{MR22}, this map is surjective, as we have the following lemma.
	
	\begin{Lem}\label{lem:coprodLik111}
		The following iterated coproduct formula holds
		\[
			\overline{\Delta}^{[k-1]} \LiL_{k \; 1, \ldots, 1}(x_1,\ldots,x_k) = - \LiL_2(x_1) \otimes \cdots \otimes \LiL_2(x_k) \in \CoLie_k(\mathcal{B}_2(F))
		\]
		
		\begin{proof}
			This is clearly true for \( k = 1 \), as 
			\[
			 \overline{\Delta}^{[0]} \LiL_{1\;1}(x_1) = \id \big({-}\IL(0; 0, 1; x_1)\big) = \IL(0; 1, 0; x_1) = -\LiL_2(x_1) \,.
			 \]
			  Then to compute \( \overline{\Delta}^{[k-1]} \LiL_{k \; 1, \ldots, 1}(x_1,\ldots,x_k) \) we first compute \( \overline{\Delta} \LiL_{k \; 1, \ldots, 1}(x_1,\ldots,x_k) \) and apply \( \overline{\Delta}^{[k-2]} \otimes \id \) to the result.  Since \( \overline{\Delta}^{[k-2]} \) vanishes on terms of depth \( k - 2 \), we only need to keep the terms of depth \( k - 1 \) in the left-hand factor (say).  This means the right hand factor must have exactly 1 non-zero entry (the depth filtration is motivic in this sense).  Since we compute the truncated cobracket \( \overline{\Delta}^{[k-1]} \), weight of the right hand factor must be 2 (as the weight \( 2k \) part of \( \CoLie_k\big( \bigoplus_{n\geq2}  \mathcal{B}_n(F) \big) \) is \( \CoLie_k(\mathcal{B}_2(F)) \), since the minimal weight is obtained by taking weight 2 in each component, which is already weight \( 2k \)).
			
			Since
			\[
				\LiL_{k\;1,\ldots,1}(x_1,\ldots,x_k) = (-1)^k \IL(0; \{0\}^{k-2}, \underbracket{0, 0 , 1, x_1}, x_1 x_2, \ldots, x_1 x_2 \cdots x_{k-1}; x_1 x_2 \cdots x_k) \,,
			\]
			the only weight 2 depth 1 term which can contribute is \( \IL(0; 0, 1; x_1) \), as highlighted above.  Further left vanishes identically as \( \IL(0; 0, 0; \bullet) = 0 \), further right removes two non-zero entries.  So the relevant part of \( \overline{\Delta} \) is
			\begin{align*}
				\overline{\Delta} \LiL_{k\;1,\ldots,1}(x_1,\ldots,x_k) &= \begin{aligned}[t]
				& (-1)^k \IL(0, \{0\}^{k-1}, x_1x_2 \ldots, x_1x_2\cdots x_{k-1}; x_1x_2\cdots x_k)\otimes \IL(0; 0, 1; x_1) \\[-1ex]
				& \quad\quad + \text{$\llbracket$terms which do not contribute to $\overline{\Delta}^{[k-1]}$$\rrbracket$} \end{aligned} \\
				& = \begin{aligned}[t] 
					&- \LiL_{k-1\;1,\ldots,1}(x_2, \ldots, x_k) \otimes \LiL_2(x_1)  \\[-1ex]
								&\quad\quad  + \text{$\llbracket$terms which do not contribute to $\overline{\Delta}^{[k-1]}$$\rrbracket$} \end{aligned}
			\end{align*}
			Applying \( \overline{\Delta}^{[k-2]} \otimes \id \), by induction gives
			\begin{align*}
				\overline{\Delta}^{[k-1]} \LiL_{k\;1,\ldots,1}(x_1,\ldots,x_k) 
				&= - \big(-\LiL_2(x_2) \otimes \cdots \otimes \LiL_2(x_k)\big) \otimes \LiL_2(x_1) \\
				&= - \LiL_2(x_1) \otimes \LiL_2(x_2) \otimes \cdots \otimes \LiL_2(x_k) \,,
			\end{align*}
			after switching the factors (picking up a minus sign), and projecting to \( \CoLie_{k}(\mathcal{B}_2(F)) \).
		\end{proof}
		\end{Lem}
	
		Since elements \( \LiL_2(x_1) \otimes \cdots \otimes \LiL_2(x_k) \) span \( \CoLie_k(\mathcal{B}_2(F)) \), this establishes the surjectivity of \autoref{eqn:deltaneq2k}.  As Matveiakin and Rudenko describe \cite[\S4.1]{MR22}, to show injectivity, one should construct a map in the opposite direction.  As \( \mathcal{B}_2(F) \) has a presentation \cite[Proposition 2.4]{MR22} (see also \cite{deJeu}) given by
		\[
			\mathcal{B}_2(F) = \frac{\mathbb{Q}[F^\times]}{R_2(F)} \,,
		\]
		where 
		\begin{equation}\label{eqn:fiveterm}
			R_2(F) = \Big\langle\sum\nolimits_{i=0}^4 (-1)^i \{ [w_0 \ldots, \widehat{w_i},\ldots, w_4] \} \Big\rangle \,,
		\end{equation}
		is generated by the five-term relation, Matveiakin and Rudenko note that following sequence is exact
		\[
			R_2(F) \otimes \CoLie_{k-2}(\mathbb{Q}[F^\times]) \to \CoLie_{k}(\mathbb{Q}[F^\times]) \to \CoLie_k(\mathcal{B}_2(F)) \to 0 \,.
		\]
		They give a map
		\begin{align*}
			\mathcal{I} \colon \CoLie_k(\mathbb{Q}[F^\times]) &\to \gr_k^\mathcal{D} \mathcal{L}_{2k}(F) \\
			\{ x_1 \} \otimes \cdots \otimes \{ x_k \} &\mapsto \Li_{k\;1,\ldots,1}(x_1,\ldots,x_k) \,.
		\end{align*}
		which is well-defined by the quasi-shuffle relations for multiple polylogarithms \cite[Proposition 3.10]{Rud20}.  To obtain a map \( \CoLie_k(\mathcal{B}_2(F)) \to \gr_k^\mathcal{D} \mathcal{L}_{2k}(F) \), hence injectivity of \autoref{eqn:deltaneq2k}, the key remaining step is to show that \( \mathcal{I} \) vanishes on \( R_2(F) \otimes \CoLie_{k-2}(\mathbb{Q}[F^\times]) \).  Whence we have the following form of the Depth Conjecture.
		
		\begin{Conj}[Depth conjecture, first obstructed case, Conjecture 1.3 \cite{MR22}]\label{conj:depthconj2k}
			For \( x_2, \ldots, x_{k} \in F^\times \), \( w_0, \ldots, w_4 \in \mathbb{P}^1(F) \), the five-term combination
			\[
				\sum_{i=0}^4 (-1)^i \LiL_{k\;1,\ldots,1}([w_0,\ldots,\widehat{w_i},\ldots,w_4], x_2,\ldots,x_k) 
			\] 
			has depth \( \leq k\-1 \).
		\end{Conj}
		Let us note here that
		\begin{equation}\label{eqn:divtoMPL}
		\Li_{k\;\underbrace{\scriptstyle 1,\ldots,1}_{k}}(x_1,\ldots,x_k) \equiv
		\LiL_{k+1, \underbrace{\scriptstyle 1,\ldots,1}_{k-1}}(\tfrac{1}{x_1\cdots x_k}, x_1,\ldots,x_{k-1}) \moddp{${<}k$} \,,
		\end{equation}
		by the cyclic symmetry from \autoref{cor:int:dihedral}, so this conjecture, and everything following, can be phrased via the usual multiple polylogarithms as well.	

		Matveiakin and Rudenko now break this conjecture into two parts, generalising the symmetry formulae of Zagier (and Gangl) \cite[\S4.1]{gangl-4} and \cite[\S6, Remark p. 7]{ganglSome} in one instance, and generalising the 122-term reduction formula of Gangl \cite[\S4.2]{gangl-4} in the other instance.  This allows them to prove the weight 6 Gangl formula holds by working modulo the hypothetical weight 6 Zagier formulae.  Despite the simpler presentation (and easier discovery in weight 4), the Zagier formulae were significantly more difficult to show in weight 6  (this is the main result in \autoref{sec:higherZagier6} below), and quite-possibly they will be the hardest part to show in general. \medskip
		
		Before explicitly presenting their formulation of the higher Zagier formulae and the higher Gangl formula, we want to introduce a precursor to the higher Zagier formulae, which describes the behaviour of \( \LiL_{k \; 1, \ldots, 1} \) when any argument is specialised to 1, reflecting that \( \LiL_2(1) = 0 \) (on the real level, it \emph{is} a product \( \frac{\pi^2}{6} \)).  This behaviour will be important to know when we come to simplify the results of degenerations of functional equations to boundary components of \( \overline{\mathfrak{M}}_{0,9} \).
			
		\begin{Conj}[Higher Nielsen formulae]\label{conj:highernielsen}
			Let \( (I_i) \in \{0,1\}^k \) be a $k$-tuple specifying which positions (marked by 1) should be specialised to 1, and which positions (marked by 0) should remain generic.  Then the elements
			\begin{equation}
			\spec_{\{ x_{i} \to 1 \mid I_i = 1\}} \LiL_{k \; 1, \ldots, 1}(x_1,\ldots, x_k) \overset{?}{\in} \mathcal{D}_{k-1} \mathcal{L}_{2k}(F)
			\quad\quad\quad  (\mathcal{N}_I) 
			\end{equation}
			have depth at most \( k - 1 \).
		\end{Conj}
		
		We term these the \emph{Nielsen} formulae because in the case of weight 4 depth 2, the function \( \LiL_{2\;1,1}(1,x) \) (i.e. the degeneration \( \mathcal{N}_{10} \)) satisfies 
		\[
		 \LiL_{2\;1,1}(1,x) = \IL(0; 0,0,x^{-1},x^{-1}; 1) = \IL(0; 0,0,1,1; x) \equiv S^\lmot_{2,2}(x)  \,,
		 \]
		  where \( S_{2,2}(x) \coloneqq I(0; 1, 1, 0, 0; x) \) is the weight 4 depth 2 Nielsen polylogarithm.  This is well-known to reduce to classical polylogarithms (see \cite[Proposition 5]{CGRNielsen}, Section 6 \cite{koelbig}, and the related formula in \cite[p.~204]{Lewin}); we show this again explicitly in \autoref{sec:higherZagier4} below, clarifying a point elided over in \cite[Definition 1.8]{GR-zeta4}.
				
		Clearly, the single reduction \( \mathcal{N}_{10\cdots0} \), namely \( \LiL_{k\;1,\ldots,1}(1, x_2,\ldots,x_k) \in \mathcal{D}_{k-1} \mathcal{L}_{2k}(D) \), would imply all other Nielsen reductions (via the quasi-shuffle relations, any argument can be moved, modulo products, to the first position).  However, as will become apparent in \autoref{sec:wt6:nielsen:11x} and \autoref{sec:wt6:nielsen:1xy}, the best route seems to show identities with more degenerations first (starting with \( \mathcal{N}_{1\cdots10} \) say) before progressing to more and more generic identities, with two generic arguments, then three, and so on, until finally the final Nielsen formula with only a single degeneration (say \(\mathcal{N}_{10\cdots0} \)). \medskip

		\paragraph{Description of independent Nielsen degenerations} We can give a more precise description of which independent degenerations will (probably) have to be considered.  The quasi-shuffle relations \( \Li_{k\;1,\ldots,1}(x_1,\ldots,x_k) \) agree, modulo lower depth, with shuffle relations amongst letters \( x_1,\ldots,x_k \).  So the linearly independent elements, modulo lower depth, is given by Lyndon words (of length $k$) over the relevant alphabet.  Consider the alphabet \( \{ 1 < x_1 < \cdots < x_{k-\ell} \} \), for some \( \ell\geq k \); the Lyndon words (of length $k$) over this alphabet describe linearly independent elements when \( \ell \) arguments are degenerated to 1.  Since 1 is chosen as the minimal letter, any Lyndon word \( \neq (1,\ldots,1) \) cannot end in 1 (otherwise rotating right by 1 letter gives a lexicographically smaller word).  Any (candidate) Lyndon word therefore has the form
		\[
			(\overbrace{1, \ldots, 1, x_{i_1}}^{m_1}, \,\, \overbrace{1, \ldots, 1, x_{i_2}}^{m_2}, \,\, \ldots, \,\, \overbrace{1,\ldots,1,x_{i_d}}^{m_d}) \,,
		\]
		with \( m_1 + m_2 + \cdots + m_d = k \), with some restrictions on \( x_{i_j} \)'s.  For a Lyndon word \( m_1 \) will be maximal (i.e. \( m_1 \geq m_i \), \( i = 2, \ldots, d \), otherwise rotating the larger \( m_j \) to the start give a lexicographically smaller word).  By taking \( x_{i_1} = x_1 < x_{i_2} = x_2 < \cdots < \), we obtain the lexicographically smallest rotation, so there can are no further restrictions on the \( m_j \).  So the set of Lyndon words (of length $k$) contains all words of the form 
		\[
		(\overbrace{1, \ldots, 1, x_1}^{m_1}, \,\, \overbrace{1, \ldots, 1, x_2}^{m_2}, \,\, \ldots, \,\, \overbrace{1,\ldots,1,x_{d}}^{m_d}) \,,
		\]
		where 
		\(
			(m_1,\ldots,m_d)
		\)
		is a composition of \( m_1 + \cdots + m_d = k \), with \( m_1 \geq m_i \), \( i = 2,\ldots,d\).  Assuming we show the degeneration \( \mathcal{N}_{1^{m_1\-1}01^{m_2\-1}0\cdots1^{m_d\-1}0} \) (exponents indicating repetition here), i.e.
		\begin{equation}\label{eqn:nielsennecessary}
			\Li_{k\;1,\ldots,1}(\overbrace{1, \ldots, 1, x_1}^{m_1}, \,\, \overbrace{1, \ldots, 1, x_2}^{m_2}, \,\, \ldots, \,\, \overbrace{1,\ldots,1,x_{d}}^{m_d}) \in \mathcal{D}_{k-1} \mathcal{L}_{2k}(F) \,,
		\end{equation}
		we are free to permute and specialise the generic \( x_i \) to obtain all other Lyndon words described by this composition.  
		
		Realistically, therefore, the minimal collection of Nielsen formulae we will need to consider are \( \mathcal{N}_{1^{m_1\-1}01^{m_2\-1}0\cdots1^{m_d\-1}0} \) where is \( (m_1,\ldots,m_d) 
		\) with \( d < k \), is a composition of \( m_1 + \cdots + m_d = k \), \( m_1 \geq m_i \), \( i = 2,\ldots,d\), giving the stipulation that \autoref{eqn:nielsennecessary} holds\footnote{In the OEIS, this is numerated by sequence \href{https://oeis.org/A184957}{\texttt{A184957}} if \( d = k \) is permitted, or sequence \href{https://oeis.org/A186807}{\texttt{A186807}}, up to the order of enumeration of the compositions.}.  (Trivially \( \LiL_{k\;1,\ldots,1}(1,\ldots,1) = 0 \) via the quasi-shuffle product, c.f. a formula for \( \zeta(\{1\}^k, k) \) via single-zeta values \cite[Eq.~(10)]{BBBkfold}, whereas the fully generic \( \LiL_{k\;1,\ldots,1}(x_1,\ldots,x_k) \) is irreducible.)
		
		\begin{Eg}[Weight 4, and weight 6]
			In weight \( 2k \), \( k = 2, 3 \), we will see below that only the following degenerations play a role.
			\begin{gather*}
				\begin{alignedat}{4}
				(\mathcal{N}_{10}) \quad & \LiL_{2\;1,1}(1, x) \in \mathcal{D}_1 \mathcal{L}_4(F) \,, \\ 
				(\mathcal{N}_{110}) \quad &\LiL_{3\;1,1,1}(1,1,y) \in \mathcal{D}_2 \mathcal{L}_6(F) \,, &&
				\quad\quad (\mathcal{N}_{100}) \quad &\LiL_{3\;1,1,1}(1,x,y) \in \mathcal{D}_2 \mathcal{L}_6(F) \,, 
				\end{alignedat}
			\end{gather*}
			corresponding to the composition  $(2)$ of \( k = 2 \), and the compositions $(3),(2,1)$  of \( k=3 \), which have maximal first elements and fewer than \( k \) parts.  (We will already see in \autoref{cor:onevar} and \autoref{lem:onexy_sym2} that some degeneration of the form \( \LiL_{3\;1,1,1}(1, z, z^{-1}) \) plays a role in proving the reduction \( \LiL_{3\;1,1,1}(1,x,y) \in \mathcal{D}_2\mathcal{L}_6(F) \), so there is still some subtlety to which degenerations we need to understand!)
		\end{Eg}
		
		\begin{Eg}[Weight 8]
			In weight \( 2k \), \( k = 4 \), the above description says to look at the following compositions \( (4), (3,1), (2,2), (2,1,1) \) of 4.  Hence we need to consider
			\begin{alignat*}{4}
			(\mathcal{N}_{1110}) \quad &\LiL_{4\;1,1,1,1}(1, 1,1, x) \overset{?}{\in} \mathcal{D}_3 \mathcal{L}_8(F) \,, &&
			\quad\quad (\mathcal{N}_{1100}) \quad &\LiL_{4\;1,1,1,1}(1,1,x,y) \overset{?}{\in} \mathcal{D}_3 \mathcal{L}_8(F) \,, \\[-0.5ex]
			(\mathcal{N}_{1010}) \quad &\LiL_{4\;1,1,1,1}(1,x,1,y) \overset{?}{\in} \mathcal{D}_3 \mathcal{L}_8(F) \,, &&
			\quad\quad (\mathcal{N}_{1000}) \quad &\LiL_{4\;1,1,1,1}(1,x,y,z) \overset{?}{\in} \mathcal{D}_3 \mathcal{L}_8(F) \,,
			\end{alignat*}
			Note: one can check explicitly that \( (1,1,x,y) \) and \( (1,x,1,y) \) are Lyndon words of the alphabet \( \{ 1 < x < y \} \), so indeed \( \LiL_{4\;1,1,1,1}(1,1,x,y) \) and \( \LiL_{4\;1,1,1,1}(1,x,1,y) \) are linearly independent, modulo lower depth.  However, after specialising \( y = x \), \( (1,x,1,x) \) is no longer a Lyndon word (it is periodic), so can be expressed via \( (1,1,x,x) \), modulo products.
		\end{Eg}
%
	
After the higher Nielsen formulae, the higher Zagier formulae tell us about the 6-fold symmetries of \( \LiL_{k\;1,\ldots,1} \) under the anharmonic ratios \( \lambda \mapsto 1-\lambda \) and \( \lambda \mapsto \lambda^{-1} \), which are symmetries of the dilogarithm \( \LiL_2 \), namely \( \LiL_2(x) + \LiL_2(1-x) = \LiL_2(x) + \LiL_2(x^{-1}) = 0 \).

\begin{Conj}[Higher Zagier formulae, Conjecture 4.3 \cite{MR22}]\label{conj:higherzagier}
	The elements
	\begin{align}
\label{eqn:zagier1}	\LiL_{k \; 1,\ldots,1}(x_1, x_2,\ldots,x_k) + \LiL_{k \; 1,\ldots,1}(1-x_1, x_2,\ldots,x_k) \overset{?}{\in} \mathcal{D}_{k-1} \mathcal{L}_{2k}(F) \\
\label{eqn:zagier2} 	\LiL_{k \; 1,\ldots,1}(x_1, x_2,\ldots,x_k) + \LiL_{k \; 1,\ldots,1}(x_1^{-1}, x_2,\ldots,x_k) \overset{?}{\in} \mathcal{D}_{k-1} \mathcal{L}_{2k}(F)
	\end{align}
	have depth at most \( k - 1 \)
\end{Conj}

In the case \( k = 2 \), \autoref{conj:higherzagier} was established by Zagier (and Gangl) \cite[\S4.1]{gangl-4} and \cite[\S6, Remark p.7]{ganglSome}, but we will reprove them in \autoref{sec:higherZagier4} for comparison (see also \cite[\S4.3]{MR22} for a related approach).  The case \( k = 3 \) is the main result of this paper; we prove this in \autoref{sec:higherZagier6}. \medskip

	Now, still following Matveiakin and Rudenko \cite[\S4.1]{MR22}, write \( \mathcal{G}_k \) for the quotient of \( \gr_k^\mathcal{D}\mathcal{L}_{2k}(F) \) by the subspace spanned by the Zagier formulae \autoref{eqn:zagier1}, and \autoref{eqn:zagier2}.  By \autoref{lem:coprodLik111}, \( \overline{\Delta}^{[k-1]} \), vanishes on the subspace, so descends to \( \mathcal{G}_k \).  Then, modulo the higher Zagier formulae, the higher Gangl formula is the following.

\begin{Conj}[Higher Gangl formula, Conjecture 4.4 \cite{MR22}]\label{conj:highergangl}
	The map
	\[
	\overline{\Delta}^{[k-1]} \colon \mathcal{G}_k \to \CoLie_k(\mathcal{B}_2(F))
	\]
	is an isomorphism.
\end{Conj}

	Employing the five-term relation, \( \sum_{i=0}^4 \LiL_2([w_0,\ldots,\widehat{w_i},\ldots,w_4]) = 0 \) which gives a presentation for \( R_2(F) \) as described above (see Proposition 2.4 \cite{MR22}, and \cite{deJeu}, wherein the folklore statement is made precise), the higher Gangl formula equivalently posits that
	\begin{equation}
	\sum_{i=0}^4 (-1)^i \LiL_{k \; 1, \ldots, 1}([w_0, \ldots, \widehat{w_i}, \ldots, w_4], x_2, \ldots, x_k)  \in \mathcal{L}_{2k}(F)
	\end{equation}
	is expressible, modulo products, via
	\begin{itemize}
		\item functions of the form \( \LiL_{k \; 1,\ldots,1}(a_1, a_2,\ldots,a_k) + \LiL_{k \; 1,\ldots,1}(1-a_1, a_2,\ldots,a_k) \), 
		\item functions of the form \( \LiL_{k \; 1,\ldots,1}(a_1, a_2,\ldots,a_k) + \LiL_{k \; 1,\ldots,1}(a_1^{-1}, a_2,\ldots,a_k) \),
		\item  and functions of depth  \( \leq k\-1 \).
	\end{itemize}
	(The first two are then also of depth  \( \leq k\-1 \), according to the Zagier formulae \autoref{conj:higherzagier}.)\medskip
	
	For \( k = 2 \), \autoref{conj:highergangl} was established by Gangl in \cite{gangl-4}; a conceptual explanation was later given by Goncharov and Rudenko in \cite{GR-zeta4} (see \cite[\S4.3]{MR22} for a more refined approach by Matveiakin and Rudenko).  We will indicate again how to derive the case \( k = 2 \) in \autoref{sec:higherZagier4}.  The case \( k = 3 \) was established by Matveiakin and Rudenko in \cite[\S4.4]{MR22}; since \autoref{conj:highergangl} is stated modulo the Zagier formulae, their result is not enough to establish \autoref{conj:depthconj2k} unconditionally for \( k = 3 \).  The main result of this paper \autoref{thm:i411sixfold} establishes the higher Zagier formulae for \( k = 3 \), so together our two results establish  \autoref{conj:depthconj2k} unconditionally for \( k =3 \), and hence \autoref{conj:depthv2} for weight 6 depth 3. \medskip
	
	As already indicated, the Nielsen formulae are necessary to simplify the degenerate terms which appear when considering how (the quadrangular) multiple polylogarithm functional equations degenerate to divisors and boundary components of \( \overline{\mathfrak{M}}_{0,N} \).  Computing and combining such degenerations seems necessary to prove the Zagier formulae and the Gangl formula thereafter; the Nielsen formulae then concentrate much the complexity into the Zagier part, which can then be neglected when investigating the Gangl formula.  This gives some indication of why the Zagier formulae proved to be significantly harder to establish in weight 6 than the Gangl formula.

	\section{Quadrangular polylogarithms and degenerations to \texorpdfstring{$\overline{\mathfrak{M}}_{0,N}$}{M\_\{0,N\}-bar}}\label{sec:quadrangular}
	
	We briefly recall the definition and construction (\autoref{sec:quadrangular:construction}) of quadrangular polylogarithms \( \QLi_{n+k}(x_0, \ldots, x_{2n+1}) \) as introduced in \cite{Rud20}, and recapitulated (as cluster polylogarithms in type \( A \)) in \cite{MR22}.  We also recall the fundamentally important (highly non-trivial) family of functional equations they give rise to in all weights.  
	
	In order to extract results for the `classical' multiple polylogarithms \( \LiL_{n_0\;n_1,\ldots,n_k} \), we need to recall (\autoref{sec:quadrangular:quadrangulation}) the quadrangulation formula \cite[Theorem 1.2]{Rud20}.  In particular we need the terms which appear in the top-depth part of the weight 4 depth 2 quadrangular polylogarithm functional equation, and the weight 6 depth 3 functional equation (\autoref{sec:quadrangular:femoddp}); we give these explicitly and pictorially (referencing the formulae given in \cite{MR-online}).
	
	Finally we recall some properties of cross-ratios (\autoref{sec:quadrangular:anharmonic}).  We also explain (\autoref{sec:quadrangular:M0nstable}) how to specialise the functional equations to boundary components of the Deligne-Mumfrod compactification of \( \mathfrak{M}_{0,N} \); these components are described combinatorially via the notion of stable curves (of genus 0).  We explain computationally how to do this using the specialisation procedure \( \spec_{\lambda\to0} \) from \autoref{sec:coalg:def}
	
	\subsection{General construction and properties.}\label{sec:quadrangular:construction}  We recall the definition of quadrangular polylogarithms via correlators from \cite{Rud20} here.  Let \( \widetilde{\mathcal{C}}_{n,k}\) be the set of all non-decreasing sequences \( \overline{s} = (i_0, \ldots, i_{n+k}) \) of indices
	\[
		0 \leq i_0 \leq i_1 \leq \cdots \leq i_{n+k} \leq 2n+1
	\]
	such that every even number \( 0 \leq k \leq 2n+1 \) appears in \( \overline{s} \) at most once.  Let \( \mathcal{C}_{n,k} \) be the subset of sequences \( \overline{s}\in \widetilde{\mathcal{C}}_{n,k} \) such that at least one element of each pair \( \{ 2i,2i+1 \} \), \( 0 \leq i \leq n \), appears in \( \overline{s} \).  For a sequence \( \overline{s} \in \widetilde{\mathcal{C}}_{n,k} \), define the sign
	\[
		\sign(\overline{s}) = (-1)^{\#\{ \text{even elements in \( \overline{s} \)} \}} \,.
	\]
	For arguments \( x_0, \ldots, x_{2n+1} \in F \), the quadrangular polylogarithm of weight \( n + k \) is defined by summing over the subset of indices \( \mathcal{C}_{n+k} \), as follows
	\[
		\QLi_{n+k}(x_0,\ldots,x_{2n+1}) = (-1)^{n+1} \!\!\! \sum_{\substack{\overline{s} = (i_0, \ldots, i_{n+k}) \\ \in \mathcal{C}_{n+k }}} \!\!\! \sign(\overline{s}) \corL(x_{i_0}, \ldots, x_{i_{n+k}}) \in \mathcal{L}_{n+k}(F) \,.
	\]
	A symmetrised version of the quadrangular polylogarithm of weight \( n+k \) is given by summing over all indices in \( \widetilde{\mathcal{C}}_{n+k} \), namely
	\[		
		\QLi^\sym_{n+k}(x_0,\ldots,x_{2n+1}) = (-1)^{n+1} \!\!\! \sum_{\substack{\overline{s} = (i_0, \ldots, i_{n+k}) \\ \in \widetilde{\mathcal{C}}_{n+k }}} \!\!\! \sign(\overline{s}) \corL(x_{i_0}, \ldots, x_{i_{n+k}}) \in \mathcal{L}_{n+k}(F) \,.
	\]
	Rudenko gives the following lemma, following immediately from the definitions, to relate the quadrangular and symmetrised quadrangular polylogarithms.
	\begin{Lem}[{Lemma 3.3, \cite{Rud20}}]\label{lem:qlisym:to:qli}
		The following relation holds
		\[
			\QLi^\sym_{n+k}(x_0, \ldots, x_{2n+1}) = \sum_{\substack{1 \leq r \leq n \\ 0 \leq i_0 <\cdots < i_r \leq n}} (-1)^{n-r} \QLi_{n+k}(x_{2i_0}, x_{2i_0+1}, x_{2i_1}, x_{2i_1+1},\ldots,x_{2i_r},x_{2i_r+1}) \,.
		\]
		\end{Lem}
	
	The first remarkable result concerning quadrangular polylogarithms is the simple and unique functional equation they satisfy, which follows purely through the combinatorics of their definitions (without any properties of correlators).
	
	\begin{Prop}[{Functional equation of quadrangular polylogarithms, Proposition 3.6 \cite{Rud20}}]\label{prop:quad:fe}
		For any \( m < N-1 \), the following functional equation holds
		\[
			\sum_{\substack{0 \leq r \leq (N-1)/2 \\ 0 \leq i_0 < \cdots < i_{2r+1} \leq N}} (-1)^{i_0 + \cdots + i_{2r+1} - r} \QLi^\sym_m(x_{i_0}, x_{i_1},\ldots, x_{i_{2r+1}}) = 0 \,.
		\]
	\end{Prop}

	A second remarkable result on quadrangular polylogarithms is their universality \cite[\S4.4]{Rud20}.  Any correlator can be expressed via quadrangular polylogarithms, in the following way.
	
	\begin{Prop}[Universality of quadrangular polylogarithms, Proposition 4.10 \cite{Rud20}]\label{prop:universal}
		The following formulae hold,
		\begin{align*}
			\corL(x_0,\ldots,x_{2n}) &= \sum_{\substack{ 0 \leq s \leq 2n+1 \\ 0 \leq i_1 < \cdots < i_s \leq 2n}} (-1)^s \spec_{x_{i_1},\ldots,x_{i_s},a\to\infty} \QLi_{2n}(x_0,\ldots,x_{2n}, a) \\
			\corL(x_0,\ldots,x_{2n+1}) &= \sum_{\substack{ 0 \leq s \leq 2n+2 \\ 0 \leq i_1 < \cdots < i_s \leq 2n+1}} (-1)^s \spec_{x_{i_1},\ldots,x_{i_s}\to\infty} \QLi_{2n+1}(x_0,\ldots,x_{2n+1}) \,. 
		\end{align*}
	\end{Prop}

	Therefore any multiple polylogarithm of weight \( w \) can be expressed via \( \QLi_w(x_0,\ldots,x_{2n+1}) \), in particular with \( n \leq \lfloor w/2 \rfloor \).

	\subsection{Quadrangulation formula}\label{sec:quadrangular:quadrangulation}  The final remarkable result is a quadrangulation formula which expresses any
	\[
		\QLi_{n+k}(x_0,\ldots,x_{2n+1})
	\]
	as a sum of depth \( {\leq}n \) multiple polylogarithms, independently of the weight \( n+k \).  We give an overview (as in \cite[\S3.2]{Rud20}), and refer the reader to \cite[\S3.5]{Rud20} for all details.  We will however give explicit (pictorial) examples for the quadrangulated polylogarithms in depths 1--3 below.
	
	Consider a convex \( (2n+2) \)-gon \( P \) with vertices labelled by distinct points \( x_0,\ldots,x_{2n+1} \in \mathbb{P}^1(F) \).  A quadrilateral inside \( P \) with vertices \( x_{i_0}, x_{i_1}, x_{i_2}, x_{i_3} \)gives rise to a cross-ratio
	\[
		[x_{i_0}, x_{i_1}, x_{i_2}, x_{i_3}] = \frac{(x_{i_0} - x_{i_1})(x_{i_2} - x_{i_3})}{(x_{i_1} - x_{i_2})(x_{i_3} - x_{i_0})} \in F^\times \,,
	\]
	Let \( \mathcal{Q}(P) \) be the set of quadrangulations of \( P \); each quadrangulation \( Q \in \mathcal{Q}(P) \) determines a dual tree \( t_Q \).  Attached to a dual tree is a certain sum of multiple polylogarithms of weight \( n+k \), and depth \( {\leq}n \), defiend via an arborification map \cite[\S3.2, and Definition 3.4]{Rud20}.  An alternative direct expression via products over certain families of subpolygons is given in \cite[Lemma 3.7]{Rud20}.
	
	Combined with the Universality of quadrangular polylogarithms explained in \autoref{prop:universal}, we obtain the important result \cite[Theorem 1.1]{Rud20} that every multiple polylogarithm of weigh \( n \geq 2 \) can be expressed as a linear combination of multiple polylogarithms of depth at most \( \lfloor n/2 \rfloor \).  This was mentioned earlier in \autoref{sec:coalg:depth}, as showing the vanishing of the associated depth graded \( \gr_k^\mathcal{D} \mathcal{L}_n(F) = 0\), for \( 2k > n \). \medskip
	
	This formula can be implemented, and expressions given in small depths.  A list has been tabulated by Matveiakin and Rudenko up to quadrangular polylogarithms of 10 points,  expressed via multiple polylogarithms of depth ${\leq}4$.  We refer the reader to their website \cite{MR-online} for authoritative expressions, given as \( {}_1\mathrm{LiQuad}^+(1, 2, \ldots, 2k-1, 2k) \), for \( 2 \leq k \leq 5 \); we reproduce the results in depths 1, 2, and 3 below, with all terms in the same (lexicographic) order.
	
	Introduce the following shorthand for cross-ratios and the higher cyclic ratio (note the minus sign, in arity 6),
	\begin{align*}
		abcd &\coloneqq [x_a,x_b,x_c,x_d] \coloneqq \frac{(x_a-x_b)(x_c-x_d)}{(x_b-x_c)(x_d-x_a)} \,, \\
		abcdef &\coloneqq [x_a,x_b,x_c,x_d,x_e,x_f] \coloneqq -\frac{(x_a-x_b)(x_c-x_d)(x_e-x_f)}{(x_b-x_c)(x_d-x_e)(x_f-x_a)} \,.
	\end{align*}
	For notational simplicity, we shall extend \( \LiL_{n_0\;n_1,\ldots,n_k} \) by linearity to formal linear combinations of arguments, written in square brackets, i.e.
	\[
		\LiL_{n_0\;n_1,\ldots,n_k}\big( \sum\nolimits_{\ell} \lambda_\ell [x_{\ell,1},\ldots,x_{\ell,k}] \big)
		 \coloneqq \sum\nolimits_{\ell} \lambda_\ell \LiL_{n_0\;n_1,\ldots,n_k}\big( x_{\ell,1},\ldots,x_{\ell,k} \big)
	\]  (We will only ever encounter triples of arguments of a form like \( [abcd, prst, xyzw] \), so there should be no confusion with the above cross-ratios.)
	
	\subsubsection*{Quadrangular polylogarithms in depth 1}  The 4 point quadrangular polylogarithm, expressed via depth 1 multiple polylogarithms, is
	\[
		\QLi_{k+1}(x_1,\ldots,x_4) = -\LiL_{k\;1}\big( + [1234]\big) \,.
	\]
	A graphical representation of the 4-point quadrangular polylogarithm \( \QLi_{k+1}(x_1,\ldots,x_4) \) is given in \autoref{fig:qli1polygon}.  The quadrilateral \( (x_1, x_2, x_3, x_4) \) determines the cross-ratio \( [x_1,x_2,x_3,x_4] \), with the marked vertices \( x_1, x_3 \) indicating offset of the quadrilateral, i.e. where to start reading the cross-ratio, which is well-defined as \( [x_1,x_2,x_3,x_4] = [x_3,x_4,x_1,x_2] \).  The label 1 inside the quadrilateral indicates it is the 1st argument to the function.
	\begin{figure}[ht!]
		\includegraphics[page=1]{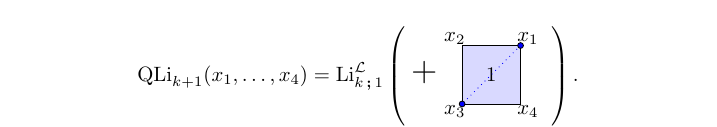}
		\caption{The 4-point quadrangular polylogarithm \( \QLi_{k+1}(x_1,\ldots,x_4)\), expressed via multiple polylogarithms of depth \( 1 \).}
		\label{fig:qli1polygon}
	\end{figure}	
	
	\subsubsection*{Quadrangular polylogarithms in depth 2} The 6 point quadrangular polylogarithm, expressed via depth 2 multiple polylogarithms, is
	\[
	\QLi_{k+2}(x_1,\ldots,x_6) = \LiL_{k\;1,1}\big( + [1236, 3456] - [1256, 3452] + [1456, 1234] \big)
	\]

	A graphical representation of the 6-point quadrangular polylogarithm \( \QLi_{k+2}(x_1,\ldots,x_6) \) is given in \autoref{fig:qli2polygon}.  In the first polygon (always with orientation anticlockwise), the quadrilateral \( (x_1, x_2, x_3, x_6) \) with label 1 determines the cross-ratio \( [x_1,x_2,x_3,x_6] \), and the second \( (x_3,x_4,x_5,x_6) \) with label 2 determines the cross-ratio \( [x_3,x_4,x_5,x_6] \).  The labels indicate into which argument of \( \Li_{k,1,1} \) these cross-ratios are placed.  Overall the first polygon gives the term \( \Li_{k\;1,1}(1236,3456) \), which is indeed first term in \( \QLi_{k+2}(x_1,\ldots,x_6) \).
\begin{figure}[ht!]
	\includegraphics[page=2]{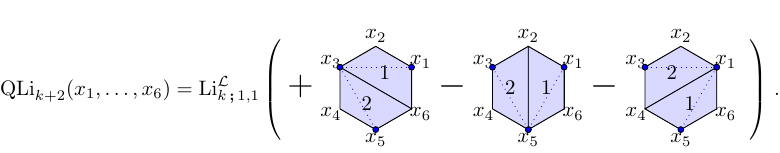}
	\caption{The 6-point quadrangular polylogarithm \( \QLi_{k+2}(x_1,\ldots,x_6)\), expressed via multiple polylogarithms of depth \( 2 \)}\label{fig:qli2polygon}
\end{figure}	

	\subsubsection*{Quadrangular polylogarithms in depth 3} The 8 point quadrangular polylogarithm, expressed via depth \( {\leq}3 \) multiple polylogarithms, is
	\begin{align*}
		&\QLi_{k+3}(x_1,\ldots,x_8) \\
		& = \begin{aligned}[t]
		& \begin{aligned}[t]
		\LiL_{k\;1,1,1}\big( 
		& -[1238,3458,5678]
		+[1238,3478,5674]
		-[1238,3678,3456]
		+[1258,3452,5678]
		\\
		& 
		+[1258,5678,3452]
		-[1278,3472,5674]
		+[1278,3672,3456]
		-[1278,5672,3452]
		\\
		&
		-[1458,1234,5678]
		-[1458,5678,1234]
		+[1478,1234,5674]
		+[1478,5674,1234]
		\\
		&
		-[1678,1236,3456]
		+[1678,1256,3452]
		-[1678,1456,1234]
		\, \smash{\big)}
		\end{aligned}
		 \\[1ex]
		& + \LiL_{k\;1,2}\big( + [1258,3452 \cdot 5678] 
		- [1278,\underbrace{345672}_{= 3456 \cdot 3672}]
		- [1458,1234 \cdot 5678]
		+ [1478,1234 \cdot 5674] \, \big)
		\end{aligned}
	\end{align*}
	
	A graphical representation of the 8-point quadrangular polylogarithm \( \QLi_{k+3}(x_1,\ldots,x_8) \) is given in \autoref{fig:qli3polygon}.  Polygons inside \( \Li_{k\;1,1,1} \) determine corresponding terms in \( \QLi_{k+3}(x_1,\ldots,x_8) \) as before.  For polygons inside \( \Li_{k\;1,2} \), quadrilaterals with the same label give cross-ratios whose \emph{product} appears in the indicated argument.
	\begin{figure}[ht!]
		\includegraphics[page=3]{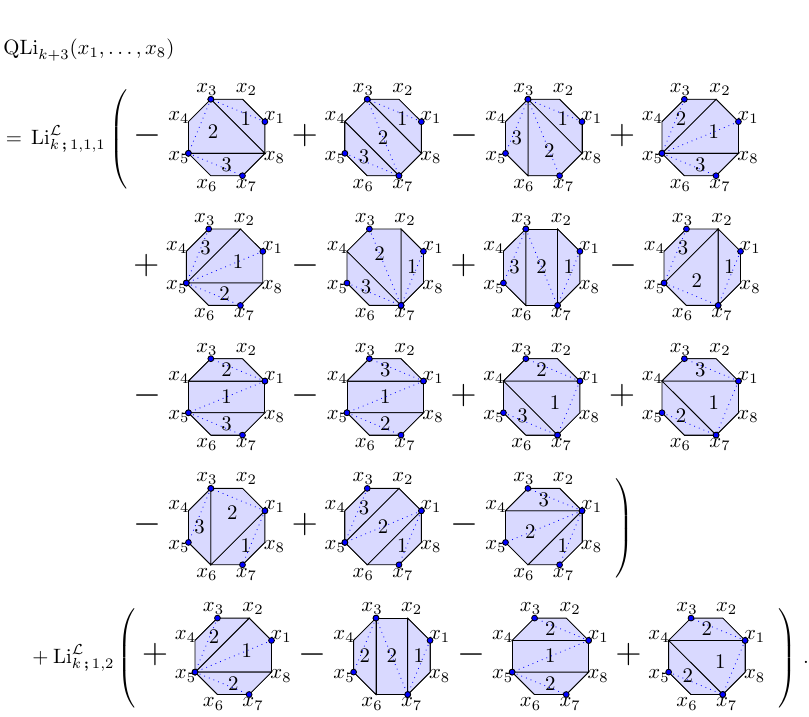}
	\caption{The 8-point quadrangular polylogarithm \( \QLi_{k+3}(x_1,\ldots,x_8)\), expressed via multiple polylogarithms of depth \( {\leq}3 \)}\label{fig:qli3polygon}
\end{figure}

	\subsection{Quadrangular polylogarithm relations, modulo lower depth}\label{sec:quadrangular:femoddp}  In order to derive the weight 4 and weight 6 Zagier formulae, we only need consider identities involving \( \LiL_{2\;1,1} \) and \( \LiL_{3\;1,1,1} \) modulo lower depth.  This allows us to greatly simplify the statement of the quadrangular functional equation from \autoref{prop:quad:fe}.  (One can, just about, keep track of the full functional equation for \( \LiL_{2\;1,1} \) by hand.  The  simplification is much more significant for \( \Li_{3\;1,1,1} \), and we will only give the results modulo depth 2, although the full identity has be manipulated with computer assistance to derive the explicit results in \autoref{app:explicit}, and the \hyperlink{filedescriptions}{ancillary text files}.)\medskip

	Firstly, by the quadrangulation formula \autoref{sec:quadrangular:quadrangulation}, \( \QLi_{n+k}(x_0, \ldots, x_{2n+1}) \) is expressed via polylogarithms of depth \( {\leq}n \); one can see the depth \( n \) part is of the form \( \LiL_{k\;1,\ldots,1} \), with \( n \)-many \( 1 \)'s.  From \autoref{lem:qlisym:to:qli}, we have
	\begin{equation}\label{eqn:qlisym:moddpn}
		\QLi^\sym_{n+k}(x_0, \ldots, x_{2n+1}) \equiv \QLi_{n_k}(x_0,\ldots,x_{2n+1}) \Mod{depth ${<}n$} \,;
	\end{equation}
	only the \( r = n \) term contains any \( \QLi \) of depth \( n \), but \( r = n \) forces the choice \( i_0 = 0, i_1 = 1, \ldots, i_n = n \), giving the above single term on the right hand side of \autoref{lem:qlisym:to:qli}.  Finally, consider \autoref{prop:quad:fe}, in the case \( m = 2k \), \( N = 2k+2 \).  Working modulo term of depth \( {<}k \), a non-zero contribution requires \( r = \lfloor (N-1)/2 \rfloor = k\), as otherwise the \( \QLi \) functions involve \( 2(k-1) + 2 = 2k  \) points, and hence are of depth \( \leq k\-1 \).  The stipulation \( 0 \leq i_0 < \cdots < i_{2r+1} \leq N \), with \( 2r + 1 = 2k + 1\), and \( N = 2k + 2 \) means we omit one point from \( x_0, \ldots, x_{2N+2} \), thus we obtain
	\[
		\sum_{i=0}^{2k+2} (-1)^i \QLi^\sym_{2k}(x_0, \ldots, \widehat{x_i}, \ldots, x_{2k+2}) \equiv 0 \Mod{depth ${<}k$} \,.
	\]
	And the same with \( \QLi^\sym \) replaced by \( \QLi \) after applying the observation from \autoref{eqn:qlisym:moddpn}.
	The final simplification comes from writing
	\[
		f_k(x_1,\ldots,x_{2k+2}) \coloneqq \text{``precisely depth \( k \) part of \( \QLi_{2k}(x_1,\ldots,x_{2k+2})  \)''} \,,
	\]
	i.e. the combination consisting of precisely the \( \LiL_{k\;1,\ldots,1} \)-terms.  In particular,
	\begin{align}
	\label{eqn:f2def} &	f_2(x_1,\ldots,x_6) \coloneqq \LiL_{2\;1,1}\big( + [1236, 3456] - [1256, 3452] + [1456, 1234] \big) \\[1em]
	\notag & f_3(x_1,\ldots,x_8) \coloneqq \\
	\label{eqn:f3def} & \begin{aligned}[c]
		\lif\big( 
		& -[1238,3458,5678]
		+[1238,3478,5674]
		-[1238,3678,3456]
		+[1258,3452,5678]
		\\
		& 
		+[1258,5678,3452]
		-[1278,3472,5674]
		+[1278,3672,3456]
		-[1278,5672,3452]
		\\
		&
		-[1458,1234,5678]
		-[1458,5678,1234]
		+[1478,1234,5674]
		+[1478,5674,1234]
		\\
		&
		-[1678,1236,3456]
		+[1678,1256,3452]
		-[1678,1456,1234]
		\, \smash{\big)} \,,
	\end{aligned}
	\end{align}
	
	Thus for \( k = 2, 3 \), we want to consider the following functional equation, and the consequences of specialising and degenerating it:
	\begin{equation}\label{eqn:qu:k}
		\QU_{2k} \colon \quad\quad \sum_{i=1}^{2k+3} (-1)^i f_k(x_1,\ldots,\widehat{x_i},\ldots,x_{2k+3}) \equiv 0 \Mod{depth ${<}k$} \,.  
	\end{equation}
	
	\noindent$\llbracket$For computational purposes, in the \texttt{Mathematica} worksheets, we work with the full versions of these identity, namely \autoref{prop:quad:fe}, for \( k = 2, 3 \), with \( m = 2k \), \( N = 2k+2 \),
	\[
		\QUf_{2k} \colon 	\sum_{\substack{0 \leq r \leq k \\ 0 \leq i_0 < \cdots < i_{2r+1} \leq 2k+2}} (-1)^{i_0 + \cdots + i_{2r+1} - r} \QLi^\sym_k(x_{i_0}, x_{i_1},\ldots, x_{i_{2r+1}}) = 0 \,.
	\]
	We can then keep track of all lower depth terms at every stage.$\rrbracket$

	\subsection{Symmetries of cross-ratios, and 6-fold anharmonic ratios}\label{sec:quadrangular:anharmonic} We point out that the cross-ratio has a number of straight-forward symmetries, immediate from the definition,
	\[
		[x_1,x_2,x_3,x_4] = [x_2,x_1,x_4,x_3] = [x_3,x_4,x_1,x_2] = [x_4,x_3,x_2,x_1] \,.
	\]
	These are obtained by any double-transposition of the entries, so the cross-ratio is invariant under the action of the Klein four group \( V_4 = \langle (1\,2)(\,34) , (1\,4)(2\,3) \rangle \) on indices.
	
	We also recall briefly that the symmetric group on 3 points \( \Sym_3 \) acts on cross-ratios as the anharmonic group (generated by \( z \mapsto 1-z \), and \( z \mapsto z^{-1} \)).  This is viewed as the action of \( \Sym_4/V_4 \) on indices in \( z \coloneqq [x_1,x_2,x_3,x_4] \); where \( V_4 \) is the above Klein four group, which leaves the cross-ratio invariant.  Explicitly we act on \( x_1,x_2,x_3 \), fixing \( x_4 \), and find for \( z = [z, 0, 1, \infty] \)
	\begin{center}
		\begin{tabular}{c|c}
			$\sigma\in\Sym_3$ &  $z^\sigma = [x_{\sigma(1)},x_{\sigma(2)},x_{\sigma(3)},x_4] $  \\ \hline
			$\id$ &	$ \hspace{-2em} [z,0,1,\infty] = {} \mathrlap{z} $ \\
			$(1 \, 2)$ & $ \hspace{-2em} [0,z,1,\infty] = {} \mathrlap{\frac{z}{z-1}} $ \\
			$(1 \, 3)$ & $ \hspace{-2em} [1,0,z,\infty] = {} \mathrlap{\frac{1}{z}} $ \\
			$(2 \, 3)$ & $ \hspace{-2em} [z,1,0,\infty] = {} \mathrlap{1-z} $ \\
			$(1 \, 2 \, 3)$ & $  \hspace{-2em}[0,1,z,\infty] = {} \mathrlap{\frac{1}{1-z}} $ \\
			$(1 \, 3 \, 2)$ & $  \hspace{-2em}[1,z,0,\infty] = {} \mathrlap{\frac{z-1}{z}} $ 
		\end{tabular}
	\end{center}
	It is useful to note that one obtains, from \( \sigma = (1\,3) \),
	\[
		[x_1,x_2,x_3,x_4]^{-1} = [x_3,x_2,x_1,x_4] \mathrel{\smash{\overset{V_4}{=}}} [x_4,x_1,x_2,x_3] \,,
	\]
	so inversion of cross-ratios corresponds to a cyclic shift of the entries by one step right (or left).

\subsection{Compactification of \texorpdfstring{\( \mathfrak{M}_{0,N} \)}{M\_\{0,N\}}, and degenerations to genus 0 stable curves}\label{sec:quadrangular:M0nstable}

	In order to extract further results from the quadrangular polylogarithm identities \( \QUf_k \), we want to specialise the identities to certain configurations with more projective symmetries so different terms agree (as in \cite[Lemma 6.4]{GR-zeta4} and \S4.3 \cite{MR22}).  Alternatively we allow (several) points to collide -- in a precise way, namely as boundary components of the Deligne-Mumford compactification of \( \mathfrak{M}_{0,N} \) -- so some terms degenerate to \( \LiL_{n_0\;n_1,\ldots,n_k}(\ldots,0,\ldots) \) or \( \LiL_{n_0\;n_1,\ldots,n_k}(\ldots,\infty,\ldots) \), which can then be understood generally via the specialisation procedure described at the start of \autoref{sec:coalg:def} (from \cite[\S2.1]{MR22}).  The latter is how we will establish the identities in \autoref{sec:higherZagier4} and \autoref{sec:higherZagier6}.  
	
	A readable account of stable $n$-pointed curves, which describe these \( \overline{\mathfrak{M}}_{0,n} \)-boundary components, can be found in Kock and Vainsencher \cite[Chapters 0--1]{bookQuantumCohom}, with Knudsen \cite{knudsenCurves} and Keel \cite{keelIntersection} the principal references therein. \medskip
	
	\paragraph{Intuitive idea} Given some (projectively invariant) multiple polylogarithm identity holding for a collection of generic points \( x_1,\ldots,x_4 \in \mathbb{P}^1(F) \), \( F = \mathbb{C} \) say, we want to find what happens as \( x_1 \to x_2 \).  However (see \cite[Example 1.2.1]{bookQuantumCohom}) there is always a projective transformation \( \mathbb{P}^1(F) \to \mathbb{P}^1(F) \), taking \( ( x_1, x_2, x_3, x_4 ) \to ( 0, t, 1, \infty) \).  But via the map \( z \mapsto \frac{z}{t} \), \( t \neq 0 \), the latter configuration is equivalent to \( (0, 1, t^{-1}, \infty) \).  In particular, if the resulting specialisation is to be consistent the condition \( x_1 = x_2 \), i.e. \( t = 0 \) in \( (0, t, 1, \infty) \), must agree with \( x_3 = x_4 \), in \( (0, 1, t^{-1}, \infty) \), arising from \( t = 0 \).
	
	As a point in the moduli space \( \overline{\mathfrak{M}}_{0,4} \), the result of \( x_1 \to x_2 \) colliding is not described by a simple configuration of points on a line, with some coincidence \( x_1 = x_2 \), as this does not capture the condition \( x_3 = x_4 \) needed to retain projective equivalence.  The solution found by Knudsen and Mumford \cite{knudsenCurves} involves admitting certain (=stable) reducible curves to describe the resulting configurations.  A more rigorous derivation is given in \cite[Examples 1.2.6 \& 1.2.7]{bookQuantumCohom}, but the upshot is that as points \( x_1 \) and \( x_2 \) try to coincide, a new copy of \( \mathbb{P}^1(F) \) splits off containing \( x_1, x_2 \), and this gives the same result as letting \( x_3 \to x_4 \) collide, whence:
	\[
		\lim_{x_1 \to x_2}
			\vcenter{\hbox{\includegraphics[page=1]{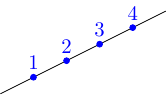}}}
		 \quad \= \quad \vcenter{\hbox{\includegraphics[page=2]{figures/stb_crvs.pdf}}} 
		 \quad \= \quad \lim_{x_3 \to x_4} \vcenter{\hbox{\includegraphics[page=1]{figures/stb_crvs.pdf}}} \,.
	\]
	From the point of view of \( x_3, x_4 \) the points \( x_1,x_2 \) are infinitesimally close to the point of intersection \( x_p \) and to each other, however \( x_1, x_2 \) retain their projective separation on their own copy of \( \mathbb{P}^1(F) \).  
	
	\paragraph{Formal definitions:}  The boundary of \( \overline{\mathfrak{M}}_{0,N} \) is described combinatorially by stable $n$-pointed curves of genus 0 (see \cite[Theorem 1.2.5]{bookQuantumCohom} originally from \cite{knudsenCurves}), which are formally defined as follows.
	
	\begin{Def}[{Stable curve, c.f. \cite[p.~24]{bookQuantumCohom}}]
		An $n$-pointed (genus 0) stable curve is connected curve \( \mathcal{C} \), with $n$ marked points, with the following properties:
		\begin{enumerate}[label=\roman*)]
			\item Every component is isomorphic to \( \mathbb{P}^1(F) \),
			\item The only singular points are simple double points, 
			\item The number of marked and singular points on each component is \( {\geq}3 \).
			\item There are no closed loops
		\end{enumerate}
	\end{Def}

	A stable curve then acts as a combinatorial device to keep track of the information of how the points \( x_i \) are collide, in a way consistent with the underlying projective invariance of the configuration.  As in \cite[\S1.2.10]{bookQuantumCohom}, one can describe the result of more complicated collisions of points:
	
	\begin{itemize}
		\item If three (or more) points collide, and their ratios of approach are distinct (the generic case), then a new \( \mathbb{P}^1(F) \) splits off, containing these points distributed according to their `ratios of approach' (i.e. the derivative or tangent when parametrising \( x_i = x_i(\lambda) \), with the limit \( \lambda \to 0 \))
		\[
		\lim_{\substack{x_1,x_3 \to x_2 \\ \text{generic}}}
		\vcenter{\hbox{\includegraphics[page=3]{figures/stb_crvs.pdf}}}
		\quad \= \quad \vcenter{\hbox{\includegraphics[page=4]{figures/stb_crvs.pdf}}} \,.
		\]
		
		\item If three (or more points) collide, and two ratios of approach agree, a new tree splits off (with extra \( \mathbb{P}^1(F) \)'s to deal with equalities of ratios of approach):
		\[
			\lim_{\substack{x_1,x_3 \to x_2 \\ \text{same rate}}}
			\vcenter{\hbox{\includegraphics[page=3]{figures/stb_crvs.pdf}}}
			\quad \= \quad \vcenter{\hbox{\includegraphics[page=5]{figures/stb_crvs.pdf}}} \,.
		\]
		
		\item If some points approach an intersection point (already on a stable curve),  then a new \( \mathbb{P}^1(F) \) arises at the intersection point, to subsume the points:
			\[
			\lim_{x_3 \to x_p}
			\vcenter{\hbox{\includegraphics[page=6]{figures/stb_crvs.pdf}}}
			\quad \= \quad \vcenter{\hbox{\includegraphics[page=7]{figures/stb_crvs.pdf}}} \,.
			\]
		
	\end{itemize}
	
	In order to compute how some functional equation specialises to the stable curve \( \mathcal{C} \), we can therefore follow the following recipe.  Interpret the stable curve as a description of a (sequence of) limits, parametrise the points \( x_i \) according to these limits, and compute the degeneration via the specialisation homomorphism from \autoref{sec:coalg:def}.
	
	\begin{Eg}[Degenerations of the five-term relation]
		Recall the five-term relation for the dilogarithm (c.f. \autoref{eqn:fiveterm}, variables shifted for convenience), \( V(x_1,\ldots,x_5) = 0 \), where 
		\begin{align*}
			&  V(x_1,\ldots,x_5) \coloneqq \sum_{i=1}^5 (-1)^i \LiL_2([x_1,\ldots,\widehat{x_i}, \ldots, x_5]) \= \\
			& \LiL_2\bigl( 
			-\bigl[\tfrac{(x_2-x_3) (x_4-x_5)}{(x_3-x_4) (x_5-x_2)}\bigr]
			{+}\bigl[\tfrac{(x_1-x_3) (x_4-x_5)}{(x_3-x_4) (x_5-x_1)}\bigr]
			{-}\bigl[\tfrac{(x_1-x_2) (x_4-x_5)}{(x_2-x_4) (x_5-x_1)}\bigr]
			{+}\bigl[\tfrac{(x_1-x_2) (x_3-x_5)}{(x_2-x_3) (x_5-x_1)}\bigr]
			{-}\bigl[\tfrac{(x_1-x_2) (x_3-x_4)}{(x_2-x_3) (x_4-x_1)}\bigr]
			 \bigr) \,.
		\end{align*}
		{\renewcommand{\crv}{{\mathcal{C}^{(2)}_\text{inv}}}
		Consider the stable curve
		\[
			 \crv = 234 \cup_p 15 \=  \vcenter{\hbox{\includegraphics[page=8]{figures/stb_crvs.pdf}}} \,.
		\]
		To compute \( V(x_1,\ldots,x_5) \big\rvert_\crv \) we can proceed as follows.  Firstly, the curve describes the (generic) limit as \( x_1, x_5 \to x_p \) collide ( or equivalently \( x_2, x_3, x_4 \to x_p \) coincide).  Parametrise this by \( x_i = x_p + \lambda y_i \), \( i = 1, 5 \) as \( \lambda \to 0 \).  We substitute this into \( V(x_1,\ldots,x_5) \), and compute \( \spec_{\lambda\to0} \), as in \autoref{sec:coalg:def}.  
		
		For term 1, we see
		\begin{align*}
			-\LiL_2\bigl(\tfrac{(x_2-x_3) (x_4-x_5)}{(x_3-x_4) (x_5-x_2)}\bigr) \big\rvert_{\crv}
			= \spec_{\lambda\to0} -\LiL_2\bigl(\tfrac{(x_2-x_3) (x_4-x_p - \lambda y_5)}{(x_3-x_4) (x_p + \lambda y_5-x_2)}\bigr)\bigr)
			= -\LiL_2\bigl(\tfrac{(x_2-x_3) (x_4-x_p)}{(x_3-x_4) (x_p -x_2)}\bigr)  \,.
		\end{align*}
		Here the specialisation can be directly computed, just by setting \( \lambda = 0 \), as there is no singularity (0 or \( \infty \)) at \( \lambda = 0 \).  Formally, one finds 
		\[
			-\LiL_2\bigl(\tfrac{(x_2-x_3) (x_4-x_p)}{(x_3-x_4) (x_p -x_2)}\bigr) = -\corL(0, 1, 0) + \corL(1, 0, \tfrac{(x_2-x_3) (x_4-x_p)}{(x_3-x_4) (x_p -x_2)}) \,,
		\]
		As the valuation \( \nu_{\lambda-0} \) of each argument is \( {\geq0} \), and attains 0 (for the argument 1, in both correlators, say), \( \spec_{\lambda\to0} \) is computed by setting \( \lambda = 0 \), as claimed.
		
		For term 2, we see
		\begin{align*}
			\LiL_2\bigl(\tfrac{(x_1-x_3) (x_4-x_5)}{(x_3-x_4) (x_5-x_1)}\bigr)  \big\rvert_{\crv}
			= \spec_{\lambda\to0} \LiL_2\bigl(\tfrac{(x_p+\lambda y_1-x_3) (x_4-x_p - \lambda y_5)}{\lambda  (x_3-x_4)(y_5-y_1)}\bigr) 
			= 0 
		\end{align*}
		Here we have to proceed carefully:
		\[
			\LiL_2\bigl(\tfrac{(x_p+\lambda y_1-x_3) (x_4-x_p - \lambda y_5)}{\lambda  (x_3-x_4)(y_5-y_1)}\bigr) = \corL(0,1,0) - \cor(1, 0, \tfrac{(x_p+\lambda y_1-x_3) (x_4-x_p - \lambda y_5)}{\lambda  (x_3-x_4)(y_5-y_1)}) \,.
		\]
		For the second term \( \nu_{\lambda-0} \geq -1 \), attaining \( -1 \) for \( \tfrac{(x_p+\lambda y_1-x_3) (x_4-x_p - \lambda y_5)}{\lambda  (x_3-x_4)(y_5-y_1)} \).  Applying the affine invariance, to multiply by \( \lambda \), we obtain
		\[
			= \corL(0,1,0) - \cor(\lambda, 0,\tfrac{(x_p+\lambda y_1-x_3) (x_4-x_p - \lambda y_5)}{(x_3-x_4)(y_5-y_1)}) 
			\xrightarrow{\lambda\to0} \corL(0,1,0) - \cor(0,0,0) = 0 \,,
		\]
		as claimed.
		
		For terms 3 and 4, we similarly have \( \spec_{\lambda\to0} \llbracket \, \text{terms 3--4} \, \rrbracket = 0 \).  Whilst for term 5, we have
		\[
			-\LiL_2\bigr(\tfrac{(x_1-x_2) (x_3-x_4)}{(x_2-x_3) (x_4-x_1)}\bigr) \big\rvert_{\crv} = -\LiL_2\bigr(\tfrac{(x_p-x_2) (x_3-x_4)}{(x_2-x_3) (x_4-x_p)}\bigr) \,.
		\]
		We have obtained
		\[
			V(x_1,\ldots,x_5) \big\rvert_{\crv} = -\LiL_2\bigl(\tfrac{(x_2-x_3) (x_4-x_p)}{(x_3-x_4) (x_p -x_2)}\bigr) - \LiL_2\bigr(\tfrac{(x_p-x_2) (x_3-x_4)}{(x_2-x_3) (x_4-x_p)}\bigr) = 0\,,
		\]
		By taking \( A = [x_2,x_3,x_4,x_p] \eqqcolon 234p \), in affine coordinates, we obtain (up to a sign)
		\[
			\LiL_2\bigl(A\bigr) + \LiL_2\bigr(A^{-1}\bigr) = 0 \,,
		\]
		which is the inversion relation for \( \LiL_2 \), and one of the two-term symmetries of the dilogarithm \( \LiL_2 \) employed to formulate the higher Zagier formulae (c.f. \autoref{conj:higherzagier}).
		}
		\medskip
	
		{
		\renewcommand{\crv}{{\mathcal{C}^{(2)}_\text{minus}}}
		Likewise, specialising \( V(x_1,\ldots,x_5) \) to the stable curve
		\[
			 \crv = 235 \cup_p 14 \=  \vcenter{\hbox{\includegraphics[page=9]{figures/stb_crvs.pdf}}} 
		\]
		produces
		\[
			V(x_1,\ldots,x_5) \big\rvert_{\crv} = -\LiL_2\bigl(\tfrac{(x_2-x_3) (x_p-x_5)}{(x_3-x_p) (x_5-x_2)}\bigr) + \LiL_2\bigr(\tfrac{(x_p-x_2) (x_3-x_5)}{(x_2-x_3) (x_5-x_p)}\bigr) \,,
		\]
		which in affine coordinates, with \( A^{-1} = [x_2,x_3,x_p,x_5] \), gives
		\(
			-\LiL_2(A^{-1}) + \LiL_2(1-A) = 0 
		\).
		This can be disentangled (using inversion, above), to give the other two-term symmetry
		\[
			\LiL_2(A) + \LiL_2(1-A) = 0 \,.
		\]
		}
	\end{Eg}

	\paragraph{Computational procedures}  For the results in this paper, we need to consider specialisation to the following types of stable curves, for some sets of points \( A, B, C \) and \( D \), as applicable.  The limit procedures used in the computation in each case (in particular in the \texttt{Mathematica} worksheets \wtfourfilename and \wtsixfilename) is described below.\medskip
	
	\begin{center}
	\begin{tabular}{c|c}
		Curve & Limit procedure \\ \hline
		\(
		A \cup_p B \=  \vcenter{\hbox{\includegraphics[page=10]{figures/stb_crvs.pdf}}} 
		\)
		& \(
			\begin{cases} a_i = x_p + \lambda \widetilde{a_i} \,,
			\end{cases}
		\) \linebreak
		as
		\( \lambda \to 0  \)
		\\ \hline
	\(
	A \cup_p B \cup_q C \=  \vcenter{\hbox{\includegraphics[page=11]{figures/stb_crvs.pdf}}} 
	\)
		& \(
		\begin{cases} a_i = x_p + \lambda \widetilde{a_i} \\
		c_j = x_q + \mu \widetilde{c_j} \,,
		\end{cases}
		\) \linebreak
		as
		\( \lambda, \mu \to 0  \)
	 \\ \hline
	\(
	(A_p, B_q, C_r) \cup D \=  \vcenter{\hbox{\includegraphics[page=12]{figures/stb_crvs.pdf}}} 
	\)
		& \(
		\begin{cases} a_i = x_p + \lambda \widetilde{a_i} \\
		b_j = x_q + \mu \widetilde{b_j} \,, \\
		c_h = x_r + \nu \widetilde{c_h} \,,
		\end{cases}
		\) \linebreak
		as
		\( \lambda, \mu, \nu \to 0  \)
	\end{tabular}
	\end{center}
	After taking these limits, the resulting identity (might) involve variables \( \widetilde{a_i}, \widetilde{b_j}, \widetilde{c_k} \), in this situation it is convenient to simply relabel the new points so as to remove the hat.
	\medskip

	\paragraph{Well-definedness of cross-ratios on \( \mathcal{C} \)}
	
	Finally, we note that under specialisation to any stable curve \( \mathcal{C} \), the (generic) cross-ratio \( \mathrm{r} \coloneqq [x_1,x_2,x_3,x_4] \), \( x_i \neq x_j \) is well defined, as an element of \( F \cup \{ \infty \} \), allowing also for the cases \( \mathrm{r} = 0, 1, \infty \), and can be determined by suitable substitutions.
	
	This holds because one can always find a component \( C \) of \( \mathcal{C} \) (see below) on which there are at least three distinct images \( z_1,\ldots,z_4 \) of \( x_1,\ldots,x_4 \), where \( z_i \) denotes the limiting value of \( x_i \) on \( C \).  If there are 4 distinct images, then 
	\[
		\mathrm{r} \big\rvert_{\mathcal{C}} = [x_1,x_2,x_3,x_4] \big\rvert_{\mathcal{C}} = [z_1,\ldots,z_4] \,.
	\]
	Otherwise some \( z_i = z_j \).  Without loss of generality, one can assume \( j = 4 \) (c.f. \autoref{sec:quadrangular:anharmonic}).  Then (in a way formalised by writing out the limit),
	\begin{align*}
		\mathrm{r} \big\rvert_{\mathcal{C}} = [x_1,x_2,x_3,x_4] \big\rvert_{\mathcal{C}} = \begin{cases}
			[z_4,z_2,z_3,z_4] = \infty \,, & \text{if \( z_1 = z_4 \)} \\
			[z_1,z_4,z_3,z_4] = 1 \,, & \text{if \( z_2 = z_4 \)} \\
			[z_1,z_2,z_4,z_4] = 0 \,, &  \text{if \( z_3 = z_4 \)}  \,. 
		\end{cases}
	\end{align*}

	To see why such \( C \) exists, consider an undirected path \( P_{i,j} \) from \( x_i \) to \( x_j \).  If \( P_{1,2} \) and \( P_{3,4} \) do not visit the same component, then \( P_{1,3} \) and \( P_{2,4} \) will (otherwise we would be able to form a non-trivial loop).  So, we can assume \( P_{1,2} \) and \( P_{3,4} \) do then visit the same component \( C_1 \) of \( \mathcal{C} \).  
	
	Find the first common component \( C_2 \) visited by these paths (in some arbitrary direction along the line of shared components), then \( C_2 \) is the component we seek.  If \( P_{1,2} \) starts at \( C_2 \), the its start point is a marked point on \( C_2 \), otherwise the image of its start point is a simple double point on \( C_2 \), likewise for \( P_{3,4} \).  For the image of the starting points: if both paths start at \( C_2 \), we get two distinct marked points, if exactly one starts at \( C_2 \) we get a marked point, and a simple double point (necessarily distinct), and if neither start at \( C_2 \) we get two distinct double points (as this component is the first common one).
	
	The image of the end points of \( P_{1,2} \) and \( P_{3,4} \) may agree on \( C_2 \) (if both paths visit the same component next), but it is necessarily distinct from the two images of the start points (the components visited before and after \( C_2 \) on the path \( P_{i,j} \) are distinct, to avoid loops). \medskip
	 
	We shall use this framework to determine \emph{combinatorially} how \( \QU_{2k} \) degenerates on certain stable curves, as all arguments of its highest depth terms are given by cross-ratios of distinct points \( x_1,\ldots,x_{2k+3} \).  The specialisations of \( \LiL_{k\;1,\ldots,1} \) with any arguments \( 0 \) or \( \infty \) necessarily degenerate to lower depth via \autoref{lem:li:argto0}, \autoref{lem:li:argtoinfy} and \autoref{rem:generaldeg}.
	
	\section{The Zagier formulae in weight 4 (revisited)}\label{sec:higherZagier4}
	
	\note{$\llbracket$Calculations} The reader can consult \texttt{Mathematica} worksheet \wtfourfilename, attached to the \texttt{arXiv} submission, to verify the calculations and proofs given in this section.  The worksheet also derives the explicit expression for the Gangl-formula proven in \autoref{cor:wt4:gangl}.$\rrbracket$%
	\medskip
	
	We start by revisiting weight 4, to re-establish the Zagier formulae (\autoref{prop:wt4:zagier}) as in \cite{MR22,GR-zeta4}. In doing so, we shall avoid an explicit projective involution, but we will instead invoke some birational map which takes us outside of purely \( \overline{\mathfrak{M}}_{0,7} \) degenerations.  We also will derive from \( \QUf_4 \) the basic symmetries and identities from \autoref{sec:mpl:identities}, as well as the (elided over) Nielsen reduction (\autoref{lem:wt4deg}) showing \( \liftwo(1,x) \equiv 0 \modonei \) needed in \cite{GR-zeta4}, all of which gives evidence that \( \QUf_{2k+2} \) is the fundamental identity for polylogarithms.  
	
	Together with the proof of the higher Zagier formulae in \autoref{sec:higherZagier6} below, this sets weight 4 and weight 6 into the same framework.  This suggests that the higher Zagier and higher Gangl formulae in all even weights can be handled similarly: by the combinatorial interplay of certain stable curve degenerations, with the (recursive) quadrangular polylogarithm definition.   \medskip
	
	Recall from \autoref{eqn:f2def} (along with the graphical viewpoint explained in \autoref{fig:qli2polygon})\vspace{-1em}
	\[
	\includegraphics[page=4]{figures/fctl_eqns.pdf}
	\]
	We shall specialise and degenerate the following functional equation, \autoref{eqn:qu:k} ($k=2$), 
	\[
	\QU_4 \coloneqq \sum_{i=1}^7 f_2(x_1,\ldots,\widehat{x_i},\ldots,x_7) \equiv 0 \modone \,,
	\]
	(or even \( \QUf_4 \), incorporating the lower depth terms) in order to derive these identities.  In the first few lemmas, we shall spell out the combinatorial details of the degenerations (which terms vanish, which terms cancel, which terms contribute) in detail.  This will allow us to be briefer in later examples, and concentrate better on the more intricate combinatorics of weight 6 \autoref{sec:higherZagier6}.
	
	We shall make use of \( \LiL_4(x) \coloneqq \Li_{0\;4}(x) = -\Li_{3\;1}(x) \) in depth 1 (as the former function is more familiar).  We shall also take for granted the depth 1 inversion \( \LiL_{4}(x^{-1}) = -\LiL_{4}(x) \), although it could be interesting to try to derive this directly from \( \QUf_4 \).
	
	{
		\renewcommand{\crv}{\mathcal{C}^{(4)}_{\mathrm{RI}}}
		\begin{Lem}[Reverse-Inverse, $\revinvsc^{(4)}$]\label{lem:wt4invrev}
			Degeneration of \( \QUf_4 \) to the stable curve 
			\[
			\crv = 17 \cup_p 246 \cup_q 35 = \,\, \vcenter{\hbox{\includegraphics[page=1]{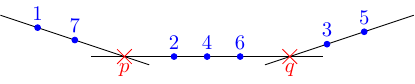}}}
			\,,
			\]
			produces the identity
			\biggerskip
			\[
			\liftwo(2q4p,4q6p) + \liftwo(p4q6, p2q4) = -\LiL_4(2q4p) + \LiL_4(4q6p) \,.
			\]
			In affine coordinates with \( A = 2q4p, B = 4q6p \), this is
			\[
			\liftwo(A, B) + \liftwo(B^{-1},A^{-1}) = -\LiL_4(A) + \LiL_4(B) \,.
			\]
		\end{Lem}
			
		\begin{Rem}
			This extends to a Reverse-Inverse result for \( \lif(x,y,z) \) in weight 6, c.f. \autoref{lem:revinvDeriv}.
		\end{Rem}
			
		\begin{proof}[Proof (modulo depth 1)]
		We treat the cases with \( 2 \leq i \leq 6 \), with \( i = 1 \), and with \( i = 7 \) separately in the functional equation 
		\begin{equation*}
			\QU_4: \qquad \sum_{i=1}^7 (-1)^i f_2(x_1,\ldots,\widehat{x_i}, \ldots, x_7) \equiv 0 \modone \,,
		\end{equation*}
		from \autoref{eqn:qu:k},  Recall \(
		f_2(x_1,\ldots,x_6) = \liftwo\big( [1236, 3456] - [1256, 3452] + [1456, 1234] \big) 
		\), as in \autoref{eqn:f2def}.  We show only \( f_2(\widehat{x_1}) \) and \( f_2(\widehat{x_7}) \) contribute; the rest vanish, modulo depth 1.

		\case{Case \( f_2(\widehat{x_i}) \), with \( 2 \leq i \leq 6 \)} Argument 1 of each term in \( f(y_1,\ldots,y_6) \) has the form \( \phi_1 \coloneqq  [y_1,a,b,y_6] \), with \( a,b \in \{ y_2,y_3,y_4,y_5 \} \).  For \( 2 \leq i \leq 6 \),  we see
		\[
			f_2(\widehat{x_i}) \coloneqq f_2(\overset{{}_{y_1}}{\boldsymbol{x_1}},x_2,\ldots,\widehat{x_i},\ldots,x_6, \overset{{}_{y_6}}{\boldsymbol{x_7}}) \,,
		\]
		so point 1 is always \( x_1 \), and point 6 is always \( x_7 \) (bold for emphasis).  Therefore, for these \( f_2(\widehat{x_i}) \), argument 
		\( \phi_1 = [x_1,a',b',x_7] \), for some \( a',b' \in \{ x_2,x_3,x_4,x_5,x_6 \} \).  Since \( x_1,x_7 \) lie on one component of \( \crv  =  17 \cup_p 246 \cup_q 35  \), and \( a', b' \) lie on some other component(s), degenerating \( \phi_1 = [x_1,a',b',x_7] \) to \( \crv \) gives
		\[
			\phi_1 \big\rvert_{\crv} = [x_1,a',b',x_7] \,\, \big\rvert_{\crv} = [p,a',b',p] = \infty \,.
		\]
		By \autoref{lem:li:argtoinfy}, we know \( \LiL_{2\;1,1}(\infty,z_2) \equiv 0 \modonei \), hence all terms in \( f_2(\widehat{x_i}) \) individually vanish, modulo depth 1.  Overall we obtain \( f_2(\widehat{x_i}) \equiv 0 \modonei \), for \( 2 \leq i \leq 6 \).
		
		\case{Case \( f_2(\widehat{x_1}) \)} We directly write down the summand \( f_2(\widehat{x_1}) \) and how it degenerates to \( \crv \).  We have
		\[
			f_2(x_2,\ldots,x_7) = \liftwo\big( [2347, 4567] - [2367, 4563] + [2567, 2345] \big) \,.
		\]
		It is convenient to extend the notation \( abcd = [x_a,x_b,x_c,x_d] \) to allow \( a,b,c,d \in \{ p, q, r \} \), with the interpretation \( x_p = p, x_q = q, x_r = r \).  Thus we may write \( 2q4p = [x_2,q,x_4,p] \), while still keeping the notation quite compact.  We now find
		\begin{align*}
			f_2(\widehat{x_1}) \, \big\rvert_{\crv} 
			& = \liftwo\big( [2q4p, 4q6p] - [2q6p, \smash{\overbrace{4q6q}^{=1}}] + [2q6p, \smash{\overbrace{2q4q}^{=1}}] \big) \\
			& = \liftwo\big(2q4p, 4q6p) \,,
		\end{align*}
		as term 2 and term 3 exactly cancel.
		
		\case{Case \( f_2(\widehat{x_7}) \)}  We directly write down the \( f_2(\widehat{x_7}) \) and how it degenerates to \( \crv \). As
		\[
		f_2(x_1,\ldots,x_6) = \liftwo\big( [1236, 3456] - [1256, 3452] + [1456, 1234]  \big) \,,
		\]
		we find
		\begin{align*}
		f_2(\widehat{x_7}) \, \big\rvert_{\crv} 
		& = \liftwo\big([p2q6, \smash{\overbrace{q4q6}^{=1}}] - [p2q6, \smash{\overbrace{q4q2}^{=1}}] + [p4q6, p2q4] \big) \\
		& = \liftwo\big(p4q6, p2q4) \,,
		\end{align*}
		as term 1 and term 2 exactly cancel.
	
		\case{Overall} Taking the alternating sum of these contributions for \( \QU_4 \), as in \autoref{eqn:qu:k}, gives
		\begin{align*}
			\QU_4 \big\rvert_{\crv} 
			&\equiv (-1)^1 \overbrace{\liftwo\big(2q4p, 4q6p)}^{f(\widehat{x_1})} {} + 5 \cdot 0 + (-1)^7 \overbrace{\liftwo\big(p4q6, p2q4)}^{f(\widehat{x_7})}  \\
			& \equiv - \liftwo(2q4p, 4q6p) - \liftwo(p4q6, p2q4) \modone \,.
		\end{align*}
		This is equivalent to the identity given in the statement of the lemma.
		\end{proof}
	
	}
	
	{
		\renewcommand{\crv}{\mathcal{C}^{(4)}_{\mathrm{sh}}}
		\begin{Lem}[The (1,1)-shuffle, $\shsym{1,1}^{(4)}$ or $\revsc^{(4)}$]\label{lem:wt4:sh11}
			Degeneration of \( \QUf_4 \) to the stable curve 
			\[ 
			\crv = 134 \cup_p 2567  = \,\,
			\vcenter{\hbox{\includegraphics[page=2]{figures/wt4.pdf}}}
			\,,
			\]
			produces the identity
			\biggerskip
			\[
			\liftwo(134p, p567) + \liftwo(p567, 134p) = -3 \LiL_4(134p \cdot p567 ) \,.
			\]
			In affine coordinates with \( A = 134p, B = p567 \), this is
			\[
			\liftwo(A, B) + \liftwo(B,A) = -3 \LiL_4(A B) \,.
			\]
		\end{Lem}

		\begin{Rem}
		This extends to a (2,1)-shuffle result for \( \lif(x,y,z) \) in weight 6, see \autoref{lem:21shuffleDeriv}.
		\end{Rem}

		\begin{proof}[Proof (modulo depth 1)]
			We treat the cases	\( f_2(\widehat{x_i}) \)  with \( 5 \leq i \leq 7 \), with \( i = 3, 4 \), with \( i = 1 \), and with \( i = 2 \), separately.  We show only \( f_2(\widehat{x_2}) \) contributes, and the rest vanish (or cancel), modulo depth 1.  In this proof, we will spell out the details very carefully.
			
			\case{Case \( f_2(\widehat{x_i}) \), with \( 5 \leq i \leq 7 \)}  Argument 2 of terms 1 and 2, in \( f(y_1,\ldots,y_6) \) has the form \( \phi_1 \coloneqq [y_3, y_4, a, b] \), with \( a, b \in \{ y_2, y_5, y_6 \} \), while argument 1 of term 3 in \( f(y_1,\ldots,y_6) \) has the form \( \phi_2 \coloneqq [y_1, y_4, c, d] \), with \( c,d \in \{ y_5, y_6 \} \).  For \( 5 \leq i \leq 7\), we see
			\begin{align*}
				f_2(\widehat{x_i}) &= f_2(\overset{{}_{y_1}}{\boldsymbol{x_1}},\overset{{}_{y_2}}{x_2},\overset{{}_{y_3}}{\boldsymbol{x_3}},\overset{{}_{y_4}}{\boldsymbol{x_4}},x_5,\ldots,\widehat{x_i},\ldots,x_7) 
			\end{align*}
			so points 1, 3 and 4 are always \( x_1, x_3 \) and \( x_4 \), respectively (bold for emphasis).  Therefore, for these \( f_2(\widehat{x_i}) \), argument \( \phi_1 = [x_3, x_4, a', b'] \), with \( a', b' \in \{ x_2, x_5, x_6, x_7 \} \), and argument \( \phi_2 = [x_1, x_4, c', d'] \) with \( c', d' \in \{ x_5, x_6, x_7 \} \).  Since \( x_1, x_3, x_4 \) are on one component of \( \crv \), and \( x_2, x_5, x_6, x_7 \) are on the other component, we have
			\begin{align*}
				\phi_1 \big\rvert_{\crv} &= [x_3, x_4, a', b'] \,\, \big\rvert_{\crv} = [p,p,a',b'] = 0 \,,  \\
				\phi_2 \big\rvert_{\crv} &= [x_1, x_4, c', d'] \,\, \big\rvert_{\crv} = [p,p,c',d'] = 0 \,.
			\end{align*}
			By \autoref{lem:li:argto0}, we know \( \liftwo(0, z_2) = \liftwo(z_1, 0) = 0 \), hence all terms in \( f_2(\widehat{x_i}) \) individually vanish, modulo depth 1 (and actually on the nose).  Overall we obtain \( f_2(\widehat{x_i}) \equiv 0 \modonei \), for \( 5 \leq i \leq 7 \).
			
			\case{Case \( f_2(\widehat{x_i}) \), with \( i = 3, 4 \):}  Each argument of each term in \( f(y_1,\ldots,y_8) \) has either the form 
			\( \phi_1 \coloneqq [y_1, a, y_3, b] \), \( \phi_2 \coloneqq [y_3, a,b,c] \) or \( \phi_3 \coloneqq [y_1,a,b,c] \), with \( a,b, c \in \{ y_2,y_4,y_5,y_6 \} \).  For 
			\[
				f_2(\widehat{x_3}) = f_2(\overset{{}_{y_1}}{\boldsymbol{x_1}},x_2,\overset{{}_{y_3}}{\boldsymbol{x_4}},x_5,x_6,x_7) \,,
			\]
			which has point 1 as \( x_1 \), and point 3 as \( x_4 \) (bold), these argument forms specialise to
			\[
				\phi_1 = [x_1, a', x_4, b'] \,, \quad \phi_2 = [x_4, a', b', c'] \,, \quad \phi_3 = [x_1, a',b',c'] \,,
			\]
			with \( a',b',c' \in \{ x_2, x_5, x_6, x_7 \} \) respectively.  As points \( x_1, x_4 \) are on one component of \( \crv \), while points \( x_2, x_5, x_6, x_7 \) are on the other, we readily compute the degenerations
			\[
				\phi_1 \big\rvert_{\crv} = [p, a', p, b'] = 1 \,, \quad \phi_2 \big\rvert_{\crv} = [p, a', b', c'] \,, \quad \phi_3 \big\rvert_{\crv} = [p, a',b',c'] \,.
			\]
			
			On the other hand, 
			\[
			f_2(\widehat{x_4}) = f_2(\overset{{}_{y_1}}{\boldsymbol{x_1}},x_2,\overset{{}_{y_3}}{\boldsymbol{x_3}},x_5,x_6,x_7) \,,
			\]
			has point 1 as \( x_1 \), and point 3 as \( x_3 \) (bold), so these argument forms specialise to
			\[
			\phi_1 = [x_1, a', x_3, b'] \,, \quad \phi_2 = [x_3, a', b', c'] \,, \quad \phi_3 = [x_1, a',b',c'] \,,
			\]
			with \( a',b',c' \in \{ x_2, x_5, x_6, x_7 \} \).  We again compute the degenerations to \( \crv \) as 
			\[
			\phi_1 \big\rvert_{\crv} = [p, a', p, b'] = 1 \,, \quad \phi_2 \big\rvert_{\crv} = [p, a', b', c'] \,, \quad \phi_3 \big\rvert_{\crv} = [p, a',b',c'] \,.
			\]
			
			Corresponding arguments in \( f_2(\widehat{x_3}) \) and \( f_2(\widehat{x_4}) \) hence degenerate to the same result on \( \crv \), which shows that \emph{term-wise}, \( f_2(\widehat{x_3}) \equiv f_2(\widehat{x_4}) \modonei \).  These summands then cancel in \( \QU_4 \).  
			
			\paragraph{Note:} One can just directly check this in weight 4, but as the combinatorics in weight 6 are similar, c.f. \autoref{lem:21shuffleDeriv}, while a direct there is more tedious, this conceptual approach can be warranted.
			
			\case{Case \( f_2(\widehat{x_1}) \)} Each term in \( f_2(y_1,\ldots,y_6) \) has an argument of the form \( \phi_1 \coloneqq [a, y_2, y_3, b] \) (up to symmetries of the cross-ratio), with \( a, b \in \{ y_1, y_4, y_5, y_6 \}\).  (Argument 2 of term 2 can be rewritten as \( [y_3, y_4, y_5, y_2] = [y_5, y_2, y_3, y_4] \), using the symmetries of the cross-ratio.)  For
			\[
				f_2(\widehat{x_1}) = f_2(x_2, \overset{{}_{y_2}}{\boldsymbol{x_3}}, \overset{{}_{y_3}}{\boldsymbol{x_4}}, x_5, x_6, x_7) \,,
			\]
			which has point 2 as \( x_3 \), and point 3 as \( x_4 \) (bold), argument \( \phi_1 = [a', x_3, x_4, b'] \), with \( a', b' \in \{ x_2, x_5, x_6, x_7 \} \).  Thus degenerating to \( \crv \), 
			\[
				\phi_1 \big\rvert_{\crv} = [a', p, p, b'] = \infty \,.
			\]
			By \autoref{lem:li:argtoinfy}, each term vanishes, modulo depth 1, giving \( f_2(\widehat{x_1}) \big\rvert_{\crv} \equiv 0 \modonei \).
			
			\case{Case \( f_2(\widehat{x_2}) \)} Directly, we have
			\begin{align*}
				f_2(\widehat{x_2}) & 
				 = \liftwo\big( [1347, 4567] - [1367, 4563] + [1567, 1345] \big) \,.
			\end{align*}
			Degenerating to \( \crv \) gives
			\begin{align*}
			f_2(\widehat{x_2}) \big\rvert_{\crv} & = \liftwo\big( [134p, p567] - [\smash{\overbrace{pp67}^{=0}}, \smash{\overbrace{p56p}^{=\infty}}] + [p567, 134p] \big) \\
			& = \liftwo\big(134p, p567) + \liftwo(p567, 134p)  \,.
			\end{align*}
			
			\case{Overall:} From the considerations above, the only surviving contribution is from \( f_2(\widehat{x_2}) \), giving
			\[
				\QU_4 \big\rvert_{\crv} \equiv (-1)^2 f_2(\widehat{x_2}) \big\rvert_{\crv} \equiv \liftwo\big(134p, p567) + \liftwo(p567, 134p)  \modone \,.
			\]
			This is the form of the identity given in the statement of the lemma.
		\end{proof}
	}
	
	\begin{Cor}[Inverse, $\invsc^{(4)}$]\label{cor:wt4:inv}
		The following identity holds
		\[
		\liftwo(A,B) - \liftwo(A^{-1},B^{-1}) = - \LiL_4(A) + \LiL_4(B) - 3 \LiL_4(A B)  \,.
		\]
		
		\begin{proof}
			Use the (1,1)-shuffle $\shsym{1,1}^{(4)}$ from \autoref{lem:wt4:sh11}, to rewrite the term 
			\[
			\liftwo(B^{-1},A^{-1}) =  - 3 \underbrace{\LiL_{4}(A^{-1} B^{-1})}_{= -\LiL_4(A B)} -  \liftwo(A^{-1},B^{-1}) \,,
			\]
			in $\revinvsc^{(4)}$ \autoref{lem:wt4invrev}.  Granting \( \LiL_4 \)-inversion: \( \LiL_4(x) = - \LiL_4(x^{-1}) \), the result follows.
		\end{proof}
	\end{Cor}

	{
		
		\renewcommand{\crv}{\mathcal{U}^{(4)}}
		\begin{Lem}[Reduction of \( \liftwo(1,x) \), $\redid{1,x}^{(4)}$]\label{lem:wt4deg}
			Degeneration of \( \QUf_4 \) to the stable curve 
			\[ 
			\crv =  16 \cup_p 35 \cup_q 247 =  \,\, 
				\vcenter{\hbox{\includegraphics[page=3]{figures/wt4.pdf}}}
			\,,
			\]
			produces the identity
			\biggerskip
			\[
			\liftwo(1, 35pq) = -\LiL_4(35pq) + \LiL_4(3pq5) + \LiL_4(3p5q) \,.
			\]
			In affine coordinates with \( A = 35pq \), this is 
			\[
			\liftwo(1, A) =  - \LiL_4(A) +  \LiL_4\big(\tfrac{A-1}{A}\big) + \LiL_4(1-A) \,.
			\]
		\end{Lem}
	
		\begin{Rem} This extends to a reduction for \( \lif(1,1,A) \) in weight 6, see \autoref{prop:11x}
			\end{Rem}
	
		\begin{proof}[Proof (modulo depth 1)]
			We treat the cases \( f_2(\widehat{x_i}) \) with \( i = 1 \), with \( i = 3 \), with \( i = 5, 6 \), with \( i = 7 \), with \( i = 2 \) and with \( i = 4 \) separately.  We show only \( f_2(\widehat{x_4}) \) contributes, and all others vanish or cancel, modulo depth 1.  Henceworth, we will start to streamline the cases where \( f_2(\widehat{x_i}) \equiv 0 \modonei \), to make the proof more compact.
			
			\case{Case \( f_2(\widehat{x_1}) \)} Each term in \( f_2(y_1,\ldots,y_6) \) has an argument of the form \( [y_1, a, b, y_6 ] \) or \( [y_3, a, b, y_6] \), with \( a, b \in \{ y_2,y_3,y_4,y_5 \} \). $\llbracket$Respectively argument 2, 1 and 1.$\rrbracket$  For \( f_2(\widehat{x_1}) \), these specialise to \( [x_2, a', b', x_7] \) and \( [x_4, a', b', x_7] \), \( a',b' \in \{ x_3, x_4, x_5, x_6 \} \), which degenerate to \( \infty \) on \( \crv \).  Hence \(f_2(\widehat{x_1}) \equiv 0 \modonei \).
			
			\case{Case \( f_2(\widehat{x_3}) \)} Argument 2 of each term in \( f_2(y_1,\ldots,y_6) \) has the form \( [y_i, a, b, y_j] \), \( i,j \in \{ 2,3,6\} \), with \( a,b \in \{ y_1,y_4,y_5 \} \) (up to cross-ratio symmetries).  For \( f_2(\widehat{x_3}) \), this specialises to \( [x_{i'}, a', b', x_{j'}] \), \( i',j' \in \{ 2,4,7 \} \), with \( a',b' \in \{ x_1, x_5, x_6 \} \), which degenerates to \( \infty \) on \( \crv \).  Hence \( f_2(\widehat{x_3}) \equiv 0 \modonei \).
			
			\case{Case \( f_2(\widehat{x_i}) \) with \( i = 5, 6 \)}  Every argument in \( f_2(y_1,\ldots,y_6) \) has the form \( [a, y_i, b, y_j] \), \( i,j \in \{ 2,4,6 \} \) with \( a,b \in \{ y_1, y_3, y_5 \} \).  For these \( f_2(\widehat{x_i}) \), this specialises to \( [a', x_{i'}, b', x_{j'}] \), \( i',j' \in \{ 2, 4, 7 \} \), with \( a', b' \in \{ x_1, x_3, x_5, x_6 \} \).  This degenerates to \( 1 \)  on \( \crv \), hence
			\[
				f_2(\widehat{x_5}) \big\rvert_{\crv} \equiv f_2(\widehat{x_6}) \big\rvert_{\crv} \equiv \liftwo(1,1) \modone \,.
			\]
			In \( \QU_4 \), these contributions cancel exactly.  (Alternatively: by \autoref{lem:wt4invrev}, with \( A = B = 1 \),  \( \liftwo(1,1) \equiv 0 \modonei \), so these contributions already vanish individually.)
			
			\case{Case \( f_2(\widehat{x_7}) \)}  Argument 1 of every term in \( f_2(y_1,\ldots,y_6) \) has the form \( [y_1, a, b, y_6] \), with \( a,b \in \{ y_2,y_3,y_4,y_5\} \).  For \( f_2(\widehat{x_7}) \), this specialises to \( [x_1, a',b',x_6] \), with \( a',b' \in \{ x_2,x_3,x_4,x_6 \} \), which degenerates to \( \infty \) on \( \crv \).  Hence \( f_2(\widehat{x_7}) \equiv 0 \modonei \).
			
			\case{Case \( f_2(\widehat{x_2}) \)} By a direct computation
			\begin{align*}
				f_2(\widehat{x_2}) \big\rvert_{\crv} &= \liftwo\big( [1347, 4567] - [1367, 4563] + [1567, 1345] \big) \big\rvert_{\crv} \\
				& = \liftwo\big( [\overbrace{p3qq}^{=0}, \overbrace{q5pq}^{=\infty}] - [\overbrace{p3pq}^{=1}, q5p3] + [\overbrace{p5pq}^{=1}, p3q5] \big) \\[1ex]
				& \equiv 0 \modone \,,
			\end{align*}
			as term 1 degenerates, and term 2 and term 3 cancel directly (via symmetries of the cross-ratio).
			
			\case{Case \( f_2(\widehat{x_4}) \)}  By a direct calculation
			\begin{align*}
				f_2(\widehat{x_4}) \big\rvert_{\crv} &= \liftwo\big( [1237, 3567] - [1267, 3562] + [1567, 1235] \big) \big\rvert_{\crv} \\
				 &= \liftwo\big( [\overbrace{pq3q}^{=1}, 35pq] - [\overbrace{pqpq}^{=1}, 35pq] + [\overbrace{p5pq}^{=1}, pq35] \big) \big\rvert_{\crv} \\[1ex]
				 &= \liftwo(1, 35pq) \,,
			\end{align*}
			as (say) term 2 and term3  cancel directly (via symmetries of the cross-ratio).
			
			\case{Overall} The only contribution is from \( f_2(\widehat{x_4}) \), and we find
			\[
				\QU_4 \big\rvert_{\crv} \equiv f_2(\widehat{x_4}) \equiv \liftwo(1, 35pq) \modone \,.
			\]
			This is the form of the identity given in the statement of the lemma.
		\end{proof}
	}

	\begin{Cor}[Reduction of \( \liftwo(x,1) \)]
		The following reduction holds
		\[
			\liftwo(A, 1) = -2\LiL_4(A) - \LiL_4(\tfrac{A-1}{A}) - \LiL_4(1-A) \,.
		\]		
		
		\begin{proof}
			This follows directly by writing
			\[
				\liftwo(A,1) = -\liftwo(1,A) - 3 \LiL(A) 
			\]
			via \autoref{lem:wt4:sh11} (with \( B = 1 \)), and using the reduction of \( \liftwo(1,A) \) from \autoref{lem:wt4deg}.
		\end{proof}
	\end{Cor}
	
	Now we present two different ways to obtain (equivalent) non-trivial symmetries of \( \liftwo \); these results extend in different ways in weight 6.
	
	{
		\renewcommand{\crv}{{\mathcal{S}^{(4)}_1}}
		\begin{Lem}[Full Symmetry 1, $\fullsym{1}^{(4)}$]\label{lem:wt4sym1}
			Degeneration of \( \QUf_4 \) to the stable curve 
			\begin{align*}
			\crv = 13  \cup_p 456 \cup_q 27 & = \,\, 
				\vcenter{\hbox{\includegraphics[page=4]{figures/wt4.pdf}}}
				\,,
			\end{align*}
			produces the following identity
			\biggerskip
			\begin{align*}
			&\liftwo(p56q, pq45) - \liftwo(pq56, p45q) = \\
			& \LiL_4\big( 
			-[4p6q]
			- [456p]
			- 2 \cdot [456q]
			-[p65q] 
			-[6p5q] 
			+ [pq45]
			+ [4p5q]
			 \big) \,.
			\end{align*}
			In affine coordinates with \( A = p56q, B = pq45 \), this is (up to \( \LiL_4 \)-inversion)
			\begin{align*}
			&\liftwo(A, B) - \liftwo(\tfrac{1}{1-A},\tfrac{B-1}{B}) = \\
			& \LiL_4\big( 
			-  \bigl[-\tfrac{A (1 - B)}{1 - A}\bigr]
			- \bigl[-\tfrac{(1 - A) B}{1 - B}\bigr]
			- 2 \cdot [A B] 
			- [1 - A]
			 - \bigl[-\tfrac{1-A}{A}\bigr]
			   + [B] + [1 - B]
			 \big) \,,
			\end{align*}
		\end{Lem}
	
		\begin{Rem}
			This extends to a Full Symmetry for \( \lif(x,y,z) \) in weight 6, see \autoref{lem:fullsym1}.
		\end{Rem}
	
	\begin{proof}[Proof (modulo depth 1)]
		We treat the cases \( f_2(\widehat{x_i}) \) with \( i = 1 \), with \( i = 2 \), with \( 4 \leq i \leq 6 \), with \( i= 3 \) and with \( i = 7 \) separately.  We show only \( f_2(\widehat{x_3}) \) and \( f_2(\widehat{x_7}) \) contribute.
		
		\case{Case \( f_2(\widehat{x_1}) \)} Argument 1 of each term in \( f_2(y_1, \ldots, y_6) \) has the form \( [y_1, a,b, y_6] \), with \( a,b \in \{y_2, y_3, y_4, y_5\} \).  For \( f_2(\widehat{x_1}) \), this specialises to \( [x_2, a', b', x_7] \), \( a',b' \in \{ x_3, x_4, x_5, x_6 \} \), which then degenerates to \( \infty \) on \( \crv \).  Hence \( f_2(\widehat{x_1}) \big\rvert_{\crv} \equiv 0 \modonei \).
		
		\case{Case \( f_2(\widehat{x_2}) \)} Each term in \( f_2(y_1, \ldots, y_6) \) has an argument of the form \( [y_1,y_2, a,b] \), with \( a,b \in \{y_3, y_4, y_5, y_6\} \). $\llbracket$Respectively argument 1, 1, 2.$\rrbracket$  For \( f_2(\widehat{x_2}) \), this specialises to \( [x_1, x_3, a', b'] \), with \( a',b' \in \{ x_4, x_5, x_6,x_7 \} \), which then degenerates to \( 0 \) on \( \crv \).  Hence \( f_2(\widehat{x_2}) \big\rvert_{\crv} \equiv 0 \modonei \).
		
		\case{Case \( f_2(\widehat{x_i}) \) with \( 4 \leq i \leq 6 \)} Each term in \( f_2(y_1, \ldots, y_6) \) has an argument of the form \( [y_1,a,y_3,b] \) with \( a,b \in \{ y_2, y_4, y_5, y_6 \} \) or \( [c, y_2, d, y_6] \) with \( c,d \in \{y_1, y_3, y_4, y_5\} \).  $\llbracket$Respectively argument 1, 1, 2.$\rrbracket$  For these \(  f_2(\widehat{x_i}) \), the arguments specialise to \( [x_1, a', x_3,  b'] \), with \( a',b' \in \{ x_2, x_4, x_5, x_6, x_7 \} \), and \( [c', x_2, d', x_7] \) with \( c',d' \in \{ x_1, x_3, x_4, x_5, x_6 \} \), respectively.  They both then degenerate to \( 1 \) on \( \crv \).  Hence, by \autoref{lem:wt4deg}, \( f_2(\widehat{x_i}) \big\rvert_{\crv} \equiv 0 \modonei \) with \( 4 \leq i \leq 6\).
		
		\case{Case \( f_2(\widehat{x_3}) \)}  By a direct computation
		\begin{align*}
			f_2(\widehat{x_3}) \big\rvert_{\crv} & = \liftwo\big( [1247, 4567] - [1267, 4562] + [1567, 1245] \big)  \big\rvert_{\crv} \\
			& = \liftwo\big( [\overbrace{pq4q}^{=1}, p56q] - [\overbrace{pq6q}^{=1}, 456q] + [p56q, pq45] \big) \\[0.5ex]
			& \equiv \liftwo(p56q, pq45) \modone \,.
		\end{align*}
		
		\case{Case \( f_2(\widehat{x_7}) \)}  By a direct computation
		\begin{align*}
		f_2(\widehat{x_7}) \big\rvert_{\crv} & = \liftwo\big( [1236, 3456] - [1256, 3452] + [1456, 1234] \big)  \big\rvert_{\crv} \\
		& = \liftwo\big( [\overbrace{pqp6}^{=1}, p456] - [pq56, p45q] + [p456, \overbrace{pqp4}^{=1}] \big) \\[0.5ex]
		& \equiv - \liftwo(pq56, p45q) \modone \,.
		\end{align*}
		
		\case{Overall} Only \( f_2(\widehat{x_3}) \) and \( f_2(\widehat{x_7}) \) contribute, and we find
		\[
			\QU_4 \big\rvert_{\crv} \equiv - \liftwo(p56q, pq45) + \liftwo(pq56, p45q) \modone \,.
		\]
		This is equivalent to the form of the identity give in the statement of the lemma.
	\end{proof} 
	}
	
	{
		\renewcommand{\crv}{{\mathcal{D}^{(4)}_1}}
		\begin{Lem}[Full Symmetry 1, alternative, $\fullsym{1}'$]\label{lem:wt4:sym1alt}
			Degeneration of \( \QUf_4 \) to the stable curve 
			\begin{align*}
			\crv = 13 \cup_p 567 \cup_q 24 = \,\, 
			 	\vcenter{\hbox{\includegraphics[page=5]{figures/wt4.pdf}}}
			 	\,\,
			\end{align*}
			produces the following identity
			\biggerskip
			\begin{align*}
			&\liftwo(pq67, p56q) - \liftwo(qp67, q56p) = \\
			& \LiL_4\big( 
			 [5q76]
			+ [5p7q]
			-[5p76]
			-[pq67]
			-[6p7q]
			-[6pq7]
			+ [5q6p]
			\big) \,.
			\end{align*}
			In affine coordinates with \( A = pq67, B = p56q \), this is 
			\begin{align*}
			&\liftwo(A, B) - \liftwo(\tfrac{A}{A-1}, 1-B) = \\
			& \LiL_4\big( 
			 \bigl[-\tfrac{A (1 - B)}{1 - A}\bigr]
			+ \bigl[-\tfrac{(1 - A) B}{1 - B}\bigr]
			-[A B]
			-[A] 
			-[1 - A]
			 -\bigl[-\tfrac{1-A}{A}\bigr]
			  + \bigl[-\tfrac{1-B}{B}\bigr]
			\big) \,.
			\end{align*}
		\end{Lem}
	
		\begin{Rem}
			This extends to a Degenerate Symmetry of \( \lif(1, x, y) \) in weight 6, see \autoref{lem:onexy_sym1} 
		\end{Rem}
	
	\begin{proof}[Proof (modulo depth 1)]
		We treat the cases \( f_2(\widehat{x_i}) \) with \( i = 2 \), with \( i = 3 \), with \( 5 \leq i \leq 7 \), with \( i = 1 \) and with \( i = 4 \) separately.  We show only \( f_2(\widehat{x_1}) \) and \( f_2(\widehat{x_4}) \) contribute.
		
		\case{Case \( f_2(\widehat{x_2}) \)} Every term in \( f_2(y_1,\ldots,y_8) \) has an argument of the form \( [y_1,y_2,a,b] \), with \( a,b \in \{ y_3, y_4, y_5, y_6 \} \).  $\llbracket$Respectively argument 1, 1, 2.$\rrbracket$  For \( f_2(\widehat{x_2}) \) on \( \crv \) these degenerate to 0, so \( f_2(\widehat{x_2}) \big\rvert_{\crv} \equiv 0 \modonei \).
		
		\case{Case \( f_2(\widehat{x_3}) \)} Every term in \( f_2(y_1,\ldots,y_8) \) has an argument of the form \( [a,y_2,y_3,b] \), with \( a,b \in \{ y_1, y_4, y_5, y_6 \} \), up to cross-ratio symmetries. $\llbracket$Respectively argument 1, 2, 2.$\rrbracket$  For \( f_2(\widehat{x_3}) \) on \( \crv \) these degenerate to 0, so \( f_2(\widehat{x_2}) \big\rvert_{\crv} \equiv 0 \modonei \).

		\case{Case \( f_2(\widehat{x_i}) \), with \( 5 \leq i \leq 7\)} Term 1, argument 1, in \( f_2(y_1,\ldots,y_8) \) has the form \( [y_1,a,y_3,b] \) with \( a,b \in \{y_2,y_4,y_5,y_6 \} \), while terms 2 and 3, argument 2, have the form \( [c, y_2, d, y_4] \) with \( c, d \in \{ y_1, y_3, y_5, y_6, \} \).  For these \( f_2(\widehat{x_i}) \) on \( \crv \), both types degenerate to 1, so \( f_2(\widehat{x_i}) \big\rvert_{\crv} \equiv 0 \modonei \), for \( 5 \leq i \leq 7 \).
		
		\case{Case \( f_2(\widehat{x_1}) \)} By a direct calculation
		\begin{align*} 
			 f_2(\widehat{x_1}) \big\rvert_{\crv} &= \liftwo\big( [2347, 4567] - [2367, 4563] + [2567, 2345]  \big) \big\rvert_{\crv} \\
			 & =  \liftwo\big( [\overbrace{qpq7}^{=1}, q567] - [qp67, q56p] + [q567, \overbrace{qpq5}^{=1}]  \big) \big\rvert_{\crv}  \\[0.5ex]
 			 & \equiv -\liftwo(qp67, q56p) \big\rvert_{\crv}  \modone
		\end{align*}
		
		\case{Case \( f_2(\widehat{x_4}) \)} By a direct calculation
		\begin{align*} 
		f_2(\widehat{x_1}) \big\rvert_{\crv} &= \liftwo\big( [1237, 3567] - [1267, 3562] + [1567, 1235] \big) \big\rvert_{\crv} \\
		& =  \liftwo\big( [\overbrace{pqp7}^{=1}, p567] - [pq67, p56q] + [p567, \overbrace{pqp5}^{=1}]  \big) \big\rvert_{\crv}  \\[0.5ex]
		& \equiv  -\liftwo(pq67, p56q) \big\rvert_{\crv}  \modone
		\end{align*}
		
		\case{Overall} Only \( f_2(\widehat{x_1}) \) and \( f_2(\widehat{x_4}) \) contribute, and we find
		\[
			\QU_4 \big\rvert_{\crv} \equiv \liftwo(qp67, q56p) -\liftwo(pq67, p56q) \modone \,.
		\]
		This is equivalent to the form of the identity given in the statement of the lemma.
	\end{proof}
	}
	
	\begin{Rem}\label{rem:wt4:equivfullsym1}
		The symmetries $\fullsym{1}^{(4)}$ from \autoref{lem:wt4sym1} and $\fullsym{1}'$ from \autoref{lem:wt4:sym1alt} are equivalent, using $\revsc^{(4)}$  \autoref{lem:wt4:sh11} and $\invsc^{(4)}$ \autoref{cor:wt4:inv}.  Symmetry $\fullsym{1}'$ from \autoref{lem:wt4:sym1alt} says
		\[
		\liftwo(A, B) - \liftwo(\tfrac{A}{A-1}, 1-B) \equiv 0 \modone \,.
		\]
		Switch \( A \leftrightarrow B \), then reverse the argument in both terms using $\revsc^{(4)}$, and finally invert the arguments in the second term using $\invsc^{(4)}$.  We obtain 
		\[
		\liftwo(A, B) - \liftwo(\tfrac{1}{1-A}, \tfrac{B-1}{B}) \equiv 0 \modone \,,
		\]
		which is exactly symmetry  $\fullsym{1}^{(4)}$ from \autoref{lem:wt4sym1}.
	\end{Rem}

	\begin{Prop}[12-fold symmetries of \( \liftwo \)]\label{rem:wt4:twelve}
		The (1,1)-shuffle $\shsym{1,1}^{(4)}$, or reversal $\revsc^{(4)}$, from \autoref{lem:wt4:sh11}, the inversion $\invsc^{(4)}$ from \autoref{cor:wt4:inv}, and the symmetry $\fullsym{1}^{(4)}$ from \autoref{lem:wt4sym1} (or the alternative version $\fullsym{1}'$ from \autoref{lem:wt4:sym1alt}) generate 12 symmetries, modulo depth 1, of 
		\[
		[\eps; x, y] \coloneqq \eps \liftwo(x,y) \,,
		\]
		The symmetries are given by the following, where the first and last lines of each batch are equivalent via $\invsc^{(4)}$, and the two batches are equivalent via $\revsc^{(4)}$.
		\begin{align*}
		\left.\begin{matrix}
		\begin{aligned}
		& [+; x, y] \\
		\overset{\mathclap{\text{\sc FS}_1^{(4)}}}{\equiv} {} \,\,\, & [+; \tfrac{1}{1-x}, \tfrac{y-1}{y}]  \\
		\overset{\mathclap{\text{\sc FS}_1^{(4)}}}{\equiv} {} \,\,\, & [+; \tfrac{x-1}{x}, \tfrac{1}{1-y}]  \\
		\overset{\mathclap{\invsc^{(4)}}}{\equiv} {}\,\,\, & [+; \tfrac{x}{x-1}, {1-y}] \\
		\overset{\mathclap{\text{\sc FS}_1^{(4)}}}{\equiv} {} \,\,\,& [+; 1-x, \tfrac{y}{y-1}]  \\
		\overset{\mathclap{\text{\sc FS}_1^{(4)}}}{\equiv} {} \,\,\, & [+; \tfrac{1}{x}, \tfrac{1}{y}]  \\
		\end{aligned}
		\end{matrix}\right\} \,\,
		\overset{\revsc^{(4)}}{\equiv}
		\,\,
		\left\{\begin{matrix}
		\begin{aligned}
		& [-; y, x] \\
		\,\,\,\, \overset{\mathclap{\text{\sc FS}_1^{(4)}}}{\equiv} {} \,\,\, & [-; \tfrac{y-1}{y}, \tfrac{1}{1-x}]  \\
		\overset{\mathclap{\text{\sc FS}_1^{(4)}}}{\equiv} {} \,\,\, & [-; \tfrac{1}{1-y},\tfrac{x-1}{x}]  \\
		\overset{\mathclap{\invsc^{(4)}}}{\equiv} {}\,\,\, & [-; {1-y}, \tfrac{x}{x-1}] \\
		\overset{\mathclap{\text{\sc FS}_1^{(4)}}}{\equiv} {} \,\,\,& [-; \tfrac{y}{y-1}, 1-x]  \\
		\overset{\mathclap{\text{\sc FS}_1^{(4)}}}{\equiv} {} \,\,\, & [-; \tfrac{1}{y}, \tfrac{1}{x}]  \\
		\end{aligned}
		\end{matrix}\right.
		\end{align*}
		
		\begin{proof}
			One sees $\invsc^{(4)}$ and $\revsc^{(4)}$ commute, whereas conjugating $\fullsym{1}^{(4)}$ by $\revsc^{(4)}$ gives $\fullsym{1}^{(4)}$ iterated twice, as does conjugating $\fullsym{1}^{(4)}$ by $\invsc^{(4)}$.  So we can then iterate $\fullsym{1}^{(4)}$ (twice), followed by $\invsc^{(4)}$ to generate all symmetries with \( (x,y) \) in the same order, and then $\revsc^{(4)}$ to generate the remaining ones.  This generates the above list.
		\end{proof}
	\end{Prop}

	{
		\renewcommand{\crv}{{\mathcal{C}^{(4)}_F}}
		\begin{Prop}[Four-term relation]\label{lem:wt4:fourterm}
			Degeneration of \( \QUf_4 \) to the stable curve 
			\[
			\crv =  13 \cup 467 \cup 25 = \,\,
			\vcenter{\hbox{\includegraphics[page=6]{figures/wt4.pdf}}}
			\,,
			\]
			produces the 4-term relation
			\[
			\liftwo( [74q6, 7pq4] + [74q6, 7qp4] + [7pq6, qp46] + [q7p6, 46pq])  \equiv 0 \modone \,.
			\]
			In affine coordinates with \( A^{-1} = 74q6, B = 7pq4\) (note $A$ inverse, to make the following neater) and terms in the same order, this is
			\begin{align*}
			&\liftwo(\big[ \tfrac{1}{A}, B \big] + \big[\tfrac{1}{A}, 1-B \big] + \big[\tfrac{B}{A}, \tfrac{1-A}{1-B} \big] +  \big[ \tfrac{B}{A}, \tfrac{1-B}{1-A}  \big]) = \\
			& \LiL_4 \big( \begin{aligned}[t]
			& 2 \, \bigl[\tfrac{(1-A) A}{(1-B) B}\bigr]
			+ \, \bigl[\tfrac{A (1-B)}{(1-A) B}\bigr]
			-2 \, \bigl[-\tfrac{1-A}{A-B}\bigr]
			-2 \, \bigl[\tfrac{1-B}{A-B}\bigr]
			-\bigl[\tfrac{1-A}{1-B}\bigr]
			+2 \, \bigl[\tfrac{A}{A-B}\bigr] \\
			&
			-2 \, \bigl[-\tfrac{A-B}{B}\bigr]
			-2 \, \bigl[\tfrac{A}{1-B}\bigr]
			+ \bigl[\tfrac{1-A}{B}\bigr]
			+\bigl[\tfrac{A}{B}\bigr]
			2 \, [A]
			+ \bigl[-\tfrac{1-A}{A}\bigr]
			+ 2 \, \bigl[-\tfrac{1-B}{B}\bigr]
			\big) \,.
			\end{aligned}
			\end{align*}
		\end{Prop}
		
		\begin{Rem}\label{rem:wt4:four}
			By making the substitution \( (A,B) = \big(\tfrac{1-Y}{1 - X Y}, \tfrac{X(1-Y)}{1 - X Y}\big) \) we have
			\[
				\big( \tfrac{B}{A}, \tfrac{1-A}{1-B} \big) = (X,Y) \,,
				\qquad
					\big( \tfrac{B}{A}, \tfrac{1-B}{1-A} \big) = (X,Y^{-1}) \,.
			\]
			Therefore establishing one of the Zagier formulae \autoref{conj:higherzagier} (with $k=2$)
			\begin{align*}
			\liftwo(x, y) + \liftwo(x, y^{-1}) & \overset{?}{\equiv} 0 \modone \,, \\
			\liftwo(x, y) + \liftwo(x, 1-y) & \overset{?}{\equiv} 0 \modone \,,
			\end{align*}
			is equivalent to establishing the other.  A similar claim holds in weight 6, see \autoref{rem:fourterm}, which comes from the extension of this to a four-term relation in weight 6, see \autoref{prop:fourterm}.
		\end{Rem}
		
		\begin{proof}[Proof of \autoref{lem:wt4:fourterm} (modulo depth 1)]
			We treat the cases \( f_2(\widehat{x_i}) \) with \( i = 2 \), with \( i = 4 \), with \( i = 6,7 \), with \( i = 1 \), with \( i = 3 \) and with \( i= 5 \) separately.  Only \( f_2(\widehat{x_1}) \), \( f_2(\widehat{x_3}) \), \( f_2(\widehat{x_5}) \) contribute.  We then show how to put the resulting identity into the given form.
			
			\case{Case \( f_2(\widehat{x_2}) \)} Every term in \( f_2(y_1,\ldots,y_6) \) has an argument of the form \( [y_1,y_2, a,b] \), (with \( a,b \) from elsewhere) which degenerates to 0 for \( f_2(\widehat{x_2}) \) on  \( \crv \).  $\llbracket$Respectively argument 1, 1, 2.$\rrbracket$ So overall  \( f_2(\widehat{x_2}) \big\rvert_{\crv} \equiv 0 \modonei \).
			
			\case{Case \( f_2(\widehat{x_4}) \)} Every term in \( f_2(y_1,\ldots,y_6) \) has an argument of the form \( [y_1, a, y_3, b] \) or \( [c, y_4, d, y_2] \) which degenerates to 1 for \( f_2(\widehat{x_4}) \) on  \( \crv \).  $\llbracket$Respectively argument 1, 2, 2.$\rrbracket$ So overall  \( f_2(\widehat{x_4}) \big\rvert_{\crv} \equiv 0 \modonei \).
			
			\case{Case \( f_2(\widehat{x_i}) \), with \( i = 6, 7 \)} Every term in \( f_2(y_1,\ldots,y_6) \) has an argument of the form \( [y_1, a, y_3, b] \) or \( [c, y_5, d, y_2] \) which both degenerate to 1 for \( f_2(\widehat{x_6}), f_2(\widehat{x_7}) \) on  \( \crv \).  $\llbracket$Respectively argument 1, either, 2.$\rrbracket$  So overall  \( f_2(\widehat{x_6}) \big\rvert_{\crv} \equiv f_2(\widehat{x_7}) \big\rvert_{\crv} \equiv 0 \modonei \).
			
			\case{Case \( f_2(\widehat{x_1}) \)} Direct calculation gives
			\begin{align*}
				f_2(\widehat{x_1}) \big\rvert_{\crv} &=  \liftwo\big( [2347, 4567] - [2367, 4563] + [2567, 2345]  \big) \big\rvert_{\crv} \\[-0.5ex]
				& = \liftwo\big( [qp47, 4q67] - [qp67, 4q6p] + [\overbrace{qq67}^{=0}, \overbrace{qp4q}^{=\infty}]  \big) \big\rvert_{\crv} \\
				& = \liftwo\big( [qp47, 4q67] - [qp67, 4q6p]\big) \,.
			\end{align*}
			
			\case{Case \( f_2(\widehat{x_3}) \)} Likewise
			\begin{align*}
			f_2(\widehat{x_3}) \big\rvert_{\crv} &=  \liftwo\big( [1247, 4567] - [1267, 4562] + [1567, 1245]  \big) \big\rvert_{\crv} \\[-0.5ex]
			& = \liftwo\big( [pq47, 4q67] - [pq67, \overbrace{4q6q}^{=1}] + [pq67, \overbrace{pq4q}^{=1}]  \big) \big\rvert_{\crv} \\
			& = \liftwo( pq47, 4q67 ) \,.
			\end{align*}
				
			\case{Case \( f_2(\widehat{x_5}) \)} Finally
			\begin{align*}
			f_2(\widehat{x_5}) \big\rvert_{\crv} &=  \liftwo\big( [1237, 3467] - [1267, 3462] + [1467, 1234]  \big) \big\rvert_{\crv} \\[-0.5ex]
			& = \liftwo\big( [\overbrace{pqp7}^{=1}, p467] - [pq67, p46q] + [p467, \overbrace{pqp4}^{=1}]  \big) \big\rvert_{\crv} \\
			& = -\liftwo( pq67, p46q ) \,.
			\end{align*}
			
			\case{Overall} The only contributions are from \( f_3(\widehat{x_1}) \), \( f_3(\widehat{x_3}) \), and \( f_3(\widehat{x_5}) \), giving
			\[
				\QU_4 \big\rvert_{\crv} \equiv  \liftwo\big( - [qp47, 4q67] + [qp67, 4q6p] - [ pq47, 4q67] + [pq67, p46q] \big) \,.
			\]
			Using the symmetry $\fullsym{1}^{(4)}$ from \autoref{lem:wt4sym1} (it cycles the last three entries of argument 1 to the right \( abcd \mapsto adbc \), and the last three entries of argument 2 to the left \( abcd \mapsto acdb \)), the inversion $\invsc^{(4)}$ from \autoref{cor:wt4:inv} (it cycles each cross-ratio by one step), and the reverse-inverse $\revinvsc^{(4)}$ from \autoref{lem:wt4invrev} (it cycles each cross-ratio by one step, and swaps the two arguments with a sign), we can rewrite this as follows.
			\begin{align*}
				\tag{term 1} -\liftwo(qp47, 4q67) & {} \overset{\textsc{RI}^{(4)}}{\equiv} \liftwo(74q6, 7qp4) \modone \\
				\tag{term 2} \liftwo(qp67, 4q6p) & {} \overset{\textsc{FS}_1^{(4)}}{\equiv} \liftwo(q7p6, 46pq) \\
									& {} \overset{\textsc{FS}_1^{(4)}}{\equiv} \liftwo(q67p, 4pq6) \modone \\
				\tag{term 3} -\liftwo(pq47, 4q67) & {} \overset{\textsc{RI}^{(4)}}{\equiv} \liftwo(74q6, 7pq4) \modone \\
				\tag{term 4} \liftwo(pq67, p46q) & {} \overset{\,\,\,\invsc^{(4)}\,\,\,}{\equiv} \liftwo(7pq6, qp46) \modone 
			\end{align*}
			We obtain the equivalent identity (terms in the same order as before),
			\[
				\liftwo([74q6, 7qp4] + [q7p6, 46pq] + [74q6, 7pq4] + [7pq6, qp46]) \equiv 0 \modone \,.
			\]
			Up to reordering the terms, this is the form of the identity presented in the statement of the lemma.
		\end{proof}
	}
	
	\begin{Def}[Four-term, weight 4]
		Define the four-term combination in weight 4 to be
		\[
		\four^{(4)}(x,y) \coloneqq \liftwo\Big(  [x, y] + [x, 1-y] + \Big[x y, \frac{1-x}{x(y-1)} \Big] + \Big[x y, \frac{x(y-1)}{1-x} \Big] \Big) \,,
		\]
		From \autoref{lem:wt4:fourterm} (with \( A = x^{-1}, B = y \)), we know this satisfies \( \four^{(4)}(x,y) \equiv 0 \modonei \).
	\end{Def}
	
	In contrast to Lemma 6.4 \cite{GR-zeta4}, or the refined version in \cite{MR22}, we will establish the Zagier identities, without recourse to a projective involution.  Instead we consider combinations of \( \four^{(4)} \) under some birational maps.  (Although this approach is undoubtedly equivalent to considering projective involutions, we found the computational aspects -- which become important and challenging in weight 6 -- conceptually easier with this approach.)
	
	\begin{Prop}[The Zagier formulae for $k=2$]\label{prop:wt4:zagier}
		The following identities hold
		\begin{align*}
		\liftwo(x, y) + \liftwo(x, 1-y) &\equiv 0 \modone \,, \\
		\liftwo(x, y) + \liftwo(x, y^{-1}) &\equiv 0 \modone \,.
		\end{align*}
		
		\begin{proof}
			\case{Identity 1} Expand the following combination of four-term relations
			\begin{align*}
			& \four^{(4)}(x, 1-y) + \four^{(4)}(\tfrac{x}{x-1}, y) \\
			& {} \equiv  
			\begin{alignedat}[t]{4}
			\liftwo\big( & \phantom{{}+{}} [x, 1-y] {} && \, +\,  [x, y] {} &&  \, +\,  \big[x(1-y), \tfrac{x-1}{xy}\big]  {} && \, +\,  \big[x(1-y), \tfrac{xy}{x-1}\big] \big) \\
			& + \big[\tfrac{x}{x-1}, y\big] {} &&  \, +\,  \big[\tfrac{x}{x-1}, 1-y\big] {} && \,  +\,  \big[\tfrac{xy}{x-1}, \tfrac{1}{x(1-y)}\big]  {} && \, +\,  \big[\tfrac{xy}{x-1}, x(1-y)\big] \big) 
			\end{alignedat}
			\end{align*}
			By the (1,1)-shuffle $\shsym{1,1}^{(4)}$ from \autoref{lem:wt4:sh11}, terms 4 and 8 cancel, modulo depth 1.  Likewise by the reverse-inverse $\revinvsc^{(4)}$ from \autoref{lem:wt4invrev} terms 3 and 7 cancel, modulo depth 1.  Then by the 12-fold symmetries from \autoref{rem:wt4:twelve}, we see term 5 is equivalent to term 1, and term 6 is equivalent to term 2, modulo depth 1:
			\begin{align*}
			\tag{term 5} \liftwo(\tfrac{x}{x-1},y) &\overset{\text{12-fold}}{\equiv} \liftwo(x,1-y) \modone \,,\\
			\tag{term 6} \liftwo(\tfrac{x}{x-1},1-y) &\overset{\text{12-fold}}{\equiv} \liftwo(x, y) \modone \,.
			\end{align*}
			Combining these terms shows
			\[
				\four^{(4)}(x, 1-y) + \four^{(4)}(\tfrac{x}{x-1}, y)  \equiv \liftwo\big( 2\,[x,1-y] + 2\, [x,y] \big) \modone \,.
			\]
			Since \( \four^{(4)} \equiv 0 \modonei \) by \autoref{lem:wt4:fourterm}, the first Zagier formula holds after dividing by 2.
			
			\case{Identity 2} The second identity follows from the equivalence discussed in \autoref{rem:wt4:four}.  For completeness, we note that
			\[
			\four^{(4)}(1-x, \tfrac{y}{y-1}) + \four^{(4)}(\tfrac{x-1}{x}, \tfrac{1}{1-y})  \,\, \overset{\text{12-fold}}{\equiv} \,\, \liftwo\big( 2\,[x, y] + 2\, [x, y^{-1}] \big) \,,
			\]
			which establishes the result directly.
		\end{proof}
	\end{Prop}
	
	Hence \( \liftwo(x,y) \) satisfies the dilogarithm 6-fold symmetries independently in each argument, giving \( 6^2 \cdot 2 = 72 \) symmetries overall.
	
	\begin{Cor}[Six-fold symmetries of $\liftwo$]  Let \( \sigma, \tau \in \Sym_3 \) be any of the six-fold symmetries acting via anharmonic ratios (see \autoref{sec:quadrangular:anharmonic}).  Then
		\[
		\liftwo(x^\sigma, y^\tau) \,\equiv\, \sgn(\sigma\tau) \liftwo(x,y) \modone \,.
		\]
		In particular \( \liftwo \) satisfies the dilogarithm 6-fold symmetries in each argument independently.
	\end{Cor}
	
	\begin{Rem}
		In their full form, the identities which arise from \autoref{prop:wt4:zagier} are very similar in form and structure to the versions originally given by Zagier (for essentially $I_{3,1}([x] + [1-x],y)$) and Gangl (for $I_{3,1}([x] + [x^{-1}], y)$), see Proposition 22, in \cite{gangl-4}, and the Remark in \cite[p. 7]{ganglSome}.  In particular, we find
		\begin{align*}
		&\liftwo(x, y) + \liftwo(x,1-y) \= \\
		&\LiL_4\big( \begin{aligned}[t]
		- \tfrac{1}{2} \, \bigl[-\tfrac{x^2 (1-y) y}{1-x}\bigr]
		+ \tfrac{1}{2} \, \bigl[-\tfrac{(1-x) (1-y)}{y}\bigr]
		+ \tfrac{1}{2} \, \bigl[-\tfrac{(1-x) y}{1-y}\bigr]
		+ \bigl[-\tfrac{x (1-y)}{1-x}\bigr]
		+ \bigl[-\tfrac{x y}{1-x}\bigr]
		- [1-x]
		-2 \, [x]
		\big)
		\end{aligned} \\[2ex]
		&\liftwo(x, y) + \liftwo(x, y^{-1}) \= \\
		&\LiL_4\big( \begin{aligned}[t]
		& +\tfrac{1}{2} \, \bigl[\tfrac{x (1-y)^2}{(1-x)^2 y}\bigr]
		+2 \, \bigl[\tfrac{1-x}{1-y}\bigr]
		-2 \, \bigl[-\tfrac{x (1-y)}{1-x}\bigr]
		-2 \, \bigl[\tfrac{x (1-y)}{(1-x) y}\bigr]
		+2 \, \bigl[-\tfrac{(1-x) y}{1-y}\bigr]
		\\
		&
		-2 \, \bigl[-\tfrac{1-x}{x}\bigr]
		+2 \, \bigl[-\tfrac{1-y}{y}\bigr]
		-\tfrac{3}{2} \, \bigl[\tfrac{x}{y}\bigr]
		-\tfrac{3}{2} \, [x y]
		- [x]
		-2 \, [1-x]
		+2 \, [1-y]
		 \big) \,.
		\end{aligned}
		\end{align*}
	\end{Rem}
	
	\paragraph{Proof of the Gangl formula in weight 4} Finally the proof that \( \liftwo \) satisfies the dilogarithm five-term relation proceeds in exactly the manner first described in \cite{GR-zeta4}, and later refined in \cite{MR22} (both of which provide a conceptual explanation of Gangl's original computer-aided derivation of the 122-term relation \cite{gangl-4}).
		
	\begin{Rem}
		We leave the verification of the details here to the interested reader; one should consult \cite[Lemma 6.6]{GR-zeta4} and \cite[Proposition 4.6]{MR22} for the pertinent details of the key degeneration of \( \QUf_4 \) to \( \mathcal{C}_{\mathrm{five}} = 12346 \cup_p 57 \).  One can also utilise the ancillary \texttt{Mathematica} worksheet \wtfourfilename, attached to the \texttt{arXiv} submission, to further explore these results.
	\end{Rem}

	Write
	\begin{align*}
	F(x_1,\ldots,x_6) & \coloneqq \liftwo( [2365] - [ 1365] + [1265] - [1235] + [1236], [3456]) \,,
	\end{align*}
	where the combination in the first argument is the dilogarithm five-term relation obtained from
	\(
		\sum\nolimits_{i=0}^4 (-1)^i [w_0,\ldots,\widehat{w_i},\ldots,w_4] 
	\)\,,
	with \( (w_0,\ldots,w_4) = (x_1, x_2, x_3, x_6, x_5) \).
	
	{\biggerskip	
	The aim is to show \( F(x_1,\ldots,x_6) \equiv 0 \modonei \).  From the Zagier formulae \autoref{prop:wt4:zagier}, one obtains the following basic  (anti-)symmetries formulated below: i) an anti-symmetry in \( x_1, x_2 \), and ii) a symmetry in \( x_3, x_5, x_6 \).  (Arrows on the left-hand side highlight how the points move in each identity, underlines on the right-hand side indicate what has changed.)%
	\begin{align*}
	\tag{1,2 anti-sym} \hspace{-2em} F(\tikzmarknode{a1}{x_1},\tikzmarknode{a2}{x_2},x_3,x_4,x_5,x_6) \,\equiv\,  -F(\uline{x_2},\uline{x_1},x_3,x_4,x_5,x_6) \modone 
	\begin{tikzpicture}[overlay,remember picture]
	\draw[arrows={latex}-{latex}] ($(a1.north) + (0,2pt)$) to[out=90,in=90,looseness=1.7] ($(a2.north) + (0,2pt)$);
	\end{tikzpicture}
	\\[3ex]
	\tag{3,5,6 sym} \hspace{-2em} \begin{aligned}[c]
	F(x_1,x_2,\tikzmarknode{c1}{x_3},x_4,\tikzmarknode{c2}{x_5},\tikzmarknode{c3}{x_6}) 
	& \,\equiv\,  F(x_1,x_2,\uline{x_5},x_4,\uline{x_3},x_6)  \\
	& \,\equiv\,F(x_1,x_2,\uline{x_6},x_4,x_5,\uline{x_3})\\ 
	& \,\equiv\, F(x_1,x_2,x_3,x_4,\uline{x_6},\uline{x_5}) \modone \,. 
	\end{aligned}
	\begin{tikzpicture}[overlay,remember picture]
	\draw[arrows={latex}-{latex}] ($(c1.north) + (0,2pt)$) to[out=90,in=110,looseness=1] ($(c2.north) + (-2pt,2pt)$);
	\draw[arrows={latex}-{latex},densely dotted] ($(c2.north) + (0,2pt)$) to[out=70,in=90,looseness=1.9] ($(c3.north) + (0,2pt)$);
	\draw[arrows={latex}-{latex},densely dashed]($(c1.south) - (0,2pt)$) to[out=-70,in=-110,looseness=1] ($(c3.south) - (0,2pt)$);
	\end{tikzpicture}
	\end{align*}%
	}%
	{%
	Degeneration of \( \QUf_4 \) to the stable curve 
	\[
		\mathcal{C}_\text{five} = 12346 \cup_p 57 = \,\, \vcenter{\hbox{\includegraphics[page=7]{figures/wt4.pdf}}}
		\,,
	\] 
	gives 
	\begin{equation*}
	\QUf_4 \rvert_{\mathcal{C}_\text{five}} \equiv F(x_3, x_4, x_2, x_1, x_6, x_p) - F(x_1, x_2, x_3, x_4, x_p, x_6) \modone \,.
	\end{equation*}
	(This stable curve describes the degeneration \( x_7 \to x_5 \), so for convenience we can just replace \( x_p \), the intersection point, with \( x_5 \) again.)
	
	This is a new ``exotic'' symmetry for \( F(x_1,\ldots,x_6) \), acting by permutation $(1\,4\,2\,3)\,(5\,6)$ as highlighted (with bold and underline to show changes, and arrows on the left-hand side indicating how the points move) in the following
	\bigskip\smallskip
	\begin{equation*}
	F(\tikzmarknode{e1}{x_1}, \tikzmarknode{e2}{x_2}, \tikzmarknode{e3}{x_3}, \tikzmarknode{e4}{x_4}, \tikzmarknode{e5}{x_5}, \tikzmarknode{e6}{x_6}) \equiv F(\boldsymbol{x_3}, \boldsymbol{x_4}, \boldsymbol{x_2}, \boldsymbol{x_1}, \uline{x_6}, \uline{x_5}) \modone \,.
	\begin{tikzpicture}[overlay,remember picture]
	\draw[arrows={latex}-{latex}] ($(e5.south) - (0,2pt)$) to[out=-90,in=-90,looseness=1.9] ($(e6.south) - (0,2pt)$);
	\draw[arrows={}-{latex}] ($(e1.north) + (0,2pt)$) to[out=70,in=110,looseness=1] ($(e4.north) + (2pt,2pt)$);
	\draw[arrows={}-{latex}] ($(e4.north) + (-2pt,2pt)$) to[out=130,in=40,looseness=1] ($(e2.north) + (0,2pt)$);
	\draw[arrows={}-{latex}] ($(e2.south) - (0,2pt)$) to[out=-70,in=-130,looseness=1.5] ($(e3.south) - (2pt,2pt)$);
	\draw[arrows={}-{latex}] ($(e3.south) - (-2pt,2pt)$) to[out=-110,in=-70,looseness=1.5] ($(e1.south) - (-2pt,2pt)$);
	\end{tikzpicture}
	\end{equation*}
	\medskip
	
	It is convenient to apply the \( (1,2) \)-anti-symmetry (to switch \( x_3, x_4 \)), and \( (5,6) \)-symmetry (to switch \( x_6,x_5 \)) on the right-hand side.  This gives an equivalent new symmetry which acts only by the double-transposition $(1\,4)\,(2\,3)$, as highlighted (arrows show how points move, bold and underline mark what has changed) here\bigskip\smallskip
	\begin{equation*}
	\tag{exotic}
		F(\tikzmarknode{f1}{x_1}, \tikzmarknode{f2}{x_2}, \tikzmarknode{f3}{x_3}, \tikzmarknode{f4}{x_4}, {x_5}, {x_6}) \equiv F(\boldsymbol{x_4}, \uline{x_3}, \uline{x_2}, \boldsymbol{x_1}, {x_5}, {x_6}) \modone \,.
		\begin{tikzpicture}[overlay,remember picture]
		\draw[arrows={latex}-{latex}] ($(f1.north) + (0,2pt)$) to[out=70,in=110,looseness=1] ($(f4.north) + (0,2pt)$);
		\draw[arrows={latex}-{latex}] ($(f2.north) + (0,2pt)$) to[out=70,in=110,looseness=1.9] ($(f3.north) + (2pt,2pt)$);
	\end{tikzpicture}
	\end{equation*}

	\noindent
	We use this to show \( F(x_1,\ldots,x_6) \) is both symmetric and anti-symmetric in \( x_1,x_2 \), hence must vanish modulo depth 1.  The idea is to use the exotic symmetry to move \( x_1, x_2 \) in turn into slot 3, which can then be moved into the last two (symmetric) slots.  We have the following. (The braces highlight which points are changing at each step, the labels ``(via symmetry)'' indicate which symmetry is being used.)
		\begin{align*}
	& \phantom{{}\equiv{} + } F(\underbrace{x_1,x_2,x_3,x_4},x_5,x_6) \\[-1ex]
	\tag{via exotic} & \equiv - F(\mathrlap{\overbrace{\mathclap{\phantom{\big(}}\phantom{x_4,x_3,x_2,x_1}}}x_4,x_3,\underbrace{x_2,x_1,x_5,x_6}) \\[-1ex]
	\tag{via 3,6 sym} & \equiv - F(\mathrlap{\underbrace{\phantom{x_4,x_3,x_6,x_1}}}x_4,x_3,\overbrace{\mathclap{\phantom{\big(}}x_6,x_1,x_5,x_2}) \\[-1ex]
	\tag{via exotic} & \equiv \phantom{+} F(\overbrace{\underbrace{\mathclap{\phantom{\big(}}x_1,x_6},x_3,x_4},x_5,x_2) \\[-1ex]
	\tag{via 1,2 anti-sym} & \equiv - F(\underbrace{\overbrace{\mathclap{\phantom{\big(}}x_6,x_1},x_3,x_4},x_5,x_2) \\[-1ex]
	\tag{via exotic} & \equiv \phantom{+} F(\mathrlap{\overbrace{\mathclap{\phantom{\big(}}\phantom{x_4,x_3,x_1,x_6}}}x_4,x_3,\underbrace{x_1,x_6,x_5},x_2) \\[-1ex]
	\tag{via 3,5 sym} &\equiv \phantom{+} F(x_4,x_3,\overbrace{\mathclap{\phantom{\big()}}x_5,x_6,x_1},x_2) \modone \,. \\[-0.3ex]
	\end{align*}
	From the first line we see that \( F \) is anti-symmetric in \( x_1,x_2 \); from the last line we see that is symmetric in \( x_1, x_2 \).  Hence it  must vanish, and we obtain \( F(x_1,\ldots,x_6) \equiv 0 \modonei \).
	}
	
	\begin{Cor}[Gangl formula in weight 4, \cite{gangl-4,GR-zeta4,MR22}]\label{cor:wt4:gangl}
		The function \( \liftwo(x,y) \) satisfies the dilogarithm five-term relation in \( x \) (correspondingly \( y \)), modulo depth 1.  That is,
		\[
		\sum_{i=0}^4 (-1)^i \liftwo([w_0,\ldots,\widehat{w_i},\ldots,w_4], y) \equiv 0 \modone \,.
		\]
	\end{Cor}
	Gangl  \cite{gangl-4} originally gave a 122-term combination of \( \LiL_4 \) terms (found with judicious investigation, and some computer assistance) which explicitly establishes this reduction.  The procedure outlined above (with some choices involved in generating the 72-fold symmetries from the Zagier formulae) gives a 303-term combination of \( \LiL_4 \) terms (of a similar type and flavour as Gangl's arguments), which reduces \( \sum_{i=0}^4 (-1)^i \liftwo([w_0, \ldots, \widehat{w_i}, \ldots, w_4], y) \) to depth 1.  Within those arguments one can find a 159-term combination of \( \LiL_4 \) terms which also gives a reduction.
	
	\section{The higher Zagier formulae in weight 6}\label{sec:higherZagier6}
	\note{$\llbracket$Calculations} The reader can consult the \texttt{Mathematica} worksheet \wtsixfilename, attached to the \texttt{arXiv} submission, to verify the calculations and proofs given in this section.  Full explicit expressions for all of the results derived using this worksheet are given in various forms in the ancillary text files (see the \hyperlink{filedescriptions}{descriptions at the end of Section 1}).$\rrbracket$\medskip
	
	We now derive the higher Zagier formulae in weight 6 (\autoref{thm:i411sixfold}); this shows \autoref{conj:higherzagier} ($k=3$) holds, and together with Matveiakin-Rudenko's proof \cite[Theorem 1.4]{MR22} of the higher Gangl formula in weight 6 \autoref{conj:highergangl} ($k=3$), we conclude that Goncharov's Depth Conjecture \autoref{conj:depthv2} (weight 6, $k=3$) holds (\autoref{cor:gon:wt6depth3}). \medskip
	
	We again outline the proof strategy, which is quite combinatorially intricate.  In \autoref{sec:wt6:sym} we re-derive the basic symmetries of \autoref{sec:mpl:identities}.  Then, in \autoref{sec:wt6:nielsen:11x} we show the Nielsen formula \autoref{conj:highernielsen} (\( \mathcal{N}_{110} \)) holds, reducing \( \lif(1,1,x) \) to depth 2.  In \autoref{sec:wt6:nielsen:1xy} we begin by finding two non-trivial symmetries for \( \lif(1,x,y) \); playing them against each other shows that the Nielsen formula \autoref{conj:highernielsen} (\( \mathcal{N}_{100} \)) holds, reducing \( \lif(1,x,y) \) to depth 2.  Then in \autoref{sec:wt6:6fold} we first establish a new non-trivial symmetry for \( \lif(x,y,z) \) (giving 12 symmetries), use this to show \( \lif(x,y,z) \) satisfies a four-term relation, and then derive a second non-trivial symmetry (producing 216 symmetries overall).  Playing these symmetries against the four-term relation shows that \( \lif(x,y,z) \) satisfies the higher Zagier formulae ($k=3$).  
	
	In \autoref{sec:wt6:conclusion}, we conclude with the higher Zagier formulae \autoref{conj:higherzagier} ($k=3$) and Goncharov's Depth Conjecture \autoref{conj:depthv2} (weight 6, $k=3$).\medskip
	
	We shall explicitly give some of the resulting identities in \autoref{app:explicit}, for the interested reader.
	
	\subsection{Re-deriving the basic symmetries}\label{sec:wt6:sym}
	
	The general results from \autoref{sec:mpl:identities} show that the basic symmetries (inverse \autoref{cor:inv} and reverse \autoref{cor:reverse}), as well as the (2,1)-shuffle \autoref{lem:21shuffleDeriv} hold for any depth 3 function of the form \( \LiL_{n_0\;1,1,1} \).  Nevertheless, we show how these identities follow directly by degeneration of \( \QU_6 \), as further evidence that the quadrangular polylogarithm functional equation (\autoref{sec:quadrangular}) is the fundamental functional equation for multiple polylogarithms.

		{
		\renewcommand{\crv}{{\mathcal{C}_{\mathrm{RI}}}}
		\begin{Lem}[Reverse-Inverse, $\revinvsc$]\label{lem:revinvDeriv}
			Degeneration of \( \QU_6 \) to the stable curve 
			\[ 
			\crv = 19 \cup_p 2468 \cup_q 357 = \,\,
			\vcenter{\hbox{\includegraphics[page=1]{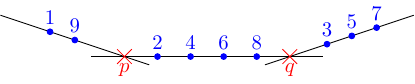}}}
			\,,
			\]
			produces the identity 
			\biggerskip
			\[
				\lif(2q4p, 4q6p, 6q8p) + \lif(6p8q, 4p6q, 2p4q) \equiv 0 \modtwo \,.
			\]
			In affine coordinates, with \( A = 2q4p \), \( B =  4q6p \), \( C = 6q8p  \), this is
			\[
				\lif(A,B,C) + \lif(C^{-1}, B^{-1}, A^{-1}) \equiv 0 \modtwo \,.
			\]
		\end{Lem}
	
		\begin{Rem}
			This is an extension of \autoref{lem:wt4invrev}.  The full identity can be found in \autoref{app:syms}.
		\end{Rem}
		
		\begin{proof}
			We treat the cases \( f_3(\widehat{x_i}) \), with \( 2 \leq i \leq 8 \), with \( i = 1 \) and with \( i = 9 \) separately.  Only \( f_3(\widehat{x_1}) \) and \( f_3(\widehat{x_9}) \) contribute.
			
			\case{Case \( f_3(\widehat{x_i}) \) with \( 2 \leq i \leq 8 \)} Argument 1 of all terms in \( f_3(y_1,\ldots,y_8) \) have the form \( [y_1,a,b,y_8] \), which degenerates to \( \infty \) for \( f_3(\widehat{x_i}) \) on \( \crv \).  Hence \( f_3(\widehat{x_i}) \equiv 0 \modtwoi \).
			
			\case{Case \( f_3(\widehat{x_1}) \)} We directly compute
			\begin{align*}				
			& f_3(\widehat{x_1}) \big\rvert_{\crv}  \\[0.5ex]
			& {} =  \begin{aligned}[t]
			\lif\big( 
			& -[2349,4569,6789]
			+[2349,4589,6785]
			-[2349,4789,4567]
			+[2369,4563,6789]
			\\
			& 
			+[2369,6789,4563]
			-[2389,4583,6785]
			+[2389,4783,4567]
			-[2389,6783,4563]
			\\
			&
			-[2569,2345,6789]
			-[2569,6789,2345]
			+[2589,2345,6785]
			+[2589,6785,2345]
			\\
			&
			-[2789,2347,4567]
			+[2789,2367,4563]
			-[2789,2567,2345]
			\, \smash{\big) \big\rvert_{\crv}} \end{aligned} \\[0.5ex]
			& {} = \begin{aligned}[t]
			\lif\big( 
			& -[2q4p,4q6p,6q8p]
			+[2q4p,4q8p,1]
			-[2q4p,4q8p,1]
			+[2q6p,1,6q8p]
			\\
			& 
			+[2q6p,6q8p,1]
			-[2q8p,1,1]
			+[2q8p,1,1]
			-[2q8p,1,1]
			\\
			&
			-[2q6p,1,6q8p]
			-[2q6p,6q8p,1]
			+[2q8p,1,1]
			+[2q8p,1,1]
			\\
			&
			-[2q8p,1,1]
			+[2q8p,1,1]
			-[2q8p,1,1]
			\, \smash{\big)} 
			\end{aligned} \\[0.5ex]
			& {} = -\lif(2q4p,4q6p,6q8p) \,.
			\end{align*}
			The terms ending with $(1,1)$ are the same up to sign, and so cancel overall as there are 4 coefficients $+1$ and 4 coefficients $-1$.  We also see term 2 and term 3 cancel, term 4 and 9 cancel, and term 5 and term 10 cancel.  Only term 1 survives.
			
			\paragraph{Note:}  One can more directly see that terms 6--8, 11--15 (the terms ending with $(1,1)$) cancel as follows.  These term in \( f_3(y_1,\ldots,y_8) \) have argument 1 of the form \( [y_1, a, y_7, y_8]  \), with \( a \in \{ y_2, y_4, y_6 \} \), and arguments 2--3 of the form \( [b,Y, c, Z] \), with \( Y, Z \in \{ y_2, y_4, y_6 \} \), and \( b,c \not\in \{ y_2, y_4, y_6 \} \).  For \( f(\widehat{x_1}) \) on \( \crv \) argument 1 degenerates to \( 2q8p = [2, q, 8, p] \), while arguments 2--3 both degenerate to \( 1 = [bqcq] \).  So all terms degenerate to the same result, which vanishes overall by the count of coefficients \( +1 \) and \( -1 \).
			
			\case{Case \( f_3(\widehat{x_9}) \)} We directly compute
			\begin{align*}
			& f_3(\widehat{x_9}) \big\rvert_{\crv}  \\[0.5ex]
			& = {}  \begin{aligned}[t]
			\lif\big( 
			& -[1238,3458,5678]
			+[1238,3478,5674]
			-[1238,3678,3456]
			+[1258,3452,5678]
			\\
			& 
			+[1258,5678,3452]
			-[1278,3472,5674]
			+[1278,3672,3456]
			-[1278,5672,3452] 
			\\
			&
			-[1458,1234,5678]
			-[1458,5678,1234]
			+[1478,1234,5674]
			+[1478,5674,1234]
			\\
			&
			-[1678,1236,3456]
			+[1678,1256,3452]
			-[1678,1456,1234]
			\, \smash{\big) \, \big\rvert_{\crv}}
			\end{aligned}
			\\[0.5ex]
			& {} = \begin{aligned}[t]
			\lif\big( 
			& -[p2q8,1,1]
			+[p2q8,1,1]
			-[p2q8,1,1]
			+[p2q8,1,1]
			\\
			& 
			+[p2q8,1,1]
			-[p2q8,1,1]
			+[p2q8,1,1]
			-[p2q8,1,1] \\
			&
			-[p4q8,p2q4,1]
			-[p4q8,1,p2q4]
			+[p4q8,p2q4,1]
			+[p4q8,1,p2q4]
			\\
			&
			-[p6q8,p2q6,1]
			+[p6q8,p2q6,1]
			-[p6q8,p4q6,p2q4]
			\, \smash{\big)} 
			\end{aligned} \\
			& {} = - \lif(p6q8,p4q6,p2q4) \,.
			\end{align*}
			We see that terms 1--8 are all the same up to sign, and so cancel overall as there are 4 coefficients $+1$ and 4 coefficients $-1$.  We also see term 9 and 11 cancel, term 10 and 12 cancel, and term 13 and 14 cancel pairwise.    Only term 15 survives.
			
			\paragraph{Note:}  One can more directly see that terms 1--8 cancel as follows.  These terms in \( f_3(y_1,\ldots,y_8) \) have argument 1 of the form \( [y_1, y_2, a, y_8]  \), with \( a \in \{ y_3, y_5, y_7 \} \), and arguments 2--3 of the form \( [Y, b, Z, c] \), with \( Y, Z \in \{ y_3, y_5, y_7 \} \) and \( b,c \not\in \{ y_3, y_5, y_7 \} \).  For \( f(\widehat{x_9}) \) on \( \crv \) argument 1 degenerates to \( p2q8 = [p, x_2, q, x_8] \), while arguments 2--3 both degenerate to \( 1 = [qbqc] \).  So terms 1--8 all degenerate to the same result, which vanishes overall by the count of coefficients \( +1 \) and \( -1 \).
			
			\case{Overall} The only contributions are from \( f_3(\widehat{x_1}) \) and \( f_3(\widehat{x_9}) \), which give
			\[
			\QU_6 \big\rvert_{\crv} = \lif(2q4p, 4q6p, 6q8p) +  \lif(p6q8,p4q6,p2q4) \,.
			\]
			This is the form of the identity give in the statement of the lemma.
		\end{proof}
		}
	
		\begin{Rem}
			 A better understanding of the combinatorics of the remaining cancellations in \( f_3(\widehat{x_1}) \) and \( f_3(\widehat{x_9}) \) will be needed, in order to generalise to higher weight.
		\end{Rem}

			{
			\renewcommand{\crv}{{\mathcal{C}_{\mathrm{sh}}}}
			\begin{Lem}[The (2,1)-shuffle, $\shsym{2,1}$]\label{lem:21shuffleDeriv}
				Degeneration of \( \QU_6 \) to the stable curve 
				\[ 
				\crv = 1348 \cup_p 25679  = \,\,
				\vcenter{\hbox{\includegraphics[page=2]{figures/wt6.pdf}}}
				\,,
				\]
				produces the identity 
				\biggerskip
				\[
					\lif\big([134p,p569,67p9] + [p569,134p,67p9] + [p569,67p9,134p] \big) \equiv 0 \modtwo \,.
				\]
				In affine coordinates, with \( A = 134p  , B_1 = p569 , B_2 = 67p9  \), this is 
				\[
					\lif\big( [A,B_1,B_2] + [B_1,A,B_2] + [B_1,B_2,A] \big) \equiv 0 \modtwo \,,
				\]
				whence \( A \) is (riffle-)shuffled into \( (B_1, B_2) \).			
			\end{Lem}	
		
			\begin{Rem}
				This is an extension of \autoref{lem:wt4:sh11}.  The full identity can be found in \autoref{app:syms}.
			\end{Rem}
		
			\begin{proof}
				We treat the cases \( f_3(\widehat{x_i}) \) with \( 5 \leq i \leq 7 \), with \( i = 8 \), with \( i = 1 \), with \( i = 3, 4 \), with \( i = 9 \), and with \( i = 2 \) separately.  Only \( f_3(\widehat{x_2}) \) contributes.
				
				\case{Case \( f_3(\widehat{x_i}) \), with \( 5 \leq i \leq 7 \)} Every term in \( f_3(y_1,\ldots,y_8) \) has an argument with exactly two consecutive entries populated from \( \{ y_1, y_3, y_4, y_7 \} \).  $\llbracket$Respectively argument 1 for terms 9--10, argument 2 for terms 1, 4, 12 and 15, and argument 3 for terms 2--3, 5--8, 11, and 13--14.$\rrbracket$ For \( f_3(\widehat{x_i}) \), with \( 5 \leq i \leq 7 \) on \( \crv \), these degenerate to 0 (here \( \infty \) does not occur), hence \( f_3(\widehat{x_i}) \big\rvert_{\crv} \equiv 0 \modtwoi \).
				
				\case{Case \( f_3(\widehat{x_8}) \)} Every  term in \( f_3(y_1,\ldots,y_8) \) has an argument with exactly two consecutive entries populated from \( \{ y_1, y_3, y_4 \} \).  $\llbracket$Respectively argument 1 for terms 9--12, argument 2 for terms 1--2, 4, 6, and 15, and argument 3 for terms 3, 5, 7--8, and 13--14.$\rrbracket$  For \( f_3(\widehat{x_8}) \) on \( \crv \), these degenerate to 0 (here \( \infty \) does not occur), hence \( f_3(\widehat{x_8}) \big\rvert_{\crv} \equiv 0 \modtwoi \).

				\case{Case \( f_3(\widehat{x_1}) \)} Every  term in \( f_3(y_1,\ldots,y_8) \) has an argument with exactly two consecutive entries populated from \( \{ y_2, y_3, y_7 \} \).  $\llbracket$Respectively argument 1 for terms 1--3 and 6--8, argument 2 for terms 4, 8--9, 11 and 13, and argument 3 for terms 5, 8, 10, 12, and 14--15.$\rrbracket$ For \( f_3(\widehat{x_1}) \) on \( \crv \) these degenerate to \( \infty \) (or \( 0 \) for term 8, argument 2).  Hence \( f_3(\widehat{x_1}) \big\rvert_{\crv} \equiv 0 \modtwoi \).

				\case{Case \( f_3(\widehat{x_i}) \), with \( i = 3, 4 \)} Every argument of \( f_3(y_1,\ldots,y_8) \) has one of the following forms: i) exactly two non-consecutive entries populated from \( \{ y_1, y_3, y_7 \} \), ii) a single entry \( y_3 \) and neither of \( y_1, y_7 \), or iii) no entry \( y_3 \).  $\llbracket$For type i) argument 1 of terms 1--3, 6--8, and 11--14, argument 2 of terms 2--3, 6--7, 9, 11 and 13, and argument 3 of terms 10, 12, and 15.  For type ii) argument 2 of terms 1 and 4, and argument 3 of terms 3, 5, 7--8, and 13--14.  For type iii) argument 1 of terms 4--5, and 9--10, argument 2 of terms 5, 8, 10, 12, and 14--15 , and argument 3 of terms 1--2, 4, 6, 9, and 11.$\rrbracket$  For \( f_3(\widehat{x_i}) \), \( i = 3, 4 \) on \( \crv \), arguments of type i) degenerate to 1, arguments of type ii) and iii) degenerate equally for \( i = 3, 4 \), respectively having 2, 1 and 0 entries lying on component $1348$ of \( \crv \).  Hence \( f_3(\widehat{x_3}) \big\rvert_{\crv} = f_3(\widehat{x_4}) \big\rvert_{\crv} \) \emph{term-wise}, and so cancel in \( \QU_6 \big\rvert_{\crv} \).
				
				\paragraph{Note:} There is something here to check.  For example, argument \( [y_1,y_3,y_7,y_2] \) would fail these conditions.  For \( f_3(\widehat{x_3}) \) on \( \crv \) it would become \( [x_1, x_4, x_8, x_2 ] \) which degenerates to \( [148p] \), while for \( f_3(\widehat{x_4}) \) on \( \crv \) it would become \( [x_1, x_3, x_8, x_2] \) which degenerates to \( [138p] \), a different value.

				\case{Case \( f_3(\widehat{x_9}) \)} Terms 3--8, and 13--15 of \( f_3(y_1,\ldots,y_8) \) have an argument containing exactly two consecutive entries populated from \( \{ y_1, y_3, y_4, y_8 \} \).  $\llbracket$Respectively argument 1 of terms 4--8, and 13--15, and argument 2 of term 3.$\rrbracket$  For \( f_3(\widehat{x_9}) \) on \( \crv \) these degenerate to 0 or \( \infty \).  We directly compute for the remaining terms 1--2, and 9--12:
				\begin{align*}
					& f_3(\widehat{x_9}) \big\rvert_\crv \\
					& = \begin{aligned}[t]
					\lif\big( 
					& -[1238,3458,5678]
					+[1238,3478,5674]
					+ \text{$\llbracket$terms 4--8$\rrbracket$}
					\\
					& 
					-[1458,1234,5678]
					-[1458,5678,1234]
					+[1478,1234,5674]
					+[1478,5674,1234] \\
					& + \text{$\llbracket$terms 13--15$\rrbracket$}
					 \smash{\big) \big\rvert_{\crv}} \,,
					 \end{aligned} \\[0.5ex]
					 & \equiv  \begin{aligned}[t]
					 \lif\big( 
					 & -[1p38,\overset{\mathclap{\text{term 1}}}{34p8},567p]
					 +[1p38,\overset{\mathclap{\text{term 2}}}{34p8},567p]
					 \\
					 & 
					 -[14p8,\underset{\mathclap{\text{term 9}}}{1p34},567p]
					 -[14p8,\underset{\mathclap{\text{term 10}}}{567p},1p34]
					 +[14p8,\underset{\mathclap{\text{term 11}}}{1p34},567p]
					 +[14p8,\underset{\mathclap{\text{term 12}}}{567p},1p34] \smash{\big)} \end{aligned} \\
					 & \equiv 0 \modtwo \,,
				\end{align*}
				as the terms all pair-wise cancel, position 1 with position 2, position 3 with position 5 and position 4 with position 6.

				\case{Case \( f_3(\widehat{x_2}) \)} We directly compute
				\begin{align*}
				& f_3(\widehat{x_2}) \big\rvert_{\crv} \\
				& = \begin{aligned}[t]
				\lif\big( 
				& -[1349,4569,6789]
				+[1349,4589,6785]
				-[1349,4789,4567]
				+[1369,4563,6789]
				\\
				& 
				+[1369,6789,4563]
				-[1389,4583,6785]
				+[1389,4783,4567]
				-[1389,6783,4563]
				\\
				&
				-[1569,1345,6789]
				-[1569,6789,1345]
				+[1589,1345,6785]
				+[1589,6785,1345]
				\\
				&
				-[1789,1347,4567]
				+[1789,1367,4563]
				-[1789,1567,1345]
				\, \smash{\big) \big\rvert_{\crv}} \end{aligned} \\[0.5ex]
				& = \begin{aligned}[t]
				\lif\big( 
					& -[134p,p569,67p9]
					+[134p,1,67p5]
					-[134p,1,p567]
					+[0,\infty,67p9]
					\\
					& 
					+[0,67p9,\infty]
					-[138p,4p83,67p5]
					+[138p,4p83,p567]
					-[138p,0,\infty]
					\\
					&
					-[p569,134p,67p9]
					-[p569,67p9,134p]
					+[1,134p,67p5]
					+[1,67p5,134p]
					\\
					&
					-[1,134p,p567]
					+[1,0,\infty]
					-[1,p567,134p]
				\, \smash{\big)} \,, \end{aligned} \\
				& \equiv \lif(-[134p,p569,67p9] -[p569,134p,67p9] -[p569,67p9,134p] ) \modtwo
				\end{align*}
				We see that terms 4--5, 8, and 14 degenerate.  We also see that term 2 and 3 cancel (after applying cross-ratio symmetries), as do term 6 and 7, term 11 and 13, and term 12 and 15.  Only terms 1, 9, and 10 survive.
				
				\case{Overall} The only contribution is from \( f_3(\widehat{x_2}) \), giving
				\[
					\QU_6 \big\rvert_{\crv} \equiv 
					\lif(-[134p,p569,67p9] -[p569,134p,67p9] -[p569,67p9,134p] ) \modtwo \,.
				\]
				Up to a global sign, this is the identity in the statement of the lemma.
			\end{proof}
			}
				
			\begin{Cor}[Reversal, \revsc]\label{cor:reverse}
				The following reversal symmetry holds
				\[
				\lif(A,B,C) \equiv \lif(C, B, A) \modtwo \,.
				\]
				
				\begin{proof}
					From \autoref{lem:21shuffleDeriv} in the case \( A \) shuffled with \( (C,B) \), and in the case \( C \) shuffled with \( (A,B) \) we obtain respectively
					\begin{align*}
					\lif\big([A, C, B] + [C, A, B] + [C, B, A]\big) \equiv 0 \modtwo \,, \\
					\lif\big([C, A, B] + [A, C, B] + [A, B, C]\big) \equiv 0 \modtwo \,.
					\end{align*}
					Subtracting these two lines leaves the required identity
					\[
					\lif(A, B, C) - \lif(C, B, A) \equiv 0 \modtwo \,. \qedhere
					\]
				\end{proof}	
			\end{Cor}
		
			\begin{Cor}[Three-term relation]\label{cor:three}
				The following three-term identity, involving cyclic permutations of \( (A,B,C) \), holds
				\[
					\lif(A,B,C) + \lif(B,C,A) + \lif(C,A,B) \equiv 0 \modtwo \,.
				\]
				
				\begin{proof}
					From \autoref{lem:21shuffleDeriv} in the case \( A \) shuffled with \( (B,C) \), we have
					\begin{align*}
						\lif(A, B, C) + \lif(B, A, C) + \lif(B, C, A) \equiv 0 \modtwo \,.
					\end{align*}
					Using \autoref{cor:reverse} the middle term satisfies \( \lif(B, A, C) \equiv \lif(C, A, B) \modtwoi \), giving the required identity.
				\end{proof}	
			\end{Cor}
		
			\begin{Cor}[Inversion, \invsc]\label{cor:inv}
				The following inverse symmetry holds
				\begin{equation}\label{eqn:inv}
					\lif(A, B, C) \equiv -\lif(A^{-1}, B^{-1}, C^{-1})  \modtwo \,.
				\end{equation}
				
				\begin{proof}
					From \autoref{lem:revinvDeriv}, we have
					\[
						\lif(A, B, C) + \lif(C^{-1}, B^{-1}, A^{-1}) \equiv 0 \modtwo \,.
					\]
					The second term satisfies \( \lif(C^{-1}, B^{-1}, A^{-1}) \equiv \lif(A^{-1},B^{-1}, C^{-1}) \modtwoi \), using \autoref{cor:reverse}.  This gives the required identity.
				\end{proof}
			\end{Cor}
		
			\begin{Rem}
				The full identities for \autoref{cor:reverse}, \autoref{cor:three} and \autoref{cor:inv} can be found in \autoref{app:syms}.
			\end{Rem}

	\subsection{Reduction of \texorpdfstring{\( \lif \)}{Li\_\{3;1,1,1\}} when two arguments equal 1}\label{sec:wt6:nielsen:11x}

	We show how \( \lif(1,1,x) \) reduces to depth 2.  This follows by degenerating to a single stable curve (without having to utilise any further symmetries or relations).

	{
	\renewcommand{\crv}{{\mathcal{U}}}
	\begin{Prop}[$\lif(1,1,x)$ reduction, $\redid{1,1,x}$]\label{prop:11x}
		Degenerating \( \mathbf{Q}_6 \) to the stable curve 
		\[ 
		\crv = 168 \cup_p 35 \cup_q 2479 = \,\,
		\vcenter{\hbox{\includegraphics[page=3]{figures/wt6.pdf}}}
		\,,
		\]
		produces the reduction
		\biggerskip
		\[
			\lif(1,1,pq35) \equiv 0 \modtwo \,.
		\]
		In affine coordinates with \( A = pq35 \), this is 
		\[
			\lif(1,1,A) \equiv 0 \modtwo \,.
		\]
	\end{Prop}

	\begin{Rem}
		This is an extension of \autoref{lem:wt4deg}.  The full identity can be found in \autoref{app:li11x}.
	\end{Rem}

	\begin{proof}
	We treat the cases \( f_3(\widehat{x_i}) \), with \( i = 1 \), with \( i = 3 \), with \( i = 5, 6 \), with \( i = 7 \), with \( i = 8 \), with \( i = 9 \), with \( i = 2 \) and with \( i= 4 \) separately.  Only \( f_3(\widehat{x_4}) \) contributes.

	\case{Case \( f_3(\widehat{x_1}) \)} Every term in \( f_3(y_1, \ldots, y_8) \) has an argument with exactly 2 consecutive entries populated from \( \{ y_1, y_3, y_6, y_8 \} \).  $\llbracket$Respectively, argument 1 of terms 4--12, argument 2 of terms 1--2, and 14--15, and argument 3 of terms 3 and 13.$\rrbracket$  For \( f_3(\widehat{x_1}) \) on \( \crv \), these degenerate to \( \infty \), hence \( f_3(\widehat{x_1}) \big\rvert_{\crv} \equiv 0 \modtwoi \).
	
	\case{Case \( f_3(\widehat{x_3}) \)} Every term in \( f_3(y_1, \ldots, y_8) \) has an argument with exactly 2 consecutive entries populated from \( \{ y_2, y_3, y_6, y_8 \} \).  $\llbracket$Respectively, argument 2 of terms 1--2, 4, 6, 9, 11, and argument 3 of terms 3, 5, 7--8, 10, 12--15.$\rrbracket$  For \( f_3(\widehat{x_3}) \) on \( \crv \), these degenerate to \( \infty \), hence \( f_3(\widehat{x_3}) \big\rvert_{\crv} \equiv 0 \modtwoi \).
	
	\case{Case \( f_3(\widehat{x_i}) \), with \( i = 5, 6\)} Every argument in \( f_3(y_1, \ldots, y_8) \) has exactly 2 non-consecutive entries populated from \( \{ y_2, y_4, y_6, y_8 \} \).  For \( f_3(\widehat{x_i}) \), \( i = 5, 6 \)  on \( \crv \), these degenerate to \( 1 \), hence every term degenerates to \( \lif(1,1,1) \) .  Totalling the coefficients 8 times \( -1 \), and 7 times \( +1 \) gives \( f_3(\widehat{x_i}) \big\rvert_{\crv} = -\lif(1,1,1) \).  In \( \QU_6 \big\rvert_{\crv} \), these contributions cancel exactly.  (Alternatively: by \autoref{lem:revinvDeriv} with \( A = B = C = 1 \), already \( \lif(1,1,1) \equiv 0 \modtwoi \).)
	
	\case{Case \( f_3(\widehat{x_7}) \)} Every term in \( f_3(y_1, \ldots, y_8) \) has an argument with exactly 2 consecutive entries populated from \( \{ y_1, y_6, y_7 \} \).  $\llbracket$Respectively, argument 2 of terms 3, 5, 7--8, 10, 12--15, and argument 3 of terms 1--2,4,6,9, and 11.$\rrbracket$  For \( f_3(\widehat{x_7}) \) on \( \crv \), these degenerate to \( \infty \), hence \( f_3(\widehat{x_7}) \big\rvert_{\crv} \equiv 0 \modtwoi \).
	
	\paragraph{Note} The symmetry between degenerations of arguments 2 and 3 of \( f_3(\widehat{x_3}) \) verses arguments 3 and 2 of \( f_3(\widehat{x_7}) \) is interesting to observe, and should be understood combinatorially.
	
	\case{Case \( f_3(\widehat{x_8}) \)} Every term in \( f_3(y_1, \ldots, y_8) \) has an argument with exactly 2 consecutive entries populated from \( \{ y_2, y_4, y_7, y_8 \} \).  $\llbracket$Respectively, argument 1 of terms 13--15, argument 2 of terms 3, 5, 7--8, , 10 and 12, and argument 3 of terms 1--2, 4, 6, 9, and 11.$\rrbracket$  For \( f_3(\widehat{x_8}) \) on \( \crv \), these degenerate to \( 0 \), hence \( f_3(\widehat{x_8}) \big\rvert_{\crv} \equiv 0 \modtwoi \).
	
	\case{Case \( f_3(\widehat{x_9}) \)} Every term in \( f_3(y_1, \ldots, y_8) \) has an argument with exactly 2 consecutive entries populated from \( \{ y_1, y_6, y_8 \} \).  $\llbracket$Respectively, argument 1 of terms 1--12, and argument 2 of terms 13--15.$\rrbracket$  For \( f_3(\widehat{x_9}) \) on \( \crv \), these degenerate to \( \infty \), hence \( f_3(\widehat{x_9}) \big\rvert_{\crv} \equiv 0 \modtwoi \).

	\case{Case \( f_3(\widehat{x_2}) \)} By direct computation 
	\begin{align*}
	& f_3(\widehat{x_2}) \rvert_\crv  \\
	& = \begin{aligned}[t]
	\lif\big(
	& - [1349, 4569, 6789] + [1349, 4589, 6785] - [1349, 4789, 4567] + [1369, 4563, 6789] \\
	& + [1369, 6789, 4563] - [1389, 4583, 6785] + [1389, 4783, 4567] - [1389, 6783, 4563] \\
	& - [1569, 1345, 6789] - [1569, 6789, 1345] + [1589, 1345, 6785] + [1589, 6785, 1345] \\
	& - [1789, 1347, 4567] + [1789, 1367, 4563] - [1789, 1567, 1345] \smash{\big) \rvert_\crv}
	\end{aligned} \\[1ex]
	& = \begin{aligned}[t]
	\lif\big(
	& - [0, \infty, 1] + [0, \infty, 1] - [0, 47p9, \infty] + [1, q5p3, 1] \\
	& + [1, 1, q5p3] - [1, q5p3, 1] + [1, 0, \infty] - [1, 1, q5p3] \\
	& - [1, p3q5, 1] - [1, 1, p3q5] + [1, p3q5, 1] + [1, 1, p3q5] \\
	& - [1, 0, \infty] + [1, 1, q5p3] - [1, 1, p3q5] \smash{\big)}
	\end{aligned}
	\end{align*}
	We see that terms 1, 2, 3, 7 and 13 degenerate.  We also see that terms 4 and 6 cancel, as do terms 5 and 8, terms 9 and 11, terms 10 and 12, and terms 14 and 15.  So \( f_3(\widehat{x_2}) \rvert_\crv \equiv 0 \modtwoi \).

	\case{Case \( f_3(\widehat{x_4}) \)} By direction computation, only one term survives
	\begin{align*}
	& f_3(\widehat{x_4}) \rvert_\crv  \\
	& = \begin{aligned}[t]
	\lif\big(
		& - [1239, 3569, 6789] + [1239, 3589, 6785] - [1239, 3789, 3567] + [1269, 3562, 6789] \\
		& + [1269, 6789, 3562] - [1289, 3582, 6785] + [1289, 3782, 3567] - [1289, 6782, 3562] \\
		& - [1569, 1235, 6789] - [1569, 6789, 1235] + [1589, 1235, 6785] + [1589, 6785, 1235] \\
		& - [1789, 1237, 3567] + [1789, 1267, 3562] - [1789, 1567, 1235]  \smash{\big) \rvert_\crv}
	\end{aligned} \\[1ex]
	& = \begin{aligned}[t]
	\lif\big(
	& - [1, 35pq, 1] + [1, 35pq, 1] - [1, 1, 35pq] + [1, 35pq, 1] \\
	& + [1, 1, 35pq]  - [1, 35pq, 1] + [1, 1, 35pq] - [1, 1, 35pq] \\
	& - [1, pq35, 1] - [1, 1, pq35] + [1, pq35, 1] + [1, 1, pq35] \\
	& - [1, 1, 35pq] + [1, 1, 35pq] - [1, 1, pq35] \smash{\big)}
	\end{aligned}
	\end{align*}
	We see terms 1--2, 4, 6, 9, and 11 combine to 0 with 3 coefficients \( +1 \) and 3 coefficients \( -1 \). Whereas the remaining terms combine to \( -\lif(1,1,pq35) \) with 4 coefficients \( +1 \) and 5 coefficients \( -1 \).  Hence \( f_3(\widehat{x_4}) \rvert_\crv \equiv -\lif(1,1,pq35) \modtwoi \).
	
	\case{Overall} The only contribution is from \( f_3(\widehat{x_4}) \), giving
	\[
		\QU_6 \big\rvert_{\crv} \equiv -\lif(1,1,pq35) \modtwo \,.
	\]
	Up to a sign, this is the reduction given in the statement of the lemma.
	\end{proof}
	}

		\begin{Cor}[Reduction of \( \lif \) with two 1's, $\redid{1,1,x}^\sigma$]\label{cor:two1s}
		The following reductions hold
		\[
		\lif(1,1,A) \equiv \lif(A,1,1) \equiv \lif(1,A,1) \equiv 0 \modtwo \,.
		\]		
		\end{Cor}
	
		\begin{proof}
			The first is just \autoref{prop:11x} for completeness.  Applying $\revsc$ from \autoref{cor:reverse} to the reduction of \( \lif(A,1,1) \) from \autoref{prop:11x} gives
			\[
			\lif(A,1,1) \equiv \lif(1,1,A) \equiv 0 \modtwo \,,
			\]
			From the (2-1)-shuffle $\shsym{2,1}$ from \autoref{lem:21shuffleDeriv} of 1 shuffled with \( (1,A) \), we have
			\[
			\lif(1,A,1)  + 2\lif(1,1,A) \equiv 0 \modtwo \,,
			\]
			which then shows \( \lif(1,A,1) \equiv 0 \modtwoi \), as the second term reduces by \autoref{prop:11x}.  This establishes the claim.
		\end{proof}
	
	\subsection{Reduction of \texorpdfstring{\( \lif \)}{Li\_\{3;1,1,1\}} when one argument equals 1}\label{sec:wt6:nielsen:1xy}

	We show that \( \lif(1,x,y) \) reduces to depth 2.  We start by establishing some non-trivial symmetries of \( \lif(1,x,y) \) in \autoref{lem:onexy_sym1}, and \autoref{lem:onexy_sym2}.  Using these symmetries and inversion (from \autoref{cor:inv}), we deduce that \( \lif(1,x,y) \) vanish, modulo depth 2, in \autoref{thm:onexy_dp2}.
	
	{
	\renewcommand{\crv}{{\mathcal{D}_1}}
	\begin{Lem}[Degenerate Symmetry 1, $\degsym{1}$]\label{lem:onexy_sym1}
		Degenerating \( \QU_6 \) to the stable curve
			\[
				\crv = (138 {} _p {} , 24 {}_q {} , 79 {}_r {} ) \cup  56 = \,\, 
		\vcenter{\hbox{\includegraphics[page=4]{figures/wt6.pdf}}}
			\,,
		\]
		produces the following symmetry
		\biggerskip
		\begin{equation}\label{eqn:degsym1}
		\lif(1, pq5r, 56pr) \equiv \lif(1, pq6r, p56q)  \modtwo \,.
		\end{equation}
		In affine coordinates with \( A = pq5r, B = 56pr \), this is
		\[
			\lif(1, A, B) \equiv \lif\Big(1, \frac{A(1-B)}{1 - AB}, \frac{AB - 1}{AB} \Big) \modtwo \,.
		\]
	\end{Lem}

	\begin{Rem}
		This is an extension of \autoref{lem:wt4:sym1alt}.  The full identity can be found in \autoref{app:degsym1}.
	\end{Rem}

	\begin{proof}
		We treat the cases \( f_3(\widehat{x_i}) \) with \( i = 2 \), with \( i = 3 \), with \( i = 5, 6 \), with \( i = 8 \), with \( i = 9 \), with \( i = 1 \), with \( i = 4 \) and with \( i= 7 \) separately.  Only \( f_3(\widehat{x_4}) \) and \( f_3(\widehat{x_7}) \) contribute.
		
		\case{Case \( f_3(\widehat{x_2}) \)} Every term in \( f_3(y_1,\ldots,y_8) \) contains an argument with exactly two consecutive entries populated from \( \{ y_1, y_2, y_7 \} \).  $\llbracket$Respectively argument 1 of terms 1--5, argument 2 of terms 6--9, 11, 13--14, and argument 3 of terms 10, 12 and 15.$\rrbracket$  For \( f_3(\widehat{x_2}) \) on \( \crv \) these degenerate to 0.
		
		\case{Case \( f_3(\widehat{x_3}) \)} Every term in \( f_3(y_1,\ldots,y_8) \) contains an argument with exactly two consecutive entries populated from \( \{ y_2, y_3 \} \).  $\llbracket$Respectively argument 1 of terms 1--3, argument 2 of terms 4, 6--7, 9, 11, and 13, and argument 3 of terms 5, 8, 10, 12, and 14--15.$\rrbracket$  For \( f_3(\widehat{x_3}) \) on \( \crv \) these degenerate to $\infty$.
		
		\case{Case \( f_3(\widehat{x_i}) \) with \( i = 5, 6\)} Every term in \( f_3(y_1,\ldots,y_8) \) contains \emph{two} arguments satisfying:\linebreak i) exactly two non-consecutive entries populated from \( \{y_1,y_3,y_7 \} \), ii) exactly two non-consecutive entries populated from \( \{ y_6,y_8 \} \), or iii) exactly two non-consecutive entries populated from \( \{ y_2, y_4 \} \). $\llbracket$Respectively arguments 1\&2 of terms 2--3, 6--7, 11 and 13, arguments 1\&3 of terms 1, 8, 12, and 14--15, and arguments 2\&3 of terms 4--5 and 9--10.$\rrbracket$  For \( f_3(\widehat{x_i}) \), \( i = 5, 6 \) on \( \crv \) these all degenerate to 1.  By \autoref{cor:two1s} these vanish, modulo depth 2.
		
		\case{Case \( f_3(\widehat{x_8}) \)} Every term in \( f_3(y_1,\ldots,y_8) \) contains an argument with exactly two consecutive entries populated from \( \{ y_7, y_8 \} \).  $\llbracket$Respectively argument 1 of terms 6--8 and 11--15, argument 2 of terms 2--3, 5 and 10, and argument 3 of terms 1, 4 and 9.$\rrbracket$  For \( f_3(\widehat{x_8}) \) on \( \crv \) these degenerate to \( 0 \).
		
		\case{Case \( f_3(\widehat{x_9}) \)} Every term in \( f_3(y_1,\ldots,y_8) \) contains an argument with exactly two consecutive entries populated from \( \{ y_1, y_3, y_8 \} \).  $\llbracket$Respectively argument 1 for terms 4--5, and argument 2 for terms 1--3.$\rrbracket$  For \( f_3(\widehat{x_9}) \) on \( \crv \), these degenerate to \(\infty\).
		
		\case{Case \( f_3(\widehat{x_1}) \)} Argument 1 of terms 6--8 in \( f_3(y_1,\ldots,y_8) \) contains exactly two consecutive entries populated from \( \{ y_2, y_7 \} \).  For \( f_3(\widehat{x_1}) \) on \( \crv \), these go to \( \infty \), so terms 6--8 vanish.  Terms 1, 3, 9--10, 13 and 15 contain \emph{two} arguments satisfying: i) exactly 2 non-consecutive entries populated from \( \{ y_6, y_8 \} \), or ii) exactly 2 non-consecutive entries populated from \( \{ y_1, y_3 \} \).  $\llbracket$Respectively arguments 1\&2 of terms 3 and 13, arguments 1\&3 of terms 1 and 15, and arguments 2\&3 of terms 9--10.$\rrbracket$  For \( f_3(\widehat{x_1}) \) on \( \crv \) these degenerate to 1, so by \autoref{cor:two1s} those terms also vanish.
		
		Computing explicitly with the remaining terms 2, 4, 5, 11, 12 and 14, we find
		\begin{align*}
		 f_3(\widehat{x_1}) \big\rvert_{\crv} 
		& \equiv \begin{aligned}[t]
		\lif\big( 
		& {} + [2349,\overset{\mathclap{\text{term 2}}}{4589},6785]
		+[2369,\overset{\mathclap{\text{term 4}}}{4563},6789]
		+[2369,\overset{\mathclap{\text{term 5}}}{6789},4563]
		\\
		& +[2589,\underset{\mathclap{\text{term 11}}}{2345},6785]
		+[2589,\underset{\mathclap{\text{term 12}}}{6785},2345]
		+[2789,\underset{\mathclap{\text{term 14}}}{2367},4563]
		\, \smash{\big) \big\rvert_{\crv}} \,, \end{aligned} \\
		& \equiv \begin{aligned}[t]
		\lif\big( & + [1,q5pr,6rp5]
				+[qp6r,q56p,1]
				+[qp6r,1,q56p]
				\\
				& +[q5pr,1,6rp5]
				+[q5pr,6rp5,1]
				+[1,qp6r,q56p]
		\, \smash{\big)} \,. \end{aligned} 
		\end{align*}
		The terms in positions 1, 4 and 5 form a (2,1)-shuffle of \( (q5pr, 6rp5) \) shuffled with 1, so by $\shsym{2,1}$ from \autoref{lem:21shuffleDeriv} they combine to 0, modulo depth 2.  Likewise the terms in position 2, 3 and 6 form a (2,1)-shuffle of \( (qp6r, q56p) \) shuffled with 1, so also combine to 0, modulo depth 2.  
		
		Overall this shows \( f(\widehat{x_1}) \big\rvert_{\crv} \equiv 0 \modtwoi \).
		
		\case{Case \( f_3(\widehat{x_4}) \)} Terms 1--3, 6--7, 9--13 and 15 contain \emph{two} arguments satisfying i) exactly two non-consecutive entries populated from \( \{ y_1, y_3, y_7 \} \), or ii) exactly two non-consecutive entries populated from \( \{ y_6, y_8 \} \).  $\llbracket$Respectively arguments 1\&2 of terms 2--3, 6--7, 11 and 13, arguments 1\&3 of terms 1, 12 and 15, and arguments 2\&3 of terms 9--10.$\rrbracket$  For \( f_3(\widehat{x_4}) \) on \( \crv \), these degenerate to 1, so by \autoref{cor:two1s} those terms vanish.
		
		Computing explicitly with the remaining terms 4--5, 8 and 14, we find
		\begin{align*}
		& f_3(\widehat{x_4}) \big\rvert_{\crv} \\
		& \equiv \begin{aligned}[t]
		\lif\big( 
		\smash{[1269,\overset{\mathclap{\text{term 4}}}{3562},6789]
		+[1269,\overset{\mathclap{\text{term 5}}}{6789},3562]
		-[1289,\overset{\mathclap{\text{term 8}}}{6782},3562]
		+[1789,\overset{\mathclap{\text{term 14}}}{1267},3562]}
		\, \smash{\big) \big\rvert_{\crv}} \,, \end{aligned}
		\\
		& \equiv \begin{aligned}[t]
		\lif\big( 
		[pq6r,p56q,1]
		+[pq6r,1,p56q]
		-[1,6rpq,p56q]
		+[1,pq6r,p56q]
		\, \smash{\big)} \modtwo \,. \end{aligned} 
		\end{align*}
		The terms in the last two positions cancel (after cross-ratio symmetries).  The terms in the first two positions are part of the (2,1)-shuffle of \( (pq6r, p56q) \) shuffled with 1, from $\shsym{2,1}$ from \autoref{lem:21shuffleDeriv},
		\[
			\lif([pq6r,p56q,1] + [pq6r,1,p56q] + [1,pq6r,p56q]) \equiv 0 \modtwo \,.
		\]
		Subtracting the (2,1)-shuffle identity from the \( f_3(\widehat{x_4}) \big\rvert_{\crv} \) result above shows that 
		\[
			 f_3(\widehat{x_4}) \big\rvert_{\crv} \equiv -\lif(1, pq6r, p56q) \modtwo \,.
		\]
		
		\case{Case \( f_3(\widehat{x_7}) \)} Terms 2--3, 6--8, and 11-15 contain \emph{two} arguments satisfying i) exactly two non-consecutive entries populated from \( \{ y_1, y_3, y_7 \} \), or ii) exactly two non-consecutive entries populated from \( \{ y_2, y_4 \} \). $\llbracket$Respectively arguments 1\&2 of terms 2--3, 6--7, and 11, and arguments 1\&3 of terms 8, 12, and 14--15.$\rrbracket$  For \( f_3(\widehat{x_7}) \) on \( \crv \), these degenerate to 1, so by \autoref{cor:two1s} those terms vanish.
		
		Computing explicitly with the remaining terms 1, 4--5 and 9--10, we find
		\begin{align*}
		f_3(\widehat{x_7})  \big\rvert_{\crv} 
		& \equiv \begin{aligned}[t]
		\lif\big( 
		& { -[1239,\overset{\mathclap{\text{term 1}}}{3459},5689]
		+[1259,\overset{\mathclap{\text{term 4}}}{3452},5689]
		+[1259,\overset{\mathclap{\text{term 5}}}{5689},3452] }
		\\
		&
		-[1459,\underset{\mathclap{\text{term 9}}}{1234},5689]
		-[1459,\underset{\mathclap{\text{term 10}}}{5689},1234]
		\, \smash{\big) \big\rvert_{\crv}} \end{aligned} \\
		& \equiv \begin{aligned}[t]
		\lif\big( 
		& -[1,pq5r,56pr]
		+[pq5r,1,56pr]
		+[pq5r,56pr,1]
		\\
		&
		-[pq5r,1,56pr]
		-[pq5r,56pr,1]
		\, \smash{\big)} \end{aligned}	\\
		& \equiv \lif(1, pq5r, 56pr) \modtwo \,,
		\end{align*}
		as the term in position 2 cancels with position 4, and position 3 cancels with position 5.
		
		\case{Overall} The only contributions are from \( f_3(\widehat{x_4}) \) and \( f_3(\widehat{x_7}) \), giving
		\begin{align*}
			\QU_6 \rvert_\crv & \equiv -\lif(1, pq6r, p56q) + \lif(1, pq5r, 56pr) \modtwo \,.
		\end{align*}
		This is equivalent to the symmetry given in the statement of the lemma.
	\end{proof}
	}

		\begin{Cor}\label{cor:onevar}
		The following one variable degenerations hold
		\begin{align*}
		&\lif(1, A, A^{-1}) \equiv 0 \modtwo \,, \text{ and }\\
		& \lif(A,A^{-1},1) \equiv \lif(A,1, A^{-1}) \equiv 0 \modtwo \,.
		\end{align*}
		\end{Cor}
	
		\begin{Rem}
			The full identity for \( \lif(1,A,A^{-1}) \) can be found in at the end of \autoref{app:degsym1}.
		\end{Rem}
	
		\begin{proof}
			Degenerating \( B \to A^{-1} \) in the degenerate symmetry $\degsym{1}$ from \autoref{lem:onexy_sym1}, gives
			\[
			\lif(1, A, A^{-1}) \equiv \lif(1, \infty, 0) \equiv 0 \modtwo \,.
			\]
			This establishes the first identity.  The second follows by applying reversal $\revsc$ from \autoref{cor:reverse} and substituting \( A \mapsto A^{-1} \).  The third follows directly (i.e. without applying $\degsym{1}$) from the (2,1)-shuffle $\shsym{2,1}$ from \autoref{lem:21shuffleDeriv} of \( A \) shuffled with \( (1, A^{-1}) \) along with inversion $\invsc$ from \autoref{cor:inv}.  Specifically, the (2,1)-shuffle reads
			\[
			\underbrace{\lif(1, A^{-1}, A) + \lif(1, A, A^{-1})}_{\invsc \, \Rightarrow \, \equiv \, 0 \modtwo} + \lif(A, 1, A^{-1}) \equiv 0 \modtwo \,.
			\]
			This establishes the result.
		\end{proof}

	{
	\renewcommand{\crv}{{\mathcal{D}_2}}
	\begin{Lem}[Degenerate Symmetry 2, $\degsym{2}$]\label{lem:onexy_sym2}
		Degenerating \( \QU_6 \) to the stable curve 
	\[
			\crv = 135 \cup_p 279 \cup_q 468 = \,\,
			\vcenter{\hbox{\includegraphics[page=5]{figures/wt6.pdf}}}
			\,,
			\]
		produces the symmetry
		\begin{equation}
		\label{eqn:degsym2}
			\lif(1, 2pq7,279p) \equiv - \lif(1, 2pq7,27q9) \modtwo \,
		\end{equation}
		In affine coordinates with  \( 	A = 2pq7, B = 279p  \), this is
		\[
			\lif(1,A,B) \equiv -\lif\Big(1, A, \frac{1 - A B}{A(1-B)}\Big) \modtwo \,.
		\]
		
		\begin{Rem}
			This is not an extension of results considered in weight 4, but one sees that degenerating \( \QU_4 \) to \( 135 \cup_p 27 \cup_q 46 \) gives another Nielsen reduction \( \liftwo(2pq7, 1) \equiv 0 \modonei \).  The full identity can be found in \autoref{app:degsym2}
		\end{Rem}
		
		\begin{proof}
			We treat the cases \( f_3(\widehat{x_i}) \) with \( i= 2 \), with \( i= 4 \), with \( i = 5 \), with \( i = 7 \), with \( i = 3 \), with \( i = 6 \), with \( i = 9 \), with \( i = 1 \) and \( i= 8 \) separately.  Only \( f_3(\widehat{x_1}) \) and \( f_3(\widehat{x_8}) \) contribute.  We then show how to put the resulting identity into the given form.
			
			\case{Case \( f_3(\widehat{x_2}) \)} Every term in \( f_3(y_1,\ldots,y_8) \) contains an argument with exactly two consecutive entries populated from \( \{ y_1, y_2, y_4 \} \).  $\llbracket$Respectively argument 1 of terms 1--12, and argument 2 of terms 13--15.$\rrbracket$  For \( f_3(\widehat{x_2}) \) on \( \crv \) these degenerate to \( 0 \).
				
			\case{Case \( f_3(\widehat{x_4}) \)} Every term in \( f_3(y_1,\ldots,y_8) \) contains an argument with exactly two consecutive entries populated from \( \{ y_1, y_3, y_4 \} \).  $\llbracket$Respectively argument 1 of terms 9--12, argument 2 of terms 1--2, 4, 6 and 15, and argument 3 of terms 3, 5, 7--8 and 13--14.$\rrbracket$  For \( f_3(\widehat{x_4}) \) on \( \crv \) these degenerate to \( 0 \).
			
			\case{Case \( f_3(\widehat{x_5}) \)} Every term in \( f_3(y_1,\ldots,y_8) \) contains an argument with exactly two consecutive entries populated from \( \{ y_4, y_5, y_7 \} \).  $\llbracket$Respectively argument 1 of terms 9--12, argument 2 of terms 1--2, 4, 6 and 15, and argument 3 of terms 3, 5, 7--8 and 13--14.$\rrbracket$  For \( f_3(\widehat{x_5}) \) on \( \crv \) these degenerate to \( \infty \).
	
			\case{Case \( f_3(\widehat{x_7}) \)} Every term in \( f_3(y_1,\ldots,y_8) \) contains an argument with exactly two consecutive entries populated from \( \{ y_4, y_6, y_7 \} \).  $\llbracket$Respectively argument 1 of terms 11--15, argument 2 of terms 2--3, 5--8, and 10, and argument 3 of terms 1, 4 and 9.$\rrbracket$  For \( f_3(\widehat{x_7}) \) on \( \crv \) these degenerate to \( \infty \).
	
			\case{Case \( f_3(\widehat{x_3}) \)} Arguments 2 and 3 in terms 1--8 \( f_3(y_1,\ldots,y_8) \) each have exactly two non-consecutive entries populated from \( \{ y_3, y_5, y_7 \} \).  For \( f(\widehat{x_3}) \) on \( \crv \), these go to 1, so by \autoref{cor:two1s} terms 1--8 disappear.  Moreover, argument 1 of terms 9--12 and argument 2 of term 15 contain exactly two consecutive entries populated from \( \{ y_1, y_4 \} \).  For \( f(\widehat{x_3}) \) on \( \crv \) these to go 0, hence terms 9--12 and term 15 also disappear. 
			
			Computing explicitly with the remaining terms 13 and 14, we find
			\begin{align*}
				f_3(\widehat{x_3}) \big\rvert_{\crv} & \equiv \lif( 
				-[1789,\overset{\mathclap{\text{term 13}}}{1247},4567]
				+[1789,\overset{\mathclap{\text{term 14}}}{1267},4562]
				 ) \big\rvert_{\crv} \\
				 & \equiv \lif( 
				 -[p7q9,p2q7,1]
				 +[p7q9,p2q7,1]
				 )  \\
				 & \equiv 0 \modtwo \,,
			\end{align*}
			as the two terms directly cancel.
	
			\case{Case \( f_3(\widehat{x_6}) \)} Terms 1, 3--5, 9--10, and 13--15 contain \emph{two} arguments which each have exactly two non-consecutive entries populated from \( \{ y_1, y_3, y_5 \} \).  $\llbracket$Respectively arguments 1\&2 of terms 1, 4 and 9, arguments 1\&3 of terms 3, 5 and 10, and arguments 2\&3 of terms 13--15.$\rrbracket$  For \( f_3(\widehat{x_6}) \) on \( \crv \) these degenerate to 1, so by \autoref{cor:two1s} those terms vanish.  Moreover, argument 1 of terms 11-12, and argument 2 of terms 2 and 6, have exactly two consecutive entries populated from \( \{ y_4, y_7 \} \).  For \( f_3(\widehat{x_6}) \) on \( \crv \) these go to \( \infty \), so those terms also vanish.
			
			Computing explicitly with the remaining terms 7 and 8, we find
			\begin{align*}
			f_3(\widehat{x_6}) \big\rvert_{\crv} & \equiv \lif( 
			-[1289,\overset{\mathclap{\text{term 7}}}{3782},3457]
			+[1289,\overset{\mathclap{\text{term 8}}}{5782},3452]
			) \big\rvert_{\crv} \\
			& \equiv \lif( 
			-[p2q9,p7q2,1]
			+[p2q9,p7q2,1]
			)  \\
			& \equiv 0 \modtwo \,,
			\end{align*}
			as the two terms directly cancel.
	
			\case{Case \( f_3(\widehat{x_9}) \)} Each argument of terms 1--5 and 9--15 of \( f_3(y_1,\ldots,y_8) \) contains either i) exactly two non-consecutive entries populated from \( \{ y_1,y_3, y_5 \} \) or ii) exactly two non-consecutive entries populated from \( \{ y_4, y_6, y_8 \} \).  For \( f_3(\widehat{x_9}) \) on \( \crv \), these degenerate to 1, so each term becomes \( \lif(1,1,1) \equiv 0 \modtwoi \), by \autoref{lem:revinvDeriv}, with \( A = B = C = 1 \).  (Actually there are 6 coefficients \( +1 \), and 6 coefficients \( -1 \), so the terms cancel directly.)
			
			Computing explicitly with the remaining terms 6--8, we find
			\begin{align*}
			f_3(\widehat{x_9}) \big\rvert_{\crv} & \equiv \lif( 
			-[1278,\overset{\mathclap{\text{term 6}}}{3472},5674]
			+[1278,\overset{\mathclap{\text{term 7}}}{3672},3456]
			-[1278,\overset{\mathclap{\text{term 8}}}{5672},3452]
			) \big\rvert_{\crv} \\
			& \equiv \lif( 
			-[p27q,pq72,1]
			+[p27q,pq72,1]
			-[p27q,pq72,1]  
			)  \\
			& \equiv -\lif(p27q, pq72, 1) \modtwo \,,
			\end{align*}
			as the terms in position 2 and 3 cancel.  Since \( p27q = (pq72)^{-1} \), this vanishes by \autoref{cor:onevar}.

			\case{Case \( f_3(\widehat{x_1}) \)} The last \emph{two} arguments of terms 1--12 in  \( f_3(y_1,\ldots,y_8) \) contain either i) exactly two non-consecutive entries populated from \( \{ y_3, y_5, y_7 \} \), or ii) exactly two non-consecutive entries populated from \( \{ y_2, y_4 \} \).  For \( f_3(\widehat{x_1}) \)  on \( \crv \), these degenerate to 1, so by \autoref{cor:two1s} those terms vanish.
			
			Computing explicitly with the remaining terms 13--15, we find
			\begin{align*}
				f_3(\widehat{x_1}) \big\rvert_{\crv} & \equiv \lif(
					-[2789,\overset{\mathclap{\text{term 13}}}{2347},4567]
					+[2789,\overset{\mathclap{\text{term 14}}}{2367},4563]
					-[2789,\overset{\mathclap{\text{term 15}}}{2567},2345]
				)\big\rvert_{\crv} \\
				& \equiv  \lif(
				-[27q9,2pq7,1]
				+[27q9,2pq7,1]
				-[27q9,2pq7,1]
				) \\
				& \equiv -\lif(27q9, 2pq7, 1) \modtwo \,.
			\end{align*}
	
			\case{Case \( f_3(\widehat{x_8}) \)}  Terms 1--5, and 9--15 in \( f_3(y_1,\ldots,y_8) \) contain \emph{two} arguments satisfying either i) exactly two non-consecutive entries populated from \( \{ y_1, y_3, y_5 \} \), or ii) exactly two non-consecutive entries populated from \( \{ y_4, y_6 \} \).  $\llbracket$Respectively arguments 1\&2 of terms 1, 4, and 9, arguments 1\&3 of terms 2--3, 5, and 10, and arguments 2\&3 of terms 11--15.$\rrbracket$  For \( f_3(\widehat{x_8}) \) on \( \crv \) these go to 1, so by \autoref{cor:two1s} those terms vanish.
			
				Computing explicitly with the remaining terms 6--8, we find
			\begin{align*}
			f_3(\widehat{x_8}) \big\rvert_{\crv} & \equiv \lif(
			-[1279,\overset{\mathclap{\text{term 6}}}{3472},5674]
			+[1279,\overset{\mathclap{\text{term 7}}}{3672},3456]
			-[1279,\overset{\mathclap{\text{term 8}}}{5672},3452]
			)\big\rvert_{\crv} \\
			& \equiv  \lif(
		-[p279,pq72,1]
		+[p279,pq72,1]
		-[p279,pq72,1]
			) \\
			& \equiv -\lif(p279, pq72, 1) \modtwo \,.
			\end{align*}
	
			\case{Overall} The only contributions are from \( f_3(\widehat{x_1}) \) and \( f_3(\widehat{x_8}) \), giving
			\begin{align*}
				\QU_6 \big\rvert_\crv 	& \equiv \lif(27q9, 2pq7, 1) - \lif(p279, pq72, 1) \modtwo \,.
			\end{align*}
			Then reverse each term using the reversal $\revsc$ from \autoref{cor:reverse}, giving
			\[
			 \equiv \lif(1, 2pq7, 27q9) -\lif(1, pq72,p279) \modtwo \,.
			\]
			Finally apply inversion $\invsc$ from \autoref{cor:inv} to the term in position 2 (which rotates each cross-ratio one step left or right)
			\[
			\equiv \lif(1, 2pq7, 27q9) + \lif(1, 2pq7,279p) \modtwo \,.
			\]
			This now equivalent the form of the symmetry given in the statement of the lemma.
			\end{proof}
	\end{Lem}
	}

	\begin{Thm}[Reduction of \( \lif(1,x,y) \), $\redid{1,x,y}$]\label{prop:onexy_120syms}\label{thm:onexy_dp2}
		The following reduction holds
		\[
			\lif(1,x,y) \equiv 0 \modtwo \,.
		\]
	\end{Thm}

		\begin{Rem}
			The full identity can be found in \autoref{app:onexy_dp2}.	
		\end{Rem}
	
		\begin{proof}
				Introduce 
				\[
				 	g(x_1,x_2,x_3,x_4,x_5) \coloneqq \lif(1,[x_3,x_1,x_4,x_2],[x_5,x_1,x_3,x_2]) \,.
				 \]
				 Then we note that each of \autoref{eqn:inv} ($\invsc$ from \autoref{cor:inv}),  as well as \autoref{eqn:degsym1} ($\degsym{1}$ from \autoref{lem:onexy_sym1}), and \autoref{eqn:degsym2} ($\degsym{2}$ from \autoref{lem:onexy_sym2}) can be written via a symmetry of \( g \), as follows.  (Arrows show how points move, bold and underline mark what has changed.)
				 \smallskip
			\begin{align*}
			\tag{$\invsc$}  g(\tikzmarknode{xi1}{x_1},\tikzmarknode{xi2}{x_2},x_3,x_4,x_5) &\equiv -g(\boldsymbol{x_2},\boldsymbol{x_1},x_3,x_4,x_5) \modtwo \,,
			\begin{tikzpicture}[overlay,remember picture]
			\draw[arrows={latex}-{latex}] ($(xi1.north) + (0,2pt)$) to[out=90,in=90,looseness=2] ($(xi2.north) + (0,2pt)$);
			\end{tikzpicture}
			\\[3ex]
			\tag{$\degsym{1}$}  g(\tikzmarknode{xa}{x_5},\tikzmarknode{xc}{p},\tikzmarknode{ya}{r},\tikzmarknode{yb}{q},\tikzmarknode{xb}{x_6}) &\equiv \phantom{+} g(\boldsymbol{p},\boldsymbol{x_6},\uline{q},\uline{r},\boldsymbol{x_5}) \modtwo \,, 
			\begin{tikzpicture}[overlay,remember picture]
			\draw[arrows={latex}-{latex}] ($(ya.south) + (0,-3pt)$) to[out=-90,in=-90,looseness=2.5] ($(yb.south) + (0,-1pt)$);
			\draw[arrows=-{latex}] ($(xa.north) + (-0.5pt,2pt)$) to[out=90,in=90,looseness=0.85] ($(xb.north) + (1pt,2pt)$);
			\draw[arrows=-{latex}] ($(xb.north) + (-1pt,2pt)$) to[out=125,in=45,looseness=1.0] ($(xc.north) + (0.5pt,2pt)$);
			\draw[arrows=-{latex}] ($(xc.north) + (-1pt,2pt)$) to[out=110,in=45,looseness=2.0] ($(xa.north) + (1pt,2pt)$);
			\end{tikzpicture}
			\\[3ex]
			\tag{$\degsym{2}$}  g(\tikzmarknode{z1}{p},\tikzmarknode{z2}{x_7},\tikzmarknode{z3}{x_2},\tikzmarknode{z4}{q},x_9) &\equiv -g(\boldsymbol{q}, \uline{x_2}, \uline{x_7}, \boldsymbol{p}, x_9) \modtwo \,. \\[1ex]
			\begin{tikzpicture}[overlay,remember picture]
			\draw[arrows={latex}-{latex}] ($(z1.south) + (0,-1pt)$) to[out=-90,in=-90,looseness=1.0] ($(z4.south) + (0,-1pt)$);
			\draw[arrows={latex}-{latex}] ($(z2.south) + (0,-1pt)$) to[out=-90,in=-90,looseness=1.75] ($(z3.south) + (0,-2pt)$);
			\end{tikzpicture}
			\end{align*}
			Notice that three applications of $\degsym{1}$ from \autoref{lem:onexy_sym1} produces a single switch of \( r, q \) overall.  The idea now is to show that \( g(x_1,\ldots,x_5) \) is symmetric and anti-symmetric in \( x_1, x_2 \), so must vanish.  By using $\degsym{1}$ we can move \( x_1, x_2 \) into the symmetric slots 3 and 4.  Explicitly, we have the following. (The braces indicate which points are changing at each step, the labels ``(via symmetry)'' indicate which symmetry is being used.)\medskip
			\begin{align*}
			& \phantom{{} \equiv {} +}g(\underbrace{x_1,\phantom{\mathclap{\big(}}x_2},x_3,x_4,x_5) \\[-1ex]
			\tag{via $\invsc$} &\equiv -g(\underbrace{\overbrace{x_2,\phantom{\mathclap{\big(}}x_1},x_3,x_4},x_5) \\[-1ex]
			\tag{via $\degsym{2}$} &\equiv \phantom{+} g(\overbrace{x_4,x_3,\underbrace{x_1,\phantom{\mathclap{\big(}}x_2}},x_5) \\[-1ex]
			\tag{via $\degsym{1}$ ($\times3$)} &\equiv \phantom{+} g(\underbrace{x_4,x_3,\overbrace{x_2,\phantom{\mathclap{\big(}}x_1}},x_5) \\[-1ex]
			\tag{via $\degsym{2}$} &\equiv -g(\overbrace{x_1,\phantom{\mathclap{\big(}}x_2,x_3,x_4},x_5) \modtwo \,. \\[-0.2ex]
			\end{align*}\medskip
			Hence \( g(x_1,\ldots,x_5) \equiv 0 \modtwoi \), and since
			\[
				\lif(1,x,y) = g(\infty, 0, 1, x, y^{-1}) \equiv 0 \modtwo \,,
			\] 
			we obtain the claimed reduction.
		\end{proof}

	\begin{Cor}[Reduction of \( \lif \) with a single 1, $\redid{1,x,y}^\sigma$]\label{cor:xyone_dp2}
			The following reductions hold
			\[
				\lif(1,x,y) \equiv \lif(x,y,1) \equiv \lif(x,1,y) \equiv 0 \modtwo \,.
			\]
	\end{Cor}

		\begin{Rem}
			A shorter identity for \( \lif(x,1,y) \) (which was found by direct search) is given at the end of \autoref{app:onexy_dp2}.
		\end{Rem}

			\begin{proof}
				The first is just \autoref{thm:onexy_dp2} for completeness.  The second follows from $\revsc$ from \autoref{cor:reverse} applied to the reduction of \( \lif(1,y,x) \) implied by \autoref{thm:onexy_dp2} (switching \( x \leftrightarrow y \)).  Then from the (2,1)-shuffle $\shsym{2,1}$ from \autoref{lem:21shuffleDeriv}, of \( (1,x) \) shuffled with \( y \), we have
				\[
					\lif(y,1,x) + \lif(1,y,x) + \lif(1,x,y)  \equiv 0 \modtwo \,,
				\]
				so the third result follows, as terms 2 and 3 reduce by \autoref{thm:onexy_dp2}.
			\end{proof}

	\noindent At this point we have established, unequivocally, that the higher Nielsen formulae in \autoref{conj:highernielsen} hold for \( k = 3 \).  We may now utilise them when continuing to degenerate \( \QU_6 \), in our quest to prove the higher Zagier formulae in \autoref{conj:higherzagier} for \( k = 3 \).
	
	\subsection{Six fold anharmonic symmetries of \texorpdfstring{\( \lif(x,y,z) \)}{Li\_\{3;1,1,1\}(x,y,z)}}\label{sec:wt6:6fold}

	We now build up to showing that \( \lif(x,y,z) \) satisfies the Zagier formulae, which reduce the two-term combinations \( [x] + [1-x] \) and \( [x] + [x^{-1}] \) to depth ${\leq}2$.  For this, we have to show \( \lif(x,y,z) \) satisfies two non-trivial symmetries \autoref{lem:fullsym1}, \autoref{lem:fullsym2} and a four-term relation \autoref{prop:fourterm} (already needed to show the second symmetry) which relates the two types of Zagier formulae.  Overall these symmetries generate a 216-fold symmetry group; we then play the symmetries against the four-term relation (\autoref{prop:zagcombRad}) to deduce one (hence, both) of the Zagier formulae.
	
	{
		\renewcommand{\crv}{{\mathcal{S}_1}}
	\begin{Lem}[Full Symmetry 1, $\fullsym{1}$]\label{lem:fullsym1}
		Degenerating \( \QU_6 \) to the stable curve  
			\[
			\crv = 29 \cup_p 4567 \cup_q 138  = \,\,
			\vcenter{\hbox{\includegraphics[page=6]{figures/wt6.pdf}}}
			\,,
		\]
		produces the symmetry\biggerskip
		\[
			\lif(p67q, pq56, p45q) \equiv -\lif(q67p, qp56, q45p) \modtwo \,.
		\]
		In affine coordinates with \( A = p67q, B = pq56, C = p45q \), this is
		\[
			\lif(A,B,C) \equiv -\lif\Big(1 - A, \frac{B}{B-1}, 1-C\Big) \modtwo \,.
		\]
	\end{Lem}

	\begin{Rem}
		The full identity can be found in \autoref{app:fullsym1:full}.
	\end{Rem}
		
		\begin{proof}
			We treat the cases \( f_3(\widehat{x_i}) \) with \( i = 1 \), with \( i = 2 \), with \( 4 \leq i \leq 7 \), with \( i= 9 \), with \( i = 3 \) and with \( i = 8 \) separately.  Only \( f_3(\widehat{x_3}) \) and \( f_3(\widehat{x_8}) \) contribute.  We then show how to put the resulting identity into the given form.
			
			\case{Case \( f_3(\widehat{x_1}) \)} The first argument of each term in \( f_3(y_1,\ldots,y_8) \) contains exactly 2 consecutive entries populated from \( \{ y_1, y_8 \} \).  For \( f_3(\widehat{x_1}) \) on \( \crv \), these degenerate to \( \infty \).
			
			\case{Case \( f_3(\widehat{x_2}) \)} Each term in \( f_3(y_1,\ldots,y_8) \) contains an argument with exactly two consecutive entries populated from \( \{ y_1, y_2, y_7 \} \).  $\llbracket$Respectively argument 1 of terms 1--5, argument 2 of terms 6--9, 11, 13--14, and argument 3 of terms 10, 12, and 15.$\rrbracket$ For \( f_3(\widehat{x_2}) \) on \( \crv \), these degenerate to 0.
			
			\case{Case \( f_3(\widehat{x_i}) \), with \( 4 \leq i \leq 7\)} The first argument of terms 1--8 in \( f_3(y_1,\ldots,y_8) \) contains exactly two non-consecutive entries populated from \( \{ y_2, y_8 \} \).  For \( f_3(\widehat{x_i}) \), with \( 4 \leq i \leq 7 \), on \( \crv \), these degenerate to 1, so by the reduction $\redid{1,x,y}^\sigma$ from \autoref{cor:xyone_dp2} these terms vanish, modulo depth 2.  Terms 9--15 in \( f_3(y_1,\ldots,y_8) \) contain an argument with exactly two non-consecutive entries populated from \( \{ y_1, y_3, y_7 \} \).  $\llbracket$Respectively argument 1 of terms 11--15, argument 2 of term 9 and argument 3 of term 10.$\rrbracket$  For \( f_3(\widehat{x_i}) \), with \( 4 \leq i \leq 7 \), on \( \crv \), these degenerate to 1, so by $\redid{1,x,y}^\sigma$ from \autoref{cor:xyone_dp2} these terms also vanish, modulo depth 2.
			
			\case{Case \( f_3(\widehat{x_9}) \)} Each term in \( f_3(y_1,\ldots,y_8) \) contains an argument with exactly two consecutive entries populated from \( \{ y_1, y_3, y_8 \} \).  $\llbracket$Respectively argument 1 of terms 4--15, and argument 2 of terms 1--3.$\rrbracket$  For \( f_3(\widehat{x_9}) \) on \( \crv \), these degenerate to \( \infty \).
			
			\case{Case \( f_3(\widehat{x_3}) \)} Argument 1 of terms 1--8 in \( f_3(y_1,\ldots,y_8) \) contains exactly two non-consecutive entries populated from \( \{ y_2, y_8 \} \).  For \( f_3(\widehat{x_3}) \) on \( \crv \), these degenerate to 1.  Likewise, argument 1 of terms 11--15 in \( f_3(y_1,\ldots,y_8) \) contains exactly two non-consecutive entries populated from \( \{ y_1, y_7 \} \).  For \( f_3(\widehat{x_3}) \) on \( \crv \), these degenerate to 1.  So by $\redid{1,x,y}^\sigma$  in \autoref{cor:xyone_dp2} these terms vanish, modulo depth 2.
			
			Computing explicitly with the remaining terms 9--10, we find
			\begin{align*}
			f_3(\widehat{x_3}) \big\rvert_{\crv} & \equiv \lif(
			{}-[1569,\overset{\mathclap{\text{term 9}}}{1245},6789]
			-[1569,\overset{\mathclap{\text{term 10}}}{6789},1245]
			)\big\rvert_{\crv} \\
			& \equiv  \lif(
				{}-[q56p,qp45,67qp]
				-[q56p,67qp,qp45]
			) \\
			& \equiv \lif(67qp,q56p,qp45) \modtwo \,.
			\end{align*}
			In the last line we have used the (2,1)-shuffle $\shsym{2,1}$ from \autoref{lem:21shuffleDeriv} of \( (q56p,qp45) \) shuffled with \( (67qp) \), to replace two terms by the third one.
			
			\case{Case \( f_3(\widehat{x_8}) \)} Argument 1 of terms 1--8 in \( f_3(y_1,\ldots,y_8) \) contains exactly two non-consecutive entries populated from \( \{ y_2, y_8 \} \).  For \( f_3(\widehat{x_8}) \) on \( \crv \), these degenerate to 1.  Likewise, terms 9--13 and 15 of \( f_3(y_1,\ldots,y_8) \) contain an argument with exactly two non-consecutive entries populated from \( \{ y_1, y_3 \} \).  $\llbracket$Respectively argument 2 of terms 9, 11 and 13, and argument 3 of terms 10, 12, and 15.$\rrbracket$.  For \( f_3(\widehat{x_8}) \) on \( \crv \), these degenerate to 1.  So by $\redid{1,x,y}^\sigma$ from \autoref{cor:xyone_dp2} these terms vanish, modulo depth 2.
			
			Computing explicitly with the remaining term 14, we find
			\begin{align*}
			f_3(\widehat{x_8}) \big\rvert_{\crv} & \equiv \lif(
						{}+[1679,\overset{\mathclap{\text{term 14}}}{1256},3452]
			)\big\rvert_{\crv} \\
			& \equiv  \lif(
						q67p,qp56,q45p
			) \modtwo \,.
			\end{align*}

			\case{Overall} The only contributions are from \( f_3(\widehat{x_3}) \) and \( f_3(\widehat{x_8}) \), giving
			\begin{align*}
				\QU_6 \big\rvert_{\crv} & \equiv - \lif(67qp,q56p,qp45) +  \lif(q67p,qp56,q45p)
			\end{align*}
			Now invert the first term using $\invsc$ from \autoref{cor:inv} (which rotates each cross-ratio left by one step), and we obtain
			\[
			\lif(p67q, pq56, p45q) + \lif(q67p, qp56, q45p) \equiv 0 \modtwo \,.
			\]		
			This is equivalent to the form of the symmetry given in the statement of the lemma.
		\end{proof}
	}

	\begin{Prop}[12-fold symmetries of \( \lif \)] \label{prop:twelve}
		The inversion $\invsc$ from \autoref{cor:inv}, reversal $\revsc$ from \autoref{cor:reverse}, and the symmetry $\fullsym{1}$ from \autoref{lem:fullsym1} generate 12 symmetries, modulo depth 2 of
		\[
					 [\eps; x, y, z] \coloneqq \eps \lif(x,y,z) \,.
		\]
		The symmetries are given as follows, where the last and first of each batch being equivalent via $\invsc$, and the two batches are equivalent via $\revsc$.
		\begin{align*}
		\left.\begin{matrix}
		\begin{aligned}
		& [+; x, y, z] \\
		\overset{\mathclap{\text{\sc FS}_1}}{\equiv} {} \,\, & [-; 1-x,\tfrac{y}{y-1},1-z]  \\
		\overset{\mathclap{\invsc}}{\equiv} {}\,\, & [+; \tfrac{1}{1-x},\tfrac{y-1}{y},\tfrac{1}{1-z}] \\
		\overset{\mathclap{\text{\sc FS}_1}}{\equiv} {} \,\,& [-; \tfrac{x}{x-1},1-y,\tfrac{z}{z-1}]  \\
		\overset{\mathclap{\invsc}}{\equiv} {} \,\,& [+; \tfrac{x-1}{x},\tfrac{1}{1-y},\tfrac{z-1}{z}]  \\
		\overset{\mathclap{\text{\sc FS}_1}}{\equiv} {} \,\, & [-; \tfrac{1}{x}, \tfrac{1}{y}, \tfrac{1}{z}] 
		\end{aligned}
		\end{matrix}\right\} \,\,
		\overset{\revsc}{\equiv}
		\,\,
		\left\{\begin{matrix}
		\begin{aligned}
		& [+; z, y, x] \\
		\,\, \overset{\mathclap{\text{\sc FS}_1}}{\equiv} {} \,\, & [-; 1-z,\tfrac{y}{y-1},1-x]  \\
		\overset{\mathclap{\invsc}}{\equiv} {}\,\, & [+; \tfrac{1}{1-z},\tfrac{y-1}{y},\tfrac{1}{1-x}] \\
		\overset{\mathclap{\text{\sc FS}_1}}{\equiv} {} \,\,& [-; \tfrac{z}{z-1},1-y,\tfrac{x}{x-1}]  \\
		\overset{\mathclap{\invsc}}{\equiv} {} \,\,& [+; \tfrac{z-1}{z},\tfrac{1}{1-y},\tfrac{x-1}{x}]  \\
		\overset{\mathclap{\text{\sc FS}_1}}{\equiv} {} \,\, & [-; \tfrac{1}{z}, \tfrac{1}{y}, \tfrac{1}{x}] 
		\end{aligned}
		\end{matrix}\right.
		\end{align*}
		
		\begin{proof}
			Firstly note that $\invsc$ and $\revsc$ commute, as do $\revsc$ and $\fullsym{1}$, and all symmetries are involutions, so we only need to consider applying $\invsc$, and $\fullsym{1}$ (plus an optional $\revsc$ afterwards) to generate everything. Therefore, we obtain the above batches of 6 terms, which are then equivalent after reversal.
		\end{proof}
	\end{Prop}

	{
	\renewcommand{\crv}{{\mathcal{C}_F}}
	\begin{Prop}[Four-term relation]\label{prop:fourterm}
		Degeneration of \( \QU_6 \) to the stable curve  
		\[
		 \crv = 13 \cup_p 4678 \cup_q 259 = \,\,
		\vcenter{\hbox{\includegraphics[page=7]{figures/wt6.pdf}}} \,,
		\]
		produces the 4-term relation\biggerskip
		\begin{align*}
			& \lif([7pq8, pq47, 7q64] + [7pq8, pq47, 4q67]) \\
			& - \lif([7pq8, pq67, p46q] + [7pq8, pq67, qp46]) \,\, \equiv \,\, 0 \modtwo \,.
		\end{align*}
		In affine coordinates, with \( A = 7pq8, B^{-1} = pq47, C = 4q67 \) (note \( B \) inverse, to make the following neater) and terms in the same order, this is
		\begin{align*}
			& \lif\Big(\Big[A, \frac{1}{B}, 1-C\Big]+ \Big[A, \frac{1}{B}, C\Big]\Big) \\
			& - \lif\Big(\Big[A, \frac{C}{B}, \frac{1-B}{1-C}\Big] + \Big[A, \frac{C}{B}, \frac{1-C}{1-B}\Big]\Big) \,\, \equiv \,\, 0 \modtwo \,.
		\end{align*}
	\end{Prop}

	\begin{Rem}\label{rem:fourterm}
		The full identity can be found in \autoref{app:four:full}.  
		By making the substitution \( B = \frac{1-Y}{X-Y}, C = \frac{X(1-Y)}{X-Y} \), we have
		\[
		\Big(\frac{C}{B},\frac{1-C}{1-B}\Big) = (X,Y) \,, \quad
		\Big(\frac{C}{B},\frac{1-B}{1-C}\Big) = (X,Y^{-1}) \,.
		\]
		Therefore, establishing one of the Zagier formulae \autoref{conj:higherzagier} (with \( k = 3 \))
		\begin{align*}
		& \lif(x,y,1-z) + \lif(x,y,z) \overset{?}{\equiv} 0 \modtwo \,, \\
		& \lif(x,y,z) + \lif(x,y,z^{-1}) \overset{?}{\equiv}0 \modtwo \,,
		\end{align*}
		is equivalent to establishing the other.  We made the same observation in weight 4, see \autoref{rem:wt4:four}, about the four-term relation in weight 4 from \autoref{lem:wt4:fourterm}, which this result extends.
		
		In \cite{MR22}, this claim (or an earlier incarnation thereof, informally communicated)  is used to show that the higher-Gangl formula \autoref{conj:highergangl} (for \( k = 3 \)) --- i.e. the \( \lif(x,y,z) \) applied to the dilogarithm five-term combination --- can be written in terms of depth two and a \emph{single} function of the form \( \lif(x,y,z) + \lif(x,y,1-z) \).
	\end{Rem}

	\begin{proof}[Proof of \autoref{prop:fourterm}]
		We treat the cases \( f_3(\widehat{x_i}) \), with \( i = 1 \), with \( i = 2 \), with \( i = 4 \), with \( 6 \leq i \leq 8 \), with \( i = 3 \), with \( i = 5 \) and with \( i = 9 \) separately.  Only \( f_3(\widehat{x_3}) \), \( f_3(\widehat{x_5}) \) and \( f_3(\widehat{x_9}) \) contribute.  We then show how to put the resulting identity into the given form.
		
		\case{Case \( f_3(\widehat{x_1}) \)} Every term in \( f_3(y_1,\ldots,y_8) \) contains an argument with exactly two consecutive entries populated from \( \{ y_1, y_4, y_8 \} \).  $\llbracket$Respectively argument 1 of terms 1--8, and 13--15, argument 2 of terms 9 and 11, and argument 3 of terms 10 and 12.$\rrbracket$  For \( f_3(\widehat{x_1}) \) on \( \crv \) these degenerate to \( \infty \).
		
		\case{Case \( f_3(\widehat{x_2}) \)} Every term in \( f_3(y_1,\ldots,y_8) \) contains an argument with exactly two consecutive entries populated from \( \{ y_1, y_2 \} \).  $\llbracket$Respectively argument 1 of terms 1--8, argument 2 of terms 9, 11, and 13--14, and argument 3 of terms 10, 12 and 15.$\rrbracket$  For \( f_3(\widehat{x_2}) \) on \( \crv \) these degenerate to \( 0  \).
		
		\case{Case \( f_3(\widehat{x_4}) \)} Every term except term 13 in \( f_3(y_1,\ldots,y_8) \) contains an argument with exactly two non-consecutive entries populated from \( \{ y_2, y_4, y_8 \} \).  $\llbracket$Respectively argument 1 of terms 1--12, and argument 3 of terms 14--15.$\rrbracket$  For \( f_3(\widehat{x_4}) \) on \( \crv \) these degenerate to \( 1\).  Argument 2 of term 13 of \( f_3(y_1,\ldots,y_8) \) is \( [y_1,y_2,y_3,y_6] \).  This has exactly two non-consecutive entries populated from \( \{ y_1, y_3 \} \), which also degenerates to \( 1 \), for \( f_3(\widehat{x_4}) \) on \( \crv \).
		
		\case{Case \( f_3(\widehat{x_i}) \), with \( 6 \leq i \leq 8\)} Every term of \( f_3(y_1,\ldots,y_8) \) contains an argument with i) exactly two entries populated from \( \{ y_2, y_5, y_8 \} \), or ii) exactly two entries populated from \( \{ y_1, y_3 \} \).  $\llbracket$ Respectively for i) argument 1 of terms 1--3, and 6--10, and argument 2 of terms 4--5 and term 14, and for ii) argument 2 of terms 11 and 13, and argument 3 of terms 12 and 15.$\rrbracket$  For \( f_3(\widehat{x_i}) \), with \( 6 \leq i \leq 8 \), on \( \crv \), these degenerate to \( 0, 1 \) or \( \infty \).
		
		\case{Case \( f_3(\widehat{x_3}) \)} Every term except term 13 in \( f_3(y_1,\ldots,y_8) \) contains an argument with exactly two non-consecutive entries populated from \( \{ y_2, y_4, y_8 \} \).  $\llbracket$Respectively argument 1 of terms 1--12, and argument 3 of terms 14 and 15.$\rrbracket$  For \( f_3(\widehat{x_3}) \) on \( \crv \) these degenerate to \( 1 \).
		
		Computing explicitly with the remaining term 13, we find
		\begin{align*}
		f_3(\widehat{x_3}) \big\rvert_{\crv} & \equiv \lif(
		{}-[1789,\overset{\mathclap{\text{term 13}}}{1247},4567]
		)\big\rvert_{\crv} \\
		& \equiv  -\lif(
		p78q,pq47,4q67
		) \modtwo \,.
		\end{align*}
		
		\case{Case \( f_3(\widehat{x_5}) \)} Every term except term 14 in \( f_3(y_1,\ldots,y_8) \) contains an argument with i) exactly two non-consecutive entries populated from \( \{ y_2, y_8 \} \) or ii) exactly two non-consecutive entries populated from \( \{ y_1, y_3 \} \).  $\llbracket$Respectively for i) argument 1 or terms 1--8, and for ii) argument 2 of terms 9, 11, and 13, and argument 3 of terms 10, 12, and 15.$\rrbracket$  For \( f_3(\widehat{x_5}) \) on \( \crv \) these degenerate to \( 1 \).
		
		Computing explicitly with the remaining term 14, we find
		\begin{align*}
		f_3(\widehat{x_5}) \big\rvert_{\crv} & \equiv \lif(
		{}+[1789,\overset{\mathclap{\text{term 14}}}{1267},3462]
		)\big\rvert_{\crv} \\
		& \equiv  \lif(
		p78q,pq67,p46q
		) \modtwo \,.
		\end{align*}

		\case{Case \( f_3(\widehat{x_9}) \)}  Every term except terms 6--7 in \( f_3(y_1,\ldots,y_8) \) contains an argument with i) exactly two non-consecutive entries populated from \( \{ y_1, y_3 \} \) or ii) exactly two consecutive entries populated from \( \{ y_2, y_5 \} \).  $\llbracket$Respectively for i) argument 1 of terms 1--3, argument 2 of terms 9, 11 and 13, and argument 3 of terms 10, 12, and 15, and for ii) arumgnent 1 of terms 4--5 and argument 2 of terms 8 and 14.$\rrbracket$
		
		Computing explicitly with the remaining term 6--7, we find
		\begin{align*}
		f_3(\widehat{x_9}) \big\rvert_{\crv} & \equiv \lif(
			{}-[1278,\overset{\mathclap{\text{term 6}}}{3472},5674]
			+[1278,\overset{\mathclap{\text{term 7}}}{3672},3456]
		)\big\rvert_{\crv} \\
		& \equiv  \lif(
			{}-[pq78,p47q,q674]
			+[pq78,p67q,p4q6]
		) \modtwo \,.
		\end{align*}
		
		\case{Overall} The only contributions are from \( f_3(\widehat{x_3}) \), \( f_3(\widehat{x_5}) \) and \( f_3(\widehat{x_9}) \), giving
		\[
			\QU_6 \big\rvert_{\crv} \equiv  \lif(
			\begin{aligned}[t]
			& [p78q,pq47,4q67] - [p78q,pq67,p46q] \\
			& \,\, + \,\, [pq78,p47q,q674] - [pq78,p67q,p4q6]
			) \modtwo \end{aligned}
		\]
		Apply inversion $\invsc$ from \autoref{cor:inv} to the terms in position 3 and 4.  This rotates each cross-ratio by one step left or right, \( abcd \mapsto bcda, \text{ or } dabc \), and flips the sign, giving
		\[
		\equiv  \lif(
		\begin{aligned}[t]
		& [p78q,pq47,4q67] - [p78q,pq67,p46q] \\
		& \,\, - \,\, [8pq7,qp47,4q67] + [8pq7,qp67,6p4q]
		) \modtwo \end{aligned}
		\]
		Now apply $\fullsym{1}$ from \autoref{lem:fullsym1} to the terms in position 3 and 4.  This switches the outside entries of the cross-ratios in argument 1 and argument 3  \( abcd \mapsto dbca \), switches the first two entries of the cross-ratio in argument 2 \( abcd \mapsto bacd \), and flips the sign, giving \[
			\equiv  \lif(
			\begin{aligned}[t]
			& [p78q,pq47,4q67] - [p78q,pq67,p46q] \\
			& \,\, + \,\, [7pq8,pq47,7q64] - [7pq8,pq67,qp46]
			) \modtwo \,. \end{aligned}
			\]
		Up to writing \( p78q = 7pq8 \), via cross-ratio symmetries, and reordering the terms (in particular taking positions 3, 1, 2, then 4), this is equivalent to the relation in the statement of the lemma.
	\end{proof}
	}

	\begin{Def}[Four-term]\label{def:four}
		Define the four-term combination in weight 6 to be
		\[
			\four(x,y,z) \coloneqq \lif\Big( - [x, y, z]  -[x, y, 1 - z] + \Big[x, y z, \frac{1-y^{-1}}{1-z} \Big] + \Big[x, y z, \frac{1-z}{1-y^{-1}}\Big] \Big) \,.
		\]
		From \autoref{prop:fourterm}, (with \( A = x, B = y^{-1}, C = z \) we know this satisfies
		\[
			\mathcal{F}(x,y,z) \equiv 0 \modtwo  \,.
		\]
		In terms of cross-ratios, we may write this as
		\[
			\four(abcd, pqrs, xyzw) = \lif\Big( \begin{aligned}[t]
			& - [abcd, pqrs, xyzw] - [abcd, pqrs, xzyw] \\
			& + \Big[abcd, pqrs \cdot xyzw, \frac{qsrp}{xzyw} \Big] +  \Big[abcd, pqrs \cdot xyzw, \frac{xzyw}{qsrp} \Big] \Big)  \,.
			\end{aligned}
		\]
	\end{Def}

	We can now use this four-term relation to simplify the results of further degenerations, in order to obtain a second non-trivial symmetry of \( \lif \).

	{
	\renewcommand{\crv}{{\mathcal{S}_2}}
	\begin{Lem}[Full symmetry 2, $\fullsym{2}$]\label{lem:fullsym2}
		Degeneration of \( \QU_6 \) to the stable curve 
		\[
			\crv = 29 \cup_p 3467 \cup_q 158  \,\,
			\vcenter{\hbox{\includegraphics[page=8]{figures/wt6.pdf}}}
			\,,
		\]
		produces the following symmetry\biggerskip
		\[
			2 \lif( [67qp, 4q6p, q34p] + [67qp, q46p, qp34]  ) \equiv 0 \modtwo \,.
		\]
		In affine coordinates with \( A = 67qp, B = 4q6p, C = q34p \), this is
		\[
			2 \lif\Big( [A,B,C] + \Big[A, \frac{B}{B-1}, \frac{1}{1-C}\Big]\Big) \equiv 0 \modtwo \,.
		\]
	\end{Lem}

		\begin{Rem}
			The full identity can be fuond in \autoref{app:fullsym2:full}.
		\end{Rem}
		
		\begin{proof}
			We treat the cases \( f_3(\widehat{x_i}) \), with \( i = 1\), with \( 3 \leq i \leq 4 \), with \( 6 \leq i \leq 7 \), with \( i = 9 \), with \( i = 2 \), with \( i = 5 \) and with \( i = 8 \) separately.  Only \( f_3(\widehat{x_2}) \), \( f_3(\widehat{x_5}) \) and \( f_3(\widehat{x_9}) \) contribute.  We then show how to put the resulting identity into the given form.
			
			\case{Case \( f_3(\widehat{x_1}) \)} Argument 1 of each term in \( f_3(y_1,\ldots,y_8) \) has exactly two non-consecutive entries populated from \( \{ y_1, y_8 \} \).  For \( f_3(\widehat{x_1}) \) on \( \crv \), this degenerates to \( \infty \).
			
			\case{Case \( f_3(\widehat{x_i}) \), with \( i = 3, 4 \)} Each term in \( f_3(y_1,\ldots,y_8) \) has an argument with i) exactly two entries populated from \( \{ y_2, y_8 \} \), or ii) exactly two entries populated from \( \{ y_1, y_4, y_7 \} \).  $\llbracket$Respectively for i) argument 1 of terms 1--8, and for ii) argument 1 of terms 9--10, and 13--14, and argument 2 of terms 11--12.$\rrbracket$  For \( f_3(\widehat{x_i}) \), with \( i = 3, 4 \) on \( \crv \) these degenerate to \( 0, 1 \), or \( \infty \).
			
			\case{Case \( f_3(\widehat{x_i}) \), with \( i = 6, 7\)} Each term in \( f_3(y_1,\ldots,y_8) \) has an argument with i) exactly two non-consecutive entries populated from \( \{ y_2, y_8 \} \), or ii) exactly two non-consecutive entries populated from \( \{ y_1, y_5, y_7 \} \).  $\llbracket$Respectively for i) argument 1 of terms 1--8, and for ii) argument 1 of terms 9--15.$\rrbracket$  For \( f_3(\widehat{x_i}) \), with \( i = 6,7 \) on \( \crv \) these degenerate to \( 0, 1 \), or \( \infty \).

			\case{Case \( f_3(\widehat{x_9}) \)} Each term in \( f_3(y_1,\ldots,y_8) \) has an argument with exactly two consecutive entries populated from \( \{ y_1, y_5, y_8 \} \)  $\llbracket$Respectively argument 1 of terms 1--3, 6--8, and 11--15, argument 2 of terms 5, and 10, and argument 3 of terms 4, and 9.$\rrbracket$  For \( f_3(\widehat{x_9}) \) on \( \crv \) these degenerate to \( \infty \).
			
			\case{Case \( f_3(\widehat{x_2}) \)} Terms 2, and 6--15 in \( f_3(y_1,\ldots,y_8) \) have an argument with exactly two entries populated from \( \{ y_1, y_4, y_7 \} \).  $\llbracket$Respectively argument 1 of terms 6--10, and 13--15, and argument 3 of terms 2, 11--12 and 15.$\rrbracket$  For \( f_3(\widehat{x_2}) \) on \( \crv \) these degenerate to \( 0 , 1 \) or \( \infty \).
			
			Computing explicitly with the remaining terms 1, and 3--5, we have
			\begin{align*}
				& f_3(\widehat{x_2}) \big\rvert_{\crv} \\
				& \equiv \lif\big(
				\begin{aligned}[t] \smash{
				{}-[1349,\overset{\mathclap{\text{term 1}}}{4569},6789]
				-[1349,\overset{\mathclap{\text{term 3}}}{4789},4567]
				+[1369,\overset{\mathclap{\text{term 4}}}{4563},6789]
				+[1369,\overset{\mathclap{\text{term 5}}}{6789},4563] }\smash{\big) \big\rvert_{\crv}}
				\end{aligned}  \\
				& \equiv \begin{aligned}[t] \lif\big( \begin{aligned}[t]
				{}-[q34p,4q6p,67qp]
				-[q34p,47qp,4q67]
				+[q36p,4q63,67qp]
				+[q36p,67qp,4q63]
				\smash{\big)} 
				\end{aligned}  \\
				\modtwo \,. \hspace{-1em} \end{aligned}
			\end{align*}
			Now use the (2,1)-shuffle $\shsym{2,1}$ from \autoref{lem:21shuffleDeriv} of \( (q36p,4q63) \) shuffled with \( 67qp \), to replace the terms in the final two positions by a single term, giving
			\begin{align*}
				& \equiv \lif\big( \begin{aligned}[t]
					{}-[q34p,4q6p,67qp]
					-[q34p,47qp,4q67]
					-[67qp,q36p,4q63]
					\smash{\big)} \modtwo \,. \end{aligned}
		\end{align*}
			
			\case{Case \( f_3(\widehat{x_5}) \)} Argument 1 of terms 1--8 in \( f_3(y_1,\ldots,y_8) \) has exactly two non-consecutive entries populated from \( \{ y_2, y_8 \} \).  Argument 1 of terms 11--15 in \( f_3(y_1,\ldots,y_8) \) has exactly two non-consecutive entries populated from \( \{ y_1, y_7 \} \).   For \( f_3(\widehat{x_5}) \) on \( \crv \) these degenerate to 1.
			
			Computing explicitly with the remaining terms 9 and 10, we have
			\begin{align*}
				 f_3(\widehat{x_5}) \big\rvert_{\crv} 
				& \equiv \lif\big(
				\begin{aligned}[t]
			\smash{	{}-[1469,\overset{\mathclap{\text{term 9}}}{1234},6789]
				-[1469,\overset{\mathclap{\text{term 10}}}{6789},1234] }
				\smash{\big) \big\rvert_{\crv}}
				\end{aligned}  \\& \equiv \lif\big(
				\begin{aligned}[t]
				{}-[q46p,qp34,67qp]
				-[q46p,67qp,qp34]
				\smash{\big) \modtwo \,.}
				\end{aligned} 
			\end{align*}
				Now use the (2,1)-shuffle $\shsym{2,1}$ from \autoref{lem:21shuffleDeriv} of \( (q46p,qp34) \) shuffled with \( 67qp \), to replace these two terms positions by a single term, giving
			\begin{align*}
			& \equiv \lif(67qp, q46p,qp34) \modtwo \,.
			\end{align*}
			
			\case{Case \( f_3(\widehat{x_8}) \)} Argument 1 of terms 1--8 in \( f_3(y_1,\ldots,y_8) \) have exactly two non-consecutive entries populated from \( \{ y_2, y_8 \} \).  Terms 9--10 and 14--15 in \( f_3(y_1,\ldots,y_8) \) have an argument with exactly non-consecutive two entries populated from \( \{ y_1, y_5 \} \).  $\llbracket$Respectively argument 1 of terms 9--10, and argument 2 of terms 14--15.$\rrbracket$  For \( f_3(\widehat{x_8}) \) on \( \crv \) these degenerate to \( 1 \).
			
			Computing explicitly with the remaining terms 11--13, we have
			\begin{align*}
			& f_3(\widehat{x_8}) \big\rvert_{\crv} \\
			& \equiv \lif\big(
			\begin{aligned}[t] \smash{
				{}+[1479,\overset{\mathclap{\text{term 11}}}{1234},5674]
				+[1479,\overset{\mathclap{\text{term 12}}}{5674},1234]
			   -[1679,\overset{\mathclap{\text{term 13}}}{1236},3456] }\smash{\big) \big\rvert_{\crv}}
			\end{aligned}  \\
			& \equiv \lif\big( \begin{aligned}[t]
				{}+[q47p,qp34,q674]
				+[q47p,q674,qp34]
				-[q67p,qp36,34q6]
			\smash{\big)} 
			\modtwo \,.  \end{aligned}
			\end{align*}
			Now use the (2,1)-shuffle $\shsym{2,1}$ from \autoref{lem:21shuffleDeriv} of \( (q47p,q674) \) shuffled with \( qp34 \), to replace the terms in the first two positions by a single term, giving
			\begin{align*}
			& \equiv \lif\big( \begin{aligned}[t]
			{}-[qp34,q47p,q674]
			-[q67p,qp36,34q6]
			\smash{\big)} \modtwo \,. \end{aligned}
			\end{align*}
			
			\case{Initial result} The only contributions are from \( f_3(\widehat{x_2}) \), \( f_3(\widehat{x_5}) \) and \( f_3(\widehat{x_8}) \), giving
			\begin{align*}
				\QU_6 \big\rvert_{\crv} 
				& \equiv \lif\big( \begin{aligned}[t]
				& 
				{}-[q34p,4q6p,67qp]
				-[q34p,47qp,4q67]
				-[67qp,q36p,4q63] 
				\\
				& {} -[67qp, q46p,qp34]
			    - [qp34,q47p,q674]
				-[q67p,qp36,34q6]
				\smash{\big)} 
				\modtwo \,.  \end{aligned}
			\end{align*}
			
			We claim that modulo the four-term relation \autoref{prop:fourterm} and the 12-fold symmetries from \autoref{prop:twelve}, this combination of 6 terms reduces to the 2 terms given in the statement of the lemma.  We explicitly show this.
			
			\case{Simplification via \( \four \)}  Consider the following four-term relations
			\begin{align*}
			\four(p43q,p6q4,pq76) &= \lif\Big( \begin{aligned}[t]
			& - [p43q, p6q4, pq76] - [p43q, p6q4, p7q6] \\[-0.5ex]
			& + \Big[p43q, p6q4 \cdot pq76, \frac{64qp}{p7q6} \Big] +  \big[p43q, p6q4 \cdot pq76, \frac{p7q6}{64qp} \Big] \Big)  \,,
			\end{aligned} \\[1ex]
			-\four(p43q, pq74, q476) &= \lif\Big( \begin{aligned}[t]
			& + [p43q, pq74, q476] + [p43q, pq74, q746] \\[-0.5ex]
			& - \Big[p43q, pq74 \cdot q476, \frac{q47p}{q746} \Big] -  \Big[p43q, pq74 \cdot q476, \frac{q746}{q47p} \Big] \Big)  \,.
			\end{aligned}
			\end{align*}
			We have the following equalities involving cross-ratio,
			\[
			p6q4 \cdot pq76 = pq74 \cdot q476 \,, \quad \frac{64qp}{p7q6} =  \frac{q746}{q47p} \,, \quad \frac{p7q6}{64qp} = \frac{q47p}{q746} \,.
			\]
			which imply that lines 2 of the above four-term combinations cancel when added.  Consider also 
			\[
			\four(p67q, pq36, q364) = \lif\Big( \begin{aligned}[t]
			& - [p67q, pq36, q364] - [p67q, pq36, q634] \\[-0.5ex]
			& + \Big[p67q, \underbrace{pq36 \cdot q364\phantom{\mathrlap{\frac{p}{q}}}}_{{} = pq46}, \underbrace{\frac{q63p}{q634}}_{=q43p} \Big] +  \Big[p67q, pq36 \cdot q364, \underbrace{\frac{q634}{q63p}}_{=pq43} \Big] \Big)  \,,
			\end{aligned}	
			\]
			wherein the products and ratios of cross-ratios all simplify to single cross-ratios in each case.  Summing these 3 four-term combinations (with the indicated signs), and directly cancelling lines 2 of the first 2 contributions, gives the following 8 terms (with the order from above preserved)
			\begin{align*}
			& \four(p43q,p6q4,pq76)	-\four(p43q, pq74, q476) + \four(p67q, pq36, q364) = \\
			& \lif\big( \begin{aligned}[t]
			&- [p43q, p6q4, pq76] - [p43q, p6q4, p7q6]
			+ [p43q, pq74, q476] + [p43q, pq74, q746] \\
			& - [p67q, pq36, q364] - [p67q, pq36, q634]
			+ [p67q, pq46, q43p] + [p67q, pq46, pq43] \big) \,.
			\end{aligned}
			\end{align*}
			We now apply the know symmetries from \autoref{prop:twelve} to rewrite this to more closely match the already-computed \( \QU_6 \big\rvert_{\crv} \).
			
			Apply $\fullsym{1}$ from \autoref{lem:fullsym1} to term in position 4, then apply $\invsc$ from \autoref{cor:inv}.  (Recall: under $\fullsym{1}$ the first and third arguments change from $abcd \mapsto dbca$, and the middle changes from $abcd \mapsto bacd$, while under inversion each argument changes from $abcd \mapsto bcda$.  Each symmetry picks up a sign.)  This gives
			\begin{align*}
			{\lif(p43q, \overset{\mathclap{\text{position 4}}}{pq74}, q746)} & \overset{\text{\sc FS}_1}{\equiv} -\lif(q43p, qp74, 674q) \\
			&\overset{\invsc}{\equiv} \lif(43pq, p74q, 74q6) \modtwo \,.
			\end{align*}
			Apply $\fullsym{1}$ from \autoref{lem:fullsym1} to term in position 5, giving
			\[
			-\lif(p67q, \overset{\mathclap{\text{position 5}}}{pq36}, q364) \overset{\text{\sc FS}_1}{\equiv} \lif(q67p, qp36, 436q) \modtwo \,.
			\]
			Then apply $\invsc$ from \autoref{cor:inv} the terms in positions 6 and 7, giving
			\begin{align*}
			& \lif( {} - [p67q, \overset{\mathclap{\text{position 6}}}{pq36}, q634]+ [p67q, \overset{\mathclap{\text{position 7}}}{pq46}, q43p] ) \\
			& \overset{\invsc}{\equiv} \lif( [67qp, q36p, 634q] - [67qp, q46p, 43pq] ) \modtwo \,.
			\end{align*}
			Finally apply $\invsc$ from \autoref{cor:inv} to the term in position 8, then apply $\fullsym{1}$ from \autoref{lem:fullsym1}, and finally the reversal $\revsc$ from \autoref{cor:reverse} (switches arguments $1\leftrightarrow3$, no sign).  This gives
			\begin{align*}
			\lif(p67q, \overset{\mathclap{\text{position 8}}}{pq46}, pq43) 
			& \overset{\invsc}{\equiv}  -\lif(67qp, q46p, q43p) \\
			& \overset{\text{\sc FS}_1}{\equiv} \lif(p7q6, 4q6p, p43q) \\
			& \overset{\revsc}{\equiv} \lif(p43q, 4q6p, p7q6) \modtwo \,.
			\end{align*}
			Substituting all of these changes into the \( 3 \times \four \) expression, we see it is equivalent, modulo depth 2, to
				\begin{align*}
			& \four(p43q,p6q4,pq76)	-\four(p43q, pq74, q476) + \four(p67q, pq36, q364) \equiv \\
			& \begin{aligned}[t] \lif\big( \begin{aligned}[t]
			&- [p43q, p6q4, pq76] - [p43q, p6q4, p7q6]
			+ [p43q, pq74, q476] + [43pq, p74q, 74q6] \\
			& + [q67p, qp36, 436q] + [67qp, q36p, 634q] 
			- [67qp, q46p, 43pq] + [p43q, 4q6p, p7q6] \big) 
			\end{aligned} \\
			\modtwo \,. \hspace{-1em}
			\end{aligned}
			\end{align*}
			Recalling that
				\begin{align*}
			\QU_6 \rvert_\crv 
			& \equiv \lif\big( \begin{aligned}[t]
			& 
			{}-[q34p,4q6p,67qp]
			-[q34p,47qp,4q67]
			-[67qp,q36p,4q63] 
			\\
			& {} -[67qp, q46p,qp34]
			- [qp34,q47p,q674]
			-[q67p,qp36,34q6]
			\smash{\big)} 
			\modtwo \,,  \end{aligned}
			\end{align*}
			we add this to the \( 3 \times \four \) expression above (abbreviated simply \( \four \) in the next lines).  In the result of ``\( \four + \QU_6 \)'', observe term 1 of \( \four \) and term 1 of \( \QU_6 \) combine via cross-ratio symmetries, as do term 7 of \( \four \) and term 4 of \( \QU_6 \).  Term 2 and term 8 of \( \four \) already cancel pair-wise.  Finally terms 3, 4, 5 and 6 of \( \four \) cancel with terms 2, 5, 6 and 3 of \( \QU_6 \), respectively.  
			
			\case{Final result} After these simplifications, we obtain
			\begin{align*}
			& \QU_6 \rvert_\crv +  \four(p43q,p6q4,pq76) - \four(p43q, pq74, q476) + \four(p67q, pq36, q364)  \\
			& \equiv \lif( \begin{aligned}[t]
			& {} -2 \, [q34p, 4q6p, 67qp] -2 \, [67qp, q46p, qp34] 			
			) \modtwo \,. \end{aligned} 
			\end{align*}
			Finally apply $\revsc$ from \autoref{cor:reverse} to the first term, giving
			\[
				\lif( {} -2 \, [67qp, 4q6p, q34p] -2 \, [67qp, q46p, qp34] ) \equiv 0 \modtwo \,.
			\]
			This is the form of the symmetry given in the statement of the lemma.
		\end{proof}
	}

	With this new symmetry established, we generate a total of 216 symmetries for \( \lif \) (out of the expected \( 6!^3 \cdot 2 = 432 \), namely the 6-fold in each argument, along with reversal).  We categorise them more carefully in the following proposition.

	\begin{Prop}[216-fold symmetries of \( \lif \)]\label{prop:216}
		Write \( [\eps; x,y,z] \coloneqq \eps \lif(x,y,z) \).  The maps 
		\begin{align*}
	\tag{\autoref{cor:inv}, $\invsc$}		&	[ \eps; x,y,z ] \mapsto [ -\eps; x^{-1}, y^{-1}, z^{-1} ] \,, \\
	\tag{\autoref{cor:reverse}, $\revsc$}		&	[ \eps; x,y,z ] \mapsto [ \eps; z,y,x ] \,, \\
	\tag{\autoref{lem:fullsym1}, $\fullsym{1}$}		&	[ \eps; x,y,z ] \mapsto [ -\eps; 1-x, \tfrac{y}{y-1}, 1-z ] \,, \\
	\tag{\autoref{lem:fullsym2}, $\fullsym{2}$}		&	[\eps; x,y,z ] \mapsto [ -\eps; x, \tfrac{y}{y-1} , \tfrac{1}{1-z} ] \,,
	\end{align*}
	generate the following group of symmetries of \( [\eps; x, y, z] \), which has cardinality 216,
		\[
			X = \begin{aligned}[t] 
			& \{ [\eps; x^\pi, y^\sigma, z^\tau] \mid \pi, \sigma,\tau \in \Sym_3, \sgn(\pi\sigma\tau) = \eps, \sgn(\pi) = \sgn(\tau) \} \\
			& \cup \{ [\eps; z^\tau, y^\sigma, x^\pi] \mid \pi, \sigma,\tau \in \Sym_3, \sgn(\pi\sigma\tau) = \eps, \sgn(\pi) = \sgn(\tau) \} \,.
			\end{aligned}
		\]
		Here the permutations \( \pi, \sigma, \tau \in \Sym_3 \) act via the anharmonic ratios, as in \autoref{sec:quadrangular:anharmonic}.
		
		\paragraph{Note} $\fullsym{1}$ is actually not needed, but is useful to obtain a more direct expression for the map \( [ \eps; x,y,z ] \mapsto [ -\eps; x,y^{-1},z ] \).
		
		\begin{proof}
			Firstly, each of the above maps lies in \( X \).  Moreover, the conditions in the definition of \( X \) define a subgroup of \( \Sym_3^3 \times \mathbb{Z}/2\mathbb{Z} \), so our maps generate a subgroup of \( X \).  The cardinality of \( X \) is easily seen to be 216: we may take arbitrary \( \pi, \sigma \in \Sym_3 \) (6 choices each), but must take \( \tau \) to have the same signature as \( \pi \) (3 choices), however we may apply the switch \( x \leftrightarrow z \) (2 choices, switch or not).  Overall we find \( 6^2 \cdot 3 \cdot 2 = 216 \) elements. \medskip
			
			To show the cardinality agrees it is sufficient to exhibit expressions for the following:
			\begin{align*}
				\tag{$\textsc{Map}_1$} [\eps; x, y, z] &\mapsto [-\eps; x, \tfrac{y}{y-1}, z] \,, \\
				\tag{$\textsc{Map}_2$} [\eps; x, y, z] &\mapsto [-\eps; x, 1-y, z] \,, \\
				\tag{$\textsc{Map}_3$} [\eps; x, y, z] &\mapsto [\eps; x, y, \tfrac{1}{1-z}]  \,.
			\end{align*}
			$\textsc{Map}_1$ and $\textsc{Map}_2$ generate the \( \Sym_3 \) anharmonic group in argument 2.  (This is not the standard choice of generators, but with \( \phi_1(y) = \frac{y}{y-1} , \phi_2(y) = 1-y \), we obtain the standard generator \( y \mapsto y^{-1} \) via \( y^{-1} = \phi_2 \circ \phi_1 \circ \phi_2(y) \); for later \emph{computational} use we will also give a direct expression below.)  $\textsc{Map}_3$ has order 3, and so it generates the 3-cycles, allowing us to change the argument 3 to any other permutation with the same signature as the first argument.
			
			\case{$\text{\sc Map}_1$ expression}  Iterating $\fullsym{2}$ 3 times produces
			\begin{align*}
			\tag{via $\fullsym{2}$ ($\times 3$) }	[\eps; x, y, z] \mapsto [-\eps; x, \tfrac{y}{y-1}, z] \,,
			\end{align*}
			as the first argument does change, the second argument is an order 2 anharmonic ratio, and the third argument is an order 3 anharmonic ratio.
			
			\case{$\text{\sc Map}_2$ expression}  Conjugating $\textsc{Map}_1$ by 
			$\invsc$, we have
			\begin{align*}
				&[\eps; x, y, z] \\
				 \tag{via $\invsc$} {} \mapsto {} & [ -\eps; x^{-1}, y^{-1}, z^{-1}] \\
				 \tag{via $\fullsym{2}$ ($\times 3$) } {} \mapsto {} & [ \eps; x^{-1}, \tfrac{1}{1-y}, z^{-1}] \\
				 \tag{via $\invsc$} {} \mapsto {} & [ -\eps; x, 1-y, z]  \,.
			\end{align*}
			
			\case{Simpler $[ \eps; x,y,z ] \mapsto [ -\eps; x,y^{-1},z ]$ expression}  Conjugating $\textsc{Map}_2$ by $\fullsym{1}$ gives
			\begin{align*}
			&[\eps; x, y, z] \\
			\tag{via $\fullsym{1}$} {} \mapsto {} & [ -\eps; 1-x, \tfrac{y}{y-1}, 1-z] \\
			\tag{via $\text{\sc Map}_2$} {} \mapsto {} & [\eps; 1-x, \tfrac{1}{1-y}, 1-z] \\
			\tag{via $\fullsym{1}$} {} \mapsto {} & [ -\eps; x, \tfrac{1}{y}, z]  \,.
			\end{align*}
			
			\case{$\text{\sc Map}_3$ expression} Finally, we have the following.  Iterating  $\fullsym{2}$ 2 times fixes arguments 1 and 2, as it acts by the identity and an order 2 anharmonic ratio, respectively; it changes the third argument to the other order 3 anharmonic ratio (i.e. its inverse as a function).  We have
			\begin{align*}
			&[\eps; x, y, \tfrac{1}{1-z}] \\
			\tag{via $\fullsym{2}$ ($\times 2$) } {} \mapsto {} & [ \eps; x, y, z] \,;
			\end{align*}
			the actual $\textsc{Map}_3$ we want is obtained by reading this backwards, i.e. using the inverse of this map.\medskip
			
			With these three maps established, the proof of the proposition is complete.
		\end{proof}
	\end{Prop}
	
	We can now give a direct proof of one of the Zagier formulae, by exploiting the interplay between these 216-fold symmetries, and the four-term relation \autoref{prop:fourterm}.  Consider the following Zagier-formula combination
	\[
		Z(x,y,z) \coloneqq \lif(x, y^{-1}, z) + \lif(x, y^{-1}, 1-z) \,,
	\]
	Observe that we can write
	\begin{align*}
		& Z(x,y,z) + Z\Big(x, \frac{z}{y}, \frac{1-z}{1-y}\Big) \\
		& = \begin{aligned}[t]
			&\lif\Big( [x, y^{-1}, z] + [x, y^{-1}, 1-z]  \,\, + \,\, \Big[x, \Big(\frac{z}{y}\Big)^{-1}, \frac{1-z}{1-y}\Big] + \Big[x, \Big(\frac{z}{y}\Big)^{-1}, 1-\frac{1-z}{1-y}\Big] \Big) \end{aligned} \\
		& {} \mathrel{\overset{\mathclap{\substack{{}_\text{216} \\ {}_\text{-fold}}}}{\equiv}} \begin{aligned}[t]
		&\lif\Big( [x, y^{-1}, z] + [x, y^{-1}, 1-z]  \,\, - \,\, \Big[x, \frac{z}{y}, \frac{1-z}{1-y}\Big] - \Big[x, \frac{z}{y}, \Big(\frac{1-z}{1-y}\Big)^{-1} \Big] \Big) \end{aligned} \modtwo \\
		&= \four(x, y^{-1}, z) \,.
	\end{align*}
	Here we have used the symmetries from \autoref{prop:216} to invert the second argument (and the sign) of the last 2 terms, via the full \( \Sym_3 \)-symmetry in argument 2.  We have then switched \( 1-w \) to \( w^{-1} \), having the same signatures as anharmonic ratios, in the third argument of the last term, via \( \textsc{Map}_3 \).
	
	As indicated, the last line is now the presentation of the four-term relation, \autoref{def:four}, given in \autoref{prop:fourterm}.  In particular it, and hence the original combination of \( Z \)'s, reduce to depth \( {\leq}2 \), namely
	\begin{equation}\label{eqn:Vcomb:depth2}
		Z(x, y, z) + Z\Big(x, \frac{z}{y}, \frac{1-z}{1-y}\Big) \equiv 0 \modtwo \,.
	\end{equation}
	
	\begin{Prop}[Radchenko]\label{prop:zagcombRad}
		Suppose that we know
		\[
		Z(x, y, z) + Z\Big(x, \frac{z}{y}, \frac{1-z}{1-y}\Big) \equiv 0 \pmod{\text{\normalfont condition $R$}} \,,
		\]
		then 
		\[
			2 \, Z(x, y, z) \equiv 0  \pmod{\text{\normalfont condition $R$}}  \,.
		\]
		
		\begin{proof}[Proof (Radchenko)]
			The transformation 
			\[
				\phi \colon (y,z) \mapsto \Big( \frac{z}{y}, \frac{1-z}{1-y} \Big)
			\]
			can be expressed as
			\[
				([\infty10y],[\infty10z]) \mapsto ([y\infty{}z0], [y\infty{}z1]) \,,
			\]
			hence is given by the order 5 permutation permutation \( (\infty \, y \, 0 \, z \, 1 \, ) \).  After 5 applications, we find
			\[
				Z(x,y,z) \equiv -Z(x,y,z)  \pmod{\text{condition $R$}}  \,,
			\]
			from which the result follows.
		\end{proof}
	\end{Prop}

	Consequently, from \autoref{eqn:Vcomb:depth2} we conclude (after dividing by 2) the following
	\begin{Cor}[Zagier formulae]\label{cor:zagier:wt6}
		The following holds
	\[
		\lif(x, y^{-1}, z) + \lif(x, y^{-1}, 1-z) \eqqcolon Z(x,y,z) \equiv 0 \modtwo \,.
	\]
	\end{Cor}
	With this, we can now change argument 3 of \( \lif(x,y,z) \) to anharmonic ratios with a different signature to argument 1, hence we obtain the full 432-fold symmetry group of \( \lif \).  More specifically, we have established one of the Zagier formulae, \autoref{conj:higherzagier} ($k=3$), and hence by \autoref{rem:fourterm} both of them.

	{
	\renewcommand{\crv}{{\mathcal{S}_3}}
	\begin{Rem}[Full Symmetry 3, $\fullsym{3}$]\label{rem:fullsym3}
		The above rewriting of \( \four \) via \( Z \)'s, \autoref{prop:zagcombRad}, and the proof thereof, were suggested by Danylo Radchenko, as a refined version of my initial proof.  Originally, I obtained the third symmetry
		\[
			\lif(A,B,C) + \lif(A,B,C^{-1}) \equiv 0 \modtwo \,,
		\]
		more directly and computationally by degenerating \( \QU_6 \) to the stable curve
		\[
			\crv = 29 \cup_p 3457 \cup_q 168  = \,\, 
		\vcenter{\hbox{\includegraphics[page=9]{figures/wt6.pdf}}}
			\,.
		  \]
		  In particular, I obtained 5 non-trivial terms (after the 216-fold symmetries from \autoref{prop:216}).  I could then add 3 four-term relations and 2 three-term relations (\autoref{cor:three}), to leave this new two-term symmetry.  For completeness, we shall briefly give the relevant results, leaving the detailed verification to the reader.
		  
		  Whether further symmetries, obtained in this \emph{computational} way, might be necessary in weight \( {\geq}8 \), is unclear so far.  It is therefore potentially useful to indicate how such possibilities work.
		  
		  \case{Overview} The only contributions are from \( f_3(\widehat{x_2}) \rvert_\crv \) (1 term, position 3), \( f_3(\widehat{x_4}) \rvert_\crv \) (2 terms, positions 9--10), and \( f_3(\widehat{x_8}) \rvert_\crv \) (4 terms, at positions 9--12), all other \( f_3(\widehat{x_i}) \) vanish, modulo depth 2, as the arguments specialise to \( 0, 1, \infty \) in some way.  Specifically we obtain
		  \begin{align*}
		  	 \QU_6 \rvert_\crv 
		  	&  \equiv  
		  	 \lif( 
		  	\begin{aligned}[t]
		  		& {} -[4pq3,47qp,45q7]
		  		 -  [5pq4,34qp,57qp]
		  		-[5pq4,57qp,34qp]
		  		\\
		  		& -[5pq4,34qp,5q7p]
		  		-[5pq4,5q7p,34qp]
		  		 +[7pq4,34qp,47q5] \\
		  		 & 
		  		+[7pq4,47q5,34qp] \, ) 
		  		\modtwo \,. \hspace{-1em} 
		  	\end{aligned}
		  \end{align*}
		   Write \( \three(x,y,z) \coloneqq \lif(x,y,z) + \lif(y,z,x) + \lif(z,x,y) \), which from \autoref{cor:three} we know satisfies \( \three(x,y,z) \equiv 0 \modtwoi \).  Recall also the four-term combination \( \four(x,y,z) \equiv 0 \modtwoi \) from \autoref{def:four}.  Consider following the four- and three-term sum, which we then expand out keeping terms in the same order as in the definitions of \( \three \) and \( \four \),
		   {\small
		  \begin{align}
		  		\label{eqn:sym3:fourthree} & 
		  		\begin{aligned}[c] 
		  		& \tfrac{1}{2} \four({} -[45q7,3q4p,4q7p] - [45q7,4p7q,3p4q] + 2 \, [57pq,3pq4,45pq]) \\
		  		& + \tfrac{1}{2} \three( {} - [3q4p,4q7p,47q5]  + [3q4p,45q7,4q7p] ) \,\, = \,\, \end{aligned}
		  		\\[1ex]
		  		\notag & \begin{aligned}[t]
		  		 \tfrac{1}{2}\lif\big( 
		  		 & {} \phantom{{} + {}} [45q7,3q4p,4q7p]+[45q7,3q4p,47qp]-[45q7,3q7p,347q]-[45q7,3q7p,47q3] \\
		  		& {} + [45q7,4p7q,3p4q]+[45q7,4p7q,34pq]-[45q7,3p7q,47q3]-[45q7,3p7q,347q] \\
		  		& {} -2\,[57pq,3pq4,45pq]-2\,[57pq,3pq4,4p5q]+2\,[57pq,45p3,3q5p]+2\,[57pq,45p3,3p5q] \\
		  		& {} -[3q4p,4q7p,47q5]-[4q7p,47q5,3q4p]-[47q5,3q4p,4q7p] \\
		  		& {} +[3q4p,45q7,4q7p]+[45q7,4q7p,3q4p]+[4q7p,3q4p,45q7] \, \smash{\big)} \,.
		  		\end{aligned}
		  \end{align}}
		  Under the 216-fold symmetries from \autoref{prop:216}, term 1, 5, and 14 combine with terms 18, 13 and 16 respectively.  Moreover, terms 2, 3, 4 and 6 cancel with terms 15, 8, 7 and 17, respectively.  Terms 9--12 remain unchanged.  This leaves us with 
		  \begin{equation}\label{eqn:fsimp}
		  	\equiv \begin{aligned}[t]
		  	\lif( \, &[45q7,3q4p,4q7p]
		  	+[45q7,4p7q,3p4q]
		  	-[4q7p,47q5,3q4p]
		  	-[57pq,3pq4,45pq] \\
		  	& -[57pq,3pq4,4p5q]
		  	+[57pq,45p3,3q5p]
		  	+[57pq,45p3,3p5q] \, ) \modtwo \,.
		  	\end{aligned}
		  \end{equation}
		  Now if we add \( \QU_6 \rvert_\crv \) to the simplified \( \four \)-combination in \autoref{eqn:fsimp}, we see
		  term 1, 2, 4, 6, and 7 of \( \QU_6 \) cancels with term 2, 4, 5, 1, and 3 of \( \four \), respectively.  Moreover, terms 3 and 5 of \( \QU_6 \) pair-wise cancel already.  We thus find
		  \begin{align*}
		  	\QU_6 \rvert_\crv  +  \text{\autoref{eqn:sym3:fourthree}} & \equiv \lif([57pq,45p3,3q5p]	+[57pq,45p3,3p5q]) 
		  \end{align*}
		  This is the claimed third symmetry, with \( A = 57pq, B = 45p3, C = 3q5p \) in affine coordinates.
	\end{Rem}
	
	\subsection{Conclusion: the higher Zagier formulae in weight 6}\label{sec:wt6:conclusion}
	
	We have now completed the hard work, and can conclude, as a corollary, the following theorem, resolving the higher Zagier formulae in weight 6 (\autoref{conj:higherzagier}, $k=3$).
	
	\begin{Thm}[Six-fold symemtries of  \( \lif \)]\label{thm:i411sixfold}
		Let \( \pi, \sigma, \tau \in \Sym_3 \) be any three of the six-fold symmetries, acting via anharmonic ratios (see \autoref{sec:quadrangular:anharmonic}).  Then
		\[
			\lif(x,y,z) \equiv \sgn(\pi \sigma \tau) \lif(x^\pi, y^\sigma, z^\tau) \modtwo \,.
		\]
		In particular \( \lif(x,y,z) \) satisfies the expected dilogarithm 6-fold symmetries in each argument independently.
		
		\begin{proof}
			From \autoref{prop:zagcombRad} and \autoref{eqn:Vcomb:depth2}, we have (replacing \( y \mapsto y^{-1} \))
			\[
				Z(x, y^{-1}, z) = \lif(x, y, z) + \lif(x, y, 1-z) \equiv 0 \modtwo \,.
			\]
			From the four-term relation \autoref{prop:fourterm} and \autoref{rem:fourterm}, we know this implies
			\[
			\lif(x', y', z') + \lif(x', y', z'^{-1}) \equiv 0 \modtwo \,.
			\]
			These generate the 6-fold symmetries in the third argument.  By the reversal $\revsc$ from \autoref{cor:reverse} we may switch this to the first argument to obtain the same result there.  The 6-fold symmetries in \( b \) already followed from \autoref{prop:216}.
		\end{proof}
	\end{Thm}

	Recall finally the result by Matveiakin and Rudenko \cite[Theorem 1.4]{MR22}, which established that, \emph{modulo the higher Zagier formulae}, the higher Gangl formula holds in weight 6.
	
	\begin{Thm}[{Matveiakin-Rudenko, Theorem 1.4 \cite{MR22}}]\label{thm:wt6:gangl}
		The function \( \lif(x,y,z) \) satisfies the dilogarithm five-term relation in \( x \) (correspondingly \( y, z \)), modulo depth 2 and the higher Zagier formulae.  That is,
		\[
		\sum_{i=0}^4 (-1)^i \lif([w_0,\ldots,\widehat{w_i},\ldots,w_4], y,z) \equiv 0 \moddp{2 \& higher Zagier formulae} \,.
		\]
	\end{Thm}

	Combining \autoref{thm:i411sixfold} with \autoref{thm:wt6:gangl} (and de Jeu's reduction of all multivariable dilogarithm functional equations to the five-term relation \cite{deJeu}) we have together established the special case of Goncharov's Depth Conjecture (\autoref{conj:depthv2}), in the case weight 6, depth \( k = 3 \).

	\begin{Cor}[Goncharov's Depth Conjecture, case: weight 6, depth 3]\label{cor:gon:wt6depth3}
		A linear combination of weight 6 multiple polylogarithms (with rational function arguments) has depth \( {\leq}2 \) if and only if its 2-nd iterated truncated coproduct \( \overline{\Delta}^{[2]} \) vanishes.
	\end{Cor}

	This now opens a route to finding highly non-trivial depth 2 identities by expanding out a reduction like
	\[
	\lif([x] + [x^{-1}], y, [z] + [1-z]) \equiv 0 \moddp{2} \,,
	\]
	in two different ways, as indicated in \cite[\S7.8.2]{CharltonThesis16}.   C.f. the 931-term \( \LiL_4 \)-functional equation found by Gangl \cite{gangl-4}, obtained by expanding \( \sum_{i,j=0}^4 (-1)^{i+j} \liftwo([w_0,\ldots,\widehat{w_i}, \ldots,w_4], [v_0,\ldots,\widehat{v_i}, \ldots, v_4]) \) in two different ways, as Goncharov outlined in \cite{Goncharov94}[p.~84].
	
	It could be useful to gaining some understanding of the appearance and cancellation of the new irreducible factors (the ``primes'') appearing in arguments \( x_i \) of functions in  these weight 6 reductions (see \autoref{app:explicit}, and the factors \( \sigma_1,\ldots,\sigma_5 \) and \( \pi_1 \) appearing in \autoref{app:fullsym1:full}, \autoref{app:four:full}, and \autoref{app:fullsym2:full}), as well as the --- potentially more complicated --- irreducible factors (the ``extra primes'') appearing in \( 1-\prod_{\ell=i}^j x_\ell \), for arguments  \( x_\ell \), (they contribute to the symbol/$\otimes^m$-invariant \cite{GoncharovGalois01}).  Via the algorithm defined by Radchenko \cite[\S2]{RadchenkoThesis16}, for generating `good' candidate arguments for polylogarithm functional equations, one can potentially obtain many interesting new arguments from these new irreducibles.  With a holistic understanding of how or why these new factors arise, one can try to generate more such new irreducibles to aid further searches for (multiple) polylogarithm functional equations.

	\bibliographystyle{habbrv2}      
\bibliography{wt6_dp3.bib}

	\appendix
	\allowdisplaybreaks
	
	\section{Explicit formulae for some weight 6 identities}\label{app:explicit}
	
	The full expressions for all of the identities in \autoref{sec:higherZagier6} can be found in the following ancillary files attached to the \texttt{arXiv} submission:
	\begin{itemize}[itemsep=0.5ex]
		\item \wtsixidentities, a text file in \texttt{Mathematica} syntax, for all identities and reductions found in \autoref{sec:wt6:sym}--\autoref{sec:wt6:6fold}, with  \autoref{sec:wt6:6fold} written via \( \lif(1, y, z) \) and \( \lif(x, 1, y) \) for simplicity,  
		\item \wtsixidentitieslist, the corresponding text file, giving expressions in the form
		\texttt{identity = [ [coeff, func, [arg1, ..., argd]], ... ];} \, ,
		\item \wtsixdepthtwo, a text file in \texttt{Mathematica} syntax, for the reductions and identities in \autoref{sec:wt6:6fold}, via purely depth 2 functions, and 
		\item \wtsixdepthtwolist, the corresponding text file, giving expressions in the form
		\texttt{identity = [ [coeff, func, [arg1, ..., argd]], ... ];} \,.
	\end{itemize}
	We emphasise that the depth \( {\leq}2 \) terms in these identities are highly non-canonical, and likely highly non-optimal; they are mainly given for the sake of interest and completeness.  \medskip
	
	For the interested and endurant reader, we explicitly present here some of the identities and reductions found through these calculations.  After deriving some identities, symmetries and reductions in \autoref{sec:higherZagier6}, we have directly searched for and \emph{algebraically verified} (using the symbol/$\otimes^m$-invariant \cite{GoncharovGalois01}) shorter identities amongst the terms (or related candidates) we previously discovered, making the results \emph{even less} canonical. 
	
	The main point of this appendix is to tangibly give some new reductions (in some arbitrary presentation), and give a glimpse of the \uline{complexity} which awaits us. \medskip
	
	In what follows, we employ the `divergent' functions \( \lif \), \( \LiL_{3\;1,2} \), \( \LiL_{4\;1,1} \), and \( \LiL_{5\;1} \), which appear in the (weight 6) quadrangular polylogarithm functions and their functional equation as in \autoref{sec:quadrangular:construction}, \autoref{sec:quadrangular:quadrangulation}, \autoref{fig:qli2polygon}, and \autoref{fig:qli3polygon}.  If one prefers, they can be directly expressed via \emph{convergent} multiple polylogarithms (see \autoref{eqn:divtoMPL} and \autoref{cor:int:dihedral}), the notation for which then usually suppresses the divergence ``\( n_0 \; \)'', including the semi-colon, when \( n_0 = 0 \).  Therefore, if we write \( \LiL_{n_1,\ldots,n_k} \coloneqq \Li_{0\;n_1,\ldots,n_k} \), one has:
	\begin{align*}
	\LiL_{3\;1,1,1}(x,y,z) &= \begin{aligned}[t] 
		 \LiL_{4,1,1}\bigl(\tfrac{1}{x y z},x,y\bigr) & {} + \LiL_{4,2}\bigl(\tfrac{1}{x y z},y z\bigr)+\LiL_{5,1}\bigl(\tfrac{1}{x y},x\bigr) \\
		 & {} - 5 \LiL_{6}(x y)-\LiL_{6}(x)+5 \LiL_{6}(y z) \,, \end{aligned} \\
	\LiL_{3\;1,2}(x,y) &= -\LiL_{4,2}\bigl(\tfrac{1}{x y},y\bigr)-5 \LiL_{6}(y) \,, \\
	\LiL_{4\;1,1}(x,y) &= \LiL_{5,1}\bigl(\tfrac{1}{x y},x\bigr)-5 \LiL_{6}(x y)-\LiL_{6}(x) \,, \\
	\LiL_{5\;1}(x) &= -\LiL_{6}(x) \,.
	\end{align*}
	
	Recall again, \( \LiL_{(n_0\;)\,n_1,\ldots,n_k} \) denotes the multiple polylogarithm function \( \Li_{(n_0\;)\,n_1,\ldots,n_k} \) viewed in the Lie coalgebra, i.e. after quotienting out by products.  Therefore all identities in this Appendix are only true if considered \emph{modulo products}.  We will also extend \( \LiL_{n_0\;n_1,\ldots,n_k} \) to formal linear combinations by linearity, viz:
	\[
	 	\LiL_{n_0\;n_1,\ldots,n_k}\big( \sum\nolimits_{\ell} \lambda_\ell [x_{\ell,1},\ldots,x_{\ell,k}] \big)
	 \coloneqq \sum\nolimits_{\ell} \lambda_\ell \LiL_{n_0\;n_1,\ldots,n_k}\big( x_{\ell,1},\ldots,x_{\ell,k} \big) \,,
	\]
	in order to write longer identities more compactly.  We shall also use capital letters for better clarity and readability in the small font below.

	\subsection{Basic symmetries and relations of \texorpdfstring{\( \lif(x,y,z) \)}{Li\_\{3;1,1,1\}(x,y,z)}} \label{app:syms} The combined symmetry $\revinvsc$ from \autoref{lem:revinvDeriv}, and the two basic symmetries, reversal $\revsc$ from \autoref{cor:reverse} and inversion $\invsc$ from \autoref{cor:inv} are given as follows.%
	\biggerskip%
	\begin{align*}
	\tag{$\revinvsc$}
	\begin{aligned}[t]
	& \LiL_{3\;1,1,1}\bigl( { 
		\scriptstyle
		\bigl[   \tfrac{1}{Z}  ,  \tfrac{1}{Y}  ,  \tfrac{1}{X}   \bigr]
		+                \bigl[   X  ,  Y  ,  Z   \bigr]
	} \bigr)
	\= 
	 \LiL_{4\;1,1}\bigl( { 
		\scriptstyle
		\bigl[   X  ,  Y   \bigr]
		-                \bigl[   Y  ,  Z   \bigr]
	} \bigr)
	+
	\LiL_{5\;1}\bigl( { 
		\scriptstyle
		-                \bigl[   X   \bigr]
		+                \bigl[   Y   \bigr]
	} \bigr)
	\end{aligned}
	\end{align*}
	\begin{align*}
	\tag{$\revsc$}
	\LiL_{3\;1,1,1}\bigl( { 
		\scriptstyle
		\bigl[   X  ,  Y  ,  Z   \bigr]
		-                \bigl[   Z  ,  Y  ,  X   \bigr]
	} \bigr)
	\=
	\LiL_{3\;1,2}\bigl( { 
		\scriptstyle
		-                \bigl[   X  ,  Y Z   \bigr]
		+                \bigl[   Z  ,  X Y   \bigr]
	} \bigr) \,,
	\end{align*}
	\begin{align*}
	\tag{$\invsc$}
	\begin{aligned}[c]
	&
	\LiL_{3\;1,1,1}\bigl( { 
		\scriptstyle
		\bigl[   \tfrac{1}{X}  ,  \tfrac{1}{Y}  ,  \tfrac{1}{Z}   \bigr]
		+                \bigl[   X  ,  Y  ,  Z   \bigr]
	} \bigr)
	\= {} \\[1ex]
	&
	\LiL_{3\;1,2}\bigl( { 
		\scriptstyle
		-                \bigl[   X  ,  Y Z   \bigr]
		+                \bigl[   X Y  ,  Z   \bigr]
	} \bigr)
	+
	\LiL_{4\;1,1}\bigl( { 
		\scriptstyle
		4                \bigl[   X Y  ,  Z   \bigr]
		+                \bigl[   X  ,  Y   \bigr]
		-                \bigl[   Y  ,  Z   \bigr]
	} \bigr)
	\\
	& {} + 
	\LiL_{5\;1}\bigl( { 
		\scriptstyle
		-10                \bigl[   X Y Z   \bigr]
		-5                \bigl[   X Y   \bigr]
		+5                \bigl[   Y Z   \bigr]
		+                \bigl[   Y   \bigr]
		-                \bigl[   Z   \bigr]
	} \bigr) \,.
	\end{aligned}
	\end{align*}

	The (2,1)-shuffle $\shsym{2,1}$, of $(X_1,X_2)$ shuffled with $Z$, from \autoref{lem:21shuffleDeriv}, and the three-term relation \autoref{cor:three} are given as follows.
	\begin{align*}
		\tag{$\shsym{2,1}$}
		& \begin{aligned}[c]
		& \LiL_{3\;1,1,1}\bigl( { 
			\scriptstyle
			\bigl[   Z  ,  X_1  ,  X_2   \bigr]
			+                \bigl[   X_1  ,  Z  ,  X_2   \bigr]
			+                \bigl[   X_1  ,  X_2  ,  Z   \bigr]
		} \bigr)
		\= 
		\\
		& \LiL_{3\;1,2}\bigl( { 
			\scriptstyle
			-                \bigl[   X_1  ,  Z X_2   \bigr]
			+                \bigl[   Z X_1  ,  X_2   \bigr]
		} \bigr)
		+
		\LiL_{4\;1,1}\bigl( { 
			\scriptstyle
			4                \bigl[   Z X_1  ,  X_2   \bigr]
		} \bigr) \,,
	\end{aligned} \\[2ex]
		\tag{Three}
		& \LiL_{3\;1,1,1}\bigl( { 
			\scriptstyle
			\bigl[   X  ,  Y  ,  Z   \bigr]
			+                \bigl[   Y  ,  Z  ,  X   \bigr]
			+                \bigl[   Z  ,  X  ,  Y   \bigr]
		} \bigr)
		\=
		\LiL_{5\;1}\bigl( { 
			\scriptstyle
			10                \bigl[   X Y Z   \bigr]
		} \bigr) \,.
		\end{align*}
	
	\subsection{Reduction of \texorpdfstring{\( \lif(1,1,A) \)}{Li\_\{3;1,1,1\}(1,1,A)}} \label{app:li11x} The reduction of \( \lif(1,1,A) \) in $\redid{1,1,x}$ from \autoref{prop:11x} is given as follows.
	\begin{align*}
	\LiL_{3\;1,1,1}\bigl( { 
		\scriptstyle
		\bigl[   1  ,  1  ,  A   \bigr]
	} \bigr)
	=
	\LiL_{4\;1,1}\bigl( { 
		\scriptstyle
		\bigl[   1  ,  -\tfrac{1-A}{A}   \bigr]
		+2                \bigl[   1  ,  A   \bigr]
	} \bigr)
	+
	\LiL_{5\;1}\bigl( { 
		\scriptstyle
		-                \bigl[   -\tfrac{1-A}{A}   \bigr]
		-                \bigl[   1-A   \bigr]
		-3                \bigl[   A   \bigr]
	} \bigr) \,.
	\end{align*}
	
	\subsection{First degenerate symmetry, \texorpdfstring{\( \lif(1,A,B) - \lif(1,\tfrac{A (1-B)}{1-A B}  ,  -\tfrac{1-A B}{A B} ) \)}{Li\_\{3;1,1,1\}(1,A,B) - Li\_\{3;1,1,1\}(1,A(1-B)/(1 - AB),-(1-AB)/(AB))}}  \label{app:degsym1}
	The first degenerate symmetry of \( \lif(1,A,B) \), $\degsym{1}$ from \autoref{lem:onexy_sym1}, is given as follows.
	\begin{align*}
	& 
	\LiL_{3\;1,1,1}\bigl( { 
		\scriptstyle
		+                \bigl[   1  ,  A  ,  B   \bigr]
		-                \bigl[   1  ,  \tfrac{A (1-B)}{1-A B}  ,  -\tfrac{1-A B}{A B}   \bigr]
	} \bigr)
	\= \\[2ex]
	& 
	\LiL_{3\;1,2}\bigl( { 
		\scriptstyle
		-                \bigl[   \tfrac{A (1-B)}{1-A B}  ,  -\tfrac{1-A B}{A B}   \bigr]
		+                \bigl[   -\tfrac{1-A}{A (1-B)}  ,  A B   \bigr]
		-                \bigl[   1  ,  -\tfrac{(1-A) B}{1-B}   \bigr]
		+                \bigl[   1-A  ,  -\tfrac{B}{1-B}   \bigr]
		+                \bigl[   1  ,  -\tfrac{1-B}{B}   \bigr]
	} \bigr)
	\\[1ex]
	& {} 
	+ 
	\LiL_{4\;1,1}\bigl( { 
		\scriptstyle
		-2                \bigl[   \tfrac{A (1-B)}{1-A B}  ,  -\tfrac{1-A B}{A B}   \bigr]
		-                \bigl[   -\tfrac{A B}{1-A B}  ,  \tfrac{1-A B}{1-A}   \bigr]
		+2                \bigl[   -\tfrac{1-A}{A (1-B)}  ,  A B   \bigr]
		+2                \bigl[   1  ,  -\tfrac{A (1-B)}{1-A}   \bigr]
	    -4                \bigl[   1  ,  -\tfrac{(1-A) B}{1-B}   \bigr]
			} \\ & \phantom{{} + \LiL_{4\;1,1}\bigl(}
			{ \scriptstyle 
		+2                \bigl[   1-A  ,  -\tfrac{B}{1-B}   \bigr]
		+                \bigl[   -\tfrac{A}{1-A}  ,  B   \bigr]
		-2                \bigl[   1  ,  -\tfrac{1-A}{A}   \bigr]
		+2                \bigl[   1  ,  -\tfrac{1-B}{B}   \bigr]
		+2                \bigl[   1  ,  A B   \bigr]
		-                \bigl[   1  ,  A   \bigr]
		+                \bigl[   A  ,  B   \bigr]
	} \bigr)
	\\[1ex]
	& {} 
	+
	\LiL_{5\;1}\bigl( { 
		\scriptstyle
		-3                \bigl[   -\tfrac{A (1-B)}{1-A}   \bigr]
		+5                \bigl[   -\tfrac{(1-A) B}{1-B}   \bigr]
		+                \bigl[   \tfrac{1-A}{1-A B}   \bigr]
		-2                \bigl[   -\tfrac{A B}{1-A B}   \bigr]
		+5                \bigl[   -\tfrac{A B}{1-A}   \bigr]
				} \\ & \phantom{{} + \LiL_{5\;1}\bigl(}
				{ \scriptstyle 
		-3                \bigl[   -\tfrac{1-B}{B}   \bigr]
		+5                \bigl[   -\tfrac{1-A}{A}   \bigr]
		-                \bigl[   1-A   \bigr]
		-5                \bigl[   A B   \bigr]
		+3                \bigl[   A   \bigr]
	} \bigr) \,.
	\end{align*}

	The one-variable degeneration from \autoref{cor:onevar}, obtained by sending \( B \to A^{-1} \) in $\degsym{1}$, is given as follows.
	\begin{align*}
		& \LiL_{3\;1,1,1}\bigl( { 
			\scriptstyle
			\bigl[   1  ,  A  ,  \tfrac{1}{A}   \bigr]
		} \bigr)
		\= \\[2ex]
		& 
		\LiL_{3\;1,2}\bigl( { 
			\scriptstyle
			\bigl[   \tfrac{1}{1-A}  ,  1-A   \bigr]
			+                \bigl[   1  ,  1-A   \bigr]
		} \bigr)
		+
		\LiL_{4\;1,1}\bigl( { 
			\scriptstyle
			\bigl[   -\tfrac{1-A}{A}  ,  A   \bigr]
			-2                \bigl[   1  ,  -\tfrac{1-A}{A}   \bigr]
			+2                \bigl[   1  ,  1-A   \bigr]
			-                \bigl[   1  ,  A   \bigr]
		} \bigr)
		\\[1ex]
		& {}
		+ 
		\LiL_{5\;1}\bigl( { 
			\scriptstyle
			4                \bigl[   -\tfrac{1-A}{A}   \bigr]
			-6                \bigl[   1-A   \bigr]
			-10                \bigl[   -(1-A)   \bigr]
			+5                \bigl[   A   \bigr]
		} \bigr) \,.
	\end{align*}
	
	\subsection{Second degenerate symmetry, \texorpdfstring{\( \lif(1,A,B)  +  \lif(1,A  ,  \tfrac{1-A B}{A (1-B)})\)}{Li\_\{3;1,1,1\}(1,A,B) + Li\_\{3;1,1,1\}(1,A,(1-AB)/(A(1-B)))}} \label{app:degsym2} The second degenerate symmetry of \( \lif(1, A, B) \), $\degsym{2}$ from \autoref{lem:onexy_sym2}, is given as follows.
	\begin{align*}
	& \LiL_{3\;1,1,1}\bigl( { 
		\scriptstyle
		                \bigl[   1  ,  A  ,  B   \bigr]
		+                \bigl[   1  ,  A  ,  \tfrac{1-A B}{A (1-B)}   \bigr]
	} \bigr)
	\= \\[2ex]
	&\LiL_{3\;1,2}\bigl( { 
		\scriptstyle
		\bigl[   -\tfrac{A}{1-A}  ,  \tfrac{(1-A) B}{1-A B}   \bigr]
		+                \bigl[   -\tfrac{1-A}{A}  ,  -\tfrac{A B}{1-A B}   \bigr]
		-                \bigl[   -\tfrac{1-A}{A}  ,  -\tfrac{A}{1-A}   \bigr]
		+                \bigl[   \tfrac{1}{A}  ,  \tfrac{A (1-B)}{1-A B}   \bigr]
			} \\ & \phantom{{} \LiL_{3\;1,2}\bigl(}
			{ \scriptstyle 
		-                \bigl[   1  ,  \tfrac{1-B}{1-A B}   \bigr]
		-                \bigl[   1  ,  -\tfrac{1-A}{A}   \bigr]
		-                \bigl[   1  ,  A B   \bigr]
		+                \bigl[   A  ,  B   \bigr]
	} \bigr)
	\\[1ex]
	& {}
	+
	\LiL_{4\;1,1}\bigl( { 
		\scriptstyle
		\bigl[   -\tfrac{1-B}{(1-A) B}  ,  -\tfrac{A B}{1-A B}   \bigr]
		+2                \bigl[   -\tfrac{A}{1-A}  ,  \tfrac{(1-A) B}{1-A B}   \bigr]
		-                \bigl[   \tfrac{1}{1-A}  ,  \tfrac{A (1-B)}{1-A B}   \bigr]
		+2                \bigl[   -\tfrac{1-A}{A}  ,  -\tfrac{A B}{1-A B}   \bigr]
				} \\ & \phantom{{} + \LiL_{4\;1,1}\bigl(}
				{ \scriptstyle 
		{}-                \bigl[   \tfrac{(1-A) B}{1-A B}  ,  1-A B   \bigr]
		+                \bigl[   -\tfrac{1-B}{(1-A) B}  ,  A B   \bigr]
		-                \bigl[   \tfrac{1-B}{1-A B}  ,  1-A B   \bigr]
		+2                \bigl[   \tfrac{1}{A}  ,  \tfrac{A (1-B)}{1-A B}   \bigr]
		-                \bigl[   1  ,  \tfrac{1-B}{1-A B}   \bigr]
				} \\ & \phantom{{} + \LiL_{4\;1,1}\bigl(}
				{ \scriptstyle 
		{}-2                \bigl[   1  ,  -\tfrac{1-A}{A}   \bigr]
		-                \bigl[   1-A  ,  B   \bigr]
		+2                \bigl[   1  ,  1-A   \bigr]
		-                \bigl[   1  ,  A B   \bigr]
		+2                \bigl[   1  ,  A   \bigr]
		+2                \bigl[   A  ,  B   \bigr]
	} \bigr)
	\\[1ex]
	& {}
	+
	\LiL_{5\;1}\bigl( { 
		\scriptstyle
		-5                \bigl[   \tfrac{A (1-B)}{(1-A) (1-A B)}   \bigr]
		-                \bigl[   \tfrac{A (1-B)}{1-A B}   \bigr]
		-2                \bigl[   \tfrac{(1-A) B}{1-A B}   \bigr]
		-3                \bigl[   -\tfrac{A (1-B)}{1-A}   \bigr]
		-3                \bigl[   \tfrac{1-B}{1-A B}   \bigr]
		-                \bigl[   -\tfrac{1-A}{A}   \bigr]
				} \\ & \phantom{{} + \LiL_{5\;1}\bigl(}
				{ \scriptstyle 
		+10                \bigl[   (1-A) B   \bigr]
		-2                \bigl[   1-A B   \bigr]
		+3                \bigl[   1-B   \bigr]
		-5                \bigl[   1-A   \bigr]
		+3                \bigl[   A B   \bigr]
		-                \bigl[   A   \bigr]
		-3                \bigl[   B   \bigr]
	} \bigr) \,.
	\end{align*}
	
	\subsection{Reduction of \texorpdfstring{\( \lif(1,A,B) \)}{Li\_\{3;1,1,1\}(1,A,B)}}\label{app:onexy_dp2} The reduction of \( \lif(1,A,B) \) to depth 2, established in $\redid{1,x,y}$ from \autoref{thm:onexy_dp2}, is given as follows.\smallskip
	\begin{align*}
	&\LiL_{3\;1,1,1}\bigl( { 
		\scriptstyle
		\bigl[   1  ,  A  ,  B   \bigr]
	} \bigr)
	\= \\[2ex]
	&\LiL_{3\;1,2}\bigl( { 
		\scriptstyle
		\tfrac{1}{2}    \bigl[   -\tfrac{1-A}{A}  ,  -\tfrac{A (1-B)}{1-A}   \bigr]
		+\tfrac{1}{2}    \bigl[   -\tfrac{A}{1-A}  ,  \tfrac{(1-A) B}{1-A B}   \bigr]
		+\tfrac{1}{2}    \bigl[   -\tfrac{B}{1-B}  ,  \tfrac{A (1-B)}{1-A B}   \bigr]
		+\tfrac{1}{2}    \bigl[   \tfrac{1}{1-B}  ,  -\tfrac{A (1-B)}{1-A}   \bigr]
				} \\ & \phantom{{} \LiL_{3\;1,2}\bigl(}
				{ \scriptstyle 
		-\tfrac{1}{2}    \bigl[   -\tfrac{1-B}{(1-A) B}  ,  A B   \bigr]
		+\tfrac{1}{2}    \bigl[   \tfrac{1-B}{1-A B}  ,  1-A B   \bigr]
		-\tfrac{1}{2}    \bigl[   1-A  ,  \tfrac{1-A B}{1-A}   \bigr]
		+\tfrac{1}{2}    \bigl[   1  ,  -\tfrac{A (1-B)}{1-A}   \bigr]
		+\tfrac{1}{2}    \bigl[   1  ,  \tfrac{A (1-B)}{1-A B}   \bigr]
				} \\ & \phantom{{} \LiL_{3\;1,2}\bigl(}
				{ \scriptstyle 
		+\tfrac{1}{2}    \bigl[   1  ,  \tfrac{(1-A) B}{1-A B}   \bigr]
		-\tfrac{1}{2}    \bigl[   A  ,  \tfrac{1-B}{1-A B}   \bigr]
		-\tfrac{1}{2}    \bigl[   B  ,  \tfrac{1-A}{1-A B}   \bigr]
		-\tfrac{1}{2}    \bigl[   1  ,  -\tfrac{A B}{1-A B}   \bigr]
		-\tfrac{1}{2}    \bigl[   \tfrac{1}{1-A}  ,  1-A B   \bigr]
				} \\ & \phantom{{} \LiL_{3\;1,2}\bigl(}
				{ \scriptstyle 
		+\tfrac{1}{4}    \bigl[   \tfrac{1}{1-A}  ,  1-A   \bigr]
		-\tfrac{1}{2}    \bigl[   1  ,  1-B   \bigr]
		-\tfrac{1}{2}    \bigl[   1  ,  A B   \bigr]
		+\tfrac{1}{2}    \bigl[   A  ,  B   \bigr]
	} \bigr)
		\\[1ex]
		& {}
	+
	\LiL_{4\;1,1}\bigl( { 
		\scriptstyle
		\tfrac{1}{2}    \bigl[   \tfrac{A (1-B)}{1-A B}  ,  -\tfrac{1-A B}{A B}   \bigr]
		-\tfrac{1}{2}    \bigl[   -\tfrac{A B}{1-A B}  ,  \tfrac{1-A B}{1-A}   \bigr]
		+                \bigl[   -\tfrac{1-A}{A}  ,  -\tfrac{A (1-B)}{1-A}   \bigr]
		+                \bigl[   -\tfrac{A}{1-A}  ,  \tfrac{(1-A) B}{1-A B}   \bigr]
				} \\ & \phantom{{} + \LiL_{4\;1,1}\bigl(}
				{ \scriptstyle 
		+                \bigl[   -\tfrac{B}{1-B}  ,  \tfrac{A (1-B)}{1-A B}   \bigr]
		+\tfrac{1}{2}    \bigl[   -\tfrac{1-A}{A}  ,  \tfrac{1-A B}{1-B}   \bigr]
		+\tfrac{1}{2}    \bigl[   -\tfrac{1-B}{B}  ,  \tfrac{1-A B}{1-A}   \bigr]
		-\tfrac{1}{2}    \bigl[   -\tfrac{1-A}{A (1-B)}  ,  1-A B   \bigr]
				} \\ & \phantom{{} + \LiL_{4\;1,1}\bigl(}
				{ \scriptstyle 
		-\tfrac{1}{2}    \bigl[   -\tfrac{1-B}{(1-A) B}  ,  1-A B   \bigr]
		+                \bigl[   \tfrac{1}{1-B}  ,  -\tfrac{A (1-B)}{1-A}   \bigr]
		-\tfrac{1}{2}    \bigl[   \tfrac{(1-A) B}{1-A B}  ,  1-A B   \bigr]
		-\tfrac{1}{2}    \bigl[   -\tfrac{1-A}{A (1-B)}  ,  A B   \bigr]
				} \\ & \phantom{{} + \LiL_{4\;1,1}\bigl(}
				{ \scriptstyle 
		+\tfrac{1}{2}    \bigl[   \tfrac{1-B}{1-A B}  ,  1-A B   \bigr]
		-                \bigl[   -\tfrac{1-B}{(1-A) B}  ,  A B   \bigr]
		+\tfrac{1}{2}    \bigl[   1-A  ,  \tfrac{1-B}{1-A B}   \bigr]
		+\tfrac{1}{2}    \bigl[   1-B  ,  \tfrac{1-A}{1-A B}   \bigr]
		-                \bigl[   1-A  ,  \tfrac{1-A B}{1-A}   \bigr]
				} \\ & \phantom{{} + \LiL_{4\;1,1}\bigl(}
				{ \scriptstyle 
		+\tfrac{3}{2}    \bigl[   1  ,  \tfrac{A (1-B)}{1-A B}   \bigr]
		+2                \bigl[   1  ,  -\tfrac{A (1-B)}{1-A}   \bigr]
		+2                \bigl[   1  ,  \tfrac{(1-A) B}{1-A B}   \bigr]
		+\tfrac{1}{2}    \bigl[   1  ,  \tfrac{1-B}{1-A B}   \bigr]
		-\tfrac{1}{2}    \bigl[   A  ,  \tfrac{1-B}{1-A B}   \bigr]
					} \\ & \phantom{{} + \LiL_{4\;1,1}\bigl(}
					{ \scriptstyle 
		-\tfrac{1}{2}    \bigl[   B  ,  \tfrac{1-A}{1-A B}   \bigr] 
		-                \bigl[   \tfrac{1}{1-A}  ,  1-A B   \bigr]
		-2                \bigl[   1  ,  -\tfrac{A B}{1-A B}   \bigr]
		+\tfrac{1}{2}    \bigl[   -\tfrac{A}{1-A}  ,  B   \bigr]
		-                \bigl[   1  ,  -\tfrac{1-A}{A}   \bigr]
					} \\ & \phantom{{} + \LiL_{4\;1,1}\bigl(}
					{ \scriptstyle 
		-\tfrac{1}{2}    \bigl[   1-A  ,  B   \bigr]
		-2                \bigl[   1  ,  1-B   \bigr]
		-\tfrac{1}{2}    \bigl[   1  ,  A B   \bigr]
		-\tfrac{1}{2}    \bigl[   1  ,  B   \bigr]
		+\tfrac{3}{2}    \bigl[   A  ,  B   \bigr]
	} \bigr)
		\\[1ex]
		& {}
	+
	\LiL_{5\;1}\bigl( { 
		\scriptstyle
		\tfrac{5}{2}    \bigl[   -\tfrac{A (1-B)}{(1-A) (1-A B)}   \bigr]
		+\tfrac{5}{2}    \bigl[   -\tfrac{(1-A) B}{(1-B) (1-A B)}   \bigr]
		-5                \bigl[   \tfrac{(1-A) (1-B)}{1-A B}   \bigr]
		-\tfrac{1}{2}    \bigl[   -\tfrac{(1-A) B}{1-B}   \bigr]
		-4                \bigl[   \tfrac{A (1-B)}{1-A B}   \bigr]
					} \\ & \phantom{{} + \LiL_{5\;1}\bigl(}
					{ \scriptstyle 
		-5                \bigl[   \tfrac{(1-A) B}{1-A B}   \bigr]
		-\tfrac{11}{2}    \bigl[   -\tfrac{A (1-B)}{1-A}   \bigr]
		-\tfrac{1}{2}    \bigl[   \tfrac{1-B}{1-A B}   \bigr]
		+                \bigl[   \tfrac{1-A}{1-A B}   \bigr]
		+\tfrac{5}{2}    \bigl[   -\tfrac{A B}{1-A B}   \bigr]
		+\tfrac{5}{2}    \bigl[   -\tfrac{A B}{1-A}   \bigr]
					} \\ & \phantom{{} + \LiL_{5\;1}\bigl(}
					{ \scriptstyle 
		-\tfrac{1}{2}    \bigl[   -\tfrac{1-B}{B}   \bigr]
		+\tfrac{5}{2}    \bigl[   -\tfrac{1-A}{A}   \bigr]
		+5                \bigl[   (1-A) B   \bigr]
		-\tfrac{1}{2}    \bigl[   1-A B   \bigr]
		-\tfrac{1}{4}    \bigl[   1-A   \bigr]
		+5                \bigl[   1-B   \bigr]
		+2                \bigl[   A B   \bigr]
		+                \bigl[   A   \bigr]
	} \bigr) \,.
	\end{align*}
	
	Let us note that the reduction for \( \lif(A,1,B) \) which follows from in \autoref{cor:xyone_dp2} is highly non-optimal (it involved approrimately 100 terms already).  One can instead find the following much shorter reduction with judicious choice of the depth 2 arguments (essentially cross-ratios of \( \{ \infty, 0, 1, A \} \) or \( \{\infty, 0, 1, B\} \), or 2-fold products thereof).
	\begin{align*}
	&\LiL_{3\;1,1,1}\bigl( { 
		\scriptstyle
		\bigl[   A  ,  1  ,  B   \bigr]
	} \bigr)
	\= \\[2ex]
	&\LiL_{3\;1,2}\bigl( { 
		\scriptstyle
		-                \bigl[   \tfrac{1}{A}  ,  A B   \bigr]
		-                \bigl[   \tfrac{1}{B}  ,  A B   \bigr]
		-                \bigl[   A  ,  B   \bigr]
	} \bigr)
		\\[1ex]
	& {}
	+
	\LiL_{4\;1,1}\bigl( { 
		\scriptstyle
		-                \bigl[   1-A  ,  -\tfrac{B}{1-B}   \bigr]
		+                \bigl[   -\tfrac{A}{1-A}  ,  1-B   \bigr]
		+                \bigl[   1  ,  -\tfrac{1-B}{B}   \bigr]
		-                \bigl[   1  ,  1-A   \bigr]
		-2                \bigl[   \tfrac{1}{A}  ,  A B   \bigr]
		-2                \bigl[   \tfrac{1}{B}  ,  A B   \bigr]
		-                \bigl[   1  ,  A   \bigr]
	} \bigr)
		\\[1ex]
	& {}
	+
	\LiL_{5\;1}\bigl( { 
		\scriptstyle
		-2                \bigl[   -\tfrac{A (1-B)}{1-A}   \bigr]
		+3                \bigl[   -\tfrac{(1-A) B}{1-B}   \bigr]
		+                \bigl[   -\tfrac{1-A}{A}   \bigr]
		-                \bigl[   -\tfrac{1-B}{B}   \bigr]
		+2                \bigl[   1-A   \bigr]
		+2                \bigl[   A   \bigr]
	} \bigr) \,.
	\end{align*}
	
	It would be interesting to try to derive this simplified version via the degenerations to stable curves procedure; perhaps it also requires a detailed understanding of depth 2 functional equations?
	
	\subsection{Full Symmetry 1, \texorpdfstring{\( \lif(A,B,C) + \lif(1-A, \tfrac{B}{B-1}, 1-C) \)}{Li\_\{3;1,1,1\}(A,B,C) + Li\_\{3;1,1,1\}(1-A, B/(B-1), 1-C)}}  \label{app:fullsym1:full} 
	For simplicity, we retain degenerate \( \lif(1,x,y) \)-type terms, which we know, from \autoref{cor:xyone_dp2} and explicitly from \autoref{app:onexy_dp2} above, are strictly depth \( {\leq}2 \).  We also write
	{\small\[
		\sigma_1 = 1 - A - C + A B C \,.
	\]}%
	The first full symmetry of \( \lif \), in $\fullsym{1}$ from \autoref{lem:fullsym1}, is given as follows.
	\begin{align*}
	& \LiL_{3\;1,1,1}\bigl( { 
		\scriptstyle
		                \bigl[   A  ,  B  ,  C   \bigr]
		+		\bigl[   1-A  ,  \tfrac{B}{B-1}  ,  1-C   \bigr]
	} \bigr)
	\= \\[2ex]
	&
	\LiL_{3\;1,1,1}\bigl( { 
		\scriptstyle
		-                \bigl[   1  ,  -\tfrac{\sigma _1}{A (1-B) C}  ,  -\tfrac{1-B}{B \sigma _1}   \bigr]
		+                \bigl[   1  ,  \tfrac{(1-A) (1-C)}{\sigma _1}  ,  -\tfrac{B \sigma _1}{1-B}   \bigr]
		+                \bigl[   1-C  ,  1  ,  \tfrac{1-A B}{A (1-B)}   \bigr]
		-                \bigl[   1  ,  1-A  ,  -\tfrac{B (1-C)}{1-B}   \bigr]
				} \\ & \phantom{{} \LiL_{3\;1,1,1}\bigl(}
				{ \scriptstyle 
		-                \bigl[   1  ,  1-C  ,  -\tfrac{(1-A) B}{1-B}   \bigr]
		-                \bigl[   \tfrac{1-C}{1-B C}  ,  1  ,  \tfrac{1}{A}   \bigr]
		-                \bigl[   1-C  ,  1  ,  -\tfrac{B}{1-B}   \bigr]
		+                \bigl[   1  ,  \tfrac{1}{A}  ,  \tfrac{1}{B C}   \bigr]
		+                \bigl[   1  ,  \tfrac{1}{C}  ,  \tfrac{1}{A B}   \bigr]
		+                \bigl[   \tfrac{1}{B}  ,  1  ,  \tfrac{1}{A}   \bigr]
	} \bigr)
	\\[1ex]
	& {}
	+
	\LiL_{3\;1,2}\bigl( { 
		\scriptstyle
		-                \bigl[   \tfrac{(1-A) (1-C)}{\sigma _1}  ,  -\tfrac{B \sigma _1}{1-B}   \bigr]
		+                \bigl[   -\tfrac{(1-A) B}{1-B}  ,  1-C   \bigr]
		+                \bigl[   1  ,  A B C   \bigr]
		-                \bigl[   A  ,  B C   \bigr]
	} \bigr)
	\\[1ex]
	& {}
	+
	\LiL_{4\;1,1}\bigl( { 
		\scriptstyle
		-                \bigl[   \tfrac{(1-A) (1-C)}{\sigma _1}  ,  -\tfrac{B \sigma _1}{1-B}   \bigr]
		+2                \bigl[   -\tfrac{A (1-B) C}{\sigma _1}  ,  -\tfrac{B \sigma _1}{1-B}   \bigr]
		+                \bigl[   1  ,  -\tfrac{A (1-B) C}{\sigma _1}   \bigr]
		+                \bigl[   \tfrac{1}{1-A}  ,  \tfrac{(1-B) C}{1-B C}   \bigr]
				} \\ & \phantom{{} + \LiL_{4\;1,1}\bigl(}
				{ \scriptstyle 
		+2                \bigl[   \tfrac{A (1-B)}{1-A B}  ,  \tfrac{1}{1-C}   \bigr]
		+                \bigl[   1-A  ,  -\tfrac{B (1-C)}{1-B}   \bigr]
		+3                \bigl[   -\tfrac{(1-A) B}{1-B}  ,  1-C   \bigr]
		-                \bigl[   1  ,  \tfrac{(1-B) C}{1-B C}   \bigr]
		-                \bigl[   1  ,  \tfrac{1-A}{1-A B}   \bigr]
				} \\ & \phantom{{} + \LiL_{4\;1,1}\bigl(}
				{ \scriptstyle 
		-                \bigl[   1  ,  \tfrac{1-C}{1-B C}   \bigr]
		-                \bigl[   \tfrac{1-A B}{1-A}  ,  C   \bigr]
		-2                \bigl[   A  ,  \tfrac{1-B C}{1-C}   \bigr]
		-3                \bigl[   -\tfrac{B}{1-B}  ,  1-C   \bigr] 
				} \\ & \phantom{{} + \LiL_{4\;1,1}\bigl(}
				{ \scriptstyle 
		+                \bigl[   1  ,  1-A   \bigr]
		+                \bigl[   1  ,  1-C   \bigr]
		-2                \bigl[   A  ,  B C   \bigr]
		+2                \bigl[   A B  ,  C   \bigr]
		+                \bigl[   1  ,  B   \bigr]
		+3                \bigl[   A  ,  B   \bigr]
	} \bigr)
	\\[1ex]
	& {}
	+
	\LiL_{5\;1}\bigl( { 
		\scriptstyle
		-4                \bigl[   -\tfrac{A (1-B) C}{\sigma _1}   \bigr]
		-2                \bigl[   \tfrac{(1-B) C}{(1-A) (1-B C)}   \bigr]
		-3                \bigl[   \tfrac{A (1-B)}{(1-A B) (1-C)}   \bigr]
		+5                \bigl[   -\tfrac{B \sigma _1}{1-B}   \bigr]
		-5                \bigl[   -\tfrac{(1-A) B (1-C)}{1-B}   \bigr]
				} \\ & \phantom{{} + \LiL_{5\;1}\bigl(}
				{ \scriptstyle 
		-2                \bigl[   \tfrac{A (1-B)}{1-A B}   \bigr]
		+2                \bigl[   \tfrac{(1-A B) C}{1-A}   \bigr]
		+2                \bigl[   \tfrac{(1-B) C}{1-B C}   \bigr]
		+3                \bigl[   \tfrac{A (1-B C)}{1-C}   \bigr]
		+9                \bigl[   -\tfrac{B (1-C)}{1-B}   \bigr]
		+                \bigl[   \tfrac{1-A}{1-A B}   \bigr]
				} \\ & \phantom{{} + \LiL_{5\;1}\bigl(}
				{ \scriptstyle 
		+4                \bigl[   \tfrac{1-C}{1-B C}   \bigr]
		-                \bigl[   1-A   \bigr]
		-4                \bigl[   1-C   \bigr]
		-5                \bigl[   A B C   \bigr]
		+                \bigl[   B C   \bigr]
		-5                \bigl[   A B   \bigr]
		+2                \bigl[   A   \bigr]
		+2                \bigl[   C   \bigr]
	} \bigr) \,.
	\end{align*}
	By expanding out the degenerate terms \( \lif(1,x,y) \), one can obtain an explicit identity for \( \lif(A,B,C) + \lif\big(1-A,\tfrac{B}{B-1},1-C\big) \) strictly in terms of depth \( {\leq}2 \) functions.  The resulting identity has around 400 terms, after using the shorter reductions for \( \lif(1,x,y) \) (69 terms), and \( \lif(x,1,y) \) (16 terms), and simplifying or combining some of the resulting depth \( {\leq}2 \) terms using their known symmetries.  
	
	Consult \autoref{tbl:symmary} below, for a complete breakdown of the structure of this identity in terms of depth \( {\leq}2 \) functions, and see the ancillary files \wtsixdepthtwo{} and\linebreak \wtsixdepthtwolist{} for the associated full expression in terms of purely depth 2.
	
	\subsection{Four-term relation for \texorpdfstring{\( \lif(x,y,z) \)}{Li\_\{3;1,1,1\}(x,y,z)}}
	\label{app:four:full} 
	If we write
	{\small \[
		\sigma_2 = A + B - A B - A C \,, \quad \sigma_3 = A - A B - A C + B C \,,
	\]}%
	for notational simplicity, then the four-term relation for \( \lif \), from \autoref{prop:fourterm}, is given as follows.
	\begin{align*}
	&\LiL_{3\;1,1,1}\bigl( { 
		\scriptstyle
		                \bigl[   A  ,  \tfrac{1}{B}  ,  1-C   \bigr]
		+                \bigl[   A  ,  \tfrac{1}{B}  ,  C   \bigr]
		-                \bigl[   A  ,  \tfrac{C}{B}  ,  \tfrac{1-B}{1-C}   \bigr]
		-                \bigl[   A  ,  \tfrac{C}{B}  ,  \tfrac{1-C}{1-B}   \bigr]
	} \bigr)
		\= \\[2ex]
		&
	\LiL_{3\;1,1,1}\bigl( { 
		\scriptstyle
		\bigl[   1  ,  \tfrac{A (1-B) (1-C)}{\sigma _3}  ,  \tfrac{\sigma _3}{(1-B) B}   \bigr]
		-                \bigl[   1  ,  \tfrac{\sigma _3}{(1-A) B C}  ,  \tfrac{(1-B) B}{\sigma _3}   \bigr]
		+                \bigl[   1  ,  \tfrac{\sigma _2}{(1-A) B}  ,  \tfrac{(1-B) B}{C \sigma _2}   \bigr]
		-                \bigl[   1  ,  \tfrac{A (1-C)}{\sigma _2}  ,  \tfrac{C \sigma _2}{(1-B) B}   \bigr]
				} \\ & \phantom{{} \LiL_{3\;1,1,1}\bigl(}
				{ \scriptstyle 
		-                \bigl[   1  ,  \tfrac{A (1-B)}{A-B}  ,  \tfrac{(A-B) C}{B (1-C)}   \bigr]
		-                \bigl[   1  ,  -\tfrac{1-B}{B-C}  ,  -\tfrac{B-C}{(1-A) C}   \bigr]
		-                \bigl[   \tfrac{1-C}{1-B}  ,  1  ,  \tfrac{B-A C}{(1-A) B}   \bigr]
		-                \bigl[   1  ,  \tfrac{A (1-B)}{A-B}  ,  -\tfrac{A-B}{B-A C}   \bigr]
				} \\ & \phantom{{} \LiL_{3\;1,1,1}\bigl(}
				{ \scriptstyle 
		+                \bigl[   1  ,  \tfrac{A (B-C)}{B-A C}  ,  -\tfrac{1-B}{B-C}   \bigr]
		-                \bigl[   \tfrac{(1-B) (1-C)}{1-B-C}  ,  1  ,  \tfrac{1}{1-A}   \bigr]
		+                \bigl[   1  ,  \tfrac{A (1-B)}{A-B}  ,  -\tfrac{A-B}{B}   \bigr]
		-                \bigl[   1  ,  -\tfrac{B-C}{C}  ,  -\tfrac{1-B}{B-C}   \bigr]
				} \\ & \phantom{{} \LiL_{3\;1,1,1}\bigl(}
				{ \scriptstyle 
		+                \bigl[   \tfrac{1-C}{1-B-C}  ,  1  ,  \tfrac{1}{1-A}   \bigr]
		+                \bigl[   1-C  ,  1  ,  -\tfrac{A-B}{(1-A) B}   \bigr]
		+                \bigl[   1  ,  \tfrac{1-B}{1-C}  ,  \tfrac{A C}{B}   \bigr]
		+                \bigl[   1  ,  \tfrac{1-C}{1-B}  ,  \tfrac{A C}{B}   \bigr]
		+                \bigl[   1  ,  A  ,  \tfrac{(1-B) C}{B (1-C)}   \bigr]
				} \\ & \phantom{{} \LiL_{3\;1,1,1}\bigl(}
				{ \scriptstyle 
		+                \bigl[   1  ,  A  ,  \tfrac{(1-C) C}{(1-B) B}   \bigr]
		+                \bigl[   \tfrac{1-C}{1-B}  ,  1  ,  \tfrac{C}{B}   \bigr]
		-                \bigl[   -\tfrac{B-C}{C}  ,  1  ,  \tfrac{1}{1-A}   \bigr]
		+                \bigl[   1  ,  \tfrac{1}{C}  ,  \tfrac{1-B}{1-A}   \bigr]
		+                \bigl[   1-B  ,  1  ,  \tfrac{1}{1-A}   \bigr]
				} \\ & \phantom{{} \LiL_{3\;1,1,1}\bigl(}
				{ \scriptstyle 
		-                \bigl[   1  ,  A  ,  -\tfrac{1-B}{B}   \bigr]
		-                \bigl[   1  ,  A  ,  \tfrac{1-C}{B}   \bigr]
		-                \bigl[   1  ,  1-C  ,  \tfrac{A}{B}   \bigr]
		+                \bigl[   1  ,  1-B  ,  \tfrac{1}{C}   \bigr]
		-                \bigl[   1-C  ,  1  ,  \tfrac{1}{B}   \bigr]
		+                \bigl[   A  ,  1  ,  \tfrac{C}{B}   \bigr]
	} \bigr)
	\\[1ex]
	& {}
	+
	\LiL_{3\;1,2}\bigl( { 
		\scriptstyle
		-                \bigl[   \tfrac{A (1-B) (1-C)}{\sigma _3}  ,  \tfrac{\sigma _3}{(1-B) B}   \bigr]
		+                \bigl[   \tfrac{A (1-C)}{\sigma _2}  ,  \tfrac{C \sigma _2}{(1-B) B}   \bigr]
		+                \bigl[   \tfrac{A (1-B)}{A-B}  ,  \tfrac{(A-B) C}{B (1-C)}   \bigr]
		+                \bigl[   \tfrac{A (1-B)}{A-B}  ,  -\tfrac{A-B}{B-A C}   \bigr]
				} \\ & \phantom{{} + \LiL_{3\;1,2}\bigl(}
				{ \scriptstyle 
		+                \bigl[   -\tfrac{B-C}{1-B}  ,  -\tfrac{(1-A) C}{B-C}   \bigr]
		-                \bigl[   \tfrac{A (1-B)}{A-B}  ,  -\tfrac{A-B}{B}   \bigr]
		-                \bigl[   \tfrac{(1-A) B}{A (1-B)}  ,  \tfrac{A C}{B}   \bigr]
		-                \bigl[   \tfrac{1-B}{1-C}  ,  \tfrac{A C}{B}   \bigr]
		-                \bigl[   \tfrac{1-C}{1-B}  ,  \tfrac{A C}{B}   \bigr]
				} \\ & \phantom{{} + \LiL_{3\;1,2}\bigl(}
				{ \scriptstyle 
		-                \bigl[   1  ,  -\tfrac{A (1-B)}{B-A C}   \bigr]
		-                \bigl[   1-A  ,  \tfrac{C}{1-B}   \bigr]
		+                \bigl[   1  ,  -\tfrac{A (1-B)}{B}   \bigr]
		+                \bigl[   1  ,  \tfrac{1-B}{C}   \bigr]
		+                \bigl[   \tfrac{A}{B}  ,  1-C   \bigr]
		-                \bigl[   \tfrac{1}{1-B}  ,  C   \bigr]
	} \bigr)
	\\[1ex]
	& {}
	+
	\LiL_{4\;1,1}\bigl( { 
		\scriptstyle
		-                \bigl[   \tfrac{A (1-B) (1-C)}{\sigma _3}  ,  \tfrac{\sigma _3}{(1-B) B}   \bigr]
		+2                \bigl[   \tfrac{(1-A) B C}{\sigma _3}  ,  \tfrac{\sigma _3}{(1-B) B}   \bigr]
		+                \bigl[   \tfrac{A (1-C)}{\sigma _2}  ,  \tfrac{C \sigma _2}{(1-B) B}   \bigr]
		-2                \bigl[   \tfrac{(1-A) B}{\sigma _2}  ,  \tfrac{C \sigma _2}{(1-B) B}   \bigr]
			} \\ & \phantom{{} + \LiL_{4\;1,1}\bigl(}
			{ \scriptstyle 
		-                \bigl[   -\tfrac{(1-A) B}{A-B}  ,  \tfrac{(A-B) C}{B (1-C)}   \bigr]
		+2                \bigl[   \tfrac{A (1-B)}{A-B}  ,  \tfrac{(A-B) C}{B (1-C)}   \bigr]
		-                \bigl[   -\tfrac{(1-A) B}{A-B}  ,  -\tfrac{A-B}{B-A C}   \bigr]
		+                \bigl[   -\tfrac{1-B}{B-C}  ,  \tfrac{A (B-C)}{B-A C}   \bigr]
			} \\ & \phantom{{} + \LiL_{4\;1,1}\bigl(}
			{ \scriptstyle 
		+2                \bigl[   \tfrac{A (1-B)}{A-B}  ,  -\tfrac{A-B}{B-A C}   \bigr]
		+4                \bigl[   -\tfrac{B-C}{1-B}  ,  -\tfrac{(1-A) C}{B-C}   \bigr]
		-                \bigl[   \tfrac{1-C}{1-B}  ,  \tfrac{(1-A) B}{B-A C}   \bigr]
		+                \bigl[   \tfrac{1-C}{B-C}  ,  \tfrac{A (B-C)}{B-A C}   \bigr]
				} \\ & \phantom{{} + \LiL_{4\;1,1}\bigl(}
				{ \scriptstyle 
		+                \bigl[   \tfrac{B-C}{1-C}  ,  -\tfrac{(1-A) C}{B-C}   \bigr]
		-2                \bigl[   \tfrac{1}{1-A}  ,  \tfrac{(1-B) (1-C)}{1-B-C}   \bigr]
		+2                \bigl[   \tfrac{1-B}{1-C}  ,  \tfrac{(1-A) B}{B-A C}   \bigr]
		-                \bigl[   -\tfrac{A-B}{B}  ,  -\tfrac{(1-A) B}{A-B}   \bigr]
				} \\ & \phantom{{} + \LiL_{4\;1,1}\bigl(}
				{ \scriptstyle 
		+                \bigl[   1  ,  \tfrac{(1-A) B C}{\sigma _3}   \bigr]
		-                \bigl[   -\tfrac{1-B}{B-C}  ,  -\tfrac{B-C}{C}   \bigr]
		-2                \bigl[   \tfrac{A (1-B)}{A-B}  ,  -\tfrac{A-B}{B}   \bigr]
		-                \bigl[   \tfrac{1-C}{B-C}  ,  -\tfrac{B-C}{C}   \bigr]
		-2                \bigl[   \tfrac{(1-A) B}{A (1-B)}  ,  \tfrac{A C}{B}   \bigr]
				} \\ & \phantom{{} + \LiL_{4\;1,1}\bigl(}
				{ \scriptstyle 
		-                \bigl[   1  ,  \tfrac{(1-B) (1-C)}{1-B-C}   \bigr]
		+2                \bigl[   -\tfrac{A-B}{(1-A) B}  ,  1-C   \bigr]
		-                \bigl[   1  ,  \tfrac{(1-A) B}{\sigma _2}   \bigr]
		+2                \bigl[   1-A  ,  \tfrac{1-B-C}{1-C}   \bigr]
		-                \bigl[   \tfrac{1-C}{1-B}  ,  \tfrac{A C}{B}   \bigr]
				} \\ & \phantom{{} + \LiL_{4\;1,1}\bigl(}
				{ \scriptstyle 
		-2                \bigl[   \tfrac{1-B}{1-C}  ,  \tfrac{A C}{B}   \bigr]
		-                \bigl[   A  ,  \tfrac{(1-C) C}{(1-B) B}   \bigr]
		-2                \bigl[   A  ,  \tfrac{(1-B) C}{B (1-C)}   \bigr]
		-3                \bigl[   \tfrac{1-B}{1-C}  ,  \tfrac{B}{C}   \bigr]
		+                \bigl[   1  ,  \tfrac{1-C}{1-B-C}   \bigr]
		+                \bigl[   \tfrac{1}{A}  ,  -\tfrac{B C}{1-B-C}   \bigr]
				} \\ & \phantom{{} + \LiL_{4\;1,1}\bigl(}
				{ \scriptstyle 
		-                \bigl[   \tfrac{A (1-B)}{A-B}  ,  \tfrac{1}{C}   \bigr]
		+                \bigl[   1  ,  \tfrac{(1-A) B}{B-A C}   \bigr]
		-                \bigl[   1  ,  \tfrac{A (1-B)}{A-B}   \bigr]
		-                \bigl[   1  ,  -\tfrac{B C}{1-B-C}   \bigr]
		+                \bigl[   1-A  ,  -\tfrac{C}{1-C}   \bigr]
		-3                \bigl[   1-A  ,  -\tfrac{C}{B-C}   \bigr]
				} \\ & \phantom{{} + \LiL_{4\;1,1}\bigl(}
				{ \scriptstyle 
		+                \bigl[   1  ,  -\tfrac{B}{1-B-C}   \bigr]
		-                \bigl[   A  ,  -\tfrac{1-B-C}{B}   \bigr]
		-2                \bigl[   1  ,  \tfrac{1-B}{1-C}   \bigr]
		+2                \bigl[   \tfrac{1}{1-A}  ,  1-B   \bigr]
		+2                \bigl[   A  ,  -\tfrac{1-B}{B}   \bigr]
		+                \bigl[   A  ,  \tfrac{1-C}{B}   \bigr]
		+3                \bigl[   \tfrac{A}{B}  ,  1-C   \bigr]
				} \\ & \phantom{{} + \LiL_{4\;1,1}\bigl(}
				{ \scriptstyle 
		-                \bigl[   \tfrac{1}{1-B}  ,  C   \bigr]
		-3                \bigl[   \tfrac{1}{B}  ,  1-C   \bigr]
		-                \bigl[   1  ,  1-A   \bigr]
		+                \bigl[   1  ,  1-B   \bigr]
		+                \bigl[   1  ,  1-C   \bigr]
		-2                \bigl[   A  ,  \tfrac{C}{B}   \bigr]
		+2                \bigl[   \tfrac{A}{B}  ,  C   \bigr]
		+                \bigl[   \tfrac{1}{A}  ,  B   \bigr]
		+                \bigl[   1  ,  A   \bigr]
	} \bigr)
	\\[1ex]
	& {}
	+
	\LiL_{5\;1}\bigl( { 
		\scriptstyle
		7                \bigl[   \tfrac{(1-B) (1-C)}{(1-A) (1-B-C)}   \bigr]
		+3                \bigl[   \tfrac{(1-A) B (1-C)}{(1-B) (B-A C)}   \bigr]
		-7                \bigl[   \tfrac{(1-A) (1-B) B}{(1-C) (B-A C)}   \bigr]
		-4                \bigl[   \tfrac{(1-A) B C}{\sigma _3}   \bigr]
		+5                \bigl[   \tfrac{(1-B) B}{C \sigma _2}   \bigr]
				} \\ & \phantom{{} + \LiL_{5\;1}\bigl(}
				{ \scriptstyle 
		-5                \bigl[   \tfrac{(1-B) B}{\sigma _3}   \bigr]
		+7                \bigl[   -\tfrac{(1-A) B}{(A-B) (1-C)}   \bigr]
		+2                \bigl[   \tfrac{(1-B) (1-C)}{1-B-C}   \bigr]
		-3                \bigl[   \tfrac{(1-A) (1-B-C)}{1-C}   \bigr]
		+4                \bigl[   \tfrac{(1-A) B}{\sigma _2}   \bigr]
				} \\ & \phantom{{} + \LiL_{5\;1}\bigl(}
				{ \scriptstyle 
		+2                \bigl[   -\tfrac{A (1-B-C)}{B C}   \bigr]
		+3                \bigl[   \tfrac{A (1-B)}{(A-B) C}   \bigr]
		+3                \bigl[   \tfrac{B (1-C)}{(A-B) C}   \bigr]
		+9                \bigl[   \tfrac{(1-B) B}{(1-C) C}   \bigr]
		-4                \bigl[   \tfrac{1-C}{1-B-C}   \bigr]
		-8                \bigl[   -\tfrac{A (1-B)}{B-A C}   \bigr]
				} \\ & \phantom{{} + \LiL_{5\;1}\bigl(}
				{ \scriptstyle 
		+                \bigl[   \tfrac{A (B-C)}{B-A C}   \bigr]
		-3                \bigl[   -\tfrac{A-B}{B-A C}   \bigr]
		-3                \bigl[   \tfrac{A (1-C)}{B-A C}   \bigr]
		+4                \bigl[   \tfrac{(1-A) B}{B-A C}   \bigr]
		+5                \bigl[   -\tfrac{(1-A) C}{B-C}   \bigr]
		+2                \bigl[   \tfrac{A (1-B)}{A-B}   \bigr]
				} \\ & \phantom{{} + \LiL_{5\;1}\bigl(}
				{ \scriptstyle 
		+2                \bigl[   -\tfrac{B C}{1-B-C}   \bigr]
		+3                \bigl[   -\tfrac{A (1-B-C)}{B}   \bigr]
		+5                \bigl[   -\tfrac{1-B}{B-C}   \bigr]
		-10                \bigl[   \tfrac{(1-A) C}{1-B}   \bigr]
		-                \bigl[   -\tfrac{B}{1-B-C}   \bigr]
		+                \bigl[   \tfrac{1-C}{B-C}   \bigr]
				} \\ & \phantom{{} + \LiL_{5\;1}\bigl(}
				{ \scriptstyle 
		+5                \bigl[   \tfrac{1-B}{1-C}   \bigr]
		+7                \bigl[   \tfrac{1-A}{1-B}   \bigr]
		+2                \bigl[   -\tfrac{1-C}{C}   \bigr]
		+2                \bigl[   -\tfrac{B-C}{C}   \bigr]
		+3                \bigl[   -\tfrac{A-B}{B}   \bigr]
		-5                \bigl[   \tfrac{A (1-C)}{B}   \bigr]
		+5                \bigl[   \tfrac{1-B}{C}   \bigr]
				} \\ & \phantom{{} + \LiL_{5\;1}\bigl(}
				{ \scriptstyle 
		-9                \bigl[   \tfrac{B}{1-C}   \bigr]
		+3                \bigl[   \tfrac{A C}{B}   \bigr]
		-2                \bigl[   1-C   \bigr]
		+3                \bigl[   1-A   \bigr]
		+3                \bigl[   \tfrac{A}{B}   \bigr]
		-3                \bigl[   \tfrac{B}{C}   \bigr]
		-4                \bigl[   1-B   \bigr]
		-                \bigl[   A   \bigr]
		+                \bigl[   B   \bigr]
		+4                \bigl[   C   \bigr]
	} \bigr) \,.
	\end{align*}
	
	\pagebreak\subsection{Full Symmetry 2, \texorpdfstring{\( \lif(A,B,C) + \lif(A, \tfrac{B}{B-1}, \tfrac{1}{1-C}) \)}{Li\_\{3;1,1,1\}(A,B,C) + Li\_\{3;1,1,1\}(1-A, B/(B-1), 1/(1-C))}} 
	\label{app:fullsym2:full} 
	If we write\smallskip
	{
	\small \begin{align*}
		\sigma_4 & = 1 - A B - C + B C \,, \quad 
		\sigma_5 = 1 - C + B C - A B C \,, \\[-0.5ex]
		\pi_1 & = 1 - C + B C^2 - A B C^2 - B^2 C^2 + A B^2 C^2 \,,
	\end{align*}
	}%
	for notational simplicity, then the second full symmetry of \( \lif \) in $\fullsym{2}$ from \autoref{lem:fullsym2}, is given as follows.
	\begin{align*}
	&\LiL_{3\;1,1,1}\bigl( { 
		\scriptstyle
		                \bigl[   A  ,  B  ,  C   \bigr]
		+                \bigl[   A  ,  \tfrac{B}{B-1}  ,  \tfrac{1}{1-C}   \bigr]
	} \bigr) 
	\= \\[2ex]
	&\LiL_{3\;1,1,1}\bigl( { 
		\scriptstyle
		-\tfrac{1}{2}    \bigl[   1  ,  \tfrac{\pi _1}{(1-A) (1-B) B C^2}  ,  \tfrac{A B (1-C) C}{\pi _1}   \bigr]
		+\tfrac{1}{2}    \bigl[   1  ,  \tfrac{1-C}{\pi _1}  ,  \tfrac{\pi _1}{A B (1-C) C}   \bigr]
				} \\ & \phantom{{} \LiL_{3\;1,1,1}\bigl(}
				{ \scriptstyle 
		+\tfrac{1}{2}    \bigl[   1  ,  -\tfrac{1-B C+A B C}{(1-A) B C}  ,  -\tfrac{A B (1-C)}{(1-B) (1-B C+A B C)}   \bigr]
		+\tfrac{1}{2}    \bigl[   1  ,  \tfrac{\sigma _4}{(1-B) (1-C)}  ,  -\tfrac{A B (1-C)}{\sigma _4}   \bigr]
				} \\ & \phantom{{} \LiL_{3\;1,1,1}\bigl(}
				{ \scriptstyle 
		+\tfrac{1}{2}    \bigl[   1  ,  \tfrac{\sigma _5}{1-C}  ,  -\tfrac{1-B}{A B \sigma _5}   \bigr]
		-\tfrac{1}{2}    \bigl[   1  ,  \tfrac{(1-A) B}{\sigma _4}  ,  -\tfrac{\sigma _4}{A B (1-C)}   \bigr]
		+                \bigl[   1  ,  \tfrac{\sigma _5}{1-C}  ,  -\tfrac{A B}{(1-B) \sigma _5}   \bigr]
		+\tfrac{1}{2}    \bigl[   1  ,  \tfrac{(1-A) B C}{\sigma _5}  ,  \tfrac{\sigma _5}{1-A B C}   \bigr]
				} \\ & \phantom{{} \LiL_{3\;1,1,1}\bigl(}
				{ \scriptstyle 
		+\tfrac{1}{2}    \bigl[   1  ,  -\tfrac{1-C+B C}{(1-B) C}  ,  -\tfrac{A (1-C)}{(1-A) (1-C+B C)}   \bigr]
		-\tfrac{1}{2}    \bigl[   1  ,  \tfrac{1}{1-B C+A B C}  ,  -\tfrac{(1-B) (1-B C+A B C)}{A B (1-C)}   \bigr]
		-                \bigl[   1  ,  \tfrac{\sigma _5}{1-C}  ,  -\tfrac{A B C}{\sigma _5}   \bigr]
				} \\ & \phantom{{} \LiL_{3\;1,1,1}\bigl(}
				{ \scriptstyle 
		+\tfrac{1}{2}    \bigl[   \tfrac{1}{1-C+B C}  ,  1  ,  \tfrac{\sigma _5}{(1-A) B C}   \bigr]
		+\tfrac{1}{2}    \bigl[   1  ,  \tfrac{(1-A) B}{1-A B}  ,  -\tfrac{1-A B}{A B (1-C)}   \bigr]
		+\tfrac{1}{2}    \bigl[   1  ,  \tfrac{1-C}{1-C+B C}  ,  -\tfrac{1-C+B C}{(1-B) C}   \bigr]
				} \\ & \phantom{{} \LiL_{3\;1,1,1}\bigl(}
				{ \scriptstyle 
		-\tfrac{3}{4}    \bigl[   1  ,  \tfrac{(1-A) B C}{\sigma _5}  ,  \sigma _5   \bigr]
		-\tfrac{3}{4}    \bigl[   1  ,  \tfrac{\sigma _5}{1-C}  ,  \tfrac{1}{\sigma _5}   \bigr]
		-\tfrac{1}{2}    \bigl[   1  ,  \tfrac{1-B}{1-A B}  ,  -\tfrac{1-A B}{A B (1-C)}   \bigr]
		+\tfrac{1}{2}    \bigl[   1  ,  \tfrac{1}{1-C+B C}  ,  -\tfrac{(1-A) (1-C+B C)}{A (1-C)}   \bigr]
				} \\ & \phantom{{} \LiL_{3\;1,1,1}\bigl(}
				{ \scriptstyle 
		-\tfrac{1}{2}    \bigl[   1  ,  -\tfrac{1-C+B C}{(1-B) C}  ,  \tfrac{A B C}{1-C+B C}   \bigr]
		-\tfrac{1}{2}    \bigl[   \tfrac{1-C}{(1-B C) (1-C+B C)}  ,  1  ,  -\tfrac{A}{1-A}   \bigr]
		-\tfrac{1}{2}    \bigl[   1  ,  \tfrac{1-B C}{1-C}  ,  -\tfrac{1-A}{A (1-B C)}   \bigr]
				} \\ & \phantom{{} \LiL_{3\;1,1,1}\bigl(}
				{ \scriptstyle 
		-\tfrac{1}{2}    \bigl[   1  ,  \tfrac{(1-A B) C}{1-A B C}  ,  \tfrac{(1-A) B}{1-A B}   \bigr]
		+\tfrac{1}{2}    \bigl[   -\tfrac{1-B}{(1-A) B}  ,  1  ,  \tfrac{1-A B C}{1-C}   \bigr]
		-\tfrac{1}{2}    \bigl[   1  ,  \tfrac{(1-B) C}{1-B C}  ,  \tfrac{1-B C}{1-A B C}   \bigr]
		+\tfrac{1}{2}    \bigl[   1  ,  \tfrac{(1-A B) C}{1-A B C}  ,  \tfrac{1-B}{1-A B}   \bigr]
				} \\ & \phantom{{} \LiL_{3\;1,1,1}\bigl(}
				{ \scriptstyle 
		-\tfrac{1}{2}    \bigl[   \tfrac{1}{1-C+B C}  ,  1  ,  \tfrac{\sigma _5}{1-C}   \bigr]
		-\tfrac{1}{2}    \bigl[   -\tfrac{(1-A) B}{1-B}  ,  1  ,  \tfrac{1-A B C}{1-C}   \bigr]
		+\tfrac{1}{2}    \bigl[   1  ,  -\tfrac{1-A B}{A B}  ,  \tfrac{(1-A) B}{1-A B}   \bigr]
		-\tfrac{1}{2}    \bigl[   1  ,  -\tfrac{1-A B}{A B}  ,  \tfrac{1-B}{1-A B}   \bigr]
				} \\ & \phantom{{} \LiL_{3\;1,1,1}\bigl(}
				{ \scriptstyle 
		-\tfrac{1}{2}    \bigl[   1  ,  \tfrac{1}{1-C+B C}  ,  \tfrac{1-C+B C}{A B C}   \bigr]
		-\tfrac{1}{2}    \bigl[   1  ,  -\tfrac{1-C}{C}  ,  -\tfrac{A}{(1-A) (1-B)}   \bigr]
		+\tfrac{1}{2}    \bigl[   1  ,  \tfrac{1-C+B C}{1-C}  ,  \tfrac{1}{1-C+B C}   \bigr]
				} \\ & \phantom{{} \LiL_{3\;1,1,1}\bigl(}
				{ \scriptstyle 
		+\tfrac{1}{2}    \bigl[   \tfrac{1}{1-C+B C}  ,  1  ,  \tfrac{1-C+B C}{1-C}   \bigr]
		-\tfrac{1}{2}    \bigl[   \tfrac{1}{1-C+B C}  ,  1  ,  \tfrac{1-C+B C}{B C}   \bigr]
		+\tfrac{1}{2}    \bigl[   1  ,  \tfrac{1}{A}  ,  -\tfrac{1-B}{B (1-C)}   \bigr]
		-\tfrac{1}{2}    \bigl[   -\tfrac{1-B}{(1-A) B}  ,  1  ,  A B   \bigr]
				} \\ & \phantom{{} \LiL_{3\;1,1,1}\bigl(}
				{ \scriptstyle 
		+\tfrac{1}{2}    \bigl[   -\tfrac{(1-A) B}{1-B}  ,  1  ,  A B   \bigr]
		+\tfrac{1}{2}    \bigl[   \tfrac{1}{1-B C}  ,  1  ,  -\tfrac{A}{1-A}   \bigr]
		-                \bigl[   1  ,  A  ,  -\tfrac{B}{(1-B) (1-C)}   \bigr]
		+                \bigl[   1  ,  1-C  ,  -\tfrac{1-B}{A B}   \bigr]
		-\tfrac{5}{2}    \bigl[   \tfrac{1}{1-C}  ,  1  ,  -\tfrac{1-B}{B}   \bigr]
				} \\ & \phantom{{} \LiL_{3\;1,1,1}\bigl(}
				{ \scriptstyle 
		+\tfrac{1}{2}    \bigl[   \tfrac{B}{1-C+B C}  ,  1  ,  A   \bigr]
		+\tfrac{7}{4}    \bigl[   \tfrac{1}{1-B+A B}  ,  1  ,  \tfrac{1}{C}   \bigr]
		+                \bigl[   1  ,  A  ,  -\tfrac{B C}{1-C}   \bigr]
		-\tfrac{1}{2}    \bigl[   1  ,  \tfrac{1}{A}  ,  \tfrac{1}{B C}   \bigr]
		+\tfrac{3}{4}    \bigl[   1  ,  C  ,  (1-A) B   \bigr]
		+\tfrac{7}{4}    \bigl[   C  ,  1  ,  B   \bigr]
	} \bigr)
	\\[1ex]
	& {}
	+
	\LiL_{3\;1,2}\bigl( { 
		\scriptstyle
		-\tfrac{1}{2}    \bigl[   \tfrac{A B (1-C) C}{\pi _1}  ,  \tfrac{\pi _1}{1-C}   \bigr]
		+\tfrac{1}{2}    \bigl[   -\tfrac{A B (1-C)}{\sigma _4}  ,  \tfrac{\sigma _4}{(1-A) B}   \bigr]
		-\tfrac{1}{2}    \bigl[   \tfrac{(1-A) B C}{\sigma _5}  ,  \tfrac{\sigma _5}{1-A B C}   \bigr]
				} \\ & \phantom{{} + \LiL_{3\;1,2}\bigl(}
				{ \scriptstyle 
		+\tfrac{1}{2}    \bigl[   -\tfrac{A B (1-C)}{(1-B) (1-B C+A B C)}  ,  1-B C+A B C   \bigr]
		+\tfrac{3}{4}    \bigl[   \tfrac{(1-A) B C}{\sigma _5}  ,  \sigma _5   \bigr]
		-\tfrac{1}{2}    \bigl[   \tfrac{1-A B}{(1-A) B}  ,  -\tfrac{A B (1-C)}{1-A B}   \bigr]
				} \\ & \phantom{{} + \LiL_{3\;1,2}\bigl(}
				{ \scriptstyle 
		-\tfrac{1}{2}    \bigl[   -\tfrac{(1-B) C}{1-C+B C}  ,  \tfrac{1-C+B C}{1-C}   \bigr]
		-\tfrac{1}{2}    \bigl[   -\tfrac{A (1-C)}{(1-A) (1-C+B C)}  ,  1-C+B C   \bigr]
		+\tfrac{1}{2}    \bigl[   \tfrac{1-A B}{1-B}  ,  -\tfrac{A B (1-C)}{1-A B}   \bigr]
				} \\ & \phantom{{} + \LiL_{3\;1,2}\bigl(}
				{ \scriptstyle 
		+\tfrac{1}{2}    \bigl[   \tfrac{(1-B) C}{1-B C}  ,  \tfrac{1-B C}{1-A B C}   \bigr]
		-\tfrac{1}{2}    \bigl[   -\tfrac{1-B}{A B}  ,  -\tfrac{(1-A) B C}{1-C}   \bigr]
		-\tfrac{1}{2}    \bigl[   -\tfrac{1-A}{A}  ,  -\tfrac{(1-B) C}{1-C}   \bigr]
		+\tfrac{1}{2}    \bigl[   -\tfrac{(1-A) (1-B)}{A}  ,  -\tfrac{C}{1-C}   \bigr]
				} \\ & \phantom{{} + \LiL_{3\;1,2}\bigl(}
				{ \scriptstyle 
		+\tfrac{1}{2}    \bigl[   -\tfrac{1-C}{(1-A) B C}  ,  A B C   \bigr]
		+\tfrac{1}{2}    \bigl[   \tfrac{A B C}{1-C+B C}  ,  1-C+B C   \bigr]
		-\tfrac{1}{2}    \bigl[   -\tfrac{1-C}{(1-B) C}  ,  A B C   \bigr]
		+\tfrac{1}{2}    \bigl[   1  ,  \tfrac{(1-A) B C}{1-A B C}   \bigr]
				} \\ & \phantom{{} + \LiL_{3\;1,2}\bigl(}
				{ \scriptstyle 
		-\tfrac{1}{2}    \bigl[   1  ,  -\tfrac{A B (1-C)}{1-B}   \bigr]
		-\tfrac{1}{2}    \bigl[   1  ,  \tfrac{(1-B) C}{1-A B C}   \bigr]
		+\tfrac{1}{2}    \bigl[   1  ,  -\tfrac{(1-B) C}{1-C}   \bigr]
		-\tfrac{1}{2}    \bigl[   -\tfrac{A B}{1-B}  ,  1-C   \bigr]
		+\tfrac{1}{2}    \bigl[   -\tfrac{A}{1-A}  ,  1-C   \bigr]
				} \\ & \phantom{{} + \LiL_{3\;1,2}\bigl(}
				{ \scriptstyle 
		+                \bigl[   1  ,  -\tfrac{A B C}{1-C}   \bigr]
		-                \bigl[   A  ,  -\tfrac{B C}{1-C}   \bigr]
		-\tfrac{3}{4}    \bigl[   1  ,  (1-A) B C   \bigr]
		-                \bigl[   A  ,  B C   \bigr]
		+                \bigl[   A B  ,  C   \bigr]
	} \bigr)
	\\[1ex]
	& {}
	+
	\LiL_{4\;1,1}\bigl( { 
		\scriptstyle
		-                \bigl[   \tfrac{A B (1-C) C}{\pi _1}  ,  \tfrac{\pi _1}{(1-A) (1-B) B C^2}   \bigr]
		-\tfrac{3}{2}    \bigl[   \tfrac{A B (1-C) C}{\pi _1}  ,  \tfrac{\pi _1}{1-C}   \bigr]
		+\tfrac{1}{2}    \bigl[   1  ,  \tfrac{(1-A) (1-B) B C^2}{\pi _1}   \bigr]
				} \\ & \phantom{{} + \LiL_{4\;1,1}\bigl(}
				{ \scriptstyle 
		-                \bigl[   -\tfrac{(1-A) B C}{1-B C+A B C}  ,  -\tfrac{(1-B) (1-B C+A B C)}{A B (1-C)}   \bigr]
		+                \bigl[   -\tfrac{A B (1-C)}{\sigma _4}  ,  \tfrac{\sigma _4}{(1-B) (1-C)}   \bigr]
		-                \bigl[   \tfrac{(1-A) B C}{\sigma _5}  ,  -\tfrac{(1-B) \sigma _5}{A B}   \bigr]
				} \\ & \phantom{{} + \LiL_{4\;1,1}\bigl(}
				{ \scriptstyle 
		-\tfrac{1}{2}    \bigl[   \tfrac{(1-A) B C}{\sigma _5}  ,  -\tfrac{A B \sigma _5}{1-B}   \bigr]
		+                \bigl[   -\tfrac{A B}{(1-B) \sigma _5}  ,  \tfrac{\sigma _5}{1-C}   \bigr]
		+                \bigl[   \tfrac{(1-A) B C}{\sigma _5}  ,  -\tfrac{\sigma _5}{A B C}   \bigr]
		-                \bigl[   \tfrac{(1-A) B C}{\sigma _5}  ,  \tfrac{\sigma _5}{1-A B C}   \bigr]
				} \\ & \phantom{{} + \LiL_{4\;1,1}\bigl(}
				{ \scriptstyle 
		+\tfrac{3}{2}    \bigl[   -\tfrac{A B (1-C)}{\sigma _4}  ,  \tfrac{\sigma _4}{(1-A) B}   \bigr]
		-\tfrac{1}{2}    \bigl[   \tfrac{1-C}{\sigma _5}  ,  -\tfrac{A B \sigma _5}{1-B}   \bigr]
		+\tfrac{1}{2}    \bigl[   \tfrac{1-C}{\sigma _5}  ,  \tfrac{\sigma _5}{1-A B C}   \bigr]
		-\tfrac{1}{2}    \bigl[   -\tfrac{1-C+B C}{(1-B) C}  ,  \tfrac{(1-A) B C}{\sigma _5}   \bigr]
				} \\ & \phantom{{} + \LiL_{4\;1,1}\bigl(}
				{ \scriptstyle 
		-                \bigl[   -\tfrac{A B C}{\sigma _5}  ,  \tfrac{\sigma _5}{1-C}   \bigr]
		+\tfrac{1}{2}    \bigl[   -\tfrac{(1-B) C}{1-C+B C}  ,  \tfrac{\sigma _5}{1-C}   \bigr]
		+\tfrac{3}{2}    \bigl[   -\tfrac{A B (1-C)}{(1-B) (1-B C+A B C)}  ,  1-B C+A B C   \bigr]
		+\tfrac{9}{4}    \bigl[   \tfrac{(1-A) B C}{\sigma _5}  ,  \sigma _5   \bigr]
				} \\ & \phantom{{} + \LiL_{4\;1,1}\bigl(}
				{ \scriptstyle 
		-\tfrac{3}{2}    \bigl[   \tfrac{1-A B}{(1-A) B}  ,  -\tfrac{A B (1-C)}{1-A B}   \bigr]
		-\tfrac{3}{2}    \bigl[   -\tfrac{(1-B) C}{1-C+B C}  ,  \tfrac{1-C+B C}{1-C}   \bigr]
		-\tfrac{1}{2}    \bigl[   -\tfrac{A (1-C)}{(1-A) (1-C+B C)}  ,  1-C+B C   \bigr]
				} \\ & \phantom{{} + \LiL_{4\;1,1}\bigl(}
				{ \scriptstyle 
		-                \bigl[   -\tfrac{1-A}{A}  ,  \tfrac{(1-B C) (1-C+B C)}{1-C}   \bigr]
		-\tfrac{1}{2}    \bigl[   -\tfrac{A}{1-A}  ,  \tfrac{(1-B) (1-C)}{1-C+B C}   \bigr]
		+\tfrac{1}{2}    \bigl[   \tfrac{(1-B) C}{1-B C}  ,  -\tfrac{A (1-B C)}{1-A}   \bigr]
				} \\ & \phantom{{} + \LiL_{4\;1,1}\bigl(}
				{ \scriptstyle 
		+                \bigl[   \tfrac{(1-B) C}{1-B C}  ,  \tfrac{1-B C}{1-A B C}   \bigr]
		-                \bigl[   1-C+B C  ,  \tfrac{(1-A) B C}{\sigma _5}   \bigr]
		+\tfrac{3}{2}    \bigl[   \tfrac{1-A B}{1-B}  ,  -\tfrac{A B (1-C)}{1-A B}   \bigr]
		+\tfrac{1}{2}    \bigl[   \tfrac{1-C}{1-B C}  ,  -\tfrac{A (1-B C)}{1-A}   \bigr]
				} \\ & \phantom{{} + \LiL_{4\;1,1}\bigl(}
				{ \scriptstyle 
		-\tfrac{1}{2}    \bigl[   \tfrac{1-C}{1-B C}  ,  \tfrac{1-B C}{1-A B C}   \bigr]
		-\tfrac{3}{2}    \bigl[   -\tfrac{(1-A) B}{1-B}  ,  \tfrac{1-C}{1-A B C}   \bigr]
		+\tfrac{3}{2}    \bigl[   -\tfrac{(1-A) B}{1-B}  ,  \tfrac{1-A B C}{1-C}   \bigr]
		-\tfrac{1}{2}    \bigl[   1  ,  \tfrac{(1-B) B C^2}{(1-B C) (1-C+B C)}   \bigr]
				} \\ & \phantom{{} + \LiL_{4\;1,1}\bigl(}
				{ \scriptstyle 
		+\tfrac{1}{2}    \bigl[   A  ,  \tfrac{(1-B) B C^2}{(1-B C) (1-C+B C)}   \bigr]
		-                \bigl[   -\tfrac{1-B}{A B}  ,  -\tfrac{(1-A) B C}{1-C}   \bigr]
		+                \bigl[   1-C+B C  ,  \tfrac{1-C}{\sigma _5}   \bigr]
		-\tfrac{1}{2}    \bigl[   1  ,  \tfrac{(1-B) (1-C)}{\sigma _4}   \bigr]
				} \\ & \phantom{{} + \LiL_{4\;1,1}\bigl(}
				{ \scriptstyle 
		+\tfrac{7}{4}    \bigl[   -\tfrac{(1-A) B}{1-B+A B}  ,  \tfrac{1}{1-C}   \bigr]
		+\tfrac{1}{2}    \bigl[   1  ,  \tfrac{(1-A) B C}{\sigma _5}   \bigr]
		+\tfrac{1}{2}    \bigl[   \tfrac{1-C}{1-C+B C}  ,  1-C+B C   \bigr]
		+                \bigl[   -\tfrac{1-C}{(1-A) B C}  ,  A B C   \bigr]
				} \\ & \phantom{{} + \LiL_{4\;1,1}\bigl(}
				{ \scriptstyle 
		-\tfrac{1}{2}    \bigl[   1  ,  \tfrac{(1-B C) (1-C+B C)}{1-C}   \bigr]
		-\tfrac{1}{4}    \bigl[   1  ,  \tfrac{1-C}{\sigma _5}   \bigr]
		+\tfrac{1}{2}    \bigl[   1  ,  \tfrac{(1-B) (1-C)}{1-C+B C}   \bigr]
		-\tfrac{1}{2}    \bigl[   1  ,  -\tfrac{(1-A) B C}{1-B C+A B C}   \bigr]
				} \\ & \phantom{{} + \LiL_{4\;1,1}\bigl(}
				{ \scriptstyle 
		+\tfrac{1}{2}    \bigl[   \tfrac{A B C}{1-C+B C}  ,  1-C+B C   \bigr]
		-\tfrac{1}{2}    \bigl[   -\tfrac{1-B}{(1-A) B}  ,  A B C   \bigr]
		-                \bigl[   -\tfrac{1-C}{(1-B) C}  ,  A B C   \bigr]
		-\tfrac{3}{2}    \bigl[   \tfrac{B C}{1-C+B C}  ,  1-C+B C   \bigr]
				} \\ & \phantom{{} + \LiL_{4\;1,1}\bigl(}
				{ \scriptstyle 
		+\tfrac{3}{2}    \bigl[   -\tfrac{1-B}{(1-A) B}  ,  A B   \bigr]
		+\tfrac{1}{2}    \bigl[   -\tfrac{(1-A) B}{1-B}  ,  A B C   \bigr]
		-\tfrac{1}{2}    \bigl[   -\tfrac{(1-A) A B^2}{1-B}  ,  C   \bigr]
		+\tfrac{7}{4}    \bigl[   1  ,  -\tfrac{(1-A) B}{1-B+A B}   \bigr]
				} \\ & \phantom{{} + \LiL_{4\;1,1}\bigl(}
				{ \scriptstyle 
		+                \bigl[   \tfrac{1}{A}  ,  -\tfrac{(1-B) (1-C)}{B}   \bigr]
		+\tfrac{3}{2}    \bigl[   A B  ,  -\tfrac{(1-A) B}{1-B}   \bigr]
		+2                \bigl[   -\tfrac{1-B}{A B}  ,  1-C   \bigr]
		+                \bigl[   -\tfrac{1-A}{A}  ,  (1-B) C   \bigr]
		+\tfrac{1}{2}    \bigl[   A  ,  -\tfrac{B (1-C)}{1-B}   \bigr]
				} \\ & \phantom{{} + \LiL_{4\;1,1}\bigl(}
				{ \scriptstyle 
		+\tfrac{1}{2}    \bigl[   -\tfrac{A (1-B)}{1-A}  ,  C   \bigr]
		+                \bigl[   -\tfrac{1-A}{A}  ,  1-B C   \bigr]
		-                \bigl[   -\tfrac{(1-A) (1-B)}{A}  ,  C   \bigr]
		-2                \bigl[   -\tfrac{A B}{1-B}  ,  1-C   \bigr]
		+\tfrac{1}{2}    \bigl[   1  ,  \tfrac{1-C}{1-B C}   \bigr]
				} \\ & \phantom{{} + \LiL_{4\;1,1}\bigl(}
				{ \scriptstyle 
		-                \bigl[   -\tfrac{1-A}{A}  ,  1-B   \bigr]
		+                \bigl[   -\tfrac{A}{1-A}  ,  1-C   \bigr]
		-                \bigl[   -\tfrac{1-B}{B}  ,  1-C   \bigr]
		+\tfrac{21}{4}    \bigl[   1-B  ,  -\tfrac{C}{1-C}   \bigr]
		-\tfrac{27}{4}    \bigl[   -\tfrac{B}{1-B}  ,  1-C   \bigr]
				} \\ & \phantom{{} + \LiL_{4\;1,1}\bigl(}
				{ \scriptstyle 
		-\tfrac{1}{2}    \bigl[   1  ,  \tfrac{B C}{1-C+B C}   \bigr]
		+\tfrac{1}{2}    \bigl[   1  ,  -\tfrac{B C}{1-B C}   \bigr]
		+\tfrac{1}{2}    \bigl[   1  ,  \tfrac{B}{1-C+B C}   \bigr]
		-\tfrac{1}{2}    \bigl[   A  ,  -\tfrac{B C}{1-B C}   \bigr]
		+                \bigl[   A  ,  \tfrac{B}{1-C+B C}   \bigr]
				} \\ & \phantom{{} + \LiL_{4\;1,1}\bigl(}
				{ \scriptstyle 
		-                \bigl[   A  ,  -\tfrac{B C}{1-C}   \bigr]
		-\tfrac{3}{2}    \bigl[   1  ,  -\tfrac{1-A}{A}   \bigr]
		-\tfrac{27}{4}    \bigl[   1  ,  -\tfrac{1-B}{B}   \bigr]
		+\tfrac{7}{4}    \bigl[   1  ,  1-B+A B   \bigr]
		-\tfrac{7}{2}    \bigl[   1-B+A B  ,  C   \bigr]
		+\tfrac{11}{4}    \bigl[   (1-A) B  ,  C   \bigr]
				} \\ & \phantom{{} + \LiL_{4\;1,1}\bigl(}
				{ \scriptstyle 
		+\tfrac{1}{2}    \bigl[   1  ,  1-B C   \bigr]
		+                \bigl[   1-A  ,  B   \bigr]
		+                \bigl[   1-B  ,  C   \bigr]
		-\tfrac{3}{2}    \bigl[   1  ,  1-B   \bigr]
		-\tfrac{1}{2}    \bigl[   A  ,  B C   \bigr]
		+4                \bigl[   A B  ,  C   \bigr]
		+\tfrac{7}{2}    \bigl[   1  ,  C   \bigr]
		-5                \bigl[   1  ,  B   \bigr]
	} \bigr)
	\\[1ex]
	& {}
	+
	\LiL_{5\;1}\bigl( { 
		\scriptstyle
		-                \bigl[   \tfrac{(1-A) (1-B) B C^2}{\pi _1}   \bigr]
		+\tfrac{1}{2}    \bigl[   \tfrac{A B (1-C) C}{\pi _1}   \bigr]
		-                \bigl[   \tfrac{1-C}{\pi _1}   \bigr]
		+                \bigl[   -\tfrac{(1-A) B (1-C+B C)}{(1-B) \sigma _5}   \bigr]
		-\tfrac{3}{2}    \bigl[   -\tfrac{(1-B) C \sigma _5}{(1-C) (1-C+B C)}   \bigr]
				} \\ & \phantom{{} + \LiL_{5\;1}\bigl(}
				{ \scriptstyle 
		+\tfrac{3}{2}    \bigl[   -\tfrac{(1-A) (1-B C) (1-C+B C)}{A (1-C)}   \bigr]
		+\tfrac{7}{2}    \bigl[   \tfrac{(1-A) B C (1-C+B C)}{\sigma _5}   \bigr]
		+                \bigl[   -\tfrac{A (1-B) (1-C)}{(1-A) (1-C+B C)}   \bigr]
				} \\ & \phantom{{} + \LiL_{5\;1}\bigl(}
				{ \scriptstyle 
		-                \bigl[   \tfrac{A (1-B) B C^2}{(1-B C) (1-C+B C)}   \bigr]
		-\tfrac{5}{2}    \bigl[   -\tfrac{(1-A) B (1-A B C)}{(1-B) (1-C)}   \bigr]
		-\tfrac{7}{2}    \bigl[   \tfrac{(1-C) (1-C+B C)}{\sigma _5}   \bigr]
		+5                \bigl[   -\tfrac{(1-A) B (1-C)}{(1-B) (1-A B C)}   \bigr]
				} \\ & \phantom{{} + \LiL_{5\;1}\bigl(}
				{ \scriptstyle 
		+\tfrac{1}{2}    \bigl[   -\tfrac{A B (1-C)}{(1-B) (1-B C+A B C)}   \bigr]
		+                \bigl[   \tfrac{(1-B) B C^2}{(1-B C) (1-C+B C)}   \bigr]
		+                \bigl[   -\tfrac{A B}{(1-B) \sigma _5}   \bigr]
		+\tfrac{1}{2}    \bigl[   -\tfrac{A (1-C)}{(1-A) (1-C+B C)}   \bigr]
				} \\ & \phantom{{} + \LiL_{5\;1}\bigl(}
				{ \scriptstyle 
		+                \bigl[   \tfrac{(1-B) (1-C)}{\sigma _4}   \bigr]
		-\tfrac{21}{4}    \bigl[   -\tfrac{(1-A) B}{(1-B+A B) (1-C)}   \bigr]
		-                \bigl[   -\tfrac{A B \sigma _5}{1-B}   \bigr]
		-\tfrac{1}{2}    \bigl[   -\tfrac{A B (1-C)}{\sigma _4}   \bigr]
		-\tfrac{3}{2}    \bigl[   \tfrac{1-A B C}{\sigma _5}   \bigr]
				} \\ & \phantom{{} + \LiL_{5\;1}\bigl(}
				{ \scriptstyle 
		-2                \bigl[   \tfrac{(1-A) B C}{\sigma _5}   \bigr]
		+\tfrac{1}{2}    \bigl[   \tfrac{1-C}{\sigma _5}   \bigr]
		-                \bigl[   \tfrac{(1-B) (1-C)}{1-C+B C}   \bigr]
		+                \bigl[   \tfrac{(1-A) B}{\sigma _4}   \bigr]
		+2                \bigl[   -\tfrac{(1-A) B C}{1-B C+A B C}   \bigr]
		-                \bigl[   -\tfrac{A B C}{\sigma _5}   \bigr]
				} \\ & \phantom{{} + \LiL_{5\;1}\bigl(}
				{ \scriptstyle 
		+\tfrac{5}{2}    \bigl[   -\tfrac{(1-A) A B^2 C}{1-B}   \bigr]
		-\tfrac{7}{2}    \bigl[   -\tfrac{(1-A) B}{1-B+A B}   \bigr]
		-\tfrac{9}{2}    \bigl[   -\tfrac{(1-A) A B^2}{1-B}   \bigr]
		+10                \bigl[   -\tfrac{A B}{(1-B) (1-C)}   \bigr]
		+\tfrac{5}{2}    \bigl[   \tfrac{(1-A) B C}{1-A B C}   \bigr]
				} \\ & \phantom{{} + \LiL_{5\;1}\bigl(}
				{ \scriptstyle 
		+                \bigl[   -\tfrac{A (1-B C)}{1-A}   \bigr]
		-                \bigl[   \tfrac{1-C}{1-C+B C}   \bigr]
		-\tfrac{3}{2}    \bigl[   -\tfrac{(1-A) (1-B C)}{A}   \bigr]
		-\tfrac{3}{2}    \bigl[   \tfrac{1-B C}{1-A B C}   \bigr]
		-\tfrac{5}{2}    \bigl[   -\tfrac{A B (1-C)}{1-B}   \bigr]
		-\tfrac{5}{2}    \bigl[   -\tfrac{A (1-B) C}{1-A}   \bigr]
				} \\ & \phantom{{} + \LiL_{5\;1}\bigl(}
				{ \scriptstyle 
		-\tfrac{5}{2}    \bigl[   -\tfrac{(1-A) B C}{1-C}   \bigr]
		-\tfrac{5}{2}    \bigl[   \tfrac{(1-B) C}{1-A B C}   \bigr]
		+5                \bigl[   -\tfrac{(1-A) (1-B) C}{A}   \bigr]
		+\tfrac{1}{2}    \bigl[   -\tfrac{(1-A) B}{1-B}   \bigr]
		-\tfrac{3}{2}    \bigl[   \tfrac{1-C}{1-A B C}   \bigr]
		-2                \bigl[   \tfrac{(1-A) B}{1-A B}   \bigr]
				} \\ & \phantom{{} + \LiL_{5\;1}\bigl(}
				{ \scriptstyle 
		+\tfrac{5}{2}    \bigl[   -\tfrac{(1-A) (1-B)}{A}   \bigr]
		-3                \bigl[   -\tfrac{A (1-B)}{1-A}   \bigr]
		+4                \bigl[   -\tfrac{(1-B) (1-C)}{B}   \bigr]
		-5                \bigl[   -\tfrac{A (1-C)}{1-A}   \bigr]
		-\tfrac{35}{4}    \bigl[   -\tfrac{(1-B) C}{1-C}   \bigr]
		+\tfrac{23}{2}    \bigl[   -\tfrac{B (1-C)}{1-B}   \bigr]
				} \\ & \phantom{{} + \LiL_{5\;1}\bigl(}
				{ \scriptstyle 
		-\tfrac{1}{2}    \bigl[   \tfrac{A B C}{1-C+B C}   \bigr]
		+2                \bigl[   \tfrac{1-B}{1-A B}   \bigr]
		-2                \bigl[   \tfrac{1-C}{1-B C}   \bigr]
		-\tfrac{3}{4}    \bigl[   \sigma _5   \bigr]
		+                \bigl[   -\tfrac{A B C}{1-B C}   \bigr]
		+2                \bigl[   \tfrac{B C}{1-C+B C}   \bigr]
		-\tfrac{7}{2}    \bigl[   \tfrac{A B}{1-C+B C}   \bigr]
		-                \bigl[   -\tfrac{B C}{1-B C}   \bigr]
				} \\ & \phantom{{} + \LiL_{5\;1}\bigl(}
				{ \scriptstyle 
		-                \bigl[   \tfrac{B}{1-C+B C}   \bigr]
		-3                \bigl[   -\tfrac{A B}{1-B}   \bigr]
		+\tfrac{7}{2}    \bigl[   -\tfrac{B C}{1-C}   \bigr]
		-\tfrac{1}{4}    \bigl[   -\tfrac{1-C}{C}   \bigr]
		-                \bigl[   1-B C+A B C   \bigr]
		-\tfrac{7}{2}    \bigl[   -\tfrac{1-A}{A}   \bigr]
		+\tfrac{31}{4}    \bigl[   -\tfrac{1-B}{B}   \bigr]
				} \\ & \phantom{{} + \LiL_{5\;1}\bigl(}
				{ \scriptstyle 
		+\tfrac{49}{4}    \bigl[   (1-B+A B) C   \bigr]
		-                \bigl[   1-C+B C   \bigr]
		-16                \bigl[   (1-A) B C   \bigr]
		-2                \bigl[   (1-B) C   \bigr]
		-\tfrac{7}{2}    \bigl[   (1-A) B   \bigr]
		-\tfrac{27}{4}    \bigl[   1-C   \bigr]
				} \\ & \phantom{{} + \LiL_{5\;1}\bigl(}
				{ \scriptstyle 
		-\tfrac{5}{2}    \bigl[   A B C   \bigr]
		-\tfrac{1}{4}    \bigl[   B C   \bigr]
		+                \bigl[   A   \bigr]
		-\tfrac{41}{4}    \bigl[   C   \bigr]
		+\tfrac{27}{2}    \bigl[   B   \bigr]
	} \bigr) \,.
	\end{align*}

	\subsection{Summary of further results}  At this point it is essentially impractical to give any further identities explicitly, even when writing them using degenerate functions \( \lif(1, x, y) \) and \( \lif(x, 1, y) \).  In \autoref{tbl:symmary} below, we summarise the complexity of all of the identities, and break this down into the number of depth \( {\leq}2 \) terms of each type, before, and after, eliminating degenerate \( \lif(\ldots, 1, \ldots)\) terms via the reductions from \autoref{thm:onexy_dp2}.  The entries in the table have been calculated from the identities given in the ancillary files \wtsixidentities{} and \wtsixdepthtwo, respectively.
	\clearpage
	
\begin{table}[ht]
			\centering
		 \begin{adjustbox}{addcode={\begin{minipage}{\width}}{%
		 			\bigskip\caption{Number of terms in the symmetries, identities and results found in \autoref{sec:higherZagier6}}
		 			\end{minipage}}, rotate=90, center}
				\begin{tabular}{rl||cccc|c||ccc|c}
		 & & Degenerate& & & \multicolumn{1}{c}{} & & \multicolumn{4}{c}{After \( \scriptstyle \LiL_{3\;1,1,1}(\ldots,1,\ldots) \)}  \\
		 \multicolumn{2}{c||}{$\lif(A,B,C)$ identity} & \( \scriptstyle \LiL_{3\!\;\!1\!,\!1\!,\!1}(-, 1, -) \) \& & & & \multicolumn{1}{c}{} & & \multicolumn{4}{c}{reductions} \\
		& &  \( \scriptstyle \LiL_{3\!\;\!1\!,\!1\!,\!1}(-, -, 1)\)'s & \( \LiL_{3\;1,2} \) & \( \LiL_{4\;1,1} \) & \( \LiL_{5\;1} \) & Total &  \( \LiL_{3\;1,2} \) & \( \LiL_{4\;1,1} \) & \( \LiL_{5\;1} \) & Total \\ \hline
		\autoref{lem:fullsym1}\,,&\!\!\!\!Full Symmetry 1  & 10 & 4 & 19 & 19 & 54 & 112 & 197 & 95 & 406 \\
		\autoref{prop:fourterm}\,,&\!\!\!\!Four-term & 28 & 15 & 54 & 45 & 146 & 329 & 583 & 239 & 1155 \\
		\autoref{lem:fullsym2}\,,&\!\!\!\!Full Symmetry 2 & 47 & 27 & 86 & 76 & 238 & 546 & 954 & 348 & 1850 \\
		\autoref{prop:216}\,,&\!\!\!\!$\textsc{Map}_1$: \( B / (B-1) \) result &  137 & 79 & 224 & 170 & 612 & 1532 & 2646 & 908 & 5088 \\
		\autoref{prop:216}\,,&\!\!\!\!$\textsc{Map}_2$: \( 1-B \) result & 137 & 79 & 225 & 169 & 612 & 1531 & 2634 & 908 & 5075 \\
		\autoref{prop:216}\,,&\!\!\!\!(Shorter) \( B^{-1} \) result & 152 & 85 & 240 & 183 & 662 & 1648 & 2822 & 952 & 5424 \\
		\autoref{prop:216}\,,&\!\!\!\!$\textsc{Map}_3$: \( (1-C)^{-1} \) result & 94 & 53 & 156 & 122 & 427 & 1046 & 1815 & 634 & 3497 \\	\autoref{rem:fullsym3}\,,&\!\!\!\!Full Symmetry 3 & 1973 & 1034 & 2590 & 1622 & 7221 & 18939 & 32840 & 9910 & 61691 \\
		\autoref{thm:i411sixfold}\,,&\!\!\!\!Zagier formula \( 1-C \) & 1880 & 1019 & 2519 & 1693 & 7113 & 18553 & 32112 & 9891 & 60558 \\
		\autoref{thm:i411sixfold}\,,&\!\!\!\!Zagier formula \( C^{-1} \) & 1887 & 1021 & 2517 & 1673 & 7100 & 18604 & 32185 & 9945 & 60736 
	\end{tabular}
	\end{adjustbox}
	\label{tbl:symmary}
	\end{table}%

\end{document}